\documentclass[french,11pt,letterpaper]{amsbook}
\usepackage{amsmath}
\usepackage{times}
\usepackage[french,english]{babel}

\usepackage{array,multirow,makecell}
\setcellgapes{1pt}
\makegapedcells

\usepackage{amsthm}
\usepackage{amssymb}

\usepackage{mathdots}

\usepackage{amscd}

\usepackage{color}
\usepackage{verbatim}
\usepackage{stmaryrd}

\usepackage[all,cmtip]{xy} 
\usepackage{tikz-cd}

\usepackage{aliascnt}

\usepackage{amsmath}
\usepackage{amsfonts}
\usepackage{amssymb}
\usepackage{graphicx}

\newcommand{\resettheoremcounters}{%
  \setcounter{theorem}{0}
}

\usepackage[T1]{fontenc}
\usepackage[utf8]{inputenc}
\makeatletter
\newenvironment{abstracts}{%
  \ifx\maketitle\relax
    \ClassWarning{\@classname}{Abstract should precede
      \protect\maketitle\space in AMS document classes; reported}%
  \fi
  \global\setbox\abstractbox=\vtop \bgroup
    \normalfont\Small
    \list{}{\labelwidth\z@
      \leftmargin3pc \rightmargin\leftmargin
      \listparindent\normalparindent \itemindent\z@
      \parsep\z@ \@plus\p@
      
      \itemsep\medskipamount
    }%
}{%
  \endlist\egroup
  \ifx\@setabstract\relax \@setabstracta \fi
}

\newcommand{\abstractin}[1]{%
  \otherlanguage{#1}%
  \item[\hskip\labelsep\scshape\abstractname.]%
}
\makeatother

\usepackage[plainpages=false,colorlinks,hyperindex,pdfpagemode=None,bookmarksopen,linkcolor=red,citecolor=blue,urlcolor=blue]{hyperref}
\usepackage{pdflscape}

\usepackage[OT2,T1]{fontenc}
\DeclareSymbolFont{cyrletters}{OT2}{wncyr}{m}{n}
\DeclareFontFamily{OT1}{rsfs}{}

\newcommand{\Hyp}{\Upsilon}

\newcommand{\codim}{\mathrm{codim}}

\newcommand{\SL}{\mathrm{SL}}

\renewcommand{\dim}{\mathrm{dim}}
\newcommand{\tors}{\mathrm{tors}}
\newcommand{\Q}{\mathbf{Q}}

\newcommand{\C}{\mathbf{C}}

\newcommand{\R}{\mathbf{R}}
\newcommand{\G}{\mathbf{G}}

\newcommand{\SU}{\mathrm{SU}}

 \DeclareFontShape{OT1}{rsfs}{n}{it}{<-> rsfs10}{}
\DeclareMathAlphabet{\mathscr}{OT1}{rsfs}{n}{it}
\newcommand{\GL}{\mathrm{GL}}
\newcommand{\vol}{\mathrm{vol}}

\newtheorem{theorem}{Th\'eor\`eme}

\usepackage{verbatim}
\usepackage{graphicx}
\usepackage{nomencl}
\DeclareFontEncoding{OT2}{}{}

\newcommand{\Z}{\mathbf{Z}}

 \newcommand{\A}{\mathbf{A}}

\DeclareFontShape{OT1}{rsfs}{n}{it}{<-> rsfs10}{}
\DeclareMathAlphabet{\mathscr}{OT1}{rsfs}{n}{it}
\newcommand{\UU}{\mathrm{U}}
\newtheorem*{theorem*}{Théorème}
\newtheorem*{cor*}{Corollaire}

\newtheorem{proposition}[theorem]{Proposition}
\newtheorem{cor}[theorem]{Corollaire}
\newtheorem{lemma}[theorem]{Lemme}
\newtheorem{lem}[theorem]{Lemme}
\theoremstyle{definition}
\newtheorem{example}[theorem]{Exemple}
\newtheorem{definition}[theorem]{Définition}
\newtheorem{remark}[theorem]{Remarque}

\definecolor{darkgreen}{RGB}{0,150,0}

\newcommand{\SO}{\mathrm{SO}}

\numberwithin{theorem}{chapter}

\begin{document}
\title[Cocycles de groupe pour $\GL_n$ et arrangements d'hyperplans]{Cocycles de groupe pour $\GL_n$ et arrangements d'hyperplans}
\author[Bergeron, Charollois, Garcia]{Nicolas Bergeron, Pierre Charollois, Luis Garcia}
%\author{Nicolas Bergeron}
\address{DMA UMR 8553 ENS / PSL et Sorbonne Université, 75005, Paris, France}
\email{nicolas.bergeron@ens.fr}
\urladdr{https://sites.google.com/view/nicolasbergeron/accueil}

%\author{Pierre Charollois}
\address{Sorbonne Universit\'e, IMJ--PRG, CNRS, Univ Paris Diderot, F-75005, Paris, France}
\email{pierre.charollois@imj-prg.fr}
\urladdr{https://webusers.imj-prg.fr/~pierre.charollois/pageperso.html}

%\author{Luis E. Garcia}
\address{Department of Mathematics, University College London, Gower Street, London WC1E 6BT, United Kingdom}
\email{l.e.garcia@ucl.ac.uk}
\urladdr{https://www.ucl.ac.uk/~ucahljg/index.htm}

\begin{abstracts}
\abstractin{french}
De nombreux auteurs, parmi lesquels Stevens \cite{Stevens}, Sczech \cite{Sczech93}, Nori \cite{Nori}, Solomon \cite{Solomon}, Hill \cite{Hill} ou plus récemment Charollois--Dasgupta--Greenbe- \\rg~\cite{CD,CDG}, Beilinson--Kings--Levin \cite{Beilinson&al} et Sharifi--Venkatesh \cite{SharifiVenkatesh}, ont construit des cocycles différents, mais apparentés, de groupes linéaires que l'on appelle habituellement ``cocycles d'Eisenstein''. L'objectif principal de ce texte est de décrire une construction topologique qui est une source commune pour tous ces cocycles. Une caractéristique intéressante de cette construction est que, partant d'une classe purement topologique, elle aboutit au monde algébrique des formes méromorphes sur des complémentaires d'hyperplans dans des produits de groupes additifs (complexes), de groupes multiplicatifs ou de (familles de) courbes  elliptiques. Cela conduit à la construction de trois types de ``cocycles de Sczech'' qui peuvent être regroupés dans le tableau suivant

\medskip
\begin{center}
\begin{tabular}{|c|c|c|c|}
\hline
{\rm Type} & {\rm Additif} & {\rm Multiplicatif} & {\rm Elliptique} \\
\hline 
{\rm Groupe}  & $\GL_n (\C)^\delta$ & $\GL_n (\Z )$ & $\GL_n (\Z)$ \\
\hline
{\rm Espace}   & $\C^n$ & $(\C^\times )^n = (\C / \Z)^n$  & $(\C / \Z +\tau \Z)^n$ \\
\hline 
$(n-1)$-{\rm cocycle} & $\mathbf{S}_{\rm aff}$ & $\mathbf{S}_{\rm mult}$ & $\mathbf{S}_{\rm ell}$ \\
\hline
\end{tabular}
\end{center}

\medskip
\noindent
Chacun de ces cocycles est à valeurs dans des formes méromorphes sur les espaces correspondants. Nous décrivons en outre explicitement ces formes méromorphes comme des polynômes homogènes de degré $n$ en des $1$-formes élémentaires des types suivants

\medskip
\begin{center}
\begin{tabular}{|c|c|c|c|}
\hline
{\rm Type} & {\rm Additif} & {\rm Multiplicatif} & {\rm Elliptique} \\
\hline 
$1$-{\rm formes} & $\frac{dz}{z}$ & $\sum_{n \in \Z} \frac{dz}{z+n}$  &  $\sum_{\lambda \in \Z + \tau \Z}  \frac{dz}{z+\lambda }$   \\
\hline
\end{tabular}
\end{center}

\medskip
\noindent
où les séries ci-dessus doivent être convenablement régularisées pour avoir un sens. Ces cocycles révèlent ainsi des relations cachées entre produits de fonctions élémentaires comme ci-dessus, relations qui sont prescrites par l'homologie des groupes linéaires.

\abstractin{english}
Many authors, among which Stevens \cite{Stevens}, Sczech \cite{Sczech93}, Nori \cite{Nori}, Solomon \cite{Solomon}, Hill \cite{Hill}, or more recently Charollois--Dasgupta--Greenberg \cite{CD,CDG}, Beilinson--Kings--Levin \cite{Beilinson&al} and Sharifi--Venkatesh \cite{SharifiVenkatesh}, have constructed different, but related, linear group cocycles that are usually referred to as ``Eisenstein cocycles.'' The main goal of this work is to describe a topological construction that is a common source for all these cocycles. One interesting feature of this construction is that, starting from a purely topological class, it leads to the algebraic world of meromorphic forms on hyperplane complements in $n$-fold products of either the (complex) additive group, the multiplicative group or a (family of) elliptic curve(s). This yields the construction of three types of ``Sczech cocycles'' that can be grouped in the following array

\medskip
\begin{center}
\begin{tabular}{|c|c|c|c|}
\hline
{\rm Type} & {\rm Additive} & {\rm Multiplicative} & {\rm Elliptic} \\
\hline 
{\rm Group}  & $\GL_n (\C)^\delta$ & $\GL_n (\Z )$ & $\GL_n (\Z)$ \\
\hline
{\rm Space}   & $\C^n$ & $(\C^\times )^n = (\C / \Z)^n$  & $(\C / \Z +\tau \Z)^n$ \\
\hline 
$(n-1)$-{\rm cocycle} & $\mathbf{S}_{\rm aff}$ & $\mathbf{S}_{\rm mult}$ & $\mathbf{S}_{\rm ell}$ \\
\hline
\end{tabular}
\end{center}

\medskip
\noindent
Each of these cocycles takes values in meromorphic $n$-forms on the corresponding space. We furthermore explicitly describe these meromorphic forms as degree $n$ homogeneous polynomials in elementary $1$-forms of the following types 

\medskip
\begin{center}
\begin{tabular}{|c|c|c|c|}
\hline
{\rm Type} & {\rm Additive} & {\rm Multiplicative} & {\rm Elliptic} \\
\hline 
$1$-{\rm forms} & $\frac{dz}{z}$ & $\sum_{n \in \Z} \frac{dz}{z+n}$  &  $\sum_{\lambda \in \Z + \tau \Z}  \frac{dz}{z+\lambda } $  \\
\hline
\end{tabular}
\end{center}

\medskip
\noindent
where the above series have to be suitably regularised to make sense. These cocycles therefore reveal hidden relations between products of elementary functions as above, relations that are prescribed by the group homology of linear groups. 
\end{abstracts}
\maketitle

\begin{otherlanguage}{french}

\tableofcontents

\chapter*{Introduction : fonctions trigonométriques et symboles modulaires}

Comme Eisenstein l'explique lui-même dans \cite{Eisenstein}, sa méthode pour construire des fonctions elliptiques s'applique de manière élégante au cas plus simple des fonctions trigonométriques. C'est par là que débute le livre que Weil \cite{Weil} consacre à ce sujet; nous suivons son exemple.

\section{La relation d'addition pour la fonction cotangente}

La méthode d'Eisenstein se base sur la considération de la série
$$\varepsilon (x) = \frac{1}{2i\pi} \sideset{}{{}^e}\sum_{m \in \Z} \frac{1}{x+m} =  \frac{1}{2i\pi} \lim_{M \to \infty}  \sum_{m = -M}^{M} \frac{1}{x+m},$$
où le symbole $\sum^e$ désigne la sommation d'Eisenstein définie par la limite de droite. 

Eisenstein démontre que\footnote{La normalisation appara\^{\i}tra plus naturellement par la suite; elle est bien évidemment liée au fait que la forme $dx/ (2i \pi x)$ possède  un résidu égal à $1$ en $0$.} $\varepsilon (x) = \frac{1}{2i} \cot \pi x$ et ce faisant retrouve la \emph{formule d'addition}, originellement découverte par Euler, selon laquelle pour tous les nombres complexes $x$ et $y$ tels que ni $x$, ni $y$, ni $x+y$ ne soient entiers, on a
\begin{equation} \label{E:addition}
\varepsilon (x) \varepsilon (y) - \varepsilon (x) \varepsilon (x+y) - \varepsilon (y) \varepsilon (x+y) = - 1/4.
\end{equation}
Le point de départ de la démonstration d'Eisenstein est une identité élémentaire entre fractions rationnelles~:
\begin{equation} \label{E:addition0}
\frac{1}{xy} - \frac{1}{x(x+y)} - \frac{1}{y (x+y)} = 0.
\end{equation}
Formellement on a en effet  
\begin{multline*}
\varepsilon (x) \varepsilon (y) - \varepsilon (x) \varepsilon (x+y) - \varepsilon (y) \varepsilon (x+y) \\ = \sum_{p,q,r} \left( \frac{1}{(x+p) (y+q)} - \frac{1}{(x+p)(x+y+r)} - \frac{1}{(y+q) (x+y+r)} \right)
\end{multline*}
où les entiers $p$, $q$ et $r$ varient dans $\Z$ tout en étant astreints à la relation $p+q-r=0$; les sommes n'étant pas absolument convergente cette décomposition n'a pas de sens, mais la relation \eqref{E:addition} se déduit d'une version régularisée de cette observation. 

Sczech \cite{Sczech92} interprète la relation d'addition \eqref{E:addition} comme une ``relation de cocycle''; nous y reviendrons. On relie d'abord cette relation aux \emph{symboles modulaires} dans le demi-plan de Poincaré $\mathcal{H}$. 

\section{Symboles modulaires}
Notons $\mathcal{H}^*$ l'espace obtenu en adjoignant à $\mathcal{H}$ les points rationnels $\mathbf{P}^1 (\Q)$ de son bord à l'infini $\mathbf{P}^1 (\R) = \R \cup \{ \infty \}$. \'Etant donné deux points distincts $r$ et $s$ dans $\mathbf{P}^1 (\Q)$, on note $\{ r,s \}$ la géodésique orientée reliant $r$ à $s$ dans $\mathcal{H}$.  

Soit $\Delta$ le groupe abélien engendré par les symboles $\{ r, s \}$ et soumis aux relations engendrées par   
$$\{ r , s \} + \{ s, r \} = 0 \quad \mbox{et} \quad \{r,s \} + \{ s,t \} + \{ t,r \} =0 .$$
On appelle \emph{symbole modulaire} l'image d'un symbole $\{r,s \}$ dans $\Delta$; on la note $[r,s]$. Manin \cite{Manin} observe que $\Delta$ est engendré par les symboles \emph{unimodulaires}, c'est à dire les $[r,s]$ avec $r=a/c$ et $s=b /d$ tels que $ad-bc = 1$, dont la géodésique associée $\{r,s\}$ est une arête de la triangulation de Farey représentée ci-dessous.

L'action de $\SL_2 (\Z)$ sur $\mathcal{H}$ par homographies se prolonge en une action sur $\mathcal{H}^*$ et induit une action naturelle sur $\Delta$ de sorte que 
$$g \cdot [\infty , 0 ] = [a/c , b/d],  \quad  \mbox{pour tout }  g = \left( \begin{smallmatrix} a & b \\ c & d \end{smallmatrix} \right) \in \SL_2 (\Z).$$
Notons 
$$\overline{\mathcal{M}} (\C^2 / \Z^2 ) = \mathcal{M} (\C^2 / \Z^2) / \C \cdot \mathbf{1}$$
le quotient de l'espace des fonctions méromorphes et $\Z^2$-périodiques sur $\C^2$ par le sous-espace des fonctions constantes. L'action linéaire de $\SL_2 (\Z)$ sur $\C^2$ induit une action sur $\C^2 / \Z^2$, et donc sur $\overline{\mathcal{M}} (\C^2 / \Z^2 )$, et l'observation suivante est simplement une reformulation de la relation d'addition \eqref{E:addition}.
 
\medskip
\noindent
{\bf Observation.} {\it L'application $\mathbf{c} : \Delta \to \overline{\mathcal{M}} (\C^2 / \Z^2)$ définie sur les symboles \emph{unimodulaires} par 
$$\mathbf{c} ([r,s]) (x, y) = \epsilon (dx-by) \epsilon (-cx + ay), \quad  \mbox{pour } r= a/c, \ s = b/d, \ ad-bc=1, $$ 
est bien définie et $\SL_2 (\Z)$-équivariante. }

\medskip

On dit alors que $\mathbf{c}$ est un \emph{symbole modulaire à valeurs dans $\overline{\mathcal{M}} (\C^2 / \Z^2)$}. 

\begin{center}
\includegraphics[width=0.9\textwidth]{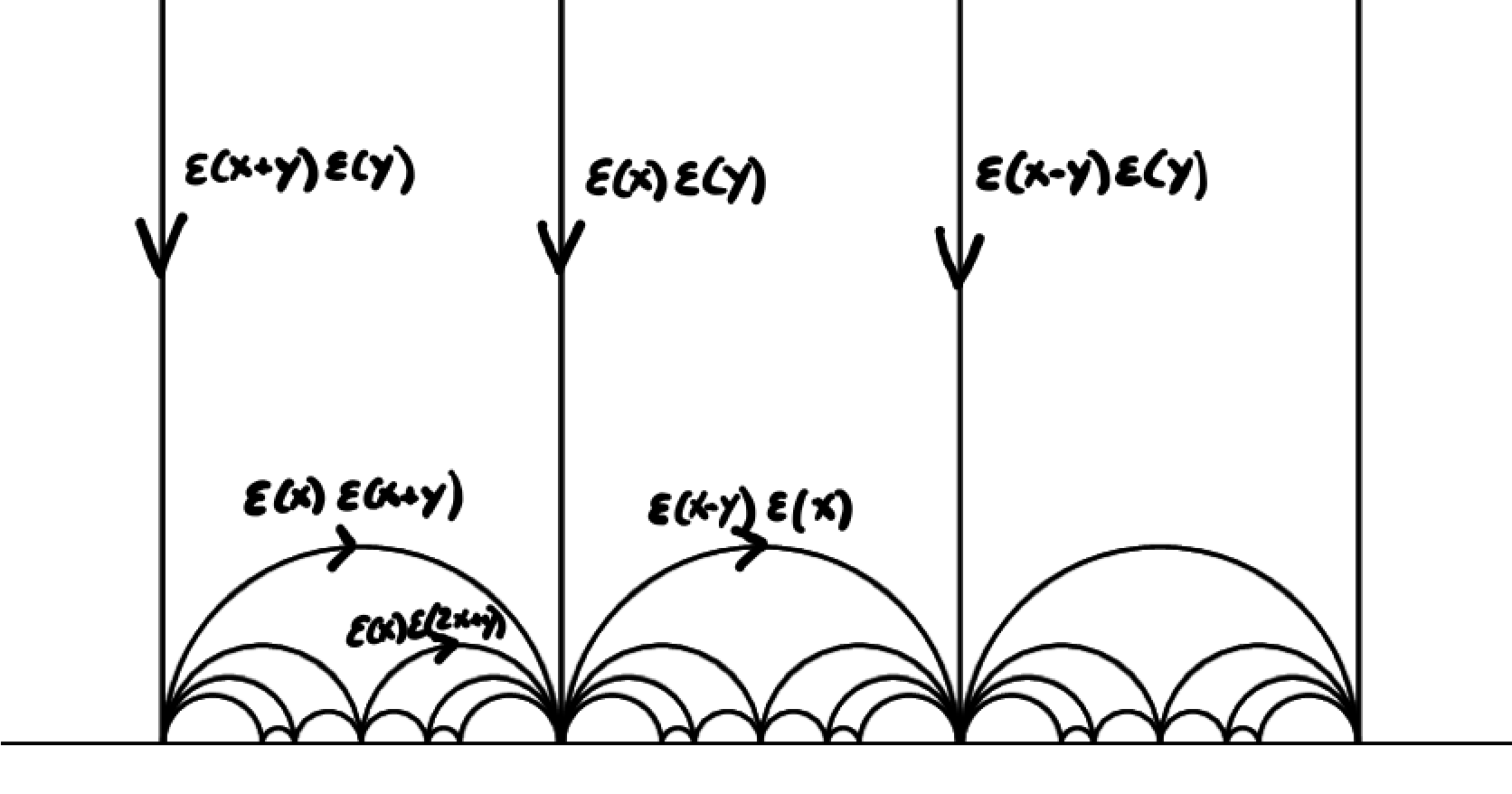}
\end{center}

Cette remarque élémentaire place la relation d'addition dans un nouveau contexte. Elle suggère d'étudier l'action des opérateurs de Hecke sur $\mathbf{c}$. 

\section{Opérateurs de Hecke}

Les actions du groupe $\SL_2 (\Z)$ sur $\Delta$ et sur $\C^2/ \Z^2$ s'étendent naturellement au monoïde $M_2 (\Z)^\circ = M_2 (\Z) \cap \GL_2 (\Q)$. 
Cela munit $\mathrm{Hom} (\Delta , \overline{\mathcal{M}} (\C^2 / \Z^2))$ d'une action à droite qui étend celle de $\SL_2 (\Z)$~:
$$\phi_{| g} ([r,s]) (x,y) = \phi (g \cdot [r,s]) \left( g \cdot \left( \begin{smallmatrix} x \\ y \end{smallmatrix} \right) \right),$$ 
$( \phi \in \mathrm{Hom} (\Delta , \overline{\mathcal{M}} (\C^2 / \Z^2)), \ g \in M_2 (\Z)^\circ ).$
L'espace 
$$\mathrm{Hom} (\Delta , \overline{\mathcal{M}} (\C^2 / \Z^2) )^{\SL_2 (\Z)}$$
des symboles modulaires à valeurs dans $\overline{\mathcal{M}} (\C^2 / \Z^2)$ hérite alors d'une action à droite de l'algèbre de Hecke associée à la paire $(M_2 (\Z)^\circ , \SL_2 (\Z))$~: étant donné un élément $g \in M_2 (\Z)^\circ$  on décompose la double classe de $g$ par $\SL_2 (\Z)$ en une union finie de classes à gauche
$$\SL_2 (\Z ) g \SL_2 (\Z) = \bigsqcup_j \SL_2 (\Z) g_j.$$
L'opérateur de Hecke associé à $g$ opère sur un symbole modulaire $\phi$ à valeurs dans $\overline{\mathcal{M}} (\C^2 / \Z^2)$ par 
$$\mathbf{T}(g) \phi = \sum_j \phi_{| g_j}.$$

\medskip
\noindent
{\it Exemples.} Notons $\mathbf{T}_p$ l'opérateur de Hecke associé à la matrice $\left( \begin{smallmatrix} p & 0 \\ 0 & 1 \end{smallmatrix} \right)$. 
Pour $p=2$, on a 
$$
(\mathbf{T}_2 \mathbf{c}) ([\infty , 0]) (x,y) = \varepsilon (2x) \varepsilon (y) + \varepsilon (x) \varepsilon (2y) + \varepsilon (2x) \varepsilon (x+y)  + \varepsilon (2y) \varepsilon (x+y).
$$
On laisse au lecteur le plaisir coupable de vérifier que, si $2x$, $2y$ et $x+y$ ne sont pas des entiers, on a 
\begin{multline*}
\varepsilon (2x) \varepsilon (y) + \varepsilon (x) \varepsilon (2y) + \varepsilon (2x) \varepsilon (x+y)  + \varepsilon (2y) \varepsilon (x+y) \\ - 2 \varepsilon (2x) \varepsilon (2y) - \varepsilon (x) \varepsilon (y) = 1/4
\end{multline*}
et donc que le symbole modulaire $\mathbf{c}$, à valeurs dans $\overline{\mathcal{M}} (\C^2 / \Z^2)$, est annulé par l'opérateur 
$$\mathbf{T}_2 - 2[2]^* -1,$$
où l'on note $[m]^*$ le tiré en arrière par l'application $\C^2/ \Z^2  \to \C^2 / \Z^2$ induite par la multiplication par $m$, autrement dit $(x,y) \mapsto (mx,my)$.

Pour $p=3$ et $p=5$, on a respectivement   
\begin{multline*}
(\mathbf{T}_3 \mathbf{c}) ([\infty , 0]) (x,y) = \varepsilon (3x) \varepsilon (y) + \varepsilon (x) \varepsilon (3y)  + \varepsilon (3x) \varepsilon (x+y)  \\ + \varepsilon (3y) \varepsilon (x+y) + \varepsilon (3x) \varepsilon (y-x)  + \varepsilon (3y) \varepsilon (x-y).
\end{multline*}
et 
\begin{multline*}
(\mathbf{T}_5 \mathbf{c}) ([\infty , 0]) (x,y) = \varepsilon (5x) \varepsilon (y) + \varepsilon (x) \varepsilon (5y)  + \varepsilon (5x) \varepsilon (x+y)  + \varepsilon (5y) \varepsilon (x+y) \\ + \varepsilon (5x) \varepsilon (y-2x)  + \varepsilon (5y) \varepsilon (x+2y)   - \varepsilon (y-2x) \varepsilon ( x+2y) 
 + \varepsilon (5x) \varepsilon (y+2x)  \\ + \varepsilon (5y) \varepsilon (x+3y) + \varepsilon (y+2x) \varepsilon ( x+3y) + \varepsilon (5x) \varepsilon (y-x) + \varepsilon (x-y) \varepsilon (5y) .
\end{multline*}
Au prix de fastidieux calculs on peut là encore vérifier que dans chacun de ces cas la relation 
\begin{equation} \label{E:HeckeTrig}
(\mathbf{T}_p - p[p]^* -1) \mathbf{c} =0
\end{equation}
est satisfaite.

\section{Un théorème et quelques questions}

Il est naturel de conjecturer que les relations \eqref{E:HeckeTrig} sont vérifiées pour tout nombre $p$ premier. Un tel énoncé rappelle une conjecture de Busuioc \cite{Busuioc} et Sharifi \cite{Sharifi}; la démonstration récente que Sharifi et Venkatesh \cite{SharifiVenkatesh} en ont donnée implique aussi que les relations \eqref{E:HeckeTrig} sont vérifiées pour tout $p$ premier.\footnote{Le lien entre les travaux de Sharifi et Venkatesh et les questions abordées ici est expliqué dans le sixième paragraphe de cette introduction.}
 
En plus de cela on aimerait relever l'application $\mathbf{c}$ en un symbole modulaire --- nécessairement \emph{partiel}, au sens de la thèse de Dasgupta, voir \cite{Evil,DD} --- à valeurs dans $\mathcal{M} (\C^2 / \Z^2)$ plutôt que son quotient par les fonctions constantes. Ces deux desiderata font l'objet du théorème ci-dessous qui s'énonce plus naturellement en termes de cohomologie des groupes et requiert en partie d'augmenter le niveau. 

L'application $\overline{\mathbf{S}} : g \mapsto \mathbf{c} ([\infty , g \cdot \infty ])$ définit en effet un $1$-cocycle\footnote{La relation de cocycle s'écrit 
$$\overline{\mathbf{S}}  (hg) = \overline{\mathbf{S}} (h) + h \cdot \overline{\mathbf{S}} (g).$$} de $\SL_2 (\Z )$ à valeurs dans 
$\overline{\mathcal{M}} (\C^2 / \Z^2)$ et donc une classe de cohomologie dans 
$$H^1 (\SL_2 (\Z ) , \overline{\mathcal{M}} (\C^2 / \Z^2)).$$  
\'Etant donné un entier $N$ strictement positif, on note comme d'habitude 
$$\Gamma_0 (N) = \left\{ \left( \begin{smallmatrix} a & b \\ c & d \end{smallmatrix} \right) \in \SL_2 (\Z) \; : \; c \equiv 0 \ (\mathrm{mod} \ N) \right\}$$ le sous-groupe de $\SL_2 (\Z)$ constitué des matrices qui fixent la droite $\langle e_1 \rangle$, engendrée par le premier vecteur de la base canonique de $\Z^2$, modulo $N$. On note enfin $\Delta^\circ_N \subset \Delta$
le sous-groupe engendré par les symboles $[r , s]$ avec $r, s \in \Gamma_0 (N) \cdot \infty \subset \mathbf{P}^1 (\Q)$. Le sous-groupe $\Delta^\circ_N$ est donc engendré par les éléments $[a/Nc , b/Nd] \in \Delta$ avec $a$ et $b$ premiers avec $N$.

On note finalement $D_N$ le groupe abélien libre engendré par les combinaisons linéaires entières formelles de diviseurs positifs de $N$ et $D_N^\circ$ le sous-groupe constitué des éléments de degré $0$, c'est-à-dire des éléments $\delta = \sum_{d | N} n_d [d]$ tels que $\sum_{d|N} n_d =0$.

\begin{theorem*}
Il existe un morphisme $\delta \mapsto \mathbf{S}_\delta$ de $D_N$ vers le groupe des $1$-cocycles de $\Gamma_0 (N)$ à valeurs dans $\mathcal{M} (\C^2 / \Z^2 )$ vérifiant les propriétés suivantes.

1. On a 
$$[\mathbf{S}_{[1]}] = [\overline{\mathbf{S}}]  \neq 0  \quad \mbox{dans} \quad H^{1} (\Gamma_0 (N),  \overline{\mathcal{M}} (\C^2 / \Z^2 ) ).$$ 

2. Pour tout entier $p$ premier ne divisant pas $N$ et pour tout $\delta \in D_N$, la classe de cohomologie de $\mathbf{S}_\delta$ dans $H^{1} (\Gamma_0 (N),  \mathcal{M} (\C^2 / \Z^2 ) )$ annule l'opérateur $\mathbf{T}_{p} - p [p]^* -  1$.

3. Pour tout entier strictement positif $N$ et pour tout $\delta = \sum_{d | N } n_d [d]$ dans $D_N^\circ$, le $1$-cocycle $\mathbf{S}_\delta$ est cohomologue à un cocycle explicite $\mathbf{S}_{\delta}^*$ défini par 
\begin{multline} \label{Sexplicit}
\mathbf{S}_{\delta}^*\left( \begin{array}{cc} a & * \\ Nc & * \end{array} \right) \\ = \left\{ \begin{array}{ll} 
0 & \mbox{si } c=0, \\
\sum_{d | N} \frac{n_d }{d' c} \sum_{j \ {\rm mod} \ d' c} \varepsilon \left( \frac{1}{d' c} (y+j) \right) \varepsilon \left( dx - \frac{a}{d' c} (y+j) \right) & \mbox{sinon},
\end{array} \right.
\end{multline}
avec $dd'=N$.
\end{theorem*}

\medskip
\noindent
{\it Remarque.} Chaque application $\mathbf{S}_\delta^*$ définit en fait un symbole modulaire \emph{partiel} dans
$$\mathrm{Hom} (\Delta_N^\circ  ,  \mathcal{M} (\C^2 / \Z^2 ))^{\Gamma_0 (N)}.$$
On le définit sur un symbole $[\infty , a/Nc]$ par $\mathbf{S}^*_\delta \left( \begin{smallmatrix} a & u \\ Nc & v \end{smallmatrix} \right)$ où $u$ et $v$ sont tels que $av-Ncu=1$.  

\medskip

Ce texte est consacré à une vaste généralisation de ce théorème. Le but est de répondre aux questions suivantes~:
\begin{enumerate}
\item Que dire des produits de $n$ fonctions cotangentes lorsque $n \geq 3$ ?
\item Y a-t-il des résultats analogues pour des fonctions elliptiques ou, plus simplement, pour des fractions rationnelles ?
\end{enumerate}
Les réponses à ces questions sont énoncées au chapitre \ref{C:2} qui est une sorte de deuxième introduction dans laquelle les résultats généraux sont énoncés. Avant cela, dans le premier chapitre, on commence par détailler la construction de classes de cohomologie, pour des sous-groupes de $\GL_n$, qui généralisent la classe de $\mathbf{S}_{[1]}$ et ont toutes une origine topologique commune. Les classes que nous construisons sont à coefficients dans une partie de la cohomologie d'arrangements d'hyperplans dans des produits $A^n$ où $A$ est isomorphe au groupe multiplicatif ou à une courbe elliptique. Un point important, pour démontrer le théorème ci-dessus ou les résultats annoncés dans le chapitre \ref{C:2} est ensuite de montrer que cette partie de la cohomologie d'arrangements d'hyperplans peut être représentée par des formes méromorphes sur $A^n$. C'est l'objet du théorème \ref{P:Brieskorn} dont la démonstration occupe le chapitre \ref{S:OrlikSolomon}. Le calcul explicite des classes ainsi obtenues occupe le reste de l'ouvrage. 

\medskip

Dans les deux derniers paragraphes de cette introduction on relie les cocycles $\mathbf{S}_\delta$, et leurs généralisations annoncées, à des objets plus classiques en théorie des nombres.

\section{\'Evaluation, terme constant, morphismes de Dedekind--Rademacher} Le slogan suivant, à retenir, distingue les avantages respectifs des cocycles $\mathbf{S}_\delta$ et $\mathbf{S}_\delta^*$.  
\begin{quote}
{\it Les cocycles notés $\mathbf{S}$ peuvent être \emph{évalués} en des points de torsion alors que les cocycles notés $\mathbf{S}^*$ peuvent eux être \emph{calculés} en le point générique.} 
\end{quote} 
Nous verrons en effet que l'on peut contrôler l'ensemble des points $P \in \C^2 / \Z^2$ en lesquels les fonctions méromorphes $\mathbf{S}_\delta (g)$ sont régulières.  Après évaluation en de tels points de torsion de $\C^2/\Z^2$ bien choisis, on retrouve alors des résultats plus classiques. 

\medskip
\noindent
{\it Exemple.} Les fonctions dans l'image de $\mathbf{S}_{[1]}$ sont régulières en les points $(j/N,0)$, pour tous $j \in \{1 , \ldots , N-1 \}$. L'application 
\begin{equation} \label{E:cocDR}
\Psi_{N} : \Gamma_0 (N ) \to \C; \quad g \mapsto  \sum_{j=1}^{N-1} \mathbf{S}_{[1]} (g ) (j/N, 0)
\end{equation}
définit donc un morphisme de groupes. On retrouve ainsi 
un multiple du morphisme de Dedekind--Rademacher \cite{Rademacher,Mazur} donné par 
\begin{equation} \label{E:DR}
\Phi_N \left( \begin{array}{cc} a & b \\ N c & d \end{array} \right) = \left\{ \begin{array}{ll}
(N -1)b/d & \mbox{si } c=0 \\
\frac{(N -1)(a+d)}{Nc} + 12 \cdot \mathrm{sign}(c) \cdot D^N \left( \frac{a}{N |c|} \right) & \mbox{si } c \neq 0,
\end{array} \right.
\end{equation}
où, en notant $D : \Q \to \Q$ la \emph{somme de Dedekind} usuelle
$$D(a/c) = \frac{1}{c} \sum_{j=1}^{c-1}\varepsilon \left( \frac{j }{c} \right) \varepsilon \left( - \frac{j a}{c} \right) \quad \mbox{pour } c>0 \quad \mbox{et} \quad (a,c)=1,$$
on a $D^N (x) = D(x) - D(N x)$. De l'expression de $\mathbf{S}_{[1]}$ que l'on donnera au chapitre~\ref{C:7}, on peut plus précisément déduire --- comme dans \cite[\S 11]{Takagi} --- que 
$$12 \cdot \Psi_N =  \Phi_N .$$

\medskip

Plutôt que d'évaluer les fonctions méromorphes $\mathbf{S}^*_\delta (g)$ on peut considérer leur terme constant en $0$. Le théorème du paragraphe précédent implique alors le corollaire qui suit.

\begin{cor*} 
Soit $\delta \in D_N^\circ$ tel que $\sum_{d | N} n_d d =0$ et $D^\delta (x) = \sum_{d | N} n_d D (dx)$. Alors l'application $\Psi_\delta : \Gamma_0 (N) \to \C$ donnée par 
\begin{equation} \label{E:DD}
\Psi_\delta \left( \begin{array}{cc} a & * \\ N c & * \end{array} \right) = \left\{ \begin{array}{ll}
0 & \mbox{si } c=0 \\
\mathrm{sign}(c) \cdot D^\delta \left( \frac{a}{N |c|} \right) & \mbox{si } c \neq 0,
\end{array} \right.
\end{equation}
définit un morphisme de groupes.
\end{cor*}

On notera que le morphisme $12 \Psi_\delta$ est à valeurs entières et qu'il coïncide avec le \emph{morphisme de Dedekind--Rademacher modifié} $\Phi_\delta$ de Darmon et Dasgupta \cite{DD}.

\section{Relations avec les travaux de Kato et de Sharifi--Venkatesh}

Kato \cite{Kato} construit des unités de Siegel sur les courbes modulaires $Y_1 (N)$ à partir de fonctions theta sur la courbe elliptique universelle $E$ au-dessus de $Y_1 (N)$. 

Il découle en effet de \cite[Proposition 1.3]{Kato} qu'étant donné un entier strictement positif $m$ premier à $6N$, il existe une fonction theta ${}_m \theta$ dans $\Q (E)^\times$ qui est une unité en dehors des points de $m$-torsion. La fonction ${}_m \theta$ est caractérisée par son diviseur dans $E[m]$ et des relations de distribution associées aux applications de multiplication par des entiers relativement premiers à $m$. Les unités de Siegel sont alors obtenues en tirant en arrière sur $Y_1 (N)$ ces fonctions theta par des sections de $N$-torsion. 

Sharifi et Venkatesh \cite{SharifiVenkatesh} considèrent des analogues des fonctions ${}_m \theta$. Leur méthode permet en fait de construire des $1$-cocycles sur des sous-groupes $\Gamma$ de $\GL_2 (\Z)$ à valeurs dans les groupes de $K$-théorie de degré $2$ des corps de fonctions de $\C^2 / \Z^2$ ou du carré $E^2$ de la courbe elliptique universelle. Les $1$-cocycles du premier type sont des applications de la forme $\Gamma \to K_2 (\Q (\C^2 / \Z^2))$. 
En les composant avec 
$$K_2 (\Q (\C^2 / \Z^2)) \to \mathcal{M} (\C^2 / \Z^2); \quad \{ f , g \} \mapsto \frac{d \log (f) \wedge d \log (g)}{dx \wedge dy},$$
où $x$ et $y$ sont les deux coordonnées de $\C^2$, on obtient des $1$-cocycles dont on peut vérifier qu'ils sont cohomologues aux cocycles du théorème énoncé ci-dessus. Le fait que les relations \eqref{E:HeckeTrig} sont bien vérifiées pour tout nombre premier $p$ découle alors de \cite[Lemma 4.2.9]{SharifiVenkatesh}. 

Un aspect intéressant de notre construction est que nous obtenons ces cocycles à partir d'une classe purement topologique; l'émergence de fonctions méromorphes se déduit au final d'un théorème ``de type Brieskorn'' qui permet de représenter certaines classes de cohomologie singulière par des formes méromorphes. La construction est suffisamment maniable pour nous permettre de considérer plus généralement l'action de $\GL_n (\Z)$ sur $\C^n / \Z^n$ ou sur $E^n$. 

Les $1$-cocycles du second type chez Sharifi et Venkatesh sont des applications de la forme $\Gamma \to K_2 (\Q (E^2))$. Comme pour les unités de Siegel, on peut tirer en arrière ces cocycles par des sections de torsion. On obtient ainsi des homomorphismes du premier groupe d'homologie de $\Gamma$ vers le $K_2$ d'une courbe modulaire. Goncharov et Brunault \cite{Gon1,Bru1,Bru2} avaient déjà construit de tels homomorphismes en associant explicitement à certains symboles modulaires des symboles de Steinberg d'unités de Siegel. Obtenir ces morphismes par spécialisation d'un $1$-cocycle à valeurs dans $K_2 (\Q (E^2))$ permet de montrer que ces homomorphismes sont Hecke-équivariants. 

En composant un $1$-cocycle $\Gamma \to K_2 (\Q (E^2))$ avec le symbole différentiel $\partial \log \wedge \partial \log$ puis en tirant le résultat en arrière par une section de torsion, on obtient un homomorphisme du premier groupe d'homologie de $\Gamma$ vers les formes modulaires de poids $2$ sur $Y_1(N)$. \'Etendu aux symboles modulaires (partiels), ce morphisme associe à un tel symbole une ``zeta modular form'' au sens de \cite[Section 4]{Kato}. 

Cette construction s'étend là encore à l'action de $\GL_n (\Z)$ sur le produit de $n$ courbes elliptiques universelles; voir chapitre \ref{C:2}. Les cocycles ainsi obtenus révèlent des relations cachées entre des produits de fonctions elliptiques  classiques, relations gouvernées par l'homologie de sous-groupes de congruence dans $\GL_n (\Z)$. On peut tirer de cela un certain nombres de conséquences arithmétiques \cite{Takagi,ColmezNous}; d'autres conséquences sont en préparation.

\section{Remerciements}

Ce texte est l'aboutissement d'une réflexion entamée il y a quelques années avec Akshay Venkatesh. Plusieurs idées sont issues de discussions avec lui (ainsi que l'impulsion d'écrire en français), un grand merci à lui pour sa générosité.

De toute évidence, ce travail doit beaucoup aux idées initiées et développées par Robert Sczech. Nous profitons de cette occasion pour lui exprimer notre gratitude.

P.C. remercie Samit Dasgupta pour lui avoir posé une question étincelle il y a 10 ans.

L.G. remercie l'IHES et le soutien de l'ERC de Michael Harris durant les premiers temps de ce projet.

N.B. tient à remercier Olivier Benoist pour de nombreuses discussions autour du chapitre 3. 

Merci à Emma Bergeron pour les dessins.

Enfin, c'est un plaisir de remercier Henri Darmon, Clément Dupont, Javier Fresan, Peter Xu et les rapporteurs anonymes pour leurs nombreux commentaires.

\numberwithin{equation}{chapter}

\numberwithin{section}{chapter}

\resettheoremcounters

\chapter{Construction de cocycles : aspects topologiques} \label{C:1}

\section{Résumé}
Le cocycle $\mathbf{S}$ discuté en introduction est relié aux ``cocycles d'Eisenstein'' qui interviennent sous différentes formes dans la littérature, par exemple dans \cite{Sczech93,Nori,CD,CDG}. Un cocycle très proche est en fait explicitement considéré par Sczech dans \cite{Sczech92,Sczech93}.

Le premier but de ce texte est de donner une construction générale de cocycles de ``type Sczech'' et de montrer qu'ils ont une source topologique commune. La méthode utilisée consiste à relever certaines classes de cohomologie dans $H^{2n-1}(X^*)$, où $X$ est un $G$-espace et $X^*$ est égal à $X$ privé d'un nombre fini de ses points, en des classes de cohomologie {\'e}quivariante dans $H_{G}^{2n-1}(X^* )$. Elle rappelle la méthode proposée par Quillen \cite{Quillen} pour calculer la cohomologie d'un groupe linéaire sur un corps fini. 

En pratique on considère essentiellement trois cas~:
\begin{itemize}
\item Le cas \emph{additif} (ou \emph{affine}). Dans ce cas $X = \C^n$, l'espace épointé $X^*$ est égal à $X$ privé du singleton $D=\{ 0 \}$ et $G = \GL_n (\C)^\delta$, le groupe des  transformations lin{\'e}aires de $\mathbf{C}^n$, consid{\'e}r{\'e} comme un groupe discret.
\item Le cas \emph{multiplicatif} (ou \emph{trigonométrique}). Dans ce cas $X = \C^n / \Z^n$, l'espace épointé $X^*$ est égal à $X$ privé d'un cycle $D$ de degré $0$ constitué de points de torsion et $G$ est un sous-groupe de $\GL_n (\Z)$ qui préserve $D$. 
\item Le cas \emph{elliptique}. Dans ce cas $X=E^n$, où est $E$ une courbe elliptique ou une famille de courbes elliptiques, l'espace épointé $X^*$ est égal à $X$ privé d'un cycle $D$ de degré $0$ constitué de points de torsion et $G$ est un sous-groupe d'indice fini de $\GL_n (\Z)$, ou $\GL_n (\mathcal{O})$ si $E$ est à multiplication par $\mathcal{O}$, qui préserve $D$. 
\end{itemize}
Dans chacun de ces cas l'action naturelle (linéaire) de $G$ sur $X$ donne lieu à un fibré d'espace total 
$$EG \times_G X$$
au-dessus de l'espace classifiant $BG=EG/G$. L'isomorphisme de Thom associe alors à $D$ une classe dans 
$H^{2n} ( EG \times_G X , EG \times_G X^* )$. 
La construction de Borel identifie ce groupe au groupe de cohomologie {\it équivariante}  $H_G^{2n} (X , X^*)$. On renvoie à l'annexe \ref{A:A} pour plus de détails sur la cohomologie équivariante; elle généralise à la fois la cohomologie des groupes et la cohomologie usuelle. La construction de Borel permet de retrouver les propriétés usuelles, comme par exemple associer une suite exacte longue à une paire de $G$-espaces. Dans la suite on considère 
$$[D] \in H_G^{2n} (X , X^*)$$
et la suite exacte longue associée à la paire $(X,X^*)$~: 
\begin{equation} \label{suiteexacte0}
H_G^{2n-1} (X ) \to H_G^{2n-1} ( X^*) \to H_G^{2n} (X , X^*) \to H_G^{2n} (X).
\end{equation}
L'origine topologique de nos cocycles repose alors sur le fait suivant~:

\medskip
\noindent
{\bf Fait.} {\it La classe $[D] \in H_G^{2n} (X , X^*)$ admet un relevé (privilégié) 
$$E[D] \in H_{G}^{2n-1}(X^*).$$
}

\medskip

Nous démontrons ce fait au cas par cas dans les paragraphes qui suivent. Dans le cas affine il résulte du fait qu'un fibré vectoriel complexe plat possède une classe d'Euler rationnelle triviale, alors que dans le cas elliptique on le déduit d'un théorème de Sullivan qui affirme que la classe d'Euler rationnelle d'un fibré vectoriel à groupe structural contenu dans $\SL_n (\Z)$ est nulle. 

L'étape suivante part d'une remarque générale~: supposons que $Y/\C$ soit une vari\'et\'e {\em affine}, de dimension $n$, sur laquelle opère un groupe $G$, et supposons donnée une classe de cohomologie équivariante $\alpha \in H_G^{2n-1}(Y(\C))$. Puisque $H^i(Y(\C))$ s'annule pour 
$i  > n$, la suite spectrale pour la cohomologie \'equivariante donne une application
$$H_G^{2n-1}(Y(\C)) \rightarrow H^{n-1}(G, H^{n}(Y(\C))).$$
Elle permet donc d'associer à $\alpha$ une classe de cohomologie du groupe $G$.

Dans les cas qui nous intéressent la variété $X^*$ n'est pas affine, mais on peut restreindre la classe $E[D]$ \`a un ouvert affine $U$.  On voudrait aussi que $U$ soit invariant par $G$; mais un tel $U$ n'existe pas. Dans le cas additif où $X = \C^n$, on peut toutefois formellement prendre $U := `` \C^n - \bigcup_{\ell} \ell^{-1}(0)$''. Plus précisément, étant donné un ensemble fini $L$ de fonctionnelles affines, on pose 
$U_L = \C^n - \cup_{\ell \in L} \ell^{-1}(0)$. En regardant $U$ comme la limite inverse des $U_L$, on associe à $E[D]$ une classe dans le groupe 
$$H^{n-1}(G, \varinjlim_{L} H^j(U_L )).$$ 

La dernière étape de notre construction consiste à représenter $\varinjlim_{L} H^j(U_L)$ par des formes méromorphes. Dans le cas affine cela résulte d'un théorème célèbre de Brieskorn \cite{Brieskorn}~: 
$$\varinjlim_{L} H^j(U_L )  = \begin{cases} 0, \ j > n, \\ \Omega^n_{\mathrm{aff}}, \ j = n, \end{cases} $$
où $\Omega^n_{\mathrm{aff}} \subset \Omega^n_{\mathrm{mer}} (\C^n )$ est une sous-algèbre de formes méromorphes, voir Définition~\ref{def:Omegamer}.
Le chapitre \ref{S:OrlikSolomon} est consacré à la démonstration d'un résultat de ce type dans les cas multiplicatif et elliptique. 

En admettant pour l'instant ce théorème ``de type Brieskorn'', on consacre le présent chapitre à détailler la construction esquissée ci-dessus. Elle conduit à des classes 
$$\mathbf{S}[D] \in H^{n-1} (G , \Omega^n_{\rm mer} (X )).$$

\section{Le cocycle additif}

\subsection{Une classe de cohomologie \'equivariante}  \label{S21}

Soit $G= \GL_n (\C)^\delta$, le groupe des  transformations lin{\'e}aires  de $\mathbf{C}^n$, consid{\'e}r{\'e} comme un groupe discret.

La repr{\'e}sentation linéaire\footnote{On identifie donc $\C^n$ à l'espace des vecteurs colonnes.} 
$$G \rightarrow \GL (\C^n ); \quad g \mapsto ( z \in \C^n \mapsto gz)$$
donne lieu {\`a} un fibr{\'e} vectoriel $\mathcal{V}$, d'espace total $EG \times_G \C^n$, sur l'espace classifiant $BG$. On renvoie à l'annexe \ref{A:A} pour des rappels sur les espaces classifiants, les espaces simpliciaux et la cohomologie équivariante. Rappelons juste ici que si $X$ est un espace topologique muni d'une action continue de $G$, on a 
$$H^*_G (X) = H^* (EG \times_G X) .$$
Lorsque $X$ est contractile, ce groupe se réduit à $H^* (BG)=H^* (G)$, la cohomologie du groupe $G$. 

On peut considérer la classe de Thom du fibré $\mathcal{V}$~:
$$u \in H^{2n}_G ( \C^n , \C^n - \{ 0 \} )$$
à coefficients dans $\C$. Dans la suite exacte
$$\xymatrix{
H_G^{2n-1} (\C^n - \{ 0 \}) \ar[r] & H_G^{2n} (\C^n , \C^n - \{ 0 \} ) \ar[r]^{ \quad \quad c} & H_{G}^{2n} (\C^n )},$$
l'image de $u$ par l'application $c$ est la classe de Chern \'equivariante
$$c_{2n}(\mathcal{V}) \in H^{2n}(BG, \C),$$
qui est nulle parce que $\mathcal{V}$ est plat.  On peut donc relever la classe $u$ en une classe dans $H_G^{2n-1} (\C^n - \{ 0 \})$.

\medskip
\noindent
{\it Remarque.} Ce relev\'e n'est pas unique, mais on peut consid\'erer la suite exacte 
$$H^{2n-1} (G) \to H^{2n-1}_G (\C^n - \{ 0 \} ) \to H^{2n-1} (\C^n -\{ 0 \})$$
associée à la fibration $\mathcal{V}^* \to BG$, où $\mathcal{V}^*$ désigne le complémentaire de la section nulle dans $\mathcal{V}$, 
Chaque relev\'e de $u$ dans $H_G^{2n-1} (\C^n - \{ 0 \})$ s'envoie sur la classe fondamentale dans $H^{2n-1} (\C^n -\{ 0 \})$. 

\medskip

Le quotient $EG \times_G (\C^n - \{ 0 \})$ est une \emph{variété simpliciale}, c'est-à-dire un ensemble semi-simplicial dont les $m$-simplexes 
$$(EG \times_G (\C^n - \{ 0 \}))_m = (EG_m \times (\C^n - \{ 0 \}) / G$$
sont des variétés et dont les applications de faces et de dégénérescences sont lisses. La \emph{réalisation grossière} $\| EG \times_G (\C^n - \{ 0 \}) \|$ de cette variété simpliciale est l'espace topologique 
$$\| EG \times_G (\C^n - \{ 0 \}) \| = \sqcup_{m \geq 0} \Delta_m \times (EG \times_G (\C^n - \{ 0 \}))_m / \sim.$$
Ici $\Delta_m$ désigne le $m$-simplexe standard et les identifications sont données par 
\begin{equation*}
(\sigma^i (t) , x) \sim (t , \sigma_i (x)), \quad t \in \Delta_{m-1} , \ x \in (EG \times_G (\C^n - \{ 0 \}))_m, \ i \in \{ 0 , \ldots , m \},
\end{equation*}
où $\sigma^i : \Delta_{m-1} \to \Delta_m$ désigne l'inclusion de la $i$-ème face et $\sigma_i :  (EG \times_G (\C^n - \{ 0 \}))_m \to (EG \times_G (\C^n - \{ 0 \}))_{m-1}$ est l'application de face correspondante. On renvoie à l'annexe \ref{A:A} pour plus de détails sur ces objets. 

Retenons que l'on a une projection continue naturelle de la réalisation grossière vers la réalisation géométrique de $EG \times_G (\C^n - \{ 0 \})$ et que cette application est une équivalence d'homotopie. En pratique nous travaillerons avec la réalisation grossière. En tirant en arrière le relevé de $u$ on obtient la proposition suivante.

\begin{proposition} \label{P4}
La classe de Thom $u$ admet un relev\'e dans 
$$H^{2n-1} ( \| EG \times_G (\C^n - \{ 0 \}) \| ).$$
\end{proposition}

{\it Via} la théorie de Chern--Weil et les travaux de Mathai et Quillen \cite{MathaiQuillen}, nous construirons au  paragraphe~\ref{S:61} du chapitre \ref{S:6} un relevé {\rm privilégié} de $u$ représenté par une forme différentielle.

\subsection{Effacer les hyperplans} \label{S:222}

\`A tout \'el\'ement $g\in G$ de premier vecteur ligne $v \in \C^n - \{ 0 \} $, on associe une forme linéaire 
$$e_1^* \circ g : \C^n \to \C; \quad z \mapsto v z.$$

Pour tout $(k+1)$-uplet $(g_0 , \ldots , g_k ) \in G^{k+1}$, on note 
$$U (g_0 , \ldots , g_k) = \{ z \in \C^n \; : \; \forall j \in \{ 0 , \ldots , k \}, \ e_1^* (g_j z) \neq 0  \}.$$
C'est un ouvert de $\C^n$ qui est égal au complémentaire d'un arrangement d'hyperplans~:
\begin{equation} \label{hypComp}
U (g_0 , \ldots , g_k) = \C^n -  \cup_j H_j, \quad H_j = \mathrm{ker} (e_1^* \circ g_j ).
\end{equation}

L'action du groupe $G$ sur $\C^n$ préserve l'ensemble de ces ouverts~:
$$g \cdot U (g_0 , \ldots , g_k) = U (g_0 g^{-1} , \ldots , g_k g^{-1}).$$
Comme les variétés \eqref{hypComp} sont affines de dimension $n$, elles n'ont pas de cohomologie en degré $>n$. Nous montrons dans l'annexe \ref{A:A} qu'il correspond alors à la classe de cohomologie fournie par la proposition \ref{P4} une classe dans 
$$H^{n-1} ( G , \lim_{\substack{\rightarrow \\ H_j}} H^n (\C^n - \cup_j H_j )).$$

\subsection{Une classe de cohomologie à valeurs dans les formes méromorphes}
\'Etant donné une forme linéaire $\ell$ sur $\C^n$, on définit une forme différentielle méromorphe sur $\C^n$ par la formule
\begin{equation}
\omega_\ell =  \frac{1}{2i \pi} \frac{d\ell }{\ell }.
\end{equation}
Pour tout $g \in G$, on a 
\begin{equation} \label{E:relg}
g^* \omega_{ g \cdot \ell} = \omega_{\ell}.
\end{equation}

D'après un théorème célèbre de Brieskorn \cite[Lemma 5]{Brieskorn}, confirmant une conjecture d'Arnold, l'application naturelle $\eta \mapsto [\eta]$ de la $\Z$-algèbre graduée engendrée par les formes $\omega_{\ell}$ et l'identité vers la cohomologie singulière à coefficients entiers de \eqref{hypComp} est un isomorphisme d'algèbre. Cela justifie d'introduire la définition suivante dans notre contexte.

\begin{definition} \label{def:Omegamer}
Soit 
$$\Omega_{\rm aff}= \bigoplus_{p=0}^n \Omega_{\rm aff}^p$$ 
la $\Z$-algèbre graduée de formes différentielles méromorphes sur $\C^n$ engendrée par les formes $\omega_{\ell}$, avec $\ell \in (\C^n)^\vee - \{0 \}$, et par l'identité en degré $0$. 
\end{definition}

Le th\'eor\`eme de Brieskorn implique que l'application naturelle 
$$\Omega_{\rm aff} \to \lim_{\substack{\rightarrow \\ H_{j}}} H^\bullet (\C^n - \cup_{j} H_{j} )$$
est un isomorphisme. Finalement, on a démontré~:

\begin{proposition} \label{P:Sa}
La classe de cohomologie fournie par la proposition \ref{P4} induit une classe 
\begin{equation} \label{E:Sa}
S_{\rm aff} \in H^{n-1} (G , \Omega_{\rm aff}^n ).
\end{equation}
\end{proposition}
Nous donnons deux représentants explicites de cette classe de cohomologie au chapitre suivant.

\section{Les cocycles multiplicatif et elliptique} \label{S:23}

On considère plus généralement une famille lisse $A \to S$ de groupes algébriques commutatifs dont les fibres sont connexes et de dimension $1$. Dans les cas multiplicatif et elliptique, chaque fibre est un groupe abélien isomorphe au groupe multiplicatif $\G_m$ dont le groupe des points complexes est isomorphe à $\C^\times = \C / \Z$, {\it via} l'application 
$$ \C  \to \C^\times ; \quad  z \mapsto q_z = e(z) := e^{2i\pi z},$$
ou à une courbe elliptique. Soit $T \to S$ le produit fibré de $n$ copies de $A$ au-dessus de $S$. Le groupe $\GL_n (\Z)$ opère sur $T$ par multiplication matricielle~: on voit un élément $\mathbf{a} \in T$ comme un vecteur colonne $\mathbf{a}=(a_1 , \ldots , a_n)$ où chaque $a_i \in A$, et un élément  $g \in G$ envoie $\mathbf{a}$ sur $g\mathbf{a}$. Soit $G$ un sous-groupe de $\GL_n (\Z)$.

\begin{definition}
Soit $c$ un entier supérieur à $1$. Un \emph{cycle invariant de $c$-torsion} sur $T$ est une combinaison linéaire formelle à coefficients entiers de sections de $c$-torsion de $T$ qui est invariante par $G$, autrement dit un élément 
$$D \in H_G^0 (T[c]).$$
On dit de plus que $D$ est \emph{de degré $0$} si la somme de ses coefficients est égale à $0$.
\end{definition}

\medskip
\noindent
{\it Exemple.} Lorsque $A$ est une famille de courbes elliptiques, l'élément 
$$[T[c] - c^{2n} \{ 0 \} ] \in H_G^0 (T[c])$$ 
est un cycle invariant de $c$-torsion de degré $0$.

\medskip

L'isomorphisme de Thom induit un isomorphisme
$$H_G^0 (T[c]) \to H_G^{2n} (T , T - T[c]);$$
on pourra se référer à \cite[Section 2]{Takagi} pour plus de détails sur cet isomorphisme et la deuxième partie du lemme ci-dessous.
 
Considérons maintenant la suite exacte longue de la paire $(T , T-T[c])$~: 
\begin{equation} \label{suiteexacte}
H_G^{2n-1} (T ) \to H_G^{2n-1} ( T - T[c]) \to H_G^{2n} (T , T - T[c]) \stackrel{\delta}{\to} H_G^{2n} (T).
\end{equation}

\medskip

\begin{lem} \label{L:Sul}

{\rm (1)} Dans le cas multiplicatif, l'image dans $H_G^{2n} (T)$ d'un cycle invariant de $c$-torsion $D$ sur $T$ est {\rm rationnellement} nulle.

{\rm (2) (Sullivan \cite{Sullivan})} Dans le cas elliptique, un cycle invariant de $c$-torsion $D$ sur $T$ est de degré $0$ si et seulement si son image dans $H_G^{2n} (T)$, par l'application $\delta$ de la suite exacte \eqref{suiteexacte}, est {\rm rationnellement} nulle. Plus précisément, si $D$ est de degré $0$ son image est nulle dans $H_G^{2n} (T , \Z [1/c])$. 
\end{lem}
\begin{proof} 1. Commençons par considérer le cas où $D = \{0\}$. On veut montrer que son image $[0]$ dans $H_G^{2n} (T)$ est nulle. Puisque cette image est contenue dans le noyau du morphisme 
$$H_G^{2n} (T) \to H_G^{2n} (T- \{ 0 \})$$
induit par l'application de restriction, il suffit de montrer que son tiré en arrière par la section nulle $e=0^* [0]$ dans $H^{2n} (BG)$ est rationnellement nul. Par définition $e$ est la classe d'Euler du fibré normal de $\{0\}$ dans $EG \times_G T$ au-dessus de $BG$. Dans le cas multiplicatif il est isomorphe au fibré
$$EG \times_G \C^n \to BG$$
qui est complexe\footnote{Ce n'est plus vrai dans le cas elliptique car $E$ peut varier au-dessus de $S$. Dans ce cas on obtient un fibré en $\R^{2n}$ qui, même plat, peut avoir une classe d'Euler non nulle.} et plat. Les classes de Chern de ce fibré sont donc nulles et donc la classe d'Euler $e$ aussi. Ainsi l'image de $\{0\}$ dans $H_G^{2n} (T)$ est bien triviale.
 
Considérons maintenant le cas général où $D$ est un cycle de $c$-torsion. Son image $[c]_* (D)$, par l'application $[c]$ de multiplication dans les fibres, est égale au cycle $\{0\}$. On vient donc de montrer que l'image de la classe de $[c]_* (D)$ dans $H_G^{2n} (T)$ est nulle. L'application $[c]: T \to T$ étant un revêtement fini de degré $c^n$, le morphisme $[c]_*: H_G^{2n} (T) \to H_G^{2n} (T)$ est rationnellement injectif. L'image de $D$ dans $H_G^{2n} (T)$ est donc aussi (rationnellement) nulle.

2. Voir par exemple \cite[Lemma 9]{Takagi} pour plus de détails.
\end{proof}

Soit $D$ un cycle invariant de $c$-torsion $D$ sur $T$ que l'on supposera de plus de degré $0$ dans le cas où $A$ est une famille de courbes elliptiques. On peut alors relever $D$ en un élément de $H_G^{2n-1} (T - T[c])$. Toutefois, en général ce relevé n'est pas uniquement déterminé; l'ambiguïté est précisément $H_G^{2n-1} (T)$. On réduit cette ambiguïté en considérant la multiplication dans les fibres par un entier $s$, voir \cite{Faltings}. La multiplication dans les fibres induit une application propre $[s] : T \to T$ qui induit à son tour une application image directe $[s]_*$ en cohomologie (équivariante). 

En supposant de plus $s$ premier à $c$, on a $[s]^{-1} T[c] = T[sc]$. L'immersion ouverte
$$j : T - T[sc]  \to T -T[c]$$
induit un morphisme 
$$j^* : H^\bullet (T-T[c]) \to H^\bullet (T-T[sc] ).$$
Avec un léger (abus de notation, on notera simplement $[s]_*$ la composition 
$$\xymatrix{
H^\bullet (T-T[c]) \ar[r]^{j^*} & H^\bullet (T-T[sc] ) \ar[r]^{[s]_*} & H^\bullet (T-T[c])}$$
de l'application de restriction à $T-T[sc]$ par l'application d'image directe en cohomologie, et de même en cohomologie équivariante. On définit de même une application 
$$[s]_* : H^\bullet (T , T - T[c]) \to H^\bullet (T , T - T[c])$$
en cohomologie et aussi en cohomologie équivariante.

Noter que, puisque $s$ est premier à $c$, on a 
$$[s]_* ( [T[c] - c^{2n} \{ 0 \} ] ) = [T[c] - c^{2n} \{ 0 \} ].$$
En général, quitte à augmenter $s$, on peut supposer que $[s]_* (D) = D$. Cela motive la définition suivante. 
 
\begin{definition} \label{Def1.7}
Soit
$$H_G^\bullet (T-T[c])^{(1)} \subset H_G^\bullet (T-T[c])$$
l'intersection, pour tout entier $s>1$ premier à $c$, des sous-espaces caractéristiques de $[s]_*$ associées à la valeur propre $1$, c'est-à-dire le sous-espace des classes de cohomologie complexes qui sont envoyées sur $0$ par une puissance de $[s]_*-1$. 
\end{definition}
On définirait de même $H_G^\bullet  (T)^{(1)}$, $H_G^\bullet ( T , T - T[c])^{(1)}$, et leurs analogues $H^\bullet (T-T[c])^{(1)}$ en cohomologie usuelle. 

Comme dans le cas affine, la construction de Borel permet de calculer la cohomologie équivariante de $T$, resp. $T-T[c]$, comme cohomologie d'un espace fibré au-dessus de $BG$ de fibre $T$, resp. $T-T[c]$. On en déduit des suites spectrales compatibles à l'action de $[s]_*$~: 
\begin{equation} \label{SST}
H^p (G , H^q (T )) \Longrightarrow H^{p+q}_G (T) 
\end{equation}
et 
\begin{equation} \label{SSTm}
H^p (G , H^q (T -T[c] )) \Longrightarrow H^{p+q}_G (T-T[c]).
\end{equation}
Dans le cas elliptique, les fibres sont compactes et $H_G^k  (T)^{(1)} = \{ 0 \}$ si $k < 2n$. Les valeurs propres de $[s]_*$ sont donc des puissances $s^j$, avec $j>1$. Il en résulte que l'on peut projeter un relevé de $D$ sur le sous-espace propre associé à la valeur $1$ dans $H_G^{2n-1} (T - T[c])$; on renvoie à \cite[\S 3.2]{Takagi} pour les détails. On obtient ainsi que le cycle $D$ possède un relevé canonique dans $H_G^{2n-1} (T - T[c])^{(1)}$. 

Dans le cas multiplicatif il n'est plus vrai que $D$ possède un relevé canonique, le relevé n'est défini que modulo $H_G^{2n-1} (T)^{(1)}$. 

Comme expliqué en introduction, on voudrait maintenant restreindre cette classe à un ``ouvert affine $G$-invariant''. Un tel ouvert n'existant pas dans $T$, on considère là encore les réalisations géométriques des espaces simpliciaux correspondants. 

Le but du prochain paragraphe est de montrer, en procédant comme dans le cas additif, les deux théorèmes qui suivent. Dans les deux cas on note $\Omega_{\rm mer} (T)$ l'algèbre graduée des formes différentielles méromorphes sur $T$ et $\Omega (T)$ la sous-algèbre constituée des formes partout holomorphes sur $T$.

\begin{theorem} \label{T:cocycleM}
Supposons que $A$ soit une famille de groupes multiplicatifs. Alors, tout cycle $G$-invariant $D$ donne lieu à une classe 
$$S_{\rm mult} [D] \in H^{n-1} (G , \Omega^n_{\rm mer} (T))$$
qui est uniquement déterminée par $D$ modulo $\Omega^n (T)$.
\end{theorem}

\begin{theorem} \label{T:cocycleE}
Supposons que $A$ soit une famille de courbes elliptiques. Alors, tout cycle $G$-invariant $D$ de degré $0$ donne lieu à une classe  
$$S_{\rm ell} [D] \in H^{n-1} (G , \Omega^n_{\rm mer} (T) )$$
qui est uniquement déterminée par $D$.
\end{theorem}

\medskip
\noindent
{\it Remarque.} La construction permet en outre de montrer que si $\mathbf{a} \in T-T[c]$ est $G$-invariant alors $S_{\rm mult} [D]$, resp. $S_{\rm ell} [D]$, est cohomologue à une classe de cohomologie à valeurs dans les formes régulières en $\mathbf{a}$. Le point de vue topologique décrit ci-dessus mène ainsi naturellement à la construction de classes de cohomologie ``à la Sczech'' comme celle évoquée en introduction. 

\medskip

Sous certaines conditions supplémentaires sur le cycle $D$, nous décrivons des représentants explicites des classes $S_{\rm mult} [D]$ et $S_{\rm ell} [D]$ dans le chapitre suivant.

\section[Démonstration des théorèmes 1.7 et 1.8]{Démonstration des théorèmes \ref{T:cocycleM} et \ref{T:cocycleE}}

\subsection{Arrangement d'hyperplans trigonométriques ou elliptiques}
On fixe un groupe algébrique $A$, isomorphe au groupe multiplicatif ou à une courbe elliptique. Soit $n$ un entier naturel et soit $T=A^n$.

On appelle {\it fonctionnelle affine} toute application $\chi: T \rightarrow A$ de la forme
$$t_0 + \mathbf{a} \mapsto \chi_0(\mathbf{a})$$
où $t_0$ est un élément de $T_{\tors}$ et $\chi_0 : A^n \rightarrow A$ un morphisme de la forme
$\mathbf{a} = (a_1, \dots, a_n) \mapsto \sum r_i a_i$
où les $r_i$ sont des entiers.
 
On dit que $\chi$ est {\it primitif} si les coordonnées  
$ (r_1, \dots, r_n) \in \mathbf{Z}^n$ de $\chi_0$ sont premières entre elles dans leur ensemble. Dans ce cas le lieu d'annulation de $\chi$ est un translaté de l'ensemble
$$\mathrm{ker}(\chi_0) := \{(a_1, \dots, a_n) \in A^n \; : \; \sum r_i a_i = 0\}.$$
Soit $\mathbf{v}_1, \dots, \mathbf{v}_{n-1}$ une base du sous-module de $\Z^n$ orthogonal au vecteur $\mathbf{r}$. Les 
$\mathbf{v}_i$ définissent une application
$$A^{n-1} \longrightarrow A^n$$
qui est un isomorphisme sur son image $\mathrm{ker}(\chi_0)$. 
(On peut en effet se ramener au cas où $(r_1, \dots, r_n) = (0,\dots, 0, 1)$.) 

\begin{definition}
On appelle \emph{hyperplan} le lieu d'annulation (ou abusivement ``noyau'') d'une fonctionnelle affine primitive. 
Un \emph{arrangement d'hyperplans} $\Hyp$ est un fermé de Zariski dans $T$ réunion d'hyperplans. La taille $\# \Hyp$ est le nombre d'hyperplans distincts de cet arrangement. 
\end{definition}

Noter que de manière équivalente, un hyperplan est l'image d'une application  
$A^{n-1} \rightarrow T$ linéaire relativement à un morphisme $A^{n-1} \rightarrow A^n$ induit par une matrice entière de taille $(n-1) \times n$.

\begin{lemma} \label{affine}
Si $\Hyp$ contient $n$ fonctionnelles affines $\chi$ dont les vecteurs associés $\mathbf{r} \in \Z^n$ sont linéairement indépendants alors 
le complémentaire $T-\Hyp$ est affine. 

Lorsque $A$ est une courbe elliptique, c'est même une équivalence. 
\end{lemma}
\proof
S'il existe $n$ fonctionnelles affines dans $\Hyp$ dont les vecteurs de $\Z^n$ associés sont linéairement indépendants alors ces fonctionnelles définissent une 
application finie $T \rightarrow A^n$ et $\Hyp$ est la pré-image de la réunion des axes de coordonnées dans $A^n$. Mais $(A-\{0\})^n$ est affine, et la pré-image d'une variété affine par une application finie est encore affine. 

Maintenant, si $A$ est une courbe elliptique et que l'espace engendré par les vecteurs des fonctionnelles affines qui définissent $\Hyp$ est un sous-module propre de $M \subset \Z^n$ alors la donnée d'un point de $T-\Hyp$ et d'un vecteur de $\Z^n$ orthogonal à $M$ définit un plongement 
$$A \longrightarrow T - \Hyp.$$ 
Comme $A$ n'est pas affine, l'espace $T-\Hyp$ ne l'est pas non plus.
\qed

\subsection[Cohomologie des arrangements d'hyperplans]{Opérateurs de dilatation et cohomologie des arrangements d'hyperplans} \label{S:dil}

On appelle {\em application de dilatation} toute application $[s] : T \rightarrow T$ associée à un entier $s>1$ et de la forme
$$[s] : \mathbf{a} \mapsto s \mathbf{a}.$$
L'image d'un hyperplan par une application de dilatation est encore un hyperplan. 

Tout hyperplan est l'image d'un sous-groupe par un translation par point de \emph{torsion} dans $T$. \'Etant donné un arrangement d'hyperplans, on peut donc trouver une application de dilatation $[s]$ qui préserve cet arrangement, c'est-à-dire telle que $[s] \Hyp \subset \Hyp$.  

Puisque $[s]$ préserve $\Hyp$, on a une immersion ouverte 
$$j :  T - [s]^{-1} \Hyp  \to T -  \Hyp .$$
La dilatation $[s]$ induit une application de $T-[s]^{-1}\Hyp$ vers $T-\Hyp$ qui est à fibres finies. En abusant légèrement des notations on note $[s]_*$ la composition 
$$\xymatrix{
H^\bullet (T-\Hyp) \ar[r]^-{j^*} & H^\bullet (T-[s]^{-1} \Hyp) \ar[r]^-{[s]_*} & H^\bullet (T-\Hyp)}$$
de l'application de restriction à $T-[s]^{-1} \Hyp$ par l'application d'image directe de $[s]$ en cohomologie.

On peut alors poser l'analogue suivant de la définition \ref{Def1.7}. 
\begin{definition} \label{D:sep1}
Soit 
$$H^*(T-\Hyp, \C)^{(1)} \subset H^*(T-\Hyp, \C)$$
l'intersection, pour tout entier $s>1$ tel que la dilatation $[s]$ préserve $\Hyp$, des sous-espaces caractéristiques de $[s]_*$ associé à la valeur propre $1$, c'est-à-dire le sous-espace des classes de cohomologie complexes qui sont envoyées sur $0$ par une puissance de $[s]_*-1$. 
\end{definition}

\subsection{Démonstration des théorèmes \ref{T:cocycleM} et \ref{T:cocycleE}}

On se place maintenant dans le cas où la famille lisse $A \to S$ de groupes algébriques commutatifs est soit simplement $A = \mathbf{G}_m$ ou une famille de courbes elliptiques au-dessus d'une courbe modulaire. 

\`A tout vecteur ligne $v\in \Z^n$ dont les coordonnées sont premières entre elles dans leur ensemble, il correspond la fonctionnelle linéaire primitive
$$\chi_{v} : T \to A; \ \mathbf{a}  \mapsto v \mathbf{a}.$$

Il découle du lemme \ref{L:Sul} que, sous les hypothèses des théorèmes \ref{T:cocycleM} et \ref{T:cocycleE}, le cycle $D$ se relève en un élément de  
$$H_G^{2n-1} (T-T[c]) = H^{2n-1} (EG \times_G (T-T[c]))$$
et donc de 
$$H^{2n-1} (\| EG \times_G (T-T[c]) \| ).$$

\`A tout $(k+1)$-uplet $\mathbf{g} \in (EG)_k$ on associe un ouvert 
\begin{equation*} \label{U}
U (\mathbf{g})  = \left\{ \mathbf{a} \in  T \; \Bigg| \; \begin{array}{l} \forall j \in \{ 0 , \ldots , k\}, \ \forall i \in \{ 0 , \ldots , n\}, \\ \chi_{e_i g_j } (\mathbf{a}) \notin A[c] \end{array} \right\} .
\end{equation*}
C'est le complémentaire d'un arrangement d'hyperplans dans $T$~:
\begin{equation} \label{E:arrgtHyp}
U (\mathbf{g})  = T - \cup_{j=0}^k \cup_{i=1}^n \cup_{a \in A[c]} \chi_{e_i g_j }^{-1} (a).
\end{equation}

L'action du groupe $G$ préserve l'ensemble de ces ouverts~:
$$h \cdot U (\mathbf{g}) = U (\mathbf{g} h^{-1}).$$
Il découle par ailleurs du lemme \ref{affine} que les variétés \eqref{E:arrgtHyp} sont affines; elles n'ont par conséquent pas de cohomologie en degré $>n$. Comme expliqué dans l'annexe~\ref{A:A}, il correspond à tout élément dans $H^{2n-1}_G (T -T[c])$ une classe de cohomologie dans 
$$H^{n-1} (G , \lim_{\substack{\rightarrow \\ \Xi }} H^n (T - \cup_{\chi \in \Xi} \cup_{a \in A[c]} \chi^{-1} (a))),$$
où $\Xi$ désigne l'ensemble des translatés par $G$ des morphismes $\chi_{e_1}, \ldots , \chi_{e_n}$. Noter que toute classe dans $H^n ( U (\mathbf{g}))$ définie un élément de la limite inductive. En pratique, nos cocycles seront représentés par des formes régulières sur des $U (\mathbf{g})$.

L'élément de $H^{2n-1}_G (T -T[c])$ que nous considérons appartient à $H^{2n-1}_G (T -T[c])^{(1)}$, on obtient donc en fait une classe de cohomologie dans 
$$H^{n-1} (G , \lim_{\substack{\rightarrow \\ \Xi }} H^n (T - \cup_{\chi \in \Xi} \cup_{a \in A[c]} \chi^{-1} (a))^{(1)}).$$
Il nous reste à représenter 
$$\lim_{\substack{\rightarrow \\ \Xi }} H^n (T - \cup_{\chi \in \Xi} \cup_{a \in A[c]} \chi^{-1} (a))^{(1)},$$
par des formes méromorphes. C'est l'objet du chapitre \ref{S:OrlikSolomon} dans lequel nous démontrons un théorème ``à la Brieskorn'' dans ce contexte, cf. Théorème \ref{P:Brieskorn}. 

Finalement le cycle $D$ donne donc lieu à un élément de $H^{n-1} (G , \Omega^n_{\rm mer} (T) )$. Dans le cas elliptique cet élément est uniquement déterminé. Ce n'est pas vrai dans le cas multiplicatif. Alors $T$ est elle-même affine et on en déduit un diagramme commutatif
$$\xymatrix{
H_G^{2n-1} (T)^{(1)} \ar[d] \ar[r] & H_G^{2n-1} (T-T[c] )^{(1)} \ar[d] \ar[r] & H_G^{2n} (T , T - T[c])^{(1)} \\ 
H^{n-1} (G , H^n (T)^{(1)} ) \ar[r] & H^{n-1} (G ,  \Omega^n_{\rm mer} (T))). & 
}$$
La classe associée à $D$ dans $H^{n-1} (G , \Omega^n_{\rm mer} (T) )$ n'est donc déterminée qu'à un élément de $H^{n-1} (G , H^n (T)^{(1)} )$ près. En invoquant encore une fois le théorème~\ref{P:Brieskorn} on identifie cette indétermination à un élément de 
$H^{n-1} (G , \Omega^n (T) )$. 

Pour conclure, notons que la remarque qui suit les théorèmes  \ref{T:cocycleM} et \ref{T:cocycleE} se démontre en partant non plus des hyperplans 
$\chi_{e_j}^{-1} (A[c])$, pour $j \in \{1, \ldots , n \}$, translatés par les éléments de $G$ mais de $n$ hyperplans ne passant pas par $\mathbf{a}$, ce que l'on peut faire de manière $G$-équivariante puisque $\mathbf{a}$ est $G$-invariant.

\chapter{Énoncés des principaux résultats : cocycles explicites} \label{C:2}

\resettheoremcounters

\numberwithin{equation}{chapter}

Dans ce chapitre on décrit dans chacun des trois cas (affine, multiplicatif, elliptique) des cocycles explicites représentants les classes de cohomologie construites au chapitre précédent. Les démonstrations des résultats énoncés ici feront l'objet des chapitres suivants. 

\section[Le cas affine]{Le cas affine : symboles modulaires universels et algèbre de Orlik--Solomon} 

Un théorème célèbre de Orlik et Solomon \cite{OrlikSolomon} fournit une présentation, par générateurs et relations, de l'algèbre graduée $\Omega_{\rm aff}$ engendrée par les formes $\omega_{\ell}$ et l'identité. En particulier dans $\Omega^n_{\rm aff}$ l'ensemble des relations, entre les monômes de degré $n$, est engendré par  
\begin{enumerate}
\item $\omega_{\ell_1} \wedge \ldots \wedge  \omega_{\ell_{n}} = 0$ si $\det (\ell_1 , \ldots , \ell_{n} ) = 0$, et 
\item $\sum_{i=0}^n (-1)^i \omega_{\ell_0} \wedge \ldots \wedge \widehat{\omega_{\ell_i}} \wedge \ldots \wedge \omega_{\ell_n} = 0$, pour tous $\ell_0 , \ldots , \ell_n$ dans $(\C^n)^\vee  - \{ 0 \}$.
\end{enumerate}
Le fait que les relations ci-dessus soient effectivement vérifiées dans $\Omega^n_{\rm aff}$ n'est pas difficile, il est par contre remarquable qu'elles engendrent \emph{toutes} les relations. 

Dans ce paragraphe on commence par expliquer que le fait que les relations soient vérifiées donne naturellement lieu à un cocycle de $G= \GL_n (\C)^\delta$ à valeurs dans $\Omega^n_{\rm aff}$. On énonce alors un théorème qui relie ce cocycle à celui construit au chapitre \ref{C:1}. Finalement on explique que ce cocycle est la spécialisation d'un symbole modulaire universel.

\subsection{Un premier cocycle explicite}

De manière générale, si $X$ est un ensemble muni d'une action transitive de $G$, si $M$ est un $G$-module, si $F : X^n \to M$ est une fonction $G$-équivariante vérifiant 
$$\sum_{i=0}^{n} (-1)^i F (x_0 , \ldots , \widehat{x}_i , \ldots , x_{n} ) = 0,$$
et si $x$ est un point dans $X$, alors 
$$f_x (g_1 , \ldots , g_{n} ) := F(g_1^{-1} x , \ldots , g_{n}^{-1} x )$$
définit un $(n-1)$-cocycle du groupe $G$ à valeurs dans $M$. De plus, la classe de cohomologie représentée par ce cocycle ne dépend pas de $x$.

En appliquant ce principe général à 
$$X = (\C^n )^\vee -\{0 \}, \quad M = \Omega^n_{\rm aff}, \quad F (\ell_1 , \ldots , \ell_{n}) = \omega_{\ell_1} \wedge \ldots \wedge \omega_{\ell_{n}} \quad  \mbox{et} \quad  x=e_1^*,$$
on obtient un $(n-1)$-cocycle homogène 
\begin{equation} \label{cocycleSa}
\mathbf{S}_{\rm aff} : G^{n} \to \Omega_{\rm aff}^n; \quad (g_1 , \ldots , g_{n} ) \mapsto \omega_{\ell_1} \wedge \ldots \wedge \omega_{\ell_{n}},
\end{equation}
où $\ell_j (z) = e_1^*(g_j z)$. 

\medskip
\noindent
{\it Remarque.} Le cocycle ainsi construit est homogène à droite, ce qui se traduit donc par la relation
\begin{equation} \label{invSa}
g^* \mathbf{S}_{\rm aff} (g_1 g^{-1} , \ldots , g_{n} g^{-1}) = \mathbf{S}_{\rm aff} (g_1 , \ldots , g_{n} ) ,
\end{equation}
qui découle de l'équation (\ref{E:relg}). 

\medskip

\begin{theorem} \label{T:Sa} 
Le cocycle $\mathbf{S}_{\rm aff}$ représente la classe (non nulle) 
$$S_{\rm aff} \in H^{n-1} (G , \Omega_{\rm aff}^n )$$ 
de la proposition \ref{P:Sa}. 
\end{theorem}

\subsection{Immeuble de Tits et symboles modulaires universels} \label{S:ARuniv}

Considérons maintenant l'immeuble de Tits $\mathbf{T}_n$. C'est le complexe simplicial dont les sommets sont les sous-espaces propres non-nuls de $\C^n$ et dont les simplexes sont les drapeaux de sous-espaces propres. Rappelons (voir \S \ref{S:Tits}) que l'immeuble de Tits s'identifie naturellement au bord à l'infini de l'espace symétrique associé à $\GL_n (\C)$ dans la compactification géodésique. D'après le théorème de Solomon--Tits \cite{SolomonTits}, l'immeuble de Tits $\mathbf{T}_n$ a le type d'homotopie d'un bouquet de $(n-2)$-sphères. Son homologie réduite en degré $n-2$ est appelé module de Steinberg de $\C^n$; on note donc 
$$\mathrm{St} (\C^n ) = \widetilde{H}_{n-2} (\mathbf{T}_n ).$$
Ash et Rudolph décrivent un ensemble explicite de générateurs de $\mathrm{St} (\C^n )$, appelés {\it symboles modulaires universels}, de la manière suivante~: soient $v_1 , \ldots , v_n$ des vecteurs non nuls de $\C^n$. Identifiant le bord de la première subdivision barycentrique $\Delta_{n-1} '$ du $(n-1)$-simplexe standard au complexe simplicial dont les sommets sont les sous-ensembles propres non vides de $\{ 1 , \ldots , n \}$, on associe aux vecteurs $v_1 , \ldots , v_n$ l'application simpliciale 
\begin{equation} \label{app1}
\partial \Delta_{n-1} ' \to \mathbf{T}_n
\end{equation}
qui envoie chaque sommet $I \subsetneq \{ 1 , \ldots , n \}$ de $\partial \Delta_{n-1} '$ sur le sommet $\langle v_i \rangle_{i \in I}$ de $\mathbf{T}_n$. Le symbole modulaire universel $[v_1 , \ldots , v_n] \in \mathrm{St}(\C^n)$ est alors défini comme l'image de la classe fondamentale de $\partial \Delta_{n-1} '$ par l'application \eqref{app1}. D'après \cite[Prop. 2.2]{AshRudolph} le symbole $[v_1 , \ldots , v_n]$ vérifie les relations suivantes.
\begin{enumerate}
\item Il est anti-symétrique (la transposition de deux vecteurs change le signe du symbole).
\item Il est homogène de degré $0$~: pour tout $a \in \C^*$, on a $[ a v_1 , \ldots , v_n] =[v_1 , \ldots , v_n]$.
\item On a $[v_1 , \ldots , v_n]=0$ si $\det (v_1 , \ldots , v_n)=0$.
\item Si $v_0 , \ldots, v_{n}$ sont $n+1$ vecteurs de $\C^n$, on a
$$\sum_{j=0}^n (-1)^j [ v_0 , \ldots , \widehat{v}_j , \ldots , v_n] =0 .$$
\item Si $g \in \GL_n (\C)$, alors $[g v_1 , \ldots , g v_n ] = g \cdot [v_1 , \ldots , v_n]$, où le point désigne l'action naturelle de $\GL_n (\C)$ sur $\mathrm{St}(\C^n)$.
\end{enumerate}
D'après \cite[Prop. 2.3]{AshRudolph} les symboles modulaires universels engendrent $\mathrm{St}(\C^n)$. Kahn et Sun \cite[Corollary 2]{KahnSun} montrent que les relations ci-dessus fournissent en fait une {\it présentation} de $\mathrm{St}(\C^n)$.

Comme pour les symboles modulaires classiques discutés en introduction, étant donné un $G$-module $M$ une application $G$-équivariante
$\Phi : \mathrm{St}(\C^n) \to M$ induit un $(n-1)$-cocycle de $G$ à coefficients dans $M$~:
$$(g_1 , \ldots , g_n ) \mapsto \Phi (g_1^{-1} e_1 , \ldots , g_n^{-1} e_1).$$ 

\subsection{Un deuxième cocycle explicite}

Dans le cas affine notre principal résultat est le suivant.

\begin{theorem} \label{T:Sa} 
L'application
\begin{multline} \label{E:applAff}
\mathrm{St}(\C^n) \to \Omega^n_{\rm aff};  \\ \quad [v_1 , \ldots , v_n] \mapsto \left\{ \begin{array}{ll}
0 & \mbox{si } \det (v_1 , \ldots , v_n)=0, \\
\omega_{v_1^*} \wedge \ldots \wedge \omega_{v_n^*} & \mbox{sinon}, \end{array} \right.
\end{multline}
où, dans le deuxième cas, $(v_1^* , \ldots , v_n^*)$ désigne la base duale à $(v_1 , \ldots , v_n)$, induit un $(n-1)$-cocycle $\mathbf{S}_{\rm aff}^*$ qui représente encore la classe (non nulle) $S_{\rm aff}$ dans $H^{n-1} (G , \Omega_{\rm aff}^n )$. 
\end{theorem}

Le cocycle $\mathbf{S}_{\rm aff}^*$ est explicitement donné par
\begin{equation} \label{cocycleSa*}
\mathbf{S}_{\rm aff}^* : G^{n} \to \Omega_{\rm aff}^n; \quad (g_1 , \ldots , g_{n} ) \mapsto \omega_{\ell_1} \wedge \ldots \wedge \omega_{\ell_{n}},
\end{equation}
où cette fois $\ell_j$ est une forme linéaire sur $\C^n$ de noyau $\langle g_1^{-1} e_1 , \ldots , \widehat{g_j^{-1} e_1} , \ldots , g_{n}^{-1} e_1 \rangle$ (et identiquement nulle si les $g_j^{-1} e_1$ ne sont pas en position générale). Il découle encore immédiatement de l'équation \eqref{E:relg} au chapitre \ref{C:1}  que $\mathbf{S}_{\rm aff}^*$ est homogène à droite, autrement dit qu'il vérifie la relation \eqref{invSa}. Il est par contre un peu moins évident qu'il définisse bien un cocycle. Les deux cocycles $\mathbf{S}_{\rm aff}$ et $\mathbf{S}_{\rm aff}^*$ sont considérés par Sczech dans une note non publiée \cite{Sczechprepub}; le cocycle $\mathbf{S}_{\rm aff}$ est le point de départ d'un article important de Sczech \cite{Sczech93} sur lequel nous reviendrons. Nous lui préférons $\mathbf{S}_{\rm aff}^*$ précisément parce qu'il provient de \eqref{E:applAff}. 

Du fait que $\mathbf{S}_{\rm aff}^*$ provienne de l'application \eqref{E:applAff} on peut penser à ce cocycle comme à une classe de cohomologie \emph{relative} au bord de Tits. Nous donnerons un sens rigoureux à cela au cours de la démonstration que nous détaillons au chapitre \ref{S:6}. Notre démarche pour démontrer le théorème \ref{T:Sa} consiste à partir de la description topologique (\ref{E:Sa}) de $S_{\rm aff}$ et d'exhiber un représentant explicite grâce à la théorie de Chern--Weil. Outre qu'il permet de montrer que les cocycles $\mathbf{S}_{\rm aff}$ et $\mathbf{S}_{\rm aff}^*$ sont cohomologues et trouvent leur origine dans le relevé dans $H_G^{2n-1} (\C^n -\{ 0 \})$ de la classe fondamentale de $\C^n - \{ 0 \}$, l'avantage de ce point de vue est qu'il se généralise naturellement dans les cas multiplicatif et elliptique que nous discutons dans les paragraphes qui suivent. 

\medskip
{\it Remarques.} 1. On peut reformuler le résultat de Ash--Rudolph \cite[Prop. 2.2]{AshRudolph} cité plus haut de la manière suivante~:
\begin{quote}
L'application qui à $(g_1 , \ldots , g_n ) \in G^n$ associe le symbole modulaire $[g_1^{-1} e_1 , \ldots , g_n^{-1} e_1]$ dans $\mathrm{St} (\C^n)$ est un $(n-1)$-cocycle homogène. 
\end{quote}
Dans l'annexe \ref{A:B} on donne une démonstration topologique de cette assertion qui réalise la classe de cohomologie associée comme une classe d'obstruction; on pourra comparer cette manière de voir avec \cite{SharifiVenkatesh}.  

2. L'application 
$$\phi : \mathrm{St} (\C^n ) \to \mathrm{St}( (\C^n )^\vee )$$
qui à un symbole $[v_1 , \ldots , v_n]$ associe $0$ si $\det (v_1 , \ldots , v_n)= 0$, et $[v_1^* , \ldots , v_n^*]$ sinon, est un isomorphisme $G$-équivariant.

Le cocycle $\mathbf{S}_{\rm aff}$ se déduit alors de l'application $G$-équivariante 
\begin{equation} \label{R:Sa}
\mathbf{S}_{\rm aff}^* \circ \phi^{-1} : \mathrm{St} ((\C^n )^\vee ) \to \Omega_{\rm aff}^n; \quad [\ell_1 , \ldots , \ell_n ] \mapsto \omega_{\ell_1} \wedge \ldots \wedge \omega_{\ell_n}.
\end{equation}

3. Les applications \eqref{R:Sa} et \eqref{E:applAff} sont $G$-équivariantes et surjectives. Or Andrew Putman et Andrew Snowden \cite{Putman} ont récemment démontré l'irréductibilité de la représentation de Steinberg $\mathrm{St} (\C^n)$. Les applications \eqref{R:Sa} et \eqref{E:applAff} sont donc des isomorphismes. Il s'en suit que les relations dans $\Omega^n_{\rm aff}$ sont engendrées par les images 
des relations de Ash--Rudolph; on retrouve donc que les relations de Orlik et Solomon engendrent {\it toutes} les relations entre les formes $\omega_\ell$ dans $\Omega^n_{\rm aff}$.

\section[Le cas multiplicatif]{Le cas multiplicatif : formes différentielles trigonométriques et symboles modulaires} \label{S:2-2}

Considérons maintenant l'algèbre graduée $\Omega_{\rm mer} = \Omega_{\rm mer} ((\C^\times )^n)$ des formes méromorphes sur le produit de $n$ copies du groupe multiplicatif $\C^\times$ que l'on identifie au quotient $\C/ \Z$ {\it via} l'application 
$$ \C / \Z \to \C^\times; \quad z \mapsto q=e^{2i\pi z}.$$
Rappelons que la fonction $\varepsilon (z) = \frac{1}{2i} \cot (\pi z)$ est égale à la somme --- régularisée au sens de Kronecker --- de la série 
$$\frac{1}{2i \pi} \sum_{m \in \Z} \frac{1}{z+m}.$$
\'Etant $\Z$-périodique la forme $\varepsilon (z) dz$ définit bien une forme méromorphe sur $\C/ \Z$. {\it Via} l'identification $\C / \Z \cong \C^\times$ rappelée ci-dessus, on a 
$$\varepsilon (z) dz = \frac{1}{2i\pi} \frac{dq}{q-1} - \frac{1}{4i \pi} \frac{dq}{q} \quad \mbox{et} \quad dz = \frac{1}{2i\pi} \frac{dq}{q}.$$
On note finalement $\overline{\Omega}_{\rm mer}$ le quotient de $\Omega_{\rm mer} ((\C^\times )^n)$ par la sous-algèbre engendrée par les formes régulières 
$$dz_j = \frac{1}{2i\pi}  \frac{dq_j}{q_j} \quad (j \in \{ 1 , \ldots , n \} ).$$
L'action (à gauche) des matrices $n \times n$ sur $\C^n$ (identifiées aux vecteurs colonnes) induit une action du monoïde $M_n (\Z )$
sur $\C^n / \Z^n$ et donc une action du groupe $\SL_n (\Z)$. On commence par rappeler dans ce contexte les définitions des opérateurs de Hecke puis, comme dans le cas additif on décrit un premier cocycle explicite avant de faire le lien avec les symboles modulaires.

\subsection{Opérateurs de Hecke} 

Soit $\Gamma$ un sous-groupe d'indice fini de $\SL_n (\Z)$ et soit $S$ un sous-monoïde de $M_n (\Z )^\circ = M_n (\Z ) \cap \GL_n (\Q)$ contenant $\Gamma$. 
L'action, \emph{à droite}, de $M_n (\Z )^\circ$ sur $\Omega_{\rm mer}^n$, par tiré en arrière, induit une action de l'algèbre de Hecke associée au couple $(S, \Gamma)$ sur $H^{n-1} (\Gamma ,\Omega_{\rm mer}^n)$~: une double classe 
$$\Gamma a \Gamma \quad \mbox{avec} \quad a \in S$$
induit un opérateur --- dit de Hecke --- sur $H^{n-1} (\Gamma ,\Omega_{\rm mer}^n)$, noté $\mathbf{T}(a)$, que l'on décrit comme suit~: 
on décompose la double classe
$$\Gamma a \Gamma  = \sqcup_j \Gamma a_j $$
où l'union est finie. Pour tout $g \in \Gamma$ on peut donc écrire  
$$a_j g^{-1}  = (g^{(j)})^{-1} a_{\sigma (j)}  \quad \mbox{avec} \quad \sigma \mbox{ permutation et } g^{(j)} \in \Gamma.$$ 
\'Etant donné un cocycle (homogène à droite) $c : \Gamma^n \to \Omega^n_{\rm mer}$, on pose alors
$$\mathbf{T} (a) c (g_1 , \ldots , g_n) = \sum_j a_{j}^*c (g_1^{(j)} , \ldots , g_n^{(j)});$$
c'est encore un cocycle homogène (à droite) et on peut montrer que sa classe de cohomologie est indépendante du choix des $a_j$; voir \cite{RhieWhaples}.  

Noter par ailleurs qu'un élément $a \in S$ induit une application $[a] : \C^n / \Z^n \to \C^n / \Z^n$. Notons $\mathrm{Div}_\Gamma$ le groupe (abélien) constitué des combinaisons linéaires entières formelles $\Gamma$-invariantes de points de torsion dans $\C^n / \Z^n$. Dans la suite on note $[\Gamma a \Gamma]$ l'application 
$$[\Gamma a \Gamma] = \sum_j [a_j] :  \mathrm{Div}_\Gamma \to \mathrm{Div}_\Gamma.$$

\subsection{Cocycles multiplicatifs I} 

\'Etant donné un élément $D \in \mathrm{Div}_\Gamma$, la théorie de Chern--Weil nous permettra de construire, au chapitre \ref{S:chap9}, des  représentants suffisamment explicites de la classe de cohomologie $S_{\rm mult} [D]  \in H^{n-1} (\Gamma , \Omega_{\rm mer}^n )$ du théorème \ref{T:cocycleM} pour démontrer le théorème qui suit au \S \ref{S:3.2.2}.

\begin{theorem} \label{T:mult}
Soit $\chi_0 : \C^{n} / \Z^n \to \C / \Z$ un morphisme primitif. Il existe une application linéaire 
$$\mathrm{Div}_\Gamma \to C^{n-1} (\Gamma , \Omega_{\rm mer}^n )^{\Gamma}; \quad D \mapsto \mathbf{S}_{\rm mult, \chi_0} [D]$$ telle que 
\begin{enumerate}
\item chaque cocycle $\mathbf{S}_{\rm mult , \chi_0} [D]$ représente $S_{\rm mult} [D]  \in H^{n-1} (\Gamma , \Omega_{\rm mer}^n )$; 
\item chaque forme différentielle méromorphe $\mathbf{S}_{\rm mult , \chi_0} [D] (g_1 , \ldots , g_n)$ est régulière en dehors des hyperplans affines passant par un point du support de $D$ et dirigés par $\mathrm{ker} ( \chi_0 \circ g_j)$ pour un $j \in \{1 , \ldots , n \}$;
\item pour tout entier $s >1$, on a 
$$\mathbf{S}_{\rm mult , \chi_0} [[s]^*D] = [s]^* \mathbf{S}_{\rm mult , \chi_0}[D] \quad \mbox{\emph{(relations de distribution)} et}$$  
\item pour tout $a \in S$, 
$$\mathbf{T} (a) \mathbf{S}_{\rm mult , \chi_0} [D]  = \mathbf{S}_{\rm mult , \chi_0} [ [\Gamma a \Gamma]^* D] \quad \mbox{dans} \quad H^{n-1} (\Gamma , \Omega_{\rm mer}^n ).$$
\end{enumerate}
\end{theorem}

\medskip
\noindent
{\it Exemple.} Lorsque $\Gamma = \SL_n (\Z)$ et $S = M_n (\Z )^\circ$, l'algèbre de Hecke est engendrée par les opérateurs 
$$\mathbf{T}^{(k)}_p = \mathbf{T} (a^{(k)}_p) \quad \mbox{avec}  \quad a^{(k)}_p=\mathrm{diag} (\underbrace{p, \ldots , p}_{k} , 1 , \ldots , 1 ),$$
où $p$ est un nombre premier et $k$ un élément de $\{ 1 , \ldots , n-1 \}$.
Maintenant, le tiré en arrière de $D_0 = \{ 0 \}$ par l'application $[\Gamma a^{(k)}_p \Gamma]$ est supporté sur l'ensemble de tous les points de $p$-torsion comptés avec multiplicité $\left( \begin{smallmatrix} n-1 \\ k-1 \end{smallmatrix} \right)_p$ sauf $0$ compté avec multiplicité $\left( \begin{smallmatrix} n \\ k \end{smallmatrix} \right)_p$.\footnote{Ici 
$$ \left( \begin{smallmatrix} n \\ k \end{smallmatrix} \right)_p = \frac{(p^n-1) \cdots (p^{n-k+1} -1)}{(p^k-1) \cdots (p-1)} = \frac{(p^n-1)(p^n-p) \cdots (p^n-p^{k-1})}{(p^k-1) (p^k-p) \cdots (p^k - p^{k-1})}$$
est le coefficient $p$-binomial de Gauss, égal au nombre de sous-espaces vectoriels de dimension $k$ dans $\mathbf{F}_p ^n$.} On en déduit que 
$$[\Gamma a^{(k)}_p \Gamma]^* D_0 = \left( \begin{smallmatrix} n-1 \\ k-1 \end{smallmatrix} \right)_p [p]^*D_0 + \left( \left( \begin{smallmatrix} n \\ k \end{smallmatrix} \right)_p - \left( \begin{smallmatrix} n-1 \\ k-1 \end{smallmatrix} \right)_p \right) D_0$$
et donc que la classe de cohomologie de $\mathbf{S}_{\rm mult, \chi_0} [D_0 ]$ annule l'opérateur 
$$\mathbf{T}^{(k)}_p - \left( \begin{smallmatrix} n-1 \\ k-1 \end{smallmatrix} \right)_p [p]^* -   \left( \begin{smallmatrix} n \\ k \end{smallmatrix} \right)_p + \left( \begin{smallmatrix} n-1 \\ k-1 \end{smallmatrix} \right)_p.$$
On retrouve ainsi les deux premiers points du théorème énoncé en introduction, le lien avec le cocycle $\mathbf{S}$ est plus précisément que pour $n=2$,
$$\mathbf{S}_{\rm mult , e_1^*} [D_0] (1 , g) = \mathbf{S}_{[1]} (g^{-1}) dx \wedge dy ,$$
où $x$ et $y$ sont les coordonnées, abscisse et ordonnée, dans $\C^2 / \Z^2$. 

\medskip

\subsection{Symboles modulaires}

Soit 
$$\Delta_n (\Z ) \subset \mathrm{St} (\C^n)$$
le sous-groupe abélien engendré par les symboles 
$$[h] = [v_1 , \ldots , v_n] \quad \mbox{où} \quad h = ( v_1 | \cdots | v_n ) \in M_n (\Z)^\circ .$$
On appelle \emph{symbole modulaire} tout élément $[h] = [v_1 , \ldots , v_n]$ de $\Delta_n (\Z)$. 

Par définition le $\Z$-module $\Delta_n (\Z)$ est égal au quotient de $\Z [M_n (\Z)^\circ ]$ par les relations (1), (2), (3) et (4) de Ash--Rudolph. On note $I_n \subset \Z [\SL_n (\Z)]$ le sous-module engendré par les éléments 
\begin{equation} \label{E:eltIdeal}
[h] + [hR] , \quad [h] + (-1)^n [hP] \quad \mbox{et} \quad [h] + [hU] + [hU^2] \quad (h \in \SL_n (\Z) ),
\end{equation}
avec
\begin{equation*}
R=(-e_2 |e_1 | e_3 | \cdots | e_n),  \quad P = (e_2 | e_3 | \cdots | e_n | (-1)^{n+1} e_1)  
\end{equation*}
et 
$$U= (-e_1 - e_2 | e_1 | e_3 | \cdots | e_n ).$$
Bykovskii \cite{Bykovskii} démontre que 
\begin{equation} \label{E:byk}
\Delta_n (\Z) = \Z [ \SL_n (\Z) ] / I_n .
\end{equation}
Un sous-groupe d'indice fini $\Gamma$ dans $\SL_n (\Z)$ opère naturellement à gauche sur $\Delta_n (\Z)$ et sur $\Z[\SL_n (\Z )]$, et \eqref{E:byk} est une identité entre $\Z [\Gamma]$-modules.

Dans l'introduction on a exhibé, dans le cas $n=2$, un premier lien entre les structures de $\Delta_2 (\Z)$ et $\Omega^2_{\rm mer}$.
La situation générale, où $n$ est arbitraire, est plus subtile; elle fait l'objet du théorème qui suit. Commençons par naturellement étendre les définitions vues en introduction. 

Le monoïde $M_n (\Z )^\circ$ opère \emph{à droite} sur $\mathrm{Hom} (\Delta_n (\Z) , \Omega_{\rm mer}^n )$ par 
$$\phi_{|g} ([v_1 , \ldots , v_n]) =  g^* \phi ([gv_1 , \ldots , gv_n ]) \quad (\phi \in \mathrm{Hom} (\Delta_n (\Z) , \Omega_{\rm mer}^n ), \ g \in M_n (\Z )^\circ).$$
Cette action induit en particulier une action de $\SL_n (\Z)$ et, étant donné un sous-groupe d'indice fini $\Gamma \subset \SL_n (\Z)$, on appelle \emph{symbole modulaire} pour $\Gamma$ à valeurs dans $\Omega_{\rm mer}^n$ un élément de 
$$\mathrm{Hom} (\Delta_n (\Z) , \Omega_{\rm mer}^n )^{\Gamma }.$$

Soit maintenant $C$ un sous-ensemble $S$-invariant\footnote{Et donc aussi $\Gamma$-invariant.} de vecteurs non-nuls dans $\Z^n$. On note $\Delta_C \subset \Delta_n (\Z)$ le sous-groupe engendré par les $[v_1 , \ldots , v_n]$, où chaque $v_j$ appartient à $C$.

\begin{definition} 
Un \emph{symbole modulaire partiel} sur $C$ pour $\Gamma$ à valeurs dans $\Omega^n_{\rm mer}$ est un élément de 
$$\mathrm{Hom} (\Delta_C , \Omega^n_{\rm mer} )^\Gamma.$$
\end{definition}
Soit $v_0 \in C$. Un symbole modulaire partiel $\phi$ donne lieu à un $(n-1)$-cocycle à valeurs dans $\Omega^n_{\rm mer}$
$$c_\phi  : (g_1 , \ldots , g_n) \mapsto \phi (g_1^{-1} v_0 , \ldots , g_n^{-1} v_0 )$$
dont la classe de cohomologie ne dépend pas du choix de $v_0$ dans $C$; dans la suite on prendra toujours $v_0 = e_1$. Les opérateurs de Hecke opèrent sur les symboles modulaires partiels par
$$\mathbf{T} (a) \phi = \sum_j \phi_{|a_j} \quad \left( a \in S , \ \phi \in \mathrm{Hom} (\Delta_C , \Omega^n_{\rm mer} )^\Gamma \right),$$
de sorte que 
$$\mathbf{T} (a) [c_\phi] = [c_{\mathbf{T} (a) \phi} ] \in H^{n-1} (\Gamma , \Omega^n_{\rm mer}) .$$

\subsection{Cocycles multiplicatifs II}

Un élément de $\mathrm{Div}_\Gamma$ est une fonction $\Gamma$-invariante $D : \Q^n / \Z^n \to \Z$ dont le support est contenu dans un réseau de $\Q^n$. 

\begin{definition}
Soit $\mathrm{Div}_\Gamma^{\circ}$ le noyau du morphisme  
$$\mathrm{Div}_\Gamma \to \Z [ (\Q/ \Z)^{n-1}]$$
induit par la projection sur les $n-1$ dernières coordonnées. 
\end{definition}

Lorsque $D \in \mathrm{Div}_\Gamma^{\circ}$ on peut représenter la classe $S_{\rm mult } [D]$ par un cocycle complètement explicite à valeurs dans $\Omega_{\rm mer}^n$, déduit d'un symbole modulaire partiel; c'est l'objet du théorème suivant que l'on démontre au \S \ref{S:8.3.3}. Pour simplifier les expressions on identifie, {\it via} la multiplication par la $n$-forme invariante $dz_1 \wedge \cdots \wedge dz_n$, l'espace $\Omega_{\rm mer}^n$ à $\mathcal{M} (\C^n / \Z^n)$, l'espace des fonctions méromorphes sur $\C^n / \Z^n$. 
 
\begin{theorem} \label{T:mult2}
Soit $\Gamma$ un sous-groupe d'indice fini dans $\GL_n (\Z)$ et $C = \Gamma \cdot \Z e_1 \subset \Z^n$. L'application linéaire 
$$\mathrm{Div}_\Gamma^\circ \to \mathrm{Hom} (\Delta_C , \mathcal{M} (\C^n / \Z^n))^\Gamma ; \quad D \mapsto \mathbf{S}^*_{\rm mult} [D],$$
où
\begin{multline*}
\mathbf{S}^*_{\rm mult} [D] : [v_1 , \ldots , v_n] \mapsto \frac{1}{\det h} \sum_{w \in \Q^n / \Z^n} D(w) \cdot \\ \sum_{\substack{\xi \in \Q^n/\Z^n \\ h \xi  = w \ (\mathrm{mod} \ \Z^n)}} \varepsilon (v_1^* + \xi_1 ) \cdots \varepsilon (v_{n}^* + \xi_n) ,
\end{multline*}
avec $h = ( v_1 | \cdots | v_n) \in M_n (\Z)^\circ$, est bien définie et vérifie les propriétés suivantes. 
\begin{enumerate}
\item Pour tout $D \in \mathrm{Div}_\Gamma^\circ$, le cocycle associé à $\mathbf{S}_{\rm mult}^* [D]$ représente la classe de cohomologie $S_{\rm mult} [D] \in H^{n-1} (\Gamma , \Omega^n_{\rm mer} )$.
\item Pour tout $a \in M_n (\Z)^\circ$ préservant $C$, on a 
$$\mathbf{T} (a) (\mathbf{S}^*_{\rm mult } [D] ) = \mathbf{S}^*_{\rm mult } [ [\Gamma a \Gamma]^* D].$$
\end{enumerate}
\end{theorem}

\medskip
\noindent
{\it Remarque.} On peut vérifier sur les formules que les relations de distributions 
$$\mathbf{S}^*_{\rm mult } [[s]^*D] = [s]^* \mathbf{S}^*_{\rm mult } [D],$$ 
pour tout entier $s>1$ et pour tout $D \in \mathrm{Div}_\Gamma^\circ$, sont encore satisfaites. On peut de même montrer que pour tous 
$g_1, \ldots , g_n \in \Gamma$, la fonction méromorphe 
$$\mathbf{S}^*_{\rm mult } [D] (g_1 , \ldots , g_n)$$ 
est régulière en dehors des hyperplans affines passant par un point du support de $D$ et dirigés par 
$$\langle g_1^{-1} e_1, \ldots , \widehat{g_j^{-1} e_1} , \ldots , g_n^{-1} e_1 \rangle \quad \mbox{pour un } j \in \{ 1 , \ldots , n \}.$$
On pourrait bien sûr ici remplacer $e_1$ par n'importe quel vecteur primitif $v_0$, à condition de prendre $C=\Gamma \cdot v_0$. 

\medskip
\noindent
{\it Exemple.} Soit $N$ un entier strictement positif. Considérons le groupe $\Gamma$ constitué des matrices de $\SL_n (\Z)$ qui fixent la droite $\langle e_1 \rangle$, engendrée par le premier vecteur de la base canonique de $\Z^n$, modulo $N$; dans la suite on note ce groupe $\Gamma_0 (N,n)$ ou simplement $\Gamma_0 (N)$ s'il n'y a pas d'ambiguïté sur la dimension. \`A toute combinaison linéaire formelle $\delta = \sum_{d | N} n_d [d]$ de diviseurs positifs de $N$ on associe 
$$D_\delta = \sum_{d | N } n_d \sum_{j=0}^{d-1} \left[ \frac{j}{d} e_1 \right]  ;$$
c'est un élément de $\mathrm{Div}_{\Gamma_0 (N)}$ qui appartient à $\mathrm{Div}_{\Gamma_0 (N)}^\circ$ si et seulement si on a $\sum_{d | N } n_d d =0$. 

Par définition $D_\delta = \sum_{d | N } n_d  \pi_d^* D_0 $, où $D_0$ désigne toujours l'élément de $\mathrm{Div}_{\Gamma_0 (N)}$ de degré $1$ supporté en $0$ et $\pi_d$ désigne la matrice diagonale $\mathrm{diag} (d, 1 , \ldots , 1)$. On a donc 
$$\mathbf{S}_{\rm mult , \chi_0} [D_\delta ] = \sum_{d | N } n_d  \mathbf{S}_{\rm mult , \chi_0} [\pi_d^* D_0 ] .$$
Noter que $D_0$ est invariant par le sous-groupe $\pi_d \Gamma_0 (N) \pi_d^{-1}$ de $\SL_n (\Z)$. 

On verra que pour $g_1 , \ldots , g_n \in \Gamma_0 (N)$ et pour tout diviseur $d$ de $N$, les cocycles
$$\mathbf{S}_{\rm mult , \chi_0} [\pi_d^* D_0 ] (g_1 , \ldots , g_n) \quad \mbox{et} \quad  \pi_d^* \mathbf{S}_{\rm mult , \chi_0} [D_0 ] (\pi_d g_1 \pi_d^{-1} , \ldots , \pi_d g_n \pi_d^{-1})$$
représentent la même classe de cohomologie dans $H^{n-1} (\Gamma , \overline{\Omega}^n_{\rm mer})$
et donc qu'il en est de même pour  
$$\mathbf{S}_{\rm mult , \chi_0} [D_\delta ] (g_1 , \ldots , g_n) \quad \mbox{et} \quad  \sum_{d | N } n_d  \pi_d^* \mathbf{S}_{\rm mult , \chi_0} [D_0 ] (\pi_d g_1 \pi_d^{-1} , \ldots , \pi_d g_n \pi_d^{-1}) .$$

\medskip
\noindent
{\it Remarque.} On retrouve le théorème énoncé en introduction en prenant $n=2$ et 
$$\mathbf{S}_{\delta} (g^{-1}) =  \mathbf{S}_{\rm mult , e_1^* } [D_{\delta^\vee} ] (1 ,  g ) \quad \mbox{et} \quad \mathbf{S}^*_{\delta} (g^{-1}) = \mathbf{S}^*_{\rm mult } [D_{\delta^\vee} ] (1 ,  g ),$$
avec $\delta^\vee = \frac{1}{N} \sum_{d |N } n_d d' [d]$. 

\medskip

Le théorème \ref{T:mult2} implique que le cocycle $\mathbf{S}_{\rm mult , \chi_0} [D_\delta ]$, qui a l'avantage d'être régulier en la plupart des points de torsion, est cohomologue au cocycle qui se déduit du symbole modulaire partiel $\mathbf{S}_{\rm mult}^* [D_\delta ]$ et dont l'avantage est d'avoir une expression simple en le point générique. Un calcul direct permet en effet de vérifier qu'il associe à un élément $[v_1 , \ldots , v_n ] \in \Delta_C$, où les $n-1$ dernières coordonnées des vecteurs $v_j $ sont toutes divisibles par $N$, l'expression 
$$\sum_{d | N } n_d  d \pi_d^* \mathbf{c} (v_1^{(d)} , \ldots , v_n^{(d)} ),$$
où $v_j^{(d)}$ désigne le vecteur de $\Z^n$ obtenu à partir de $v_j$ en divisant par $d$ ses $n-1$ dernières coordonnées et 
$$\mathbf{c} (v_1 , \ldots , v_n ) = \frac{1}{\det h} \sum_{\substack{\xi \in \Q^n/\Z^n \\ h \xi  \in \Z^n }} \varepsilon (v^*_1 + \xi_1 ) \cdots \varepsilon (v^*_{n} + \xi_n ) , $$
avec toujours $h= ( v_1 | \cdots | v_n)$.

\subsection{Cocycles de Dedekind--Rademacher généralisés}

Soit $N$ un entier strictement positif et $\delta = \sum_{d | N} n_d [d]$ une combinaison linéaire formelle entière de diviseurs positifs de $N$ comme dans l'exemple précédent. Puisque $D_\delta$ est $\Gamma_0 (N)$-invariant, il lui correspond un cocycle (régulier) $\mathbf{S}_{\rm mult , e_1^*} [D_\delta ]$ du groupe $\Gamma_0 (N)$. On note 
\begin{equation*}
\mathbf{\Psi}_\delta : \Gamma_0 (N)^n \to \mathcal{M} (\C^n / \Z^n); \quad (g_1 , \ldots , g_n ) \mapsto \mathbf{S}_{\rm mult , e_1^*} [D_\delta ] (g_1 , \ldots , g_n ).
\end{equation*}
C'est un $(n-1)$-cocycle du groupe $\Gamma_0 (N)$ à valeurs dans les fonctions méromorphes sur $\C^n / \Z^n$. 

Sous l'hypothèse que $\delta$ est de degré $0$, c'est-à-dire $\sum_{d | N} n_d =0$, les points de $D_\delta$ ont tous une première coordonnée non nulle dans $\frac{1}{N} \Z / \Z$. Il découle donc du théorème \ref{T:mult} que l'image de $\mathbf{\Psi}_\delta$ est contenue dans les fonctions régulières en $0$. 

\begin{proposition} \label{P:DRgen}
L'application
$$\Phi_\delta : \Gamma_0 (N )^n \to \C; \quad (g_1 , \ldots , g_n ) \mapsto \left[ \mathbf{\Psi}_\delta (g_1 , \ldots , g_n ) \right] (0) $$ 
définit un $(n-1)$-cocycle (à valeurs scalaires) qui représente une classe de cohomologie non-nulle $[\Phi_\delta] \in H^{n-1} (\Gamma_0 (N) , \Q)$ telle que 
\begin{enumerate}
\item la classe $d_n [\Phi_\delta ]$, où $2d_n$ désigne le dénominateur du $n$-ème nombre de Bernoulli, est entière, et
\item pour tout nombre premier $p$ qui ne divise pas $N$ et pour tout $k \in \{ 0 , \ldots , n-1 \}$, on a
$$\mathbf{T}_p^{(k)}   [\Phi_\delta ] = \left( \left( \begin{smallmatrix} n-1 \\ k-1 \end{smallmatrix} \right)_p  +   \left( \begin{smallmatrix} n \\ k \end{smallmatrix} \right)_p - \left( \begin{smallmatrix} n-1 \\ k-1 \end{smallmatrix} \right)_p \right) \cdot [\Phi_\delta].$$
\end{enumerate}
\end{proposition}

Les cocycles $\Phi_\delta$ sont des généralisations à $\SL_n (\Z)$ des cocycles de Dedekind--Rademacher pour $\SL_2 (\Z )$. Si $F$ est un corps de nombres totalement réel $F$ de degré $n$ au-dessus de $\Q$ et $L$ est un idéal fractionnaire de $F$, on peut considérer le groupe $U$ des unités totalement positives de $\mathcal{O}_F$ préservant $L$. Le choix d'une identification de $L$ avec $\Z^n$, et donc de $F$ avec $\Q^n$, permet de plonger $U$ dans $\GL_n (\Z)$ {\it via} la représentation régulière de $U$ sur $L$. Dans le dernier paragraphe du chapitre \ref{S:chap9} on montre que l'évaluation de $\Psi_\delta$ sur la classe fondamentale dans $H_{n-1} (U , \Z)$ est égale à la valeur en $0$ d'une stabilisation d'une fonction zeta partielle de $F$. Le théorème \ref{T:mult2} permet alors de retrouver les expressions de ces valeurs comme sommes de Dedekind explicites obtenues par Sczech \cite{Sczech93}; voir aussi \cite{CD,CDG}.

\subsection{Lien avec le cocycle de Sczech}

Les cocycles $\mathbf{S}^*_{\rm mult} [D]$ ($D \in \mathrm{Div}_\Gamma^\circ$) sont très proches du cocycle de Sczech \cite[Corollary p. 598]{Sczech93}. Il y a toutefois quelques différences notables~:
\begin{enumerate}
\item Par souci de simplicité nous nous sommes restreints ici\footnote{Mais pas dans \cite{ColmezNous} motivé par l'étude de toutes les valeurs critiques des fonctions $L$ correspondantes.} à ne considérer que de la cohomologie à coefficients triviaux et non tordue. De sorte qu'en prenant les notations de Sczech on a $P=1$ et $v=0$.
\item L'hypothèse que $D \in \mathrm{Div}_\Gamma^\circ$ conduit à une stabilisation du cocycle de Sczech qui fait disparaitre, au prix d'une augmentation du niveau, son paramètre $Q$ mais n'est défini qu'en un point \emph{générique} de la fibre $\C^n / \Z^n$.
\end{enumerate}

\section[Le cas elliptique]{Le cas elliptique : formes différentielles elliptiques et symboles modulaires}

Considérons maintenant une courbe elliptique $E$ au-dessus d'une base $Y$ et l'algèbre graduée $\Omega_{\rm mer} (E^n)$ des formes méromorphes sur le produit symétrique 
$$E^n = E \times_Y \cdots \times_Y E$$ 
de $n$ copies de $E$.  

Une matrice entière induit, par multiplication à gauche, une application $E^n \to E^n$. On en déduit donc là encore une action du semi-groupe $M_n (\Z)^\circ = M_n (\Z ) \cap \GL_n (\Q)$ sur $\Omega_{\rm mer} (E^n)$. 

\subsection{Opérateurs de Hecke} 

Soit $\Gamma$ un sous-groupe d'indice fini de $\SL_n (\Z)$ et soit $S$ un sous-monoïde de $M_n (\Z )^\circ$ contenant $\Gamma$. 
L'action, \emph{à droite}, de $M_n (\Z )^\circ$ sur $\Omega_{\rm mer} (E^n)$, par tiré en arrière, induit là encore une action de l'algèbre de Hecke associée au couple $(S, \Gamma)$ sur $H^{n-1} (\Gamma ,\Omega_{\rm mer} (E^n))$. On note encore $\mathbf{T}(a)$ l'opérateur de Hecke sur $H^{n-1} (\Gamma ,\Omega_{\rm mer} (E^n))$ associé à la double classe 
$$\Gamma a \Gamma \quad \mbox{avec} \quad a \in S.$$

Noter par ailleurs qu'un élément $a \in S$ induit une application $[a] : E^n \to E^n$. 
Notons $\mathrm{Div}_\Gamma(E^n)$ le groupe (abélien) constitué des combinaisons linéaires entières formelles $\Gamma$-invariantes de points de torsion dans $E^n$. Dans la suite on note $[\Gamma a \Gamma]$ l'application 
$$[\Gamma a \Gamma] = \sum_j [a_j] :  \mathrm{Div}_\Gamma(E^n) \to \mathrm{Div}_\Gamma(E^n).$$

Supposons maintenant que la base $Y$ soit une courbe modulaire $\Lambda \backslash \mathcal{H}$ avec $\Lambda$ sous-groupe d'indice fini dans $\SL_2 (\Z)$.
\`A travers l'identification 
\begin{equation} \label{E:ident}
\mathcal{H} \times \R^2 \stackrel{\simeq}{\longrightarrow} \mathcal{H} \times \C; \quad (\tau , (u,v)) \mapsto (\tau , u \tau +v )
\end{equation}
l'action de $\Z^2$ sur $\R^2$ par translation induit une action de $\Z^2$ sur $\mathcal{H} \times \C$ par 
$$(\tau , z )  \stackrel{(m,n)}{\longmapsto} (\tau , z+ m \tau + n).$$
Le groupe $\SL_2 (\R)$ opère à gauche sur $\mathcal{H} \times \R^2$ par 
$$(\tau , (u,v)) \stackrel{g}{\longmapsto} (g \tau , (u,v) g^{-1} ).$$
\`A travers \eqref{E:ident} cette action devient
$$(\tau , z) \stackrel{\left( \begin{smallmatrix} a & b \\ c & d \end{smallmatrix} \right)}{\longmapsto} \left( \frac{a\tau +b}{c\tau +d} , \frac{z}{c\tau + d} \right).$$
La courbe elliptique $E$, et plus généralement le produit fibré $E^n$, s'obtiennent comme double quotient 
\begin{equation}
E^n= \Lambda \backslash \left[ ( \mathcal{H} \times \C^n) / \Z^{2n} \right] = \Lambda \backslash \left[ \mathcal{H} \times (\R^{2n} / \Z^{2n}) \right].
\end{equation}
Considérons maintenant un sous-monoïde $\Delta \subset M_2 (\Z )^\circ$ contenant $\Lambda$. L'action de $\Lambda$ sur \eqref{E:ident} s'étend au monoïde $\Delta$ par la formule
$$(\tau , (u,v)) \stackrel{g}{\longmapsto} (g \tau , \det (g) (u,v) g^{-1} ).$$
Cette action préserve le réseau $\Z^2$ dans $\R^2$ de sorte que tout élément $g \in \Delta$ induit une application 
$$[g] : \mathcal{H} \times \R^{2n}/ \Z^{2n} \to \mathcal{H} \times \R^{2n} /\Z^{2n}.$$

Une double classe
$$\Lambda b \Lambda \quad \mbox{avec} \quad b \in \Delta$$ 
induit alors un op\'erateur 
$$[\Lambda b \Lambda] : \Omega^n_\mathrm{mer}(E^n) \to \Omega^n_\mathrm{mer}(E^n)$$ 
lui aussi dit de Hecke, sur $\Omega_{\rm mer}^n (E^n)$ et donc aussi sur $H^{n-1} (\Gamma ,\Omega_{\rm mer} (E^n))$, que l'on notera simplement $T(b)$.

\subsection{Cocycles elliptiques I} \label{S:ellI} Soit $\mathrm{Div}_\Gamma^0 (E^n)$  le groupe des combinaisons linéaires entières formelles $\Gamma$-invariantes \emph{de degré $0$} de points de torsion dans $E^n$.

\'Etant donné un élément $D \in \mathrm{Div}_\Gamma^0 (E^n )$, la théorie de Chern--Weil nous permettra de construire, au chapitre \ref{S:chap10}, des  représentants suffisamment explicites de la classe de cohomologie $S_{\rm ell} [D]  \in H^{n-1} (\Gamma , \Omega_{\rm mer} (E^n) )$ du théorème \ref{T:cocycleE} pour démontrer le théorème suivant. 

Dans la suite on fixe $n$-morphismes primitifs linéairement indépendants
$$\chi_1 , \ldots , \chi_n : \Z^n \to \Z.$$ 
On note encore $\chi_j : \Q^n \to \Q$ les formes linéaires correspondantes et $\chi_j : E^n \to E$ les morphismes primitifs qu'ils induisent. On pose 
$$\chi = (\chi_1 , \ldots , \chi_n).$$
On démontre le théorème suivant au \S \ref{S:9.3.3}.

\begin{theorem} \label{T:ell}
Il existe une application linéaire 
$$\mathrm{Div}_\Gamma^0 (E^n ) \to C^{n-1} (\Gamma , \Omega^n_{\rm mer}(E^n) )^{\Gamma}; \quad D \mapsto \mathbf{S}_{\rm ell , \chi} [D]$$
telle que 
\begin{enumerate}
\item pour tout entier $s >1$, on a $\mathbf{S}_{\rm ell , \chi} [[s]^*D] = [s]^* \mathbf{S}_{\rm ell , \chi} [D]$ (\emph{relations de distribution});
\item chaque cocycle $\mathbf{S}_{\rm ell , \chi} [D]$ représente $S_{\rm ell} [D]  \in H^{n-1} (\Gamma , \Omega_{\rm mer}^n (E^n) )$; 
\item chaque forme différentielle méromorphe $\mathbf{S}_{\rm ell , \chi} [D] (g_1 , \ldots , g_n)$ est régulière en dehors des hyperplans affines passant par un point du support de $D$ et dirigés par $\mathrm{ker} ( \chi_i \circ g_j)$ pour $i,j \in \{1 , \ldots , n \}$; 
\item pour tout $a \in S$, 
$$\mathbf{T} (a) \left[\mathbf{S}_{\rm ell , \chi} [D] \right] = \left[\mathbf{S}_{\rm ell , \chi } [ [\Gamma a \Gamma]^* D] \right];$$
\item pour tout $b \in \Delta$, 
$$T (b) \left[ \mathbf{S}_{\rm ell , \chi } [D] \right] = \left[ \mathbf{S}_{\rm ell , \chi } [ [\Lambda b \Lambda]^* D] \right].$$
\end{enumerate}
\end{theorem}

\medskip
\noindent
{\it Exemple.} Soit $c$ un entier supérieur à $2$. La combinaison linéaire de points de $c$-torsion $E^n [c] - c^{2n} \{0 \}$ dans $E^n$ est de degré $0$ et invariante par $\Gamma = \SL_n (\Z)$; elle définit donc un élément $D_c \in \mathrm{Div}_\Gamma^0 (E^n)$.  L'algèbre de Hecke associée au couple $(S , \Gamma)$, avec $S = M_n (\Z )^\circ$, est engendrée par les opérateurs 
$$\mathbf{T}^{(k)}_p = \mathbf{T} (a^{(k)}_p) \quad \mbox{avec}  \quad a^{(k)}_p=\mathrm{diag} (\underbrace{p, \ldots , p}_{k} , 1 , \ldots , 1 ),$$
où $p$ est premier et $k$ appartient à $\{ 1 , \ldots , n-1 \}$.
Soit $p$ un nombre premier. La base canonique de $\C$ au-dessus du point $\tau=i \in \mathcal{H}$ fournit une base de $E[p]$ et permet d'identifier $E^n [p]$ au groupe abélien des matrices $M_{n,2} (\mathbf{F}_p)$. Le tiré en arrière de $\{ 0 \}$ par l'application $[\Gamma a^{(k)}_p \Gamma]$ est supporté sur l'ensemble de tous les points de $p$-torsion et une matrice dans $M_{n,2} (\mathbf{F}_p)$ est comptée avec multiplicité\footnote{\'Egale au nombre de $k$-plans contenant un $2$-plan donné. Noter que ce nombre est $0$ si $k \leq 1$ il n'y a alors que des matrices de rang $\leq 1$ dans le support.} $\left( \begin{smallmatrix} n-2 \\ k-2 \end{smallmatrix} \right)_p$ si elle est de rang $2$, avec multiplicité 
\footnote{\'Egale au nombre de $k$-plans contenant une droite donnée.} $\left( \begin{smallmatrix} n-1 \\ k-1 \end{smallmatrix} \right)_p$ si elle est de rang $1$ et avec multiplicité $\left( \begin{smallmatrix} n \\ k \end{smallmatrix} \right)_p$ si elle est nulle.

D'un autre côté on peut considérer la double classe $\Lambda \left( \begin{smallmatrix} p & 0 \\ 0 & 1 \end{smallmatrix} \right) \Lambda$ avec 
$$\Lambda = \SL_2 (\Z).$$ 
La pré-image de $\{ 0 \}$ par l'application induite $E^n \to E^n$ est cette fois égale à l'ensemble des matrices de rang $1$ comptées avec multiplicité $1$ et la matrice nulle comptée avec multiplicité 
$\left( \begin{smallmatrix} 2 \\ 1 \end{smallmatrix} \right)_p = p+1$. 

On en déduit que pour $p$ premier à $c$, on a 
\begin{multline*}
[\Gamma a_p^{k} \Gamma]^* D_c  = \left( \begin{smallmatrix} n-2 \\ k-2 \end{smallmatrix} \right)_p [p]^*D_c + \left( \left( \begin{smallmatrix} n-1 \\ k-1 \end{smallmatrix} \right)_p - \left( \begin{smallmatrix} n-2 \\ k-2 \end{smallmatrix} \right)_p \right)  [ \Lambda \left( \begin{smallmatrix} p & 0 \\ 0 & 1 \end{smallmatrix} \right) \Lambda ]^* D_c \\ + \left( \left( \begin{smallmatrix} n \\ k \end{smallmatrix} \right)_p -  \left( \left( \begin{smallmatrix} n-1 \\ k-1 \end{smallmatrix} \right)_p - \left( \begin{smallmatrix} n-2 \\ k-2 \end{smallmatrix} \right)_p \right) (p+1) - \left( \begin{smallmatrix} n-2 \\ k-2 \end{smallmatrix} \right)_p \right) D_c
\end{multline*}
et donc que la classe de cohomologie de $\mathbf{S}_{\rm ell , \chi} [D_c]$ annule l'opérateur
\begin{multline*}
\mathbf{T}_p^{(k)} - \left( \begin{smallmatrix} n-2 \\ k-2 \end{smallmatrix} \right)_p [p]^* - \left( \left( \begin{smallmatrix} n-1 \\ k-1 \end{smallmatrix} \right)_p - \left( \begin{smallmatrix} n-2 \\ k-2 \end{smallmatrix} \right)_p \right) T_p \\ - \left( \begin{smallmatrix} n \\ k \end{smallmatrix} \right)_p +  \left( \left( \begin{smallmatrix} n-1 \\ k-1 \end{smallmatrix} \right)_p - \left( \begin{smallmatrix} n-2 \\ k-2 \end{smallmatrix} \right)_p \right) (p+1) + \left( \begin{smallmatrix} n-2 \\ k-2 \end{smallmatrix} \right)_p.
\end{multline*}

\medskip

\subsection{Cocycles elliptiques II}

La $1$-forme différentielle $dz$ induit une trivialisation du fibré des $1$-formes relatives $\Omega^1_{E / Y}$. On utilise cette trivialisation pour identifier les sections holomorphes du fibré des formes méromorphes relatives $\Omega_{{\rm mer}, E^n / Y}^n$ à l'espace $\mathcal{M}_n (E^n)$ des fonctions $F$ méromorphes sur $(\mathcal{H} \times \C^n) / \Z^{2n}$ qui vérifie la propriété de modularité
$$F \left( \frac{a\tau +b }{c \tau +d } , \frac{z}{c\tau +d} \right) = (c\tau +d )^n F (z,\tau ) \quad \mbox{pour tout} \quad \left( \begin{smallmatrix} a & b \\ c & d \end{smallmatrix} \right) \in \Lambda .$$ 

Sous certaines hypothèses naturelles sur $D$ on peut, comme dans le cas multiplicatif, représenter la classe $S_{\rm ell} [D]$ par un cocycle complètement explicite déduit d'un symbole modulaire partiel.  Le long d'une fibre $E_\tau$ au-dessus d'un point $[\tau] \in Y$ il est naturel de remplacer la fonction $\varepsilon$ par la série d'Eisenstein
\begin{equation} 
E_1 (\tau , z) = \frac{1}{2i\pi} \sideset{}{{}^e}\sum_{\omega \in \Z + \Z \tau} \frac{1}{z+\omega} = \frac{1}{2i\pi} \lim_{M \to \infty} \sum_{m=-M}^M \left( \lim_{N \to \infty} \sum_{n=-N}^N \frac{1}{z+ m \tau + n } \right).
\end{equation}
Toutefois, deux problèmes se présentent: 
\begin{enumerate}
\item La série d'Eisenstein n'est pas périodique en $z$ de période $\Z\tau + \Z$ mais vérifie seulement
\begin{equation} \label{E1per}
E_1 ( \tau , z + 1) = E_1 (\tau , z) \quad \mbox{et} \quad E_1 (\tau , z+ \tau) = E_1 (\tau  , z ) - 1,
\end{equation}
voir \cite[III \ \S 4 \ (5)]{Weil}. 

\item La série d'Eisenstein n'est pas non plus modulaire mais vérifie seulement que pour toute matrice $\left( \begin{smallmatrix} a & b \\ c & d \end{smallmatrix} \right) \in \SL_2 (\Z)$ on a 
\begin{equation} \label{E1mod}
E_1 \left( \frac{a\tau + b}{c\tau +d} , \frac{z}{c\tau +d} \right) = (c\tau + d) E_1 (\tau , z) + cz,
\end{equation}
voir \cite[III \ \S 5 \ (7)]{Weil}.
\end{enumerate}
Pour remédier à ces deux problèmes on suppose dorénavant que $Y=Y_0 (N)$, avec $N$ entier et supérieur à $2$. La courbe $E$ est alors munie d'un sous-groupe $K \subset E[N]$ cyclique d'ordre $N$ et on peut considérer la fonction 
\begin{equation*}
E_1^{(N)} (\tau , z) =  \sum_{\xi \in K} E_1 \left( \tau , z + \xi \right) - N  E_1 (\tau , z) =  \sum_{j=0}^{N-1} E_1 \left( \tau , z + \frac{j}{N}  \right) - N  E_1 (\tau , z).
\end{equation*}
D'après \eqref{E1per} et \eqref{E1mod} cette fonction est en effet périodique et modulaire relativement au groupe $\Gamma_0 (N)$. 
La fonction $E_1^{(N)} (\tau , z)$ est associée au cycle de torsion $K - N \{ 0 \}$ dans $E$ qui est \emph{de degré $0$}. 

\medskip

Revenons maintenant au produit fibré $E^n$. On note $\mathrm{Div}_{\Gamma , K} (E^n) \subset \mathrm{Div}_\Gamma (E^n)$ le sous-groupe constitué des combinaisons linéaires entières formelles $\Gamma$-invariantes de points de torsion dans $K^n \subset E[N]^n$. Identifiant ces points de torsion au sous-groupe $\left( \frac{1}{N} \Z / \Z \right)^n \subset (\Q / \Z )^n$, on verra un élément de $\mathrm{Div}_{\Gamma , K} (E^n)$ comme une fonction $\Gamma$-invariante $D : (\Q / \Z )^n \to \Z$ dont le support est contenu dans le réseau $\left( \frac{1}{N} \Z \right)^n \subset \Q^n$. 

\begin{definition}
Soit $\mathrm{Div}_{\Gamma , K}^{\circ} (E^n)$ l'intersection des noyaux des $n$ morphismes 
$$\mathrm{Div}_{\Gamma , K} (E^n) \to \Z [(\Q /\Z)^{n-1}]$$
induits par la projection sur les $n-1$ dernières coordonnées. 
\end{definition}
Noter que les éléments de $\mathrm{Div}_{\Gamma , K}^{\circ} (E^n)$ sont de degré $0$.

\medskip
\noindent
{\it Exemple.} Supposons $\Gamma = \Gamma_0 (N,n)$. \`A toute combinaison linéaire formelle $\delta = \sum_{d | N} n_d [d]$ de diviseurs positifs de $N$ on associe 
$$D_\delta = \sum_{d | N } n_d \sum_{\xi \in K[d]} (\xi , 0 , \ldots , 0),$$
où $K[d]$ désigne l'ensemble des éléments de $d$-torsion dans $K$. La combinaison formelle de points de torsion $D_\delta$ définit un élément de $\mathrm{Div}_{\Gamma , K}^{\circ}(E^n)$ si et seulement si $\sum_{d| N} n_d d =0$. 

\medskip

Notre construction donne lieu à des classes de cohomologie à valeurs dans les formes méromorphes $\Omega_{\rm mer}^n (E^n)$. En restreignant ces formes aux fibres on obtient des classes de cohomologie à valeurs dans les sections holomorphes du fibré des formes méromorphes relatives $\Omega^n_{{\rm mer}, E^n /Y}$ et donc dans $\mathcal{M}_n (E^n )$. Le théorème suivant, que nous démontrons au \S \ref{S:9.4.2}, donne des représentants explicites de ces classes de cohomologie. 

\begin{theorem} \label{T:ellbis}
Supposons $Y=Y_0 (N)$ avec $N$ entier supérieur à $2$ et notons $K \subset E[N]$ le sous-groupe cyclique d'ordre $N$ correspondant. Soit $\Gamma$ un sous-groupe d'indice fini dans $\SL_n (\Z)$ et $C = \Gamma \cdot \Z e_1 \subset \Z^n$. L'application linéaire 
$$\mathrm{Div}_{\Gamma , K}^{\circ} (E^n) \to \mathrm{Hom} (\Delta_C , \mathcal{M}_n (E^n) )^\Gamma ; \quad D \mapsto \mathbf{S}^*_{\rm ell} [D],$$
où
\begin{equation} \label{E:S*ell} 
\mathbf{S}^*_{\rm ell} [D] : [v_1 , \ldots , v_n]  \mapsto \frac{1}{\det h} \sum_{w \in E^n} D(w)   \sum_{\substack{\xi \in E^n \\ h \xi  = w}}  E_1 (\tau ,  v_1^* +  \xi_{1})  \cdots  E_1 (\tau , v_n^* + \xi_{n})   ,
\end{equation}
avec $h = ( v_1 | \cdots | v_n) \in M_n (\Z)^\circ$, est bien définie et représente la classe de cohomologie 
$$S_{\rm ell} [D] \in H^{n-1} (\Gamma , \mathcal{M}_n (E^n) ).$$
\end{theorem}
Le fait que $\mathbf{S}^*_{\rm ell} [D]$ définit bien un cocycle est démontré d'une manière différente dans la thèse de Hao Zhang \cite[Theorem 4.1.6]{Zhang}. La démonstration repose sur la construction, par le second auteur, d'un $(n-1)$-cocycle du groupe $\GL_n (\Z)$ à valeurs dans les fonctions méromorphes sur $\mathcal{H} \times \C^n \times \C^n$, voir \cite{CS}. 

\medskip
\noindent
{\it Remarques.} 1. Le fait que les fonctions dans l'image de $\mathbf{S}^*_{\rm ell } [D]$, qui ne sont {\it a priori} définies que sur 
$\mathcal{H} \times \C^n $, soient en fait $\Z^{2n}$-invariantes et modulaires de poids $n$ relativement à l'action de $\Lambda$ fait partie de l'énoncé du théorème mais peut être vérifié à la main.  Comme dans le cas multiplicatif on peut aussi vérifier directement sur les formules que les relations de distribution 
$$\mathbf{S}^*_{\rm ell } [[s]^*D] = [s]^* \mathbf{S}^*_{\rm ell } [D],$$ 
pour tout entier $s>1$ et pour tout $D \in \mathrm{Div}_{\Gamma , K}^{\circ} (E^n)$, sont encore vérifiées et que pour tous 
$g_1, \ldots , g_n \in \Gamma$, la forme différentielle méromorphe $\mathbf{S}^*_{\rm ell } [D] (g_1 , \ldots , g_n)$ est régulière en dehors des hyperplans affines passant par un point du support de $D$ et dirigés par 
$$\langle g_1^{-1} e_1, \ldots , \widehat{g_j^{-1} e_1} , \ldots , g_n^{-1} e_1 \rangle \quad \mbox{pour un } j \in \{ 1 , \ldots , n \}.$$

2. Pour tout entier $m$ premier à $N$, l'opérateur de Hecke $T_m$ correspondant à la double classe 
$$\Lambda\left( \begin{array}{cc} m & 0 \\ 0 & 1 \end{array} \right)  \Lambda = \bigsqcup_{\substack{a, d >0 \\ ad=m}} \bigsqcup_{b=0}^{d-1} \Lambda \left( \begin{array}{cc} a & b \\ 0 & d \end{array} \right)$$
opère sur une fonction méromorphe $F$ du type 
$$F = E_1 (\tau , z_1 ) \ldots E_n (\tau , z_n),$$ 
par l'expression familière
$$T_m F = m^n \sum_{\substack{a, d >0 \\ ad=m}} \sum_{b=0}^{d-1} \frac{1}{d^n} E_1 \left( \frac{a\tau+b}{d} , a z_1 \right) \ldots E_n \left( \frac{a\tau+b}{d} , az_n \right).$$ 

\medskip

\medskip
\noindent
{\it Exemple.} Soit $\Gamma = \Gamma_0 (N,n)$ et soit $\delta = \sum_{d | N} n_d [d]$ une combinaison linéaire formelle de diviseurs positifs de $N$ telle que $\sum_{d| N} n_d d =0$. Après restriction aux fibres, le théorème \ref{T:ell} associe à tout $n$-uplet $\chi$ de morphismes primitifs et linéairement indépendants $\Z^n \to \Z$ un $(n-1)$-cocycle régulier $\mathbf{S}_{\rm ell , \chi } [D_\delta]$ du groupe $\Gamma$ à valeurs dans $\mathcal{M}_n (E^n )$.  Le théorème \ref{T:mult2} implique que ce cocycle est cohomologue au cocycle qui se déduit du symbole modulaire partiel $\mathbf{S}_{\rm mult}^* [D_\delta ]$ et dont un calcul permet de vérifier qu'il associe à un élément $[v_1 , \ldots , v_n ] \in \Delta_C$, où les $n-1$ dernières coordonnées des vecteurs $v_j $ sont toutes divisibles par $N$, l'expression 
$$\sum_{d | N } n_d d \pi_d^*\mathbf{E}( \tau ; v_1^{(d)} , \ldots , v_n^{(d)} )  ,$$
où $v_j^{(d)}$ désigne le vecteur de $\Z^n$ obtenu à partir de $v_j$ en divisant par $d$ ses $n-1$ dernières coordonnées et 
\begin{equation*}
\mathbf{E} ( \tau ; v_1 , \ldots , v_n ) = \frac{1}{\det h} \sum_{\substack{\xi \in E^n \\ h \xi  = 0 }} E_1 (\tau ,  v_1^* + \xi_{1}) \cdots E_1 (\tau , v_n^* + \xi_{n}) , 
\end{equation*}
avec toujours $h= ( v_1 | \cdots | v_n)$.

\subsection{Relèvement theta explicite}

Soit $N$ un entier supérieur à $4$ et soit $E$ la courbe elliptique universelle au-dessus de la courbe modulaire $Y_1 (N) = \Gamma_1 (N) \backslash \mathcal{H}$, où
$$\Gamma_1 (N) = \left\{ \left( \begin{array}{cc} a & b \\ c & d \end{array} \right) \in \SL_2 (\Z) \; : \; \left( \begin{array}{c} a \\ c \end{array} \right) \equiv \left( \begin{array}{c} 1 \\ 0 \end{array} \right) \ (\mathrm{mod} \ N ) \right\};$$
c'est l'espace des modules fin qui paramètre la donnée d'une courbe elliptique $E_\tau = \C / (\Z \tau + \Z)$ et d'un point de $N$-torsion $[1/N]$. 

Soit $\Gamma$ un sous-groupe de congruence contenu dans $\Gamma_0 (N,n)$ et $\delta = \sum_{d | N} n_d [d]$ une combinaison de diviseurs positifs de $N$ telle que $\sum_{d | N} n_d d=0$. La fonction $D_\delta$ appartient alors à $\mathrm{Div}_\Gamma^0 (E^n)$. Prenons
$$\chi = (e_1^*, e_1^* + e_2^* , \ldots , e_1^* + e_n^*).$$
Sous l'hypothèse que $\delta$ est de degré $0$, c'est-à-dire $\sum_{d | N} n_d =0$, les points de $D_\delta$ ont tous une première coordonnée non nulle dans $\frac{1}{N} \Z / \Z$. Il découle donc du théorème \ref{T:ell} que l'image de $\mathbf{S}_{\rm ell , \chi} [D_\delta ]$ est contenue dans les formes régulières en $0$. 
Après restriction de ces formes aux fibres et évaluation en $0$ on obtient une classe de cohomologie à valeurs dans l'espace $M_n (Y_1 (N))$ des formes modulaires de poids $n$ sur $Y_1 (N)$. On note $\Theta [D_\delta ] : \Gamma^n \to M_n (Y_1 (N))$ le cocycle correspondant.

\begin{theorem}
Le cocycle $\Theta [D_\delta ]$ représente une classe de cohomologie dans $H^{n-1} (\Gamma , M_n (Y_1 (N)))$ qui réalise le relèvement  theta des formes modulaires paraboliques de poids $n$ pour $Y_1 (N)$ dans la cohomologie $H^{n-1} (\Gamma , \C)$.
Le relèvement associe à une forme modulaire parabolique $f \in M_n (Y_1 (N))$ la classe de cohomologie du cocycle 
$$\Theta_{f} [D_\delta ] = \langle \Theta  [D_\delta ] , f \rangle_{\rm Petersson} : \Gamma^n \to \C$$
et pour tout entier $p$ premier ne divisant pas $N$ et pour tout $k\in \{ 1 , \ldots , n-1 \}$, on a
\begin{multline*}
\mathbf{T}_p^{(k)} \left[ \Theta_{ f} [D_\delta ] \right]   =  \left( \left( \begin{smallmatrix} n-1 \\ k-1 \end{smallmatrix} \right)_p - \left( \begin{smallmatrix} n-2 \\ k-2 \end{smallmatrix} \right)_p \right) \left[\Theta_{ T_p f} [D_\delta ] \right] \\ + \left( \left( \begin{smallmatrix} n \\ k \end{smallmatrix} \right)_p -  (p+1) \left( \left( \begin{smallmatrix} n-1 \\ k-1 \end{smallmatrix} \right)_p -   \left( \begin{smallmatrix} n-2 \\ k-2 \end{smallmatrix} \right)_p \right) \right) \left[ \Theta_{ f} [D_\delta ] \right] .
\end{multline*}
\end{theorem}

\medskip
\noindent
{\it Remarques.} 1. Si $f$ est une forme propre pour $T_p$ de valeur propre $a_p$, le relevé de $f$ est propre pour les opérateurs de Hecke $\mathbf{T}_p^{(k)}$ Lorsque $n=3$ les valeurs propres correspondantes aux opérateurs $\mathbf{T}_p^{(1)}$ et $\mathbf{T}_p^{(2)}$ sont respectivement $a_p+p^2$ et $pa_p+1$; ce relèvement est notamment étudié dans \cite{AshGraysonGreen}.

2. L'hypothèse $\sum_{d | N} n_d d =0$ faite sur $D_\delta$ implique que les formes modulaires dans l'image de $\Theta [D_\delta ]$ sont paraboliques à l'infini. Après évaluation en l'infini de $\Theta [D_\delta ]$ on retrouve le cocycle de Dedekind-Rademacher généralisé $\Psi_\delta$.

\medskip

Si $F$ est un corps de nombres totalement réel de degré $n$ au-dessus de $\Q$ et $L$ est un idéal fractionnaire de $F$, on peut là encore considérer le groupe $U$ des unités totalement positives de $\mathcal{O}_F$ préservant $L$ et le plonger dans $\GL_n (\Z)$. On peut cette fois montrer que l'évaluation de $\Theta_\delta [D]$ sur la classe fondamentale dans $H_{n-1} (U , \Z)$ est une forme modulaire de poids $n$ égale à la restriction à la diagonale d'une série d'Eisenstein partielle\footnote{Associée à la classe de $L$.} sur une variété modulaire de Hilbert; voir \cite[\S 13.3]{Takagi}. Le théorème \ref{T:ellbis} permet alors d'exprimer cette restriction comme une combinaison linéaire de produits de séries d'Eisenstein $E_1$. 

\medskip
\noindent
{\it Remarque.} Lorsque $n=2$, le théorème \ref{T:ellbis} est à rapprocher d'un théorème de Borisov et Gunnells \cite[Theorem 3.16]{BorisovGunnells} selon lequel l'application qui à un symbole unimodulaire $[a/c : b/d]$, avec $c,d \neq 0$ modulo $N$, associe 
$$E_1 \left(\tau , \frac{c}{N}  \right) E_1 \left(\tau , \frac{d}{N} \right)$$
induit un symbole modulaire partiel Hecke équivariant à valeurs dans les formes modulaires de poids $2$ sur $Y_1 (N)$ \emph{modulo les séries d'Eisenstein de poids $2$}. Borisov et Gunnells considèrent en effet les séries d'Eisenstein, de poids $1$  et niveau $Y_1(N)$,
$$E_1 (\tau , a/N)  \quad  ( a \in \{ 1 , \ldots , N-1 \} ),$$ 
et remarquent que si $a+b+c = 0$ modulo $N$ et si $a,b,c \neq 0$, l'expression 
$$E_1 (\tau , a/N) E_1 (\tau , b/N) + E_1 (\tau , b/N) E_1 (\tau , c/N) + E_1 (\tau , c/N) E_1 (\tau , a/N)$$ 
est une combinaison linéaire de séries d'Eisentein de poids $2$, voir \cite[Proposition 3.7 et 3.8]{BorisovGunnells}.\footnote{Cela se déduit d'une formule plus générale démontrée dans \cite{BorisovGunnellsInventiones} et que l'on peut également trouver chez Eisenstein \cite{Eisenstein}.}  

\medskip

Notons finalement que la construction de ce chapitre s'applique de manière similaire à une courbe elliptique plutôt qu'à une famille de courbes elliptiques. Lors-\ que cette courbe elliptique est à multiplication complexe par l'anneau des entiers $\mathcal{O}$ d'un corps quadratique imaginaire $k$ on obtient des cocycles de degré $n-1$ de sous-groupes de congruence, non plus de $\GL_n (\Z)$ mais, de $\GL_n (\mathcal{O})$. Ce sont des généralisations des cocycles considérés par Sczech \cite{SczechBianchi} et Ito \cite{Ito} lorsque $n=2$. Comme d'après un théorème classique de Damerell les évaluations des séries d'Eisenstein $E_1$ en des points CM sont, à des périodes transcendantes explicites près, des nombres algébriques, on peut montrer que les cocycles associés aux corps quadratiques imaginaires sont à valeurs algébriques. Leur étude fait l'objet de l'article  \cite{ColmezNous} dans lequel nous démontrons une conjecture de Sczech et Colmez relative aux valeurs critiques des fonctions $L$ attachées aux caractères de Hecke d'extensions de $k$. 

\chapter[Cohomologie d'arrangements d'hyperplans]{Cohomologie d'arrangements d'hyperplans~: représentants canoniques} \label{S:OrlikSolomon}

\resettheoremcounters

\numberwithin{equation}{chapter}

Dans ce chapitre, qui peut se lire de manière indépendante du reste de l'ouvrage, on démontre un théorème ``à la Orlik--Solomon'' pour les arrangements d'hyperplans dans des produits de $\C^\times$ ou des produits de courbes elliptiques. Le résultat principal que nous démontrons est le théorème \ref{P:Brieskorn}.

\section{Arrangement d'hyperplans trigonométriques ou elliptiques}
On fixe un groupe algébrique $A$, isomorphe au groupe multiplicatif ou à une courbe elliptique. Soit $n$ un entier naturel. On considère un $A^n$-torseur $T$ muni d'un sous-ensemble $T_{\tors} \subset T$ fixé qui est un torseur pour le groupe des points de torsions de $A^n$.

\medskip
\noindent
{\it Remarque.} Même dans le cas (qui nous intéresse principalement) où $T=A^n$, nous serons amené à considérer des hyperplans affines comme 
$\{a \} \times A^{n-1}$ où $a \in A_{\rm tors}$. Ces derniers sont encore des $A^{n-1}$-torseurs. 

\medskip

Comme dans le cas où $T=A^n$, on appelle {\it fonctionnelle affine} toute application $\chi: T \rightarrow A$ de la forme
$$t_0 + \mathbf{a} \mapsto \chi_0(\mathbf{a})$$
où $t_0$ est un élément de $T_{\tors}$ et $\chi_0 : A^n \rightarrow A$ un morphisme de la forme
$\mathbf{a} = (a_1, \dots, a_n) \mapsto \sum r_i a_i$
où les $r_i$ sont des entiers.\footnote{Autrement dit un morphisme standard, sauf dans le cas où $A$ est une courbe elliptique à multiplication complexe.} Rappelons que $\chi$ est {\it primitif} si les coordonnées  
$ (r_1, \dots, r_n) \in \mathbf{Z}^n$ de $\chi_0$ sont premières entre elles dans leur ensemble. Dans ce cas le lieu d'annulation de $\chi$ est un translaté de $\mathrm{ker}(\chi_0)$ qui est isomorphe à l'image de l'application linéaire $A^{n-1} \to A^n$ associée à une base du sous-module de $\Z^n$ orthogonal au vecteur $\mathbf{r}$. 
Le lieu d'annulation de $\chi$ a donc, comme $T$, une structure de torseur sous l'action de $A^{n-1}$ et il est muni d'une notion de points de torsion. 

On appelle \emph{hyperplan} le lieu d'annulation (ou abusivement ``noyau'') d'une fonctionnelle affine primitive, ou de manière équivalente l'image d'une application  
$A^{n-1} \rightarrow T$ linéaire relativement à un morphisme $A^{n-1} \rightarrow A^n$ induit par une matrice entière de taille $(n-1) \times n$.

\begin{definition}
Un \emph{arrangement d'hyperplans} $\Hyp$ est un fermé de Zariski dans $T$ réunion d'hyperplans. La taille $\# \Hyp$ est le nombre d'hyperplans distincts de cet arrangement. 
\end{definition}

Comme dans le cas linéaire (Lemme \ref{affine}) on a le lemme suivant.

\begin{lemma} \label{affineb}
Si $\Hyp$ contient $n$ fonctionnelles affines $\chi$ dont les vecteurs associés $\mathbf{r} \in \Z^n$ sont linéairement indépendants alors 
le complémentaire $T-\Hyp$ est affine. 

Lorsque $A$ est une courbe elliptique, c'est même une équivalence. 
\end{lemma}

\section[Cohomologie des arrangements d'hyperplans]{Opérateurs de dilatation et cohomologie des arrangements d'hyperplans} \label{S:dil}

On appelle {\em application de dilatation} toute application $[s] : T \rightarrow T$ associée à un entier $s>1$ et de la forme
$$[s] : t+ \mathbf{a} \mapsto t  + s \mathbf{a}$$
pour un certain $t \in T_{\tors}$.
L'image d'un hyperplan par une application de dilatation est encore un hyperplan. 

\'Etant donné un arrangement d'hyperplans $\Hyp$ définis par des fonctionnelles affines $\chi_1 , \ldots , \chi_r$ on peut, comme dans le cas $T=A^n$, trouver une application de dilatation $[s]$ qui préserve $\Hyp$. Quitte à augmenter $s$ on peut de plus supposer que si $\chi_1 , \ldots , \chi_r$ est une collection de fonctionnelles affines définissant des hyperplans dans $\Hyp$, alors $[s]$ préserve les composantes connexes de la préimage de $0$ par 
$$(\chi_1 , \ldots , \chi_r) : T \to A^r;$$
ces composantes connexes sont en effet des fermés de Zariski dans $T$. 
On fixe dorénavant un tel choix de dilatation $[s]$. Pour ne pas alourdir les notations on notera simplement 
$$H^*(T-\Hyp, \C)^{(1)} \subset H^*(T-\Hyp, \C)$$
le sous-espace caractéristique de $[s]_*$ associé à la valeur propre $1$. Noter que cette fois ce sous-espace dépend {\it a priori} du choix de la dilatation. 
 
On utilisera les mêmes notations pour les espaces de formes différentielles.

\subsection{Faisceau de formes différentielles et un théorème de Clément Dupont} 
Une forme méromorphe $\omega$ sur $T$ est {\it à pôles logarithmiques le long de $\Hyp$} si au voisinage de chaque point $p \in T$ on peut décomposer $\omega$ comme combinaison linéaire sur $\C$ de formes du type  
\begin{equation} \label{E31}
\nu \wedge  \bigwedge_{J} \frac{d f_j}{f_j}
\end{equation}
où $\nu$ est une forme holomorphe au voisinage de $p$ et chaque indice $j \in J$ paramètre un hyperplan $H_j$ de $\Hyp$ passant par $p$ défini par une équation linéaire locale $f_j = 0$. Notons que cette définition est indépendante du choix des $f_j$ puisque ceux-ci sont uniquement déterminés à une unité locale près.  

Les formes méromorphes sur $T$ à pôles logarithmiques le long de $\Hyp$ forment un complexe de faisceaux d'espaces vectoriels complexes sur $T$ que, suivant Dupont \cite{DupontC}, nous notons $\Omega^\bullet_{\langle T , \Hyp \rangle}$. On prendra garde au fait que ce complexe est en général strictement contenu dans le complexe de Saito $\Omega^\bullet_T (\log \Hyp)$ lorsque le diviseur $\Hyp$ n'est pas à croisements normaux. 

Dupont \cite[Theorem 1.3]{DupontC} démontre que l'inclusion entre complexes de faisceaux 
\begin{equation} \label{Dqi} 
\Omega^\bullet_{\langle T , \Hyp \rangle} \hookrightarrow j_* \Omega^\bullet_{T - \Hyp},
\end{equation}
où $j$ désigne l'inclusion de $T-\Hyp$ dans $T$, est un quasi-isomorphisme. On rappelle ci-dessous les ingrédients principaux de sa démonstration. 
   
\subsection{Opérateurs de dilatation} Il découle du lemme suivant que l'application de dilatation $[s]$ induit un opérateur $[s]_*$ sur $\Omega^\bullet_{\langle T , \Hyp \rangle}$. 
  
\begin{lemma}
Pour tout ouvert $V \subset T$, le morphisme de trace
$$\Omega^*_{\mathrm{mer}}([s]^{-1} V) \rightarrow \Omega^*_{\mathrm{mer}}(V)$$
induit par $[s]$ envoie $\Omega^\bullet_{\langle T , \Hyp \rangle} ([s]^{-1} V)$ dans $\Omega^\bullet_{\langle T , \Hyp \rangle} (V)$. 
Autrement dit, la trace définit un morphisme de faisceaux 
$$[s]_* \Omega^\bullet_{\langle T , \Hyp \rangle} \rightarrow \Omega^\bullet_{\langle T , \Hyp \rangle}.$$
\end{lemma}
\proof 
Soient $p_1$ et $p_2$ deux points de $T$ tels que $[s] p_1 = p_2$ et soit $\omega_1$ une section locale de $\Omega^\bullet_{\langle T , \Hyp \rangle}$ définie sur un voisinage de $p_1$. Puisque $[s]$ préserve $\Hyp$, l'image par $[s]$ d'un hyperplan $H_1$ de $\Hyp$ passant par $p_1$ est un hyperplan $H_2$ de $\Hyp$ passant par $p_2$. 

Soit maintenant $f_2 = 0$ une équation locale pour l'hyperplan $H_2$ dans un voisinage $W$ de $p_2$. Comme l'application $[s]$ est un biholomorphisme local, quitte à rétrécir $W$, l'application $f_1 = [s]^* f_2$ définit une équation locale pour l'hyperplan $H_1$ au voisinage de $p_1$. De plus, le morphisme trace induit par $[s]$ envoie $f_1 \in \Omega^0  ([s]^{-1} W)$ sur $\deg([s]) f_2$ et donc 
$$[s]_* \frac{df_1}{f_1} = \frac{df_2}{f_2}.$$
Comme $\Omega^\bullet_{\langle T , \Hyp \rangle}([s]^{-1} V)$ est engendré par des formes qui localement peuvent s'exprimer comme produits extérieurs de formes régulières et de formes du type $df_1/f_1$, le lemme s'en déduit. 
\qed 
  
\begin{definition}
Soit 
$$H^* (T , \Omega^j_{\langle T , \Hyp \rangle})^{(1)} \subset H^*(T , \Omega^j_{\langle T , \Hyp \rangle})$$
le sous-espace caractéristique de $[s]_*$ associé à la valeur propre $1$, c'est-à-dire le sous-espace des classes de cohomologie qui sont envoyées sur $0$ par une puissance de $[s]_*-1$. 
\end{definition}

Le but de ce chapitre est la démonstration du théorème qui, dans les contextes multiplicatifs et elliptiques, remplace le théorème de Brieskorn invoqué dans le cas affine pour construire la classe \eqref{E:Sa}. 
  
\begin{theorem} \label{P:Brieskorn}
Supposons $T-\Hyp$ affine. Pour tout degré $j \leq n = \dim(T)$ les formes dans $H^0(T, \Omega^j_{\langle T , \Hyp \rangle})^{(1)}$ sont fermées et l'application naturelle
$$H^0(T, \Omega^j_{\langle T , \Hyp \rangle})^{(1)} \rightarrow H^j(T-\Hyp)^{(1)}$$
est un isomorphisme. 
\end{theorem}

Après avoir terminé la rédaction de ce chapitre nous avons appris que, dans le cas multiplicatif, il existe en fait, comme dans le cas affine, une algèbre de formes différentielles algébriques fermées sur $T-\Hyp$, qui est isomorphe à la cohomologie de $T-\Hyp$. Dans le cas des arrangements d’hyperplans affines, cette algèbre est celle d’Arnol’d et Brieskorn. Dans le cas des arrangements toriques, elle est plus difficile à décrire, notamment parce que la cohomologie n’est pas engendrée en degré 1 en général. De Concini et Procesi \cite[Theorem 5.2]{DeConciniProcesi} traitent le cas unimodulaire; ce cas-là est analogue au cas affine, l’algèbre est engendrée par des représentants d’une base du $H^1$ du tore et les formes $d\log (\chi-a)$, où $\{\chi =a\}$ est une équation d’un des hyperplans de l'arrangement. Une description de l'algèbre dans le cas général est donnée par Callegaro, D'Adderio, Delucchi, Migliorini et Pagaria \cite{CDDMP}.
  
\subsection{Mise en place de la démonstration (par récurrence)}
 
On démontre le théorème \ref{P:Brieskorn} par récurrence sur la paire $(n=\dim \ T , \#(\Hyp))$, relativement à l'ordre lexicographique. L'initialisation de la récurrence se fera en démontrant directement le théorème 
\begin{itemize}
\item lorsque $\# \Hyp = n$, dans le cas elliptique, et
\item lorsque $\Hyp$ est vide, dans le cas multiplicatif.
\end{itemize}

\medskip
\noindent
{\it Remarque.} Lorsque $A$ est elliptique et que $\# \Hyp = n$ --- cas d'initialisation de la récurrence --- les vecteurs de $\Z^n$ associés aux hyperplans de $\Hyp$ sont nécessairement linéairement indépendants d'après le lemme \ref{affineb}. Le diviseur $\Hyp  \subset T$ est {\it à croisements normaux simples}.

\medskip

\`A chaque étape de récurrence on procède comme suit~: supposons que $\Hyp'$ soit un ensemble non vide d'hyperplans avec $T-\Hyp'$ affine. 
Supposons de plus $\# \Hyp' > n$ dans le cas elliptique. Considérons un hyperplan $H$ dans $\Hyp'$ tel que $T-\Hyp$, avec $\Hyp = \Hyp '-H$, soit affine. 
Notons $\Hyp \cap H$ l'ensemble des hyperplans de $H$ --- vu comme $A^{n-1}$-torseur --- obtenus par intersection de $H$ avec les hyperplans (de $T$) appartenant à $\Hyp$. L'arrangement 
$$H - (\Hyp \cap H),$$
étant un fermé de Zariski de la variété affine $T-\Hyp$, est affine.

On peut maintenant considérer le diagramme 
\begin{equation} \label{E:diag}
T - \Hyp '  \stackrel{j}{\hookrightarrow} T - \Hyp
\stackrel{\iota}{\hookleftarrow}  H -  (\Hyp \cap H),
\end{equation}
où $\iota$ est une immersion fermée et $j$ une injection ouverte. 
Dans la suite $\iota$ désignera plus généralement l'inclusion de $H$ dans $T$. Noter que $\iota (H -  (\Hyp \cap H))$ est un diviseur lisse de $T-\Hyp$ de complémentaire l'image de $j$.

Le point clé de la démonstration est le fait, non trivial et démontré par Clément Dupont, que la suite de faisceaux
\begin{equation} \label{dupont}  
0 \rightarrow \Omega^q_{\langle T , \Hyp \rangle} \rightarrow \Omega^{q}_{\langle T , \Hyp ' \rangle}  \stackrel{\mathrm{Res}}{\longrightarrow} \iota_* \Omega^{q-1}_{\langle H , \Hyp \cap H \rangle} \rightarrow 0
\end{equation}
est \emph{exacte}. 
Dans le paragraphe suivant on décrit les grandes lignes de la démonstration de Dupont.

\section{Travaux de Dupont} \label{travaux dupont}
 
Suivant Dupont \cite[Definition 3.4]{DupontC}, on munit le faisceau $\Omega^\bullet_{\langle T , \Hyp \rangle}$ d'une filtration ascendante 
$$W_0 \Omega^\bullet_{\langle T , \Hyp \rangle} \subset W_1 \Omega^\bullet_{\langle T , \Hyp \rangle} \subset \ldots $$
appelé {\it filtration par les poids}. Le $k$-ème terme de cette filtration $W_k \Omega^\bullet_{\langle T , \Hyp \rangle}$ est engendré par les formes du type \eqref{E31} avec $\# J \leq k$. On a donc 
$$W_0 \Omega^\bullet_{\langle T , \Hyp \rangle} = \Omega^\bullet_T \quad \mbox{et} \quad W_q  \Omega^q_{\langle T , \Hyp \rangle} =  \Omega^q_{\langle T , \Hyp \rangle}.$$

Pour comprendre le complexe gradué associé considérons $\mathrm{gr}_{n}^W  \Omega^\bullet_{\langle T , \Hyp \rangle}$, 
c'est-à-dire le quotient 
$$W_n  \Omega^\bullet_{\langle T , \Hyp \rangle} /W_{n-1}  \Omega^\bullet_{\langle T , \Hyp \rangle}.$$ 
Localement, la classe d'une forme de type \eqref{E31} dans $\mathrm{gr}_{n}^W  \Omega^\bullet_{\langle T , \Hyp \rangle}$ ne dépend pas du choix des $f_j$ puisque ceux-ci sont uniquement déterminés à une unité locale près. Un choix approprié de coordonnées locales permet alors de vérifier que la fibre de $\mathrm{gr}_{n}^W  \Omega^\bullet_{\langle T , \Hyp \rangle}$ au-dessus d'un point $p$ est nulle à moins que $p$ ne soit contenu dans $n$ hyperplans linéairement indépendants. 
   
Maintenant, si $p$ appartient à $n$ hyperplans linéairement indépendants, Dupont \cite[Theorem 3.11]{DupontC} construit une application naturelle de l'algèbre de Orlik--Solomon locale en $p$ vers $\mathrm{gr}_{n}^W  \Omega^\bullet_{\langle T , \Hyp \rangle}$. Les hyperplans de $\Hyp$ passant par $p$ induisent un arrangement d'hyperplans (linéaires) $\Hyp^{(p)}$ dans l'espace tangent en $p$ à $T$ qu'un choix de coordonnées locales identifie à $\C^n$. Orlik et Solomon \cite{OrlikSolomon} définissent une algèbre graduée $\mathrm{OS}_\bullet (\Hyp^{(p)})$ par générateurs et relations. En énumérant $H_1 , \ldots , H_\ell$ les hyperplans de $\Hyp^{(p)}$, autrement dit les hyperplans de $\Hyp$ passant par $p$, Orlik et Solomon définissent en particulier $\mathrm{OS}_n (\Hyp^{(p)})$ 
comme étant engendrée par des éléments $e_J$ pour $J = \{ j_1 , \ldots , j_n \} \subset  \{1 , \ldots , \ell \}$ avec $j_1 < \ldots < j_n$, quotientée par l'espace des relations
engendré par les combinaisons linéaires
\begin{equation} \label{RelOS}
\sum_{i=0}^n (-1)^i e_{K-\{j_i \}},
\end{equation}
où $K=\{ j_0 , \ldots , j_n \} \subset \{1 , \ldots , \ell \}$ avec $j_0 < j_1 < \ldots < j_n$. 

L'application 
\begin{equation} \label{appDupont}
\mathrm{OS}_n (\Hyp^{(p)}) \to \mathrm{gr}_{n}^W  \Omega^\bullet_{\langle T , \Hyp \rangle}
\end{equation}
construite par Dupont est alors définie de la manière suivante. Pour chaque $j$ dans $\{1 , \ldots , \ell \}$ on choisit une équation locale $f_j =0$ définissant l'hyperplan $H_j$ au voisinage de $p$, et on envoie chaque générateur $e_J$ de $\mathrm{OS}_n (\Hyp^{(p)})$ sur la forme 
$$\wedge_{j \in J} \frac{df_j}{f_j} \in \mathrm{gr}_{n}^W  \Omega^\bullet_{\langle T , \Hyp \rangle}.$$
Rappelons que ces éléments ne dépendent pas des choix faits pour les $f_j$, le fait qu'ils vérifient les relations \eqref{RelOS} s'obtient, comme dans le cas classique des arrangements d'hyperplans dans $\C^n$, en explicitant la relation de dépendance entre les $H_j$  $(j \in K)$. Cela montre que l'application \eqref{appDupont} est bien définie. Par définition, elle est surjective au voisinage de $p$.

Dupont définit plus généralement une application
\begin{equation} \label{appDupont2}
\bigoplus_{\codim(S) = k} \iota_{S*} \Omega_S^{\bullet} \otimes \mathrm{OS}_S (\Hyp )   \rightarrow \mathrm{gr}_{k}^W  \Omega^\bullet_{\langle T , \Hyp \rangle},
\end{equation}
où la somme porte maintenant sur les {\it strates}, c'est-à-dire les composantes connexes d'intersections d'hyperplans, de codimension $k$ et $\iota_S$ désigne l'inclusion de la strate $S$ dans $T$. \`A tout point $p$ d'une strate $S$ il correspond une algèbre d'Orlik--Solomon locale $\mathrm{OS}_\bullet (\Hyp^{(p)})$ et on note 
$$\mathrm{OS}_S (\Hyp) \subset \mathrm{OS}_\bullet (\Hyp^{(p)})$$
le sous-espace engendré par les monômes $e_J$ où $J$ décrit les sous-ensembles d'hyperplans appartenant à $\Hyp^{(p)}$ et dont l'intersection est exactement la trace $S^{(p)}$ de la strate $S$ dans l'espace tangent en $p$ à $T$. Noter que $\mathrm{OS}_S (\Hyp)$ ne dépend pas du choix du point $p$ dans $S$. 

Le facteur correspondant à une strate $S = \{ p \}$ de codimension $n$ dans \eqref{appDupont2} correspond à l'application \eqref{appDupont}.
Pour construire l'application \eqref{appDupont2} en général, considérons une strate $S$ de codimension $k$ et un point $p \in S$. Un choix d'équations locales 
$f_j =0$ pour les hyperplans de $\Hyp$ contenant $S$ permet --- par un argument similaire à celui détaillé ci-dessus --- de construire une application
\begin{equation} \label{E:app1}
\underline{\mathrm{OS}_S (\Hyp)} \rightarrow \mathrm{gr}_{k}^W  \Omega^\bullet_{\langle T , \Hyp \rangle}.
\end{equation}
(Ici $\underline{\mathrm{OS}_S (\Hyp)}$ désigne le faisceau constant associé à $\mathrm{OS}_S (\Hyp)$.)
On étend maintenant \eqref{E:app1} en une application  
\begin{equation} \label{E:app2}
\Omega_T^{\bullet} \otimes \mathrm{OS}_S (\Hyp)   \rightarrow \mathrm{gr}_{k}^W  \Omega^\bullet_{\langle T , \Hyp \rangle}
\end{equation}
par linéarité. Finalement \eqref{E:app2} se factorise par $\iota_{S*} \Omega_S^{\bullet } \otimes \mathrm{OS}_S (\Hyp)$~: étant donné une forme $\nu$ dans le noyau de l'application $\Omega_T^{\bullet } \to \iota_{S*} \Omega_S^{\bullet }$ et un monôme $e_J \in \mathrm{OS}_S (\Hyp)$ il s'agit de montrer que \eqref{E:app2} envoie $\nu \otimes e_J$ sur $0$. Notons $f_j = 0$ les équations locales des hyperplans de $J$. Dans les coordonnées les $f_j$ engendrent l'idéal de $S$, on peut donc écrire $\nu$ comme une combinaison linéaire 
$$\nu \sum \omega_j f_j + \sum \omega_j ' (df_j).$$
Il nous reste alors finalement à vérifier que
$$f_j \wedge \left( \frac{df_1}{f_1} \wedge \dots \wedge \frac{df_k}{f_k} \right) \in W_{k-1}  \Omega^\bullet_{\langle T , \Hyp \rangle}$$
et
$$df_j \wedge \left( \frac{df_1}{f_1} \wedge \dots \wedge \frac{df_k}{f_k} \right) \in W_{k-1}  \Omega^\bullet_{\langle T , \Hyp \rangle}$$
ce qui résulte des définitions. 

L'application \eqref{appDupont2} est surjective, ce qui peut se voir en calculant les fibres. Dupont \cite[Theorem 3.6]{DupontC} montre en fait que  \eqref{appDupont2} est un isomorphisme de complexes. On extrait ici de sa démonstration le lemme crucial (pour nous) suivant.

\begin{lemma} \label{L:exactdupont}
La suite courte de faisceaux \eqref{dupont} est exacte.
\end{lemma}
\begin{proof} 
La suite de complexes \eqref{dupont} induit des suites courtes
\begin{equation} \label{dupontW}
\tag{$\mathbf{W}_k$} 0 \rightarrow W_k \Omega^q_{\langle T , \Hyp \rangle} \rightarrow W_k \Omega^{q}_{\langle T , \Hyp ' \rangle}  \rightarrow W_{k-1} \Omega^{q-1}_{\langle H , \Hyp \cap H \rangle} \rightarrow 0
\end{equation}
et
\begin{equation} \label{dupontgr}
\tag{$\mathbf{gr}^W_k$} 0 \rightarrow \mathrm{gr}^W_k \Omega^q_{\langle T , \Hyp \rangle} \rightarrow \mathrm{gr}^W_k \Omega^{q}_{\langle T , \Hyp' \rangle}  \rightarrow  \mathrm{gr}^W_{k-1} \Omega^{q-1}_{\langle H , \Hyp \cap H \rangle} \rightarrow 0.
\end{equation}
Puisque $W_k \Omega^\bullet_{\langle T , \Hyp \rangle} = \Omega^\bullet_{\langle T , \Hyp \rangle}$ pour $k$ suffisamment grand, pour démontrer le lemme il suffit de montrer que les suites \eqref{dupontW} sont exactes. Et comme par ailleurs les suites courtes \eqref{dupontW} sont exactes à gauche et à droite, il nous faut seulement vérifier qu'elles sont exactes au milieu. 

Remarquons maintenant que les suites \eqref{dupontW} et \eqref{dupontgr} donne lieu à une suite exacte courte de complexes
\begin{equation} \label{SEcx}
0 \rightarrow (\mathbf{W}_{k-1}) \rightarrow (\mathbf{W}_k) \rightarrow (\mathbf{gr}_k^W) \to 0 .
\end{equation}
La suite exacte longue associée en cohomologie implique que si les suites  $(\mathbf{W}_{k-1})$ et \eqref{dupontgr} sont exactes au milieu alors \eqref{dupontW} l'est aussi. Puisque $(\mathbf{W}_{0})$ est évidemment exacte, une récurrence sur $k$ réduit la démonstration du lemme à la démonstration que les suites \eqref{dupontgr} sont exactes au milieu.

Pour abréger les expressions, on notera simplement $\mathrm{OS}^{(k)} (\Hyp)$ le faisceau apparaissant à la source de l'application \eqref{appDupont2}. En tensorisant avec les faisceaux $\iota_{S*} \Omega_S^\bullet$ les suites exactes entre algèbres d'Orlik--Solomon obtenues par ``restriction et effacement'' on obtient que les suites de complexes
$$0 \to \mathrm{OS}^{(k)} (\Hyp) \to \mathrm{OS}^{(k)} (\Hyp' ) \to \mathrm{OS}^{(k)} (\Hyp \cap H ) \to 0$$
sont exactes. En considérant les diagrammes commutatifs
$$\xymatrix{
0 \ar[r] & \mathrm{OS}^{(k)} (\Hyp) \ar[r]  \ar[d] & \mathrm{OS}^{(k)}_{\Hyp'} \ar[r] \ar[d]  &  \mathrm{OS}^{(k-1)}_{\Hyp \cap H} \ar[d] \ar[r]  &0 \\
0 \ar[r] & \mathrm{gr}_k^W \Omega^\bullet_{\langle T ,\Hyp \rangle} \ar[r] & \mathrm{gr}_k^W \Omega^\bullet_{\langle T , \Hyp' \rangle} \ar[r]&  \mathrm{gr}_{k-1}^W \Omega^\bullet_{\langle H ,  \Hyp \cap H \rangle} \ar[r] & 0.
}
$$
on montre, dans un même élan, par récurrence sur le cardinal de l'arrangement d'hyperplans que les flèches verticales sont des isomorphismes et que la ligne du bas est exacte au milieu. Par hypothèse de récurrence on peut supposer que les flèches verticales à gauche et à droite sont des isomorphismes. Une chasse au diagramme montre alors que la ligne du bas est exacte au milieu pour tout $k$. Comme expliqué ci-dessus cela suffit à montrer que les suites courtes \eqref{dupontW} sont exactes. La suite exacte longue en cohomologie associée à la suite exacte courte de complexes \eqref{SEcx} implique alors que la suite \eqref{dupontgr} est enfin partout exacte (pas seulement au milieu). Finalement les deux lignes des diagrammes commutatifs ci-dessus sont exactes et, puisque les flèches verticales à gauche et à droite sont des isomorphismes, le lemme des cinq implique que la flèche du milieu est également un isomorphisme. Ce qui permet de continuer la récurrence.

\end{proof}
  
Terminons cette section par un autre résultat important (pour nous en tout cas) de Clément Dupont \cite[Theorem 3.13]{DupontC}.
  
\begin{lemma} \label{L:exactdupont}
L'inclusion \eqref{Dqi} est un quasi-isomorphisme.
\end{lemma}
\begin{proof}[Esquisse de démonstration]
La cohomologie de complexe $j_* \Omega_{T-\Hyp}^\bullet$ en $p$ est précisément la cohomologie de l'arrangement local $\Hyp^{(p)}$ d'hyperplans dans $\C^n$, c'est-à-dire l'algèbre de Orlik--Solomon $\mathrm{OS}_\bullet (\Hyp^{(p)})$. D'un autre côté on peut calculer la cohomologie de $\Omega^\bullet_{\langle T , \Hyp \rangle}$ en se servant de la filtration par les poids et de l'isomorphisme \eqref{appDupont2}, on obtient le même résultat. 
\end{proof}

\section[Démonstration du théorème 3.5]{Démonstration du théorème \ref{P:Brieskorn}}

\subsection{Initialisation dans le cas multiplicatif} \label{B1}
On l'a dit, on démontre le théorème \ref{P:Brieskorn} par récurrence. Dans le cas multiplicatif où $A =\mathbf{G}_m$, l'initialisation de la récurrence correspond au cas où $\Hyp$ est vide. 
   
\begin{lemma} \label{L:B1}
Le sous-espace 
\begin{equation} \label{E:H0(1)}
H^0(\mathbf{G}_m^n, \Omega^\bullet)^{(1)}
\end{equation}
est de dimension $1$ engendré par la forme $\wedge_{i=1}^{n} dz_i/z_i$ où $z_i$ désigne la $i$-ème coordonnée sur $\mathbf{G}_m^n$. 

En particulier \eqref{E:H0(1)} est concentré en degré $n$.
\end{lemma}
\proof 
Une forme différentielle holomorphe sur $\mathbf{G}_m^n$ se décompose de manière unique en série de Laurent à plusieurs variables. Il suffit donc de considérer les monômes 
\begin{equation} \label{E:monome}
\left( \prod_{i=1}^n z_i^{a_i} \right) \cdot \bigwedge_{j \in J} \frac{dz_j}{z_j} \quad (J \subset \{1 , \ldots , n \} , \ a_1 , \ldots , a_n  \in \Z).
\end{equation}
Maintenant, pour tout entier $a \in \Z$ on a
$$\sum_{w : w^s=z} w^a = \left\{ \begin{array}{cl}
s z^{a/s} & \mbox{si } s | a, \\
0 & \mbox{sinon}.
\end{array} \right.$$
En poussant en avant le monôme \eqref{E:monome} par $[s]_*$ on obtient donc
\begin{multline*}
\sum_{w_i^s = z_i} \left( \left( \prod_{i=1}^n w_i^{a_i} \right) \cdot \bigwedge_{j \in J} \frac{dw_j}{w_j}  \right) \\
= \left\{ \begin{array}{ll}
s^{n- \# J} \left( \prod_{i=1}^n z_i^{a_i /s} \right) \cdot \bigwedge_{j \in J} \frac{dz_j}{z_j} & \mbox{si } s | a_i \mbox{ pour tout } i, \\
0 & \mbox{sinon}. 
\end{array} \right. 
\end{multline*}
Il s'en suit que tout élément de $H^0(\mathbf{G}_m^n , \Omega^\bullet )$ appartient à un sous-espace $[s]_*$-invariant de dimension fini et qu'un monôme \eqref{E:monome} appartient à $H^0(\mathbf{G}_m^n, \Omega^\bullet)^{(1)}$ si et seulement si $\# J = n$ et tous les $a_i$ sont nuls. 
\qed 
 
La cohomologie de $\mathbf{G}_m^n$ est égale à l'algèbre extérieure sur les formes fermées $dz_i /z_i$. Le sous-espace $H^\bullet ( \mathbf{G}_m^n )^{(1)}$ est donc de dimension $1$ engendré par la forme $\wedge_{i=1}^{n} dz_i/z_i$. Le lemme \ref{L:B1} implique donc que, dans ce cas où $A =\mathbf{G}_m$ et $\Hyp$ est vide, le théorème \ref{P:Brieskorn} est bien vérifié.

\subsection{Initialisation dans le cas elliptique} \label{B2}
Dans le cas elliptique où $A =E$, l'initialisation de la récurrence correspond au cas où $\Hyp$ est constitué de $n$ hyperplans linéairement indépendants. Dans ce cas, le faisceau $\Omega^\bullet_{\langle T , \Hyp \rangle}$ coïncide avec le faisceau $\Omega^\bullet_T (\log \Hyp)$ des formes différentielles logarithmiques. C'est un complexe de faisceaux et, d'après Atiyah--Hodge \cite{AtiyahHodge}, Griffiths \cite{Griffiths} et Deligne \cite{DeligneH2}, on a un isomorphisme canonique
\begin{equation} \label{G-D}
H^k (T-\Hyp , \C)  = \mathbf{H}^k (T , \Omega^\bullet_T (\log \Hyp) );
\end{equation}
cf. \cite[Corollary 8.19]{Voisin}. De plus, la suite spectrale de Hodge--de Rham
$$E_1^{pq} = H^p (T , \Omega^q_T (\log \Hyp)) \Longrightarrow \mathbf{H}^{p+q} (T , \Omega^\bullet_T (\log \Hyp ) )$$
dégénère à la première page, cf. \cite[Corollaire 3.2.13, (ii)]{DeligneH2}.
 
Le lemme suivant implique donc que, dans ce cas, le théorème \ref{P:Brieskorn} est bien vérifié.
\begin{lemma} \label{L:affirmation}
Supposons que $\Hyp$ soit constitué de $n$ hyperplans linéairement indépendants. Alors le sous-espace $H^p(T, \Omega^q_T (\log \Hyp) )^{(1)}$ est nul pour tout $p > 0$.
\end{lemma} 
\begin{proof}
On démontre le lemme en explicitant les groupes de cohomologie $H^p(T, \Omega^q_T (\log \Hyp) )^{(1)}$. On commence par remarquer que l'on a un isomorphisme
\begin{equation} \label{E:isomF}
\bigoplus_{i=1}^n \mathcal{O}(H_i) \stackrel{\sim}{\longrightarrow} \Omega^1_T (\log \Hyp) ,
\end{equation}
où $\Hyp = \{ H_1, \dots, H_n \}$. Pour tout $i \in \{ 1 , \ldots , n \}$ soit $\omega_i$ une $1$-forme différentielle sur $T$ obtenue comme tirée en arrière par un caractère définissant $H_i$ d'une $1$-forme partout non-nulle sur $A$. 
L'application 
$$(f_i) \in \oplus_i \mathcal{O} (H_i) \longmapsto \sum f_i \omega_i $$
est globalement définie et induit un isomorphisme au niveau des fibres; elle réalise l'isomorphisme \eqref{E:isomF}. 

Noter que $[s]^* \omega_i = s \omega_i$; pour tout $f \in \mathcal{O} (H_i )$ on a donc
$$[s]_*( f \omega_i) = (s^{-1} [s]_* f) \omega_i.$$
Comme par ailleurs 
$$\Omega^q_T (\log \Hyp) = \wedge^q \Omega^1_T (\log \Hyp) \simeq  \bigoplus_{\substack{J \subset \{1, \ldots , n \} \\ \# J = q}} \mathcal{O}(\sum_{j \in J} H_j )$$
le lemme \ref{L:affirmation} découle du lemme qui suit.
\end{proof} 

\begin{lemma} 
Soit $J \subset \{1, \ldots , n \}$ de cardinal $q$ et soit $p$ un entier strictement positif. Le morphisme $[s]_*$  opère sur $H^p(T, \mathcal{O}(\sum_{j \in J} H_j))$ par le scalaire $s^{2n-p}$ et l'espace $H^p(T, \mathcal{O}(\sum_{j \in J} H_j))$ est réduit à $0$ si $n\leq p+q$. 
\end{lemma}
En particulier $s^q$ n'est pas un valeur propre généralisée.
\proof
Lorsque $q = n$, le faisceau $\mathcal{O}(\sum_{i=1}^n H_i)$ est ample. Or il découle du théorème d'annulation de Mumford \cite[\S III.16]{Mumford} que la cohomologie supérieure, de degré $p>0$, des fibrés amples est nulle sur une variété abélienne.

Supposons maintenant $q<n$ et supposons pour simplifier $T=A^n$. Les caractères définissant les hyperplans $H_j$ ($j \in J$) induisent un morphisme de variétés abéliennes 
$$\pi: A^n \rightarrow A^q$$
et le fibré en droites $\mathcal{O}(\sum_{j \in J} H_j )$ sur $A^n$ est égal au tiré en arrière $\pi^* \mathcal{L}$ du fibré en droites 
$\mathcal{L}$ sur $A^q$ associé au diviseur $z_1 \dots z_q = 0$ égal à la somme des hyperplans de coordonnées. 

Commençons par remarquer qu'en général si $f : C \to B$ est un morphisme de variétés abéliennes à fibres connexes alors les $R^j f_* \mathcal{O}_C$ sont des fibrés vectoriels triviaux. Pour le voir, on note $F$ le noyau de $f$; c'est une variété abélienne et il existe une sous-variété abélienne $B' \subset C$ telle que 
\begin{itemize}
\item l'application somme $F \times B' \to C$ et
\item l'application induite $B' \to B$
\end{itemize} 
soient des isogénies (voir par exemple \cite[Proposition 12.1]{Milne}). Notons $C' = F \times B' $ et $f' : C' = F \times B' \to B'$. Le faisceau de cohomologie $R^j f'_* \mathcal{O}_{C'}$ est un faisceau constant de fibre $H^j (F , \mathcal{O}_F)$. Comme les translations de $F$ opèrent trivialement sur $H^j (F , \mathcal{O}_F)$, ce faisceau constant sur $B'$ descend sur $B$ en un faisceau, nécessairement égal à $R^j f_* \mathcal{O}_C$, constant de fibre $H^j (F, \mathcal{O}_F)$. 

En général le morphisme $\pi$ n'est pas à fibres connexes mais, comme $\pi$ est un morphisme de groupes, les composantes connexes des fibres sont toutes de même dimension et le quotient $B$ de $A^n$ par les composantes connexes des fibres est un revêtement fini de $A^q$. On peut donc factoriser $\pi$ en la composition $\pi = g \circ \pi'$ de morphismes $\pi': A^n \to B$ et $g : B \to A^q$ avec $\pi'$ à fibres connexes et $g$ fini. Notons $F$ la fibre de $\pi'$. La suite spectrale de Leray associée au morphisme $\pi '$ s'écrit~:
\begin{equation} \label{suiteLeray}
H^r (B , R^p  \pi '_*(\pi {}^* \mathcal{L} )) \Longrightarrow H^{p+r} (A^n , \pi {}^* \mathcal{L}).
\end{equation}
Par la formule de projection, on a un isomorphisme
\begin{equation} \label{ProjFormula}
R^p  \pi '_*(\pi {}^* \mathcal{L} ) = R^p  \pi '_*(\pi' {}^* g^* \mathcal{L} ) \cong g^*\mathcal{L} \otimes R^p \pi ' _* \mathcal{O}_{A^n} = g^*\mathcal{L} \otimes H^p (F ,  \mathcal{O}_F),
\end{equation}
où la dernière égalité découle du paragraphe précédent. 

Puisque les composantes connexes des fibres de $\pi$ sont des copies de $A^{n-q}$, on obtient que 
$$H^r (B , R^p  \pi '_*(\pi^* \mathcal{L} )) \cong H^r (B , g^*\mathcal{L}) \otimes H^p (A^{n-q} , \mathcal{O}_{A^{n-q}} ).$$
Il découle à nouveau du théorème d'annulation de Mumford que ce groupe s'annule pour $r>0$ car $g^*\mathcal{L}$ est ample. La suite spectrale \eqref{suiteLeray} dégénère donc et on obtient un isomorphisme 
\begin{equation} \label{HpDec}
H^p (A^n , \pi^* \mathcal{L}) \cong H^0 (B , g^*\mathcal{L} ) \otimes H^p (A^{n-q} , \mathcal{O}_{A^{n-q}} ).
\end{equation}
Notons de plus que 
$$H^0 (B, g^*\mathcal{L} ) = H^0 (A^q , g_* (g^* \mathcal{L})) = H^0 (A^q , \mathcal{L} \otimes g_* \mathcal{O}_B)$$
où $g_* \mathcal{O}_B$ est une somme de fibrés en droites de torsion $\mathcal{T}_1 , \ldots  ,\mathcal{T}_r$. On a donc
$$H^0 (B, g^*\mathcal{L} ) = \bigoplus_{j=1}^r H^0 (A^q , \mathcal{L}_j')$$ 
où $\mathcal{L}_j ' = \mathcal{L} \otimes \mathcal{T}_j$.

Par hypothèse, la multiplication par $s$ sur $A^n$ préserve les composantes connexes des fibres de $\pi$ et donc sa décomposition de Stein. L'endomorphisme induit $[s]_*$ sur
$$H^p (A^n , \mathcal{O}(\sum_{j \in J} H_j)) = H^p (A^n , \pi^* \mathcal{L})$$
préserve donc la décomposition \eqref{HpDec} et chaque terme $H^0 (A^q , \mathcal{L}_j')$ de la décomposition de $H^0 (B, g^*\mathcal{L} )$. On étudie l'action de $[s]_*$ sur chacun des facteurs du produit tensoriel~:
\begin{itemize}
\item L'endomorphisme $[s]_*$ opère sur\footnote{Ici on utilise que la suite spectrale de Hodge--de Rham pour $A^{n-q}$ dégénère en $E_1$.}
$$H^p (A^{n-q} , \mathcal{O}_{A^{n-q}} ) \hookrightarrow H^p(A^{n-q}, \C)$$
par $s^{2n-2q-p}$. De plus, le groupe de gauche est nul si $n \leq p+q$. 
\item  Comme $H^r (A^q , \mathcal{L}_j ')=0$ pour $r>0$, il découle par exemple de \cite[Theorem 13.3]{Milne} que chaque groupe $H^0 (A^q , \mathcal{L}_j')$ est de rang $1$. La section canonique est une fonction theta sur laquelle l'homomorphisme $[s]_*$ opère par multiplication par $s^{2q}$; voir par exemple \cite[Eq. (3.1)]{Beauville}. 
\end{itemize}
Finalement $[s]_*$ opère sur $H^p (A^n , \mathcal{O}(\sum_{j \in J} H_j))$ par le scalaire $s^{2n-p}$ et le groupe $H^p (A^n , \mathcal{O}(\sum_{j \in J} H_j))$ est nul si $n\leq p+q$. Il est donc exclu d'obtenir $s^q$ comme valeur propre. 
\qed

\subsection{Pureté, par récurrence}
 
\begin{lemma}  \label{C1}
La partie invariante de la cohomologie $H^j(T -\Hyp )^{(1)}$ est pure (comme structure de Hodge \cite{DeligneH2}) de poids $2j$ pour $j \leq \dim(T)$.  En particulier, si $T-\Hyp$ est affine, 
la cohomologie est pure en tout degré. 
\end{lemma}
\proof
Pour $n=0$ le lemme est immédiat. Lorsque $n$ est strictement positif mais que $\# \Hyp=0$,
la cohomologie invariante n'intervient qu'en degré cohomologique maximal. Dans le cas elliptique ce degré est $2 \dim (T)$ et il n'y a rien à démontrer. Dans le cas multiplicatif le lemme se déduit du fait que $H^1(\mathbf{G}_m)$ est pur de poids $2$. 

Pour l'étape de récurrence (sur $(n, \# \Hyp )$), on considère un triplet \eqref{E:diag} tel que $(T,\Hyp )$ et $(H , (\Hyp \cap H))$ vérifient tous les deux les conclusions du lemme \ref{C1}. La suite exacte longue de Gysin en cohomologie associée à ce triplet s'écrit
\begin{multline} \label{LESGysin}
H^{j-2}( H - (H \cap \Hyp ),  \C(-1)) \stackrel{\delta_1}{\rightarrow} H^j( T - \Hyp ) \rightarrow  H^j(T - \Hyp ' ) \\ \rightarrow H^{j-1}(H - (H \cap \Hyp),  \C(-1)) 
\stackrel{\delta_2}{\rightarrow} 
H^{j+1}( T - \Hyp ).
\end{multline}
Comme $\C (-1)$ est de poids $2$ et que les applications de la suite exacte ci-dessus préservent la filtration par le poids et sont compatibles avec l'action de $[s]_*$, on conclut que $H^j(T - \Hyp ' )^{(1)}$ est pure de poids $2j$ pour $j \leq \dim(T)$. Cela prouve la première assertion du lemme.
 
La deuxième partie du lemme s'en déduit puisque, si $T-\Hyp$ est affine, la cohomologie de $T-\Hyp$ s'annule pour $j > \dim(T)$. 
\qed

\begin{lemma} \label{exactness in topology}
Supposons que $T - \Hyp$ et $H - (H \cap \Hyp)$ soient affines. Alors  
$$0 \rightarrow H^j( T - \Hyp)^{(1)} \rightarrow  H^j(T - \Hyp')^{(1)} \stackrel{\mathrm{Res}}{\longrightarrow} H^{j-1}(H - (H \cap \Hyp),  \C(-1))^{(1)} \rightarrow 0$$
est une suite exacte courte. 
\end{lemma}
\proof
Puisque $T - \Hyp$ et $H - (H \cap \Hyp)$ sont affines, il découle du lemme \ref{C1} que pour tout degré $j$ les parties invariantes de la cohomologie $H^j(T -\Hyp )^{(1)}$ et $H^j(H- (H \cap \Hyp) )^{(1)}$ sont pures de poids $2j$.

Les membres de la suite \eqref{LESGysin} appartiennent à la catégorie des $\C$-espaces vectoriels de dimension finie munis d'une action linéaire de $\Z$ (le groupe engendré par $[s]_*$). Le foncteur qui à un tel espace $H$ associe $H^{(1)}$ est exact. On peut en effet décomposer $H$ en la somme directe $\oplus H^{(\lambda)}$ des sous-espaces caractéristiques de $[s]_*$ et l'application 
$$\alpha  \mapsto \frac{1}{\prod_{\lambda \neq 1} (1- \lambda)^m} \prod_{\lambda \neq 1} ([s]_* - \lambda)^m \quad \left( m = \dim \ H \right)$$
est un projecteur sur $H^{(1)}$.\footnote{C'est pour cette raison qu'on considère le sous-espace caractéristique $H^{(1)}$ plutôt que le sous-espace des vecteurs $1$-propre.}  

En passant aux parties invariantes dans la suite \eqref{LESGysin}, on obtient donc une suite exacte dont le terme de gauche est pur de poids $2j-2$ alors que le terme suivant est pur de poids $2j$. L'application $\delta_1$ est donc nulle. De même le terme de droite est pur de poids $2j+2$ alors que le terme précédent $H^{j-1}(H - (H \cap \Hyp),  \C(-1))^{(1)}$ est pur de poids $2j$. L'application $\delta_2$ est donc nulle elle aussi. 
\qed

\medskip

On compare maintenant ces suites exactes à la suite exacte longue associée à \eqref{dupont}, c'est-à-dire, 
\begin{multline} \label{dupont2}  
H^p(T, \Omega^q_{\langle T ,\Hyp \rangle} )^{(1)} \rightarrow H^p(T, \Omega^q_{\langle T , \Hyp ' \rangle })^{(1)}  \rightarrow H^p(H,  \Omega^{q-1}_{\langle H , \Hyp \cap H\rangle })^{(1)} \\
	\stackrel{\delta}{\rightarrow} H^{p+1}(T, \Omega^q_{\langle T ,\Hyp \rangle})^{(1)}. 
\end{multline}
Le quasi-isomorphisme \eqref{Dqi} donne plus précisément lieu à une suite spectrale qui calcule la cohomologie de $T-\Hyp$ et le lemme suivant montre que la première page de cette suite spectrale permet de calculer la partie invariante de la cohomologie de $T-\Hyp$.

\begin{lemma}\label{C2}
Supposons que $A$ soit une courbe elliptique et que $T-\Hyp$ soit affine. Alors la suite spectrale de Hodge--de Rham 
\begin{equation} \label{SSHdR}
H^p(T, \Omega^q_{\langle T , \Hyp \rangle})^{(1)} \Longrightarrow H^{p+q}(T-\Hyp)^{(1)}
\end{equation} 
dégénère à la première page et tous les morphismes de connexions $\delta$ dans \eqref{dupont2} sont nuls.    
\end{lemma}
\proof
On procède à nouveau par récurrence sur $(n , \# \Hyp)$. 
 
On initialise la récurrence avec le cas où $\Hyp$ est constitué de $n$ hyperplans linéairement indépendants. Alors $H$ est un diviseur à croisements normaux et $\Omega^q_{\langle T , \Hyp \rangle} = \Omega^q_T (\log \Hyp )$. Dans ce cas le lemme est une conséquence de la dégénérescence de la suite spectrale de Hodge--de Rham à pôles logarithmiques, cf. \S \ref{B2}. 

Supposons maintenant par récurrence que le lemme est démontré pour les arrangements $(H, H \cap \Hyp)$ et $(T, \Hyp)$; nous le vérifions alors pour l'arrangement $(T, \Hyp ')$ avec toujours $\Hyp '=\Hyp \cup H$. 

De la suite exacte courte \eqref{dupont} on tire
\begin{multline} \label{dps}  
\dim \ H^p(T, \Omega^q_{\langle T , \Hyp ' \rangle})^{(1)} \\ \leq \dim \ H^p(T, \Omega^q_{\langle T , \Hyp \rangle})^{(1)} +  \dim \ H^p(H, \Omega_{\langle H , H\cap \Hyp \rangle }^{q-1})^{(1)}.
\end{multline}
En sommant sur tous les couples $(p,q)$ tels que $p+q=j$ on obtient tour à tour les inégalités suivantes~:
\begin{equation*} 
\begin{split}
\dim \  H^j(T - \Hyp ' )^{(1)} & \stackrel{(i)}{\leq} \sum_{p+q=j} \dim \ H^p(T, \Omega^q_{\langle T , \Hyp ' \rangle} )^{(1)}   \\ 
& \stackrel{(ii)}{ \leq } \sum_{p+q=j} \left( \dim \ H^p(T, \Omega^q_{\langle T , \Hyp \rangle})^{(1)} +  \dim \ H^p(H, \Omega_{\langle H , H\cap \Hyp \rangle }^{q-1})^{(1)} \right) \\
& \stackrel{(iii)}{=} \dim \ H^j(T - \Hyp)^{(1)} + \dim  \ H^{j-1}(H - (\Hyp \cap H))^{(1)},
\end{split}
\end{equation*}
où l'on explique chacune des (in)égalités ci-dessous.
\begin{itemize}
\item[(i)] Découle de l'existence de la suite spectrale \eqref{SSHdR}; on a égalité pour tous les $j$ si et seulement si la suite spectrale dégénère à la première page. 
\item[(ii)] Découle de \eqref{dps}; on a égalité si et seulement les morphismes de connexions dans la suite exacte longue \eqref{dupont2} sont nuls. 
\item[(iii)] C'est l'étape de récurrence. 
\end{itemize}
Finalement le lemme \ref{exactness in topology} implique que toutes les inégalités ci-dessus sont en fait des égalités et donc que la suite spectrale pour $\Hyp '=\Hyp \cup H$ dégénère et que la suite exacte longue \eqref{dupont2} se scinde en suites exactes courtes. 
\qed

\medskip
\noindent
{\it Remarque.} Comme nous l'a fait remarquer un rapporteur, le formalisme général des complexes de Hodge mixtes \cite[Scholie 8.1.9, (v)]{DeligneH3} et les travaux de Dupont impliquent en fait directement que la suite spectrale de Hodge--de Rham \eqref{SSHdR} dégénère à la première page.

\medskip
	
\begin{lemma}\label{C3}
Supposons $A = \mathbf{G}_m$. La suite \eqref{dupont} donne lieu à des suites exactes courtes :
\begin{equation} \label{dupont3}  
0 \rightarrow H^0(T,  \Omega^j_{\langle T,\Hyp \rangle }) \rightarrow H^0(T, \Omega^{j}_{\langle T , \Hyp '  \rangle } ) \rightarrow H^0(H , \Omega^{j-1}_{\langle H , H \cap \Hyp \rangle}) \rightarrow 0.
\end{equation}
De plus, l'action de $[s]_*$ sur chacun des espaces impliqués est localement finie. 
\end{lemma}
\proof
La première assertion découle du fait que qu'il n'y a pas de cohomologie supérieure puisque $T$ est affine. La seconde affirmation s'obtient par récurrence; l'initialisation est conséquence du lemme \ref{L:B1}. 
\qed

\subsection{Fin de la démonstration du théorème \ref{P:Brieskorn}}
On peut maintenant démontrer le théorème par récurrence sur $(n, \# \Hyp)$. 

L'initialisation a été vérifiée aux \S \ref{B1} et \ref{B2}. Supposons maintenant par récurrence que le théorème est démontré pour les arrangements $(H, H \cap \Hyp)$ et $(T, \Hyp)$ (avec $T-\Hyp$ et $H-(H\cap \Hyp)$ affines donc). Vérifions alors le théorème pour l'arrangement $(T, \Hyp ')$ avec toujours $\Hyp '=\Hyp \cup H$. 

D'après les lemmes \ref{exactness in topology}, \ref{C1} et \ref{C2} on a des diagrammes commutatifs
$$
\xymatrix{
H^0(T,  \Omega^j_{\langle T,\Hyp \rangle })^{(1)}  \ar[d] \ar@{^{(}->}[r] & H^0(T, \Omega^{j}_{\langle T , \Hyp '  \rangle } )^{(1)}  \ar[d] \ar@{->>}[r] & H^0(H , \Omega^{j-1}_{\langle H , H \cap \Hyp \rangle})^{(1)} \ar[d]  \\ 
H^j(T - \Hyp )^{(1)} \ar@{^{(}->}[r] & H^j (T- \Hyp ' )^{(1)} \ar@{->>}[r] & H^{j-1}( H - (H \cap \Hyp ))^{(1)}  
}
$$
où les suites horizontales sont exactes et, par hypothèse de récurrence, les morphismes verticaux à gauche et à droite sont des isomorphismes. Le morphisme vertical du milieu est donc lui aussi un isomorphisme, ce qui démontre le théorème.
\qed

\medskip
\noindent
{\it Remarque.} Comme nous l'a fait remarquer un rapporteur, la théorie de Hodge mixte et les travaux de Dupont permettent de déduire directement du lemme \ref{C1} de pureté que 
$$H^0(T,  \Omega^\bullet_{\langle T,\Hyp \rangle })^{(1)} \to H^\bullet (T - \Hyp )^{(1)}$$
est un isomorphisme. Nous avons préféré maintenir notre approche un peu plus pédestre. Dans tous les cas, le c{\oe}ur de l'argument repose sur \cite{DupontC}.

\medskip

\chapter{Formes différentielles sur l'espace symétrique associé à $\SL_n (\C)$} \label{C:4}

\resettheoremcounters

Dans ce chapitre, on fixe un entier $n\geq 2$ et on note $V = \C^n$; on voit les éléments de $V$ comme des vecteurs colonnes. 

Une matrice $g \in \GL_n (\C)$ définit une forme hermitienne sur $\C^n$ de matrice hermitienne associée $g^{-*}g^{-1}$.\footnote{On note $M^\top$ la transposée d'une matrice $M$, et $M^*=\overline{M}^\top$. Lorsque $M$ est inversible on note enfin $M^{-\top}$ et $M^{-*}$ la transposée et la conjuguée de son inverse.} On en déduit une bijection 
\begin{equation*}
S:=\mathrm{GL}_n (\C ) / \mathrm{U}_n \simeq \left\{H  \; : \; H \text{ forme hermitienne définie positive sur } \C^n \right\}.
\end{equation*}
Dans ce chapitre on note 
$$G = \GL_n (\C) \quad \mbox{et} \quad K = \UU_n$$ 
de sorte que $S=G/K$. En identifiant $\R_{>0}$ au centre réel de $\GL_n  (\C)$ {\it via} l'application $s \mapsto s \cdot 1_n$, le quotient
$$X = \mathrm{GL}_n (\C) / \UU_n \R_{>0}$$ 
est l'espace symétrique associé au groupe $\mathrm{SL}_n (\C)$.

On prendra pour point base dans $S$ la métrique hermitienne $|\cdot |$. Pour $z = (z_1 , \ldots , z_n) \in \C^n$, vu comme vecteur colonne, on a $|z|^2 = |z_1 |^2 + \ldots + |z_n|^2$ et la forme hermitienne $H$ associée à une matrice $g \in \GL_n (\C)$ est donc donnée par $H(z,z) = | g^{-1} z |^2$.

Dans ce chapitre on explique que la théorie de Mathai--Quillen permet de construire deux formes différentielles $G$-invariantes naturelles 
$$\psi \in A^{2n-1}(S \times \C^n ) \quad \mbox{et} \quad \varphi \in A^{2n} (S \times \C^n )$$
qui décroissent rapidement le long des fibres de $S \times \C^n \to S$ et vérifient
\begin{enumerate}
\item $\varphi$ est une \emph{forme de Thom} au sens qu'elle est fermée et d'intégrale $1$ dans les fibres $\C^n$, et
\item la transformée de Mellin de $\psi$ en $0$ définit est une forme fermée sur $S\times (\C^n - \{ 0 \} )$ et sa restriction aux fibres $\C^n - \{0 \}$ représente la classe fondamentale. 
\end{enumerate}

\section{Formes de Mathai--Quillen}

Dans la suite de ce chapitre on note $\mathfrak{g}$ et $\mathfrak{k}$ les algèbres de Lie respectives de $G$ et $K$ et on note $\mathfrak{p}$ le supplémentaire orthogonal de $\mathfrak{k}$ dans $\mathfrak{g}$ relativement à la forme de Killing.

\subsection{Fibré $K$-principal au-dessus de $S$} La projection $G \to G/K$ permet de voir le groupe $G$ comme un fibré $K$-principal au-dessus de $S$. Une connexion sur $G$ est une $1$-forme $\theta \in A^1(G) \otimes \mathfrak{k}$ telle que
\begin{equation}
\begin{split}
\mathrm{Ad}(k)(k^*\theta) &= \theta, \qquad k \in K, \\
\iota_X \theta &=X, \qquad X \in \mathfrak{k}.
\end{split}
\end{equation}
Le fibré $G \to S$ est naturellement $G$-équivariant (relativement aux actions naturelles, à gauche, de $G$ sur lui-même et sur $S$). Soit $\theta$ la connexion $G$-invariante sur $G$ définit comme suit~: {\it via} les isomorphismes
\[
(A^1(G) \otimes \mathfrak{k})^{G \times K} \simeq (\mathfrak{g}^* \otimes \mathfrak{k})^K \simeq \mathrm{Hom}_K(\mathfrak{g},\mathfrak{k}),
\]
une connexion $G$-invariante correspond à une section $K$-équivariante de l'inclusion $\mathfrak{k} \hookrightarrow \mathfrak{g}$. On définit $\theta$ comme étant la connexion associée à la projection $p : \mathfrak{g} \to \mathfrak{k}$. La forme de connexion $\theta$ est donc explicitement donnée par la formule
\[
\theta = p(g^{-1}dg) = \frac{1}{2}(g^{-1}dg - g^* d( g^{-*})).
\]
Sa forme de courbure associée
\[
\Omega = (\Omega_{ij} )_{1 \leq i, j \leq n} = d\theta+\theta^2 \in A^2(G) \otimes \mathfrak{k}
\]
est $G$-invariante et horizontale, autrement dit $\iota_X \Omega = 0$ pour tout $X \in \mathfrak{k}$ --- identifié au champs de vecteur engendré par l'action à droite de $\mathfrak{k}$ on $G$.

\subsection{Fibré vectoriel associé} 
Soit 
$$G \times^K V = [ G \times V ] / K,$$ 
où l'action (à droite) de $K$ on $G \times V$ est $(g,v)k = (gk,k^{-1}v)$. C'est un fibré vectoriel au-dessus de $S=G/K$ et l'action 
$$h \cdot [g,v]=[hg,v]$$
l'équipe d'une structure de fibré $G$-équivariant. La forme hermitienne standard $v \mapsto v^*v$ sur $V$ le munit finalement d'une métrique $G$-\'equivariante.

\subsection{Une forme de Thom explicite} On rappelle ici l'expression de la forme de Thom construite par Mathai et Quillen \cite{MathaiQuillen} sur $G \times^K V$. Soient $z_1,\overline{z}_1, \ldots,z_n, \overline{z}_n$ les coordonnées standard sur $V$. On note respectivement $z$ et $dz$ les vecteurs colonnes $(z_1,\ldots,z_n)^\top$ et $(dz_1,\ldots,dz_n)^\top$. 

\'Etant donné un sous-ensemble $I=\{i_1,\ldots,i_p\} \subseteq \{1 \ldots,n\}$, avec $i_1 < \cdots < i_p$, et un vecteur $\xi=(\xi_1,\ldots,\xi_n )$ on note 
$$\xi^I = \xi_{i_1} \cdots \xi_{i_p} \quad \mbox{et} \quad \xi^{I*} = \xi_{i_p} \cdots \xi_{i_1}.$$ 
On note $I'=\{1,\ldots,n \}-I$ le complémentaire de $I$ et on définit une signe $\epsilon(I,I')$ par l'égalité $dz^I dz^{I'} = \epsilon(I,I') dz_1 \cdots dz_n$.

L'expression
\begin{equation} \label{E:UMQ}
U = \left(\frac{i}{2\pi}\right)^n e^{-|z|^2} \sum_{I, J} \epsilon(I,I') \epsilon(J,J') \det(\Omega_{IJ}) (dz+\theta z)^{I'} \overline{(dz+\theta z)}^{J' *},
\end{equation}
où la somme porte sur tous les couples $(I,J)$ de sous-ensembles de $\{1,\ldots,n\}$ avec $|I|=|J|$ et $\Omega_{IJ}=(\Omega_{ij})_{i \in I, j \in J}$, définit une forme $K$-invariante, fermée, horizontale et de degré $2n$ sur $G \times V$ et donc une forme dans $A_{d=0}^{2n}(G \times^K V)$; c'est la forme de Thom de Mathai--Quillen.

\subsection{Formes différentielles sur $S \times V$} 
La représentation standard de $G$ dans $V$ fait du fibré trivial 
$$E = S \times V \to S$$
un fibré $G$-équivariant --- pour tout $g \in G$, l'isomorphisme $\alpha(g):E \xrightarrow{\sim} g^*E$ est donné, dans chaque fibre, par la multiplication par $g$. Le fibré $E$ est naturellement muni d'une métrique hermitienne $G$-équivariante $|\cdot|_h^2$ donnée par 
$|v|_H^2 = v^*Hv$, où $H=H(g)=g^{-*}g^{-1}$. 
L'application
\[
\Phi: E \to G \times^{K} V, \qquad \Phi(gK,v) = (g,g^{-1}v)
\]
est une isométrie $G$-équivariante. 

\begin{definition}
Soit 
\[
\varphi = \Phi^*( U) \in A^{2n} (E)^G \quad \mbox{et} \quad \psi = \iota_X \varphi \in A^{2n-1} (E)^G, 
\]
où $X = \sum_i (z_i \partial_{i} + \overline{z}_i \overline{\partial}_i )$ est le champs de vecteur radial sur $E$.
\end{definition}
La forme $\varphi$ est fermée, rapidement décroissante et d'intégrale $1$ le long des fibres de $S \times V \to V$; c'est une \emph{forme de Thom} en ce sens. 

\section{Une $(2n-1)$-forme fermée sur $X \times (\C^n - \{ 0 \})$} \label{S:42}

La multiplication par un réel strictement positif $s$ sur $V=\C^n$ induit une application 
$$[s] :  S \times V \to S \times V.$$
Il découle des définitions que 
\begin{equation} \label{E:tddt1}
d ([s]^* \psi ) =  s \frac{d}{ds} ( [s]^* \varphi ).
\end{equation}

La proposition suivante est essentiellement due à Mathai et Quillen \cite[\S 7]{MathaiQuillen}.
\begin{proposition} \label{P:eta}
L'intégrale 
\begin{equation} \label{E:eta}
\eta =  \int_0^{+\infty} [s]^* \psi \frac{ds}{s} 
\end{equation}
est convergente et définit une forme \emph{fermée} dans $A^{2n-1} (X \times (\C^n - \{ 0 \}))^G$ dont la restriction à chaque fibre $\C^n -\{0 \}$ de $X \times (\C^n - \{ 0 \} )$ représente la classe fondamentale. 
\end{proposition} 
\begin{proof} L'expression explicite \eqref{E:UMQ} de la forme de Thom de Mathai--Quillen permet de montrer que l'intégrale converge en $0$; en fait lorsque $s$ tend vers $0$ la forme $[s]^* \psi$ est un $O(s)$. 

Maintenant si $v \in V$ est un vecteur non nul, l'image $s v$ tend vers l'infini avec $s$ et le fait que $\psi$ soit rapidement décroissante dans les fibres de $S \times V \to S$ implique que l'intégrale \eqref{E:eta} converge sur $S \times (\C^n - \{ 0 \})$. Comme $\eta$ est de plus invariante, par construction, par multiplication dans les fibres, elle définit finalement une forme sur $X \times (\C^n - \{ 0 \} )$. 

La forme $\eta$ est de degré $2n-1$ et $G$-invariante. On calcule sa différentielle à l'aide de \eqref{E:tddt1}~:
\begin{equation*}
\begin{split}
d \eta & = d \left( \int_0^{+\infty} [s]^* \psi \frac{ds}{s}  \right) \\
& = \int_0^{+\infty}  d([s]^* \psi)  \frac{ds}{s} \\
& =  \int_0^{+\infty} \frac{d}{ds} ([s]^* \varphi )  ds .
\end{split}
\end{equation*} 
La décroissance rapide de $\varphi$ le long des fibres de $E \to S$ montre comme ci-dessus que sur $S \times (\C^n - \{ 0 \} )$ les formes $[s]^* \varphi$ tendent vers $0$ quand $s$ tend vers l'infini. Enfin, il découle de \eqref{E:UMQ} que lorsque $s$ tend vers $0$ la forme $[s]^* \varphi$ tend vers 
\begin{equation} \label{E:chern}
\omega = \left( \frac{i}{2\pi} \right)^n \Phi^* (\det \Omega).
\end{equation}
La différentielle $d\eta$ est donc égale à \eqref{E:chern}. Montrons maintenant que cette forme est identiquement nulle. 

La $2n$-forme $G$-invariante \eqref{E:chern} est égale au tiré en arrière, par la projection $E \to S$, du représentant de Chern--Weil de la classe de Chern de degré maximal $c_n (E)$.\footnote{Le $G$-fibré $E$ étant plat la forme \eqref{E:chern} admet donc nécessairement une primitive $G$-invariante, ce qu'est précisément le forme $\eta$. Montrer que $\eta$ est fermée est un problème analogue à l'existence du relevé canonique dans la proposition \ref{P4}.} 

Soit $B  \subset \GL_n (\C)$ le sous-groupe de Borel des matrices triangulaires supérieures. La décomposition $G = B K$ induit un difféomorphisme $B$-équivariant
$$f: B \times^{B \cap K} V \to G \times^K V.$$
Il suffit donc de montrer que la $2n$-forme différentielle 
\begin{equation} \label{E:f*MQ}
\left( \frac{i}{2\pi} \right)^n f^* (\det \Omega)
\end{equation}
est identiquement nulle. Maintenant, comme $B \cap K = \UU_1^n$, le fibré hermitien $B \times^{B \cap K} V$ se scinde métriquement en une somme directe de fibrés en droite au-dessus de $B/(B \cap K)$. Par fonctorialité la forme \eqref{E:f*MQ} au-dessus d'un point $[b]$ dans $B/(B \cap K)$ s'obtient comme $b$-translaté du produit de $n$ formes sur les facteurs $\UU_1$ de $B \cap K$, chacune de ces formes égale à \eqref{E:f*MQ} avec $n=1$. On est ainsi réduit à considérer le cas $n=1$ et la nullité de \eqref{E:f*MQ} résulte du fait qu'il n'y a pas de $2$-forme non nulle sur le cercle. 

\medskip
\noindent
{\it Remarque.} Une démonstration plus conceptuelle est possible : par l'astuce de Weyl il suffit en effet de démontrer que la forme invariante correspondant à \eqref{E:chern} sur le dual compact de $S$ 
$$S^c = \UU_n = (\UU_n \times \UU_n) /\mathrm{U}_n$$ 
est identiquement nulle. Mais celle-ci est harmonique car fermée invariante et représente la classe de Chern maximale d'un fibré plat. Elle est donc nécessairement nulle. 

\medskip

Il nous reste à voir que $\eta$ représente la classe fondamentale de $S \times (\C^n - \{ 0 \} )$. Il suffit pour cela de se restreindre à la fibre au-dessus du point base de $S$. Alors $\varphi$ est simplement 
$$
\left( \frac{i}{2\pi} \right)^n e^{-|z|^2} dz_1 \wedge \ldots \wedge dz_n \wedge d\overline{z}_n \wedge \ldots \wedge d \overline{z}_1 
 = \frac{1}{\pi^n} e^{- \sum_j x_j^2} dx_1 \wedge \ldots \wedge dx_{2n}$$
où les $x_j$ sont les coordonnées dans une base orthonormée et orienté positivement du $\R$-espace vectoriel $\C^n$. Dans ces coordonnées le champ radial $X$ s'écrit $\sum_j x_j \partial_{x_j}$ et la forme $\psi$ est égale à 
$$\frac{1}{\pi^n} e^{- \sum_j x_j^2} \sum_{j=1}^{2n} (-1)^{j-1} x_j dx_1 \wedge \ldots \wedge \widehat{dx_j} \wedge \ldots \wedge dx_{2n}.$$
On trouve finalement que la forme $\eta$, en restriction à la fibre $\C^n - \{ 0 \}$ au-dessus du point base de $S$, est égale à 
$$\frac{\Gamma (n)}{2 \pi^n} \frac{ \sum_{j=1}^{2n} (-1)^{j-1} x_j dx_1 \wedge \ldots \wedge \widehat{dx_j} \wedge \ldots \wedge dx_{2n}}{ (x_1^2 + \ldots + x_{2n}^2)^n}$$
qui est la forme volume normalisée sur la sphère $\mathbf{S}^{2n-1}$. 
\end{proof}

\medskip
\noindent
{\it Remarque.} On peut identifier $S$ à $X \times \R_{>0}$ {\it via} l'application $\GL_n (\C )$-équivariante $gU_n \mapsto ( [g] , | \det (g) |^{1/n})$. Les formes $\varphi$ et $\psi$ définissent alors deux formes invariantes sur $X \times \R_{>0} \times V$ et on a  
$$\varphi = \alpha - \psi \wedge \frac{dr}{r},$$
où 
$\alpha$ est une forme différentielle sur $X \times \R_{>0} \times V$ qui ne contient pas de composante $dr$ le long de $\R_{>0}$ et dont la restriction à chaque $X \times \{s \} \times V$ $(s \in \R_{>0})$ est égale à $[s]^* \varphi$. Le fait que $\varphi$ soit fermée est équivalent à 
\eqref{E:tddt1} et, au-dessus de $V-\{0 \}$, la forme $\eta$ est le résultat de l'intégration partielle
$$\eta =  \int_{\R_{>0}} \varphi.$$

\section{Calculs explicites dans le cas $n=1$} 

On suppose dans ce paragraphe que $n=1$. Dans ce cas $S= \R_{>0}$ et $X$ est réduit à un point. On peut alors calculer explicitement les formes de Mathai--Quillen, cf. \cite{Takagi}. On trouve que pour $(r,z) \in \R_{>0} \times \C$ on a
\begin{equation} \label{E:phiN1}
\varphi = \frac{i}{2\pi} e^{- r^2 |z|^2 } \left( r^2 dz \wedge d \overline{z} -r^2 ( z d\overline{z} - \overline{z}dz ) \wedge \frac{dr}{r} \right)
\end{equation} 
de sorte que 
\begin{equation} \label{E:psiN1}
\psi = - \frac{i}{2\pi} \left(  r^2 |z|^2 e^{- r^2 |z|^2 } \right) \left( \frac{dz}{z} - \frac{d\overline{z}}{\overline{z}} \right)
\end{equation}
et 
\begin{equation} \label{E:etaN1}
\eta = - \frac{i}{4\pi}  \left( \frac{dz}{z} - \frac{d\overline{z}}{\overline{z}} \right).
\end{equation}

\section{Formes de Schwartz et représentation de Weil} \label{S:44}

L'espace total du fibré $E$ est un espace homogène sous l'action du groupe affine $G \ltimes V = \GL_n (\C) \ltimes \C^n$, où 
$$(g,v) \cdot (g' , v' ) = (gg' , g v' + v).$$
L'action à gauche du groupe affine sur $E= S \times V$ est transitive et le stabilisateur du point base $([K] , 0)$ est le groupe $K=\UU_n$ plongé dans le groupe affine {\it via} l'application $k \mapsto (k,0)$. 

Soit $\mathcal{S} (V)$ l'espace de Schwartz de $V$. On appelle {\it représentation de Weil} du groupe affine dans $\mathcal{S} (V)$ la représentation $\omega$ donnée par~:
$$\omega (g , v) : \mathcal{S} (V) \to \mathcal{S} (V) ; \quad \phi \mapsto \left( w \mapsto \phi ( g^{-1} (w-v)) \right).$$

Soit\footnote{L'isomorphisme est induit par l'évaluation en le point base $(eK, 0)$ dans $S \times V$.}
$$A^{k} (E , \mathcal{S}(V))^{G \ltimes V} := \left[ \mathcal{S} (V) \otimes A^k (E) \right]^{G \ltimes V} \cong \left[\mathcal{S}(V) \otimes \wedge^\bullet (\mathfrak{p} \oplus V)^*   \right]^K$$
l'espace des $k$-formes différentielles invariantes sur $E$ à valeurs dans $\mathcal{S}(V)$. 

\begin{definition}
Pour tout $w \in V$, on pose 
$$\widetilde{\varphi}  (w) = t_{-w}^* \varphi \in A^{2n} (E)  \quad \mbox{et} \quad \widetilde{\psi}(w) = t_{-w}^* \psi \in A^{2n-1} (E),$$
où $t_{-w} : E \to E$ désigne la translation par $-w$ dans les fibres de $E$.
\end{definition}

\begin{lemma}
Les applications $w \mapsto \widetilde{\varphi}  (w)$ et $w \mapsto \widetilde{\psi}(w)$ définissent des éléments 
$$\widetilde{\varphi} \in A^{2n} (E , \mathcal{S}(V))^{G \ltimes V} \quad \mbox{et} \quad \widetilde{\psi} \in A^{2n-1} (E , \mathcal{S}(V))^{G \ltimes V}.$$
\end{lemma}
\begin{proof} Montrons d'abord l'invariance, c'est à dire que pour tous $(g,v)$ et $(h,w)$ dans $G \ltimes V$, on a
\begin{equation} \label{E:invariancephipsi}
(g,v)^* \widetilde{\varphi} (gw+v) = \widetilde{\varphi} (w) \quad \mbox{et} \quad (g,v)^* \widetilde{\psi} (gw+v) = \widetilde{\psi} (w).
\end{equation}
Ici $(g,v)^*$ désigne le tiré en arrière par le difféomorphisme de $E$ induit par l'élément $(g,v) \in G \ltimes V$.

L'identité \eqref{E:invariancephipsi} résulte des définitions; on la vérifie pour $\widetilde{\varphi}$, le cas de $\widetilde{\psi}$ se traitant de la même manière~:
\begin{equation*}
\begin{split}
(g,v)^* \widetilde{\varphi} (gw+v) & = \left[ (1,v) \cdot (g , 0) \right]^* \widetilde{\varphi} (gw+v) \\
& = g^* \left[ (1,v)^* \widetilde{\varphi} (gw+v) \right] \\
& = g^* \left[ (1,v)^* ((1, -gw -v)^* \varphi ) \right] \\
& = g^* \left[ (1 , -gw )^* \varphi \right] = \left[ (1, -gw) \cdot (g, 0 ) \right]^* \varphi \\
& = (g , -gw)^* \varphi = \left[ (g,0) \cdot (1 , -w) \right]^* \varphi \\
& = (1 , -w)^* (g^* \varphi ) = (1,-w)^* \varphi = \widetilde{\varphi} (w).
\end{split}
\end{equation*}

Il nous reste à vérifier que $\widetilde{\varphi}$ --- et de la même manière $\widetilde{\psi}$ --- définit bien une forme différentielle à valeurs dans $\mathcal{S}(V)$. Par $(G \ltimes V)$-invariance, il suffit pour cela de remarquer que pour tout $X$ dans $\wedge^{2n} (\mathfrak{p} \oplus V)$, la fonction 
$$(\widetilde{\varphi} (w))_{(eK , 0)} (X) = (t_{-w}^* \varphi )_{(eK , 0)} (X) = [\varphi (-w)] (X) $$ 
est bien une fonction de Schwartz de $w$. 
\end{proof}

\medskip
\noindent
{\it Remarque.} De notre point de vue, l'apport majeur du formalisme de Mathai--Quillen est de fournir ces fonctions tests ``archimédiennes'' auxquelles on peut alors appliquer le formalisme automorphe. Il est en effet notoirement délicat de construire les bonnes fonctions tests à l'infini en général. 

\medskip

Par définition on a 
$$\widetilde{\varphi} (0) = \varphi \quad \mbox{et} \quad \widetilde{\psi} (0) = \psi$$
et il découle en particulier de \eqref{E:invariancephipsi} que 
\begin{equation} \label{E:pullback}
v^* \widetilde{\varphi} (w) = (v-w)^* \varphi  \quad \mbox{et} \quad v^* \widetilde{\psi} (w) = (v-w)^* \psi \in A^\bullet (S) ,
\end{equation} 
où l'on a identifié un vecteur de $V$ à la section constante $S \to E$ qu'il définit. 

Finalement, les formes $[s]^* \varphi$ et $[s]^* \psi$ définissent à leur tour des formes différentielles $\widetilde{[s]^* \varphi}$ et $\widetilde{[s]^* \psi}$ dans $A^{\bullet} (E, \mathcal{S}(V))^{G \ltimes V}$ qui vérifient 
\begin{equation} \label{E:invariancephipsi2}
\widetilde{[s]^* \varphi} (w) = [s]^* \widetilde{\varphi} (s w) \quad \mbox{et} \quad \widetilde{[s]^* \psi} (w) = [s]^* \widetilde{\psi} (s w).
\end{equation}

Le lemme suivant découle de la construction. 

\begin{lemma} \label{L:convcourant}
1. Lorsque $s$ tend vers $+\infty$, les formes $\widetilde{[s]^*\varphi} (0) = [s]^* \varphi$ convergent uniformément sur tout compact de $S \times (\C^n - \{ 0 \})$ vers la forme nulle et convergent au sens des courants vers le courant d'intégration le long de $S \times \{ 0 \}$.

2. Soit $v \in \C^n - \{ 0 \}$. Lorsque $s$ tend vers $+\infty$, les formes $$\widetilde{[s]^*\varphi} (v) = [s]^* \widetilde{\varphi} (sv)$$ convergent uniformément exponentiellement vite sur tout compact de $S \times \C^n$ vers la forme nulle. 

3. Lorsque $s$ tend vers $0$, les formes $\widetilde{[s]^*\varphi} (0) = [s]^* \varphi$ convergent uniformément sur tout compact de $S \times \C^n$ vers la forme nulle.

\end{lemma}

\medskip
\noindent
{\it Remarque.} Dans le dernier cas, les formes $\widetilde{[s]^*\varphi} (0) = [s]^* \varphi$ convergent uniformément sur tout compact de $S \times \C^n$ vers la forme $\omega$ définie en \eqref{E:chern} et dont on a montré qu'elle est identiquement nulle. Dans le dernier chapitre, on considère la forme $\varphi$ associée au produit 
$$\GL_n  (\R ) / \SO_n \times \SL_2 (\R ) /\SO_2 \times \C^n.$$
Dans ce cas les formes $\widetilde{[s]^*\varphi} (0) = [s]^* \varphi$ convergent uniformément sur tout compact de $S \times \C^n$ vers une forme invariante non nulle en générale.

\chapter{Compactifications de Satake, de Tits et symboles modulaires} \label{C:5}

\resettheoremcounters

Les espaces symétriques admettent de nombreuses compactifications équivariantes, cf. \cite{BorelJi}. On rappelle ici deux d'entre elles, la compactification (minimale) de Satake et la compactification de Tits. On étudie ensuite le comportement de la forme $\eta$, définie au chapitre précédent, lorsque l'on s'approche du bord de ces compactifications. 

\section{Compactification de Satake}

L'espace $S$ est un cône ouvert dans l'espace vectoriel réel $\mathcal{H}$ des matrices hermitiennes de rang $n$. L'action de $G$ sur $S$ 
$$g : H \mapsto g^{-*} H g^{-1}$$
s'étend en une action sur $\mathcal{H}$ qui induit une action de $G$ sur l'espace projectif $\mathrm{P} (\mathcal{H})$. L'application 
$$i : X \to \mathrm{P} (\mathcal{H})$$
est un plongement $G$-équivariant. L'adhérence de son image $i(X)$ dans l'espace compact $\mathrm{P} (\mathcal{H})$ est donc une compactification $G$-équivariante de $X$; elle est appelée {\it compactification (minimale) de Satake} et notée $\overline{X}^S$. Elle est convexe et donc contractile.

La compactification de Satake $\overline{X}^S$ se décompose en une union disjointe 
\begin{equation} \label{E:Xsatake1}
\overline{X}^S = X \bigcup_{W} b(W) ,
\end{equation}
où $W$ parcourt les sous-espaces non nuls et propres de $\C^n$ et $b(W)$ désigne l'image dans $\mathrm{P} (\mathcal{H})$ du cône sur l'ensemble des matrices hermitiennes semi-définies positives dont le noyau est exactement $W$. 

On note $P_W$ le sous-groupe parabolique de $\SL_n (\C)$ qui préserve $W$. Si $g \in \SL_n (\C)$ est tel que 
$$W = g \langle e_1 , \ldots , e_{j} \rangle \quad (j=\dim W) ,$$
le groupe $P_W$ est obtenu en conjuguant par $g$ le groupe 
$$P=P_j = P_{\langle e_1 , \ldots , e_{j} \rangle}.$$
Soient
\begin{equation*}
\begin{split}
N=N_j &= \left\{ \left. \begin{pmatrix} 1_j & x \\ 0 & 1_{n-j}\end{pmatrix} \ \right| \  x \in M_{j, (n-j)} (\C ) \right\} \\
M=M_j &=  \left\{ \left. \begin{pmatrix} A & 0 \\ 0 & B \end{pmatrix} \ \right| \ A \in \mathrm{GL}_j(\C), \ B \in \mathrm{GL}_{n-j}(\C) , \ |\det(A)| =|\det(B)|=1
 \right\} \\ 
A=A_j &=  \left\{ \left. a(t_1,t_2):=\begin{pmatrix} t_1 1_j & 0 \\ 0 & t_2 1_{n-j} \end{pmatrix} \ \right| \ t_1, t_2 \in \R_{>0}, \ \det a(t_1,t_2) = 1 
\right\} .
\end{split}
\end{equation*}
Un élément $g \in \mathrm{SL}_n(\C)$ peut s'écrire 
\[
g = u m a k, \quad u \in N, \ m \in M, \ a \in A, \ k \in \SU_n
\]
Dans cette décomposition $u$ et $a$ sont uniquement déterminés par $g$, et $m$ et $k$ sont déterminés à un élément de $M \cap \SU_n$ près. On peut donc écrire $a=a(t_1(g),t_2(g))$ avec $t_1(g)$ et $t_2(g)$ déterminé par $g$; on a plus précisément
\[
t_1(g)^{-j} = \det (H(g)|_{ \langle e_1,\ldots,e_j \rangle }),
\]
où $H(g)|_{\langle e_1,\ldots,e_j \rangle}$ désigne la restriction de la métrique hermitienne $v \mapsto |g^{-1}v|^2$ déterminée par $g$ au sous-espace $\langle e_1,\ldots, e_j \rangle$.

\'Etant donné un réel $t \in \R^+$, on pose
\[
A_t = \{a(t_1,t_2) \in A \; : \;  t_1 /t_2  \geq t \}.
\]

\begin{definition}
On appelle \emph{ensemble de Siegel} associé à un sous-ensemble relativement compact $\omega \subset NM$ le sous-ensemble
\[
\mathfrak{S}_j (t,\omega) := \omega A_t \cdot \SU_n \subset \mathrm{SL}_n(\C);
\]
on parlera aussi d'ensemble de Siegel pour son image dans $X$.

On appelle plus généralement, \emph{ensemble de Siegel associé à $W$} tout sous-ensemble de la forme
\[
\mathfrak{S}_W (g,t,\omega) := g \omega A_t \cdot \SU_n
\]
où $g \in \mathrm{GL}_n(\C)$ vérifie $g \langle e_1,\ldots,e_{\mathrm{dim} W} \rangle =W$.
\end{definition}

Dans la décomposition \eqref{E:Xsatake1}, les composantes de bord $b(W)$ appartiennent toutes à une même $G$-orbite $G\cdot X_{n-j}$ où l'on note $X_{n-j}$ l'espace symétrique associé au sous-espace $\langle e_{j +1},\ldots, e_n \rangle$ plongé dans $P (\mathcal{H})$ {\it via} l'inclusion
$$\GL_{n-j} (\C) \to \mathcal{H} ; \quad A \mapsto \begin{pmatrix} 0 & 0 \\ 0 & A \end{pmatrix}$$
de sorte que les formes hermitiennes de l'image aient pour noyau $\langle e_1,\ldots, e_j \rangle$.

Le groupe $G$ a en fait exactement $n$ orbites~:
\begin{equation} \label{E:Xsatake2}
\overline{X}^S = X \sqcup G \cdot X_{n-1} \sqcup \ldots \sqcup G  \cdot X_1.
\end{equation}
Finalement la topologie sur $\overline{X}^S$ peut être comprise inductivement~: un ouvert $U$ relativement compact voisinage d'un point $p$ dans $X_{n-1}$ se relève en un ouvert relativement compact de $\GL_{n-1} (\C) \subset M_1$. Le produit de cet ouvert avec un sous-ensemble relativement compact de $N_1$ définit un sous-ensemble relativement compact $\omega \subset N_1 M_1$. On construit alors un voisinage de $p$ dans $\overline{X}^S$ en prenant la réunion de $U$ avec l'ensemble de Siegel $\mathfrak{S} (t , \omega )$.

\section{Compactification de Tits} \label{S:Tits}

L'immeuble de Tits $\mathbf{T}=\mathbf{T}_n$ associé au groupe $\SL_n (\C)$ est un ensemble simplicial dont les simplexes non dégénérés sont en bijection avec les sous-groupes paraboliques propres de $\SL_n (\C)$, ou de manière équivalente avec les drapeaux propres 
\begin{equation} \label{E:flag}
W_\bullet: 0 \subsetneq W_1 \subsetneq \cdots \subsetneq W_k \subsetneq \C^n, \quad k \geq 0.
\end{equation}
Le stabilisateur $Q(W_\bullet) \subset \SL_n (\C )$ d'un tel drapeau est un sous-groupe parabolique propre qui définit un simplexe de dimension $k$ dans $\mathbf{T}$. La $i$-ème face de ce simplexe correspond au drapeau déduit de $W_\bullet$ en enlevant $W_i$.

Notons $Q= N_Q A_Q M_Q$ la décomposition de Langlands (associé au choix fixé du sous-groupe compact $\SU_n$) d'un sous-groupe parabolique $Q$ de $\SL_n (\C)$ et $\mathfrak{q}$, $\mathfrak{n}_Q$, $\mathfrak{a}_Q$ et $\mathfrak{m}_Q$ les algèbres de Lie correspondantes. 

Soit $\Phi^+(Q ,A_Q )$ l'ensemble des racines pour l'action adjointe de $\mathfrak{a}_Q$ on $\mathfrak{n}_Q$. Ces racines définissent une chambre positive
$$
\mathfrak{a}_Q^+ = \left\{ H \in \mathfrak{a}_Q \; : \;  \alpha(H)>0, \quad  \alpha \in \Phi^+(Q,A_Q) \right\}.
$$
En notant $\langle \cdot, \cdot \rangle$ la forme de Killing sur $\mathfrak{sl}_N (\C)$, on définit un simplexe ouvert 
$$
\mathfrak{a}_Q^+(\infty) = \left\{ H \in \mathfrak{a}_Q^+ \; :  \; \langle H, H \rangle = 1  \right\} \subset \mathfrak{a}_Q^+
$$
et un simplexe fermé
$$
\overline{\mathfrak{a}_Q^+}(\infty) = \left\{ H \in \mathfrak{a}_Q \; : \; \alpha(H) \geq 0, \ \langle H, H \rangle = 1, \quad  \alpha \in \Phi^+(Q,A_Q) \right\}
$$
dans $\mathfrak{a}_Q$. 

\medskip
\noindent
{\it Remarque.} Lorsque $Q= P_W$ est maximal, l'algèbre de Lie $\mathfrak{a}_Q$ est de dimension $1$ et $\overline{\mathfrak{a}_Q^+}(\infty)$ se réduit à un point. 

\medskip

Si $Q_1$ et $Q_2$ sont deux sous-groupes paraboliques propres, alors $\overline{\mathfrak{a}_{Q_1}^+}(\infty)$ est une face de $\overline{\mathfrak{a}_{Q_2}^+}(\infty)$ si et seulement si $Q_2 \subseteq Q_1$. L'immeuble de Tits peut donc être réalisé géométriquement comme 
\begin{equation} \label{eq:Tits_building_geometric_realization}
\mathbf{T}_n \sim \coprod_{Q} \overline{\mathfrak{a}_Q^+}(\infty)/\sim,
\end{equation}
où l'union porte sur tous les sous-groupes paraboliques propres $Q$ dans $\SL_n (\C)$ et $\sim$ désigne la relation d'équivalence induite par l'identification de $\overline{\mathfrak{a}_{Q_1}^+}(\infty)$ à une face de $\overline{\mathfrak{a}_{Q_2}^+}(\infty)$ dès que $Q_2 \subseteq Q_1$. Au niveau ensembliste on peut décomposer
$$
\mathbf{T}_n = \coprod_Q  \mathfrak{a}_Q^+(\infty)
$$
en l'union disjointe des simplexes ouverts $\mathfrak{a}_Q^+(\infty)$.

La compactification de Tits $\overline{X}^T$ de $X$ a pour bord l'immeuble de Tits $\mathbf{T}_n$~: au niveau ensembliste on a 
$$
\overline{X}^T = X \cup \coprod_Q \mathfrak{a}_Q^+(\infty).
$$
On renvoie à \cite[\S I.2]{BorelJi} pour une description détaillée de la topologie sur $\overline{X}^T$ et son identification avec la compactification géodésique de $X$, on se contente ici d'énoncer les trois propriétés suivantes qui caractérisent cette topologie~:
\begin{enumerate}
\item La topologie induite sur le bord $\mathbf{T}_n$ est la topologie quotient donnée par \eqref{eq:Tits_building_geometric_realization}.

\item Soit $x \in X$. Une suite $x_j \in \overline{X}^T$, $n \geq 1$, converge vers $x$ si et seulement si $x_j \in X$ pour $j \gg 1$ et  $x_j$ converge vers $x$ dans $X$ muni de sa topologie usuelle. 

\item Soit $H_\infty \in \mathfrak{a}_Q^+(\infty)$ et soit $(x_j)_{j \geq 1}$ une suite dans $X$. La décomposition de Langlands de $Q$ permet d'écrire $x_j=u_j \exp(H_j)m_j$ avec $u_j \in N_Q$ et $H_j \in \mathfrak{a}_Q$ uniquement déterminés et $m_j \in M_Q$ uniquement déterminé modulo $\SU_n \cap M_Q$. Alors $x_j \to H_\infty$ si et seulement si $x_j$ est non bornée et 
\begin{enumerate}
\item[(i)] $H_j/||H_j|| \to H_\infty$ dans $\mathfrak{a}_Q$,
\item[(ii)] $d(u_j m_j x_0,x_0)/||H_j|| \to 0$,
\end{enumerate}
où $d$ désigne la métrique symétrique sur $X$.
\end{enumerate}

Muni de cette topologie, l'espace $\overline{X}^T$ est séparé et l'action de $\SL_n (\C)$ sur $X$ s'étend naturellement en une action continue sur $\overline{X}^T$.

\begin{definition}
\'Etant donné deux points $x \in X$ et $x' \in \overline{X}^T$, on note $[x,x']$ l'unique segment géodésique orienté joignant $x$ à $x'$. 
\end{definition}
Si $x' \in X$, on définit plus explicitement $[x,x']$ comme étant égal à l'image de l'application
\begin{equation} \label{E:segment}
s(x,x') : [0,1] \to \overline{X}^T; \quad t \mapsto s(t;x,x'),
\end{equation}
l'unique segment géodésique orienté, paramétré à vitesse constante par l'intervalle unité, reliant $x$ à $x'$ dans $X$ avec $s(0;x,x') = x$ et $s(1; x,x' )= x'$.
Si $x'$ appartient au bord de $\overline{X}^T$, il existe un unique sous-groupe parabolique $Q$ tel que $x'$ corresponde à $H_\infty \in \mathfrak{a}_Q^+ (\infty)$. Dans les coordonnées horocycliques associées à la décomposition de Langlands de $Q$, on a $x= u \exp (H)m$. On définit alors $[x,x']$ comme étant égal à l'image de l'application
\begin{equation*} 
s(x,x') : [0,1] \to \overline{X}^T; \quad t \mapsto s(t;x,x') = \left\{ \begin{array}{ll}
u \exp \left( H + \frac{t}{1-t} H_\infty \right) m, & \mbox{si } t<1 ,\\
x' , & \mbox{si } t=1.
\end{array} \right.
\end{equation*}

\section{Ensembles de Siegel généralisés}

Soit $J = \{ j_1 < \ldots < j_r \}$ une suite strictement croissante d'entiers dans $\{1 , \ldots , n-1 \}$. On associe à $J$ le drapeau 
$$W_J : 0 \subsetneq W_{j_1} \subsetneq \cdots \subsetneq W_{j_r} \subsetneq \C^n,$$
où $W_{j_k} = \langle e_1 , \ldots , e_{j_k} \rangle$, et on note $Q_J$ le sous-groupe parabolique qui stabilise $W_J$. On peut décrire explicitement la décomposition de Langlands $Q_J$~: soient
\begin{equation*} 
\begin{split}
N &= N_{J} = \left\{ \begin{pmatrix} 1_{j_1} & * & \cdots & *  \\ 0 & 1_{j_2} & \cdots & * \\ 0 & 0 & \ddots & * \\ 0 & 0 & 0 & 1_{j_{r+1}} \end{pmatrix}\right\} \\
M &= M_{J} =   \left\{ \left. \begin{pmatrix} A_1 & 0 & \cdots & 0 \\ 0 & A_2 & \cdots & 0 \\ 0 & 0 & \ddots & 0 \\ 0 & 0 & 0 & A_{r+1} \end{pmatrix} \ \right| \ A_k \in \mathrm{GL}_{j_k}(\mathbb{C}), \ |\det(A_k)|=1 \right\} \\ 
A &= A_{J} = \left\{ a(t_1,\ldots,t_{r+1}) \; : \;  t_k>0, \ \det a(t_1,\ldots,t_{r+1}) = 1 \right\},
\end{split}
\end{equation*}
où 
$$a(t_1,\ldots,t_{r+1}) =\begin{pmatrix} t_1 1_{j_1} & 0 & \cdots & 0 \\ 0 & t_2 1_{j_2} & \cdots & 0 \\ 0 & 0 & \ddots & 0 \\ 0 & 0 & 0 & t_{r+1} 1_{j_{r+1}} \end{pmatrix}.$$
On a $Q_J = N A M$. Un élément $g \in \SL_n (\C)$ peut être décomposé en un produit 
\[
g = u ma k, \quad u \in N, \ m \in M, \ a \in A, \ k \in \SU_n,
\]
où $u$ et $a$ sont uniquement déterminés par $g$, et $m$ et $k$ sont déterminés à un élément de $M \cap \SU_n$ près.

\'Etant donné un nombre réel strictement strictement positif $t$, on pose 
\[
A_t = \{a(t_1,\ldots,t_{r+1}) \in A \; : \;  t_k/t_{k+1}  \geq t \text{ for all } k \}.
\]
L'ensemble de Siegel généralisé déterminé par $t>0$ et un sous-ensemble relativement compact $\omega \subset NM$ est 
$$\mathfrak{S}_J (t , \omega) = \omega A_t \cdot \SU_n;$$ 
on le verra aussi bien comme un sous-ensemble de $\SL_n (\C)$ que de $X$. 

Soit
$$\overline{\mathfrak{S}_{J} (t , \omega)}^T \subset \overline{X}^T$$
l'adhérence de $\mathfrak{S}_{J} (t , \omega)$.

Les deux affirmations suivantes découlent du fait que la compactification de Tits coïncide avec la compactification géodésique, cf. \cite[\S 1.2 et Proposition I.12.6]{BorelJi}.

\begin{quote}
{\it Affirmation 1 :} Soit $H \in \mathfrak{a}_{Q_J}^+(\infty)$ et soit $\omega \subset N_J M_J$ un sous-ensemble ouvert relativement compact. Alors pour tout réel $t>0$, l'ensemble des $x \in X$ tels qu'il existe $y \in [x, H[$ tel que 
$$[ y , H ] \subset \overline{\mathfrak{S}_{J} (t , \omega)}^T$$
est ouvert.
\end{quote} 

\medskip

\begin{quote}
{\it Affirmation 2 :}
Soit $x$ un point dans $X$ et soit $H \in \mathfrak{a}_{Q_J}^+(\infty)$. Il existe un sous-ensemble ouvert relativement compact $\omega \subset N_J M_J$ tel que pour tout $t>0$, il existe $y \in [x, H [$ tel que 
$$[ y , H ] \subset \overline{\mathfrak{S}_{J} (t , \omega)}^T.$$
\end{quote}

\medskip

Il découle en particulier de ces deux affirmations que si $\kappa \subset X$ un sous-ensemble compact et si $t$ est un réel strictement positif, le cône 
$$C (\kappa , [W_J]) = \bigcup_{\substack{x \in \kappa \\ x' \in [W_J]}} [x,x']$$
est asymptotiquement contenu dans une réunion finie d'ensembles de Siegel généralisés, autrement dit il existe un compact $\Omega \subset X$ et des ensembles relativement compacts $\omega_{J '} \subset N_{J'} M_{J'}$, pour tout sous-ensemble $J' \subset J$, tels que 
$$C(\kappa , [W_J]) \subset \Omega \cup \bigcup_{J ' \subset J} \overline{\mathfrak{S}_{J'} (t , \omega_{J'} )}^T.$$
Le dessin ci-dessous représente schématiquement la décomposition en ensemble de Siegel lorsque $\# J=2$ (de sorte que $[W_J]$ est un simplexe de dimension $1$).

\begin{center}
\includegraphics[width=0.3\textwidth]{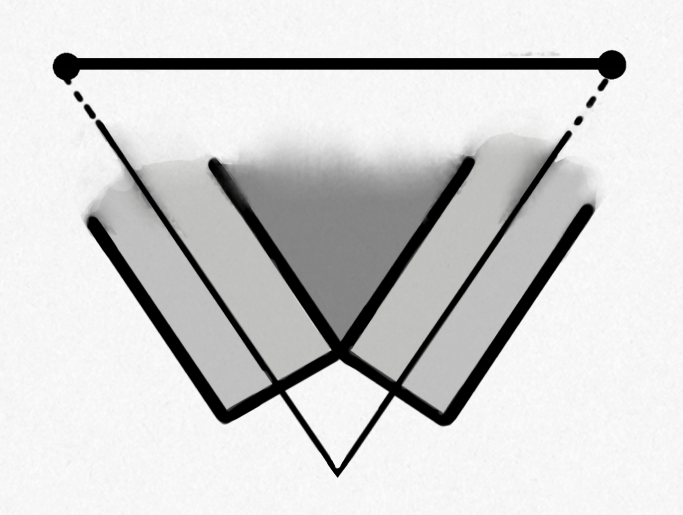}
\end{center}

On pose plus généralement la définition suivante. 

\begin{definition}
Soit $W_\bullet$ un drapeau propre de $\C^n$. On appelle \emph{ensemble de Siegel généralisé} associé au cusp $W_\bullet$ tout ensemble de la forme 
$$\mathfrak{S}_{W_\bullet} (g,t , \omega) = g \omega A_t  \cdot \SU_n $$ 
dans $\SL_n (\C)$ (ou dans $X$), où $g \in \SL_n (\C)$ est tel que $g^{-1} W_\bullet$ est un drapeau standard, c'est-à-dire de la forme $W_J$ pour un certain $J$. 
\end{definition}

Après translation par $g$, l'observation ci-dessus implique~:

\begin{proposition} \label{P:reduction}
Soit $W_\bullet$ un drapeau propre de $\C^n$ de simplexe associé $[W_\bullet] \subset \mathbf{T}$, soit $t$ un réel strictement positif et soit $\kappa \subset X$ un sous-ensemble compact. Il existe un sous-ensemble relativement compact $\Omega \subset X$, un élément $g \in \SL_n (\C)$ et des ensembles relativement compacts $\omega_{J '} \subset N_{J'} M_{J'}$, pour tout sous-ensemble $J' \subset J$, tels que 
$$C(\kappa , [W_\bullet]) \subset \Omega \cup \bigcup_{W_\bullet ' \subset W_\bullet} \overline{\mathfrak{S}_{W_\bullet '} (g, t , \omega_{J'} )}^T.$$
\end{proposition}
\begin{proof} 
On prend pour $g$ un élément tel que $g^{-1} W_\bullet = W_J$. Alors 
$$C(\kappa , [W_\bullet]) = g C(g^{-1} \kappa , [W_J])$$
et on est ramené au cas du drapeau standard $W_J$ détaillé ci-dessus. 
\end{proof}

\section{Comportement à l'infini de $\eta$}

\'Etant donné un sous-espace $W \subset V=\C^n$ on note $W^\perp$ l'orthogonal de $W$ relativement à la métrique hermitienne standard $| \cdot |$ sur $\C^n$ et $p_{W} : V \to W$ la projection orthogonale. 

\begin{lem} \label{L8}
Soit $Q_J \subset \SL_n (\C)$ le sous-groupe parabolique propre associé au drapeau $W_J$. Soit $\omega \subset N_J M_J$ un sous-ensemble relativement compact. Il existe alors des constantes strictement positives $C$, $\alpha$ et $\beta$ telles que pour tout réel $t >0$, on ait
\begin{equation*}
\| v^*\varphi \|_\infty \leq C e^{- \alpha t^\beta  | p_{W_{j_r}^\perp }(v) |^2} t^{\beta n/2} \max (1 , |v|^n ) 
\end{equation*}
et 
\begin{equation*}
\| v^* \psi \|_\infty \leq C e^{- \alpha t^\beta  | p_{W_{j_r}^\perp }(v) |^2}  t^{\beta n/2} \max ( |v| , |v|^n) \quad (v \in V)
\end{equation*}
en restriction à $\mathfrak{S}_J (t , \omega) \subset X (\subset S)$. 
\end{lem}
\begin{proof}
Le carré de la métrique hermitienne sur $V$ associée à un élément $g$ dans $\omega A_t \cdot \SU_n$ est bi-Lipschitz à 
$$v \mapsto \frac{1}{t_1^2} | p_{W_{j_1}} (v) |^2 + \ldots + \frac{1}{t_r^2} | p_{W_{j_{r-1}}^\perp \cap W_{j_r}} (v) |^2 + \frac{1}{t_{r+1}^2} | p_{W_{j_r}^\perp} (v) |^2 \geq \frac{1}{t_{r+1}^2} | p_{W_{j_r}^\perp} (v) |^2.$$
Mais puisque chaque $t_j$ est supérieur à $t^{r+1-j} t_{r+1}$ et que $t_1 \cdots t_{r+1} =1$, on a $1 \geq t^{j_1 + \ldots + j_r} t_{r+1}^n$ et donc 
$$\frac{1}{t_{r+1}^2} \geq t^\beta \quad \mbox{avec} \quad \beta = \frac{2}{n} (j_1 + \ldots + j_r) >0.$$ 
Le lemme se déduit alors des expressions explicites de $\varphi$ et $\psi$ déduites de \eqref{E:UMQ}.\footnote{Noter que $\varphi(v)$ tend vers une constante quand $v$ tend vers $0$ alors que $\psi (v)$ tend vers $0$ linéairement.}
\end{proof}

On déduit de ce lemme la proposition suivante.

\begin{proposition} \label{P32}
Soit $Q_J \subset \SL_n (\C)$ le sous-groupe parabolique propre associé au drapeau $W_J$. Soit $\omega \subset N_J M_J$ un sous-ensemble relativement compact, soit $t$ un réel strictement positif et soit $\kappa \subset V - W_{j_r}$ un sous-espace compact. La restriction de $\eta$ à $\mathfrak{S}_J (t,\omega) \times \kappa$ s'étend en une forme fermée, nulle à l'infini, à l'adhérence 
$$\overline{\mathfrak{S}_J (t,\omega)}^S \times \kappa, \quad \mbox{resp. } \overline{\mathfrak{S}_J (t,\omega)}^T \times \kappa,$$ dans $\overline{X}^S \times \kappa$, resp. $\overline{X}^T \times \kappa$. 
\end{proposition}
\begin{proof} Pour tout $s >0$ on a
$$v^* ([s]^* \psi ) = (sv)^* \psi  .$$
Il découle donc du lemme précédent qu'il existe des constantes strictement positives $C$, $\alpha$ et $\beta$ telles que, en restriction à $\mathfrak{S}_J (t ,\omega) \times \kappa$, la forme $[s]^*\psi$ soit de norme 
\begin{equation} \label{E:psis}
\| [s]^* \psi \|_\infty \leq s C  e^{- s^2 \alpha t^\beta } .
\end{equation}
L'intégrale 
$$\int_0^{+\infty} [s]^* \psi \frac{ds}{s}$$
est donc uniformément convergente sur $\mathfrak{S}_J (t ,\omega) \times \kappa$. Il découle enfin de la proposition \ref{P:reduction} et de \eqref{E:psis} que la forme $\eta$ tend uniformément vers $0$ lorsque l'on s'approche du bord de Tits (ou de Satake) dans $\mathfrak{S}_J (t ,\omega) \times \kappa$.
\end{proof}

\section{Symboles modulaires}

Soit $k$ un entier naturel. On note $\Delta_k '$ la première subdivision barycentrique du $k$-simplexe standard. On identifie chaque sommet $v$ de $\Delta_k '$ à un sous-ensemble non vide de $\{0,\ldots,k\}$ de sorte qu'un ensemble de sommets $\{v_0,\ldots,v_r\}$ forme un $r$-simplexe de $\Delta_k '$ si et seulement si  
$$v_0 \subseteq \cdots \subseteq v_r.$$ 
On notera $\Delta_{v_0,\ldots,v_r}$ ce simplexe.

\`A tout $(k+1)$-uplet $\mathbf{q} = (q_0 , \ldots , q_{k})$ de vecteurs non nuls dans $V$ avec $k \leq n-1$, on associe maintenant une application continue 
\begin{equation} \label{E:appDelta}
\Delta(\mathbf{q}) : \Delta_{k} ' \to \overline{X}^T.
\end{equation}
Supposons dans un premier temps que $\langle q_0 , \ldots, q_{k} \rangle$ soit un sous-espace propre de $V = \C^n$. Pour toute chaîne $v_0 \subseteq \cdots \subseteq v_r$ définissant un $r$-simplexe de $\Delta_{k}'$, le drapeau associé
\begin{equation} \label{E:flagq}
0 \subsetneq \langle q_i \; | \; i \in v_0 \rangle \subseteq \langle q_i \; | \; i \in v_1 \rangle \subseteq \cdots \subseteq \langle q_i \; | \; i \in v_r \rangle \subsetneq \C^n
\end{equation}
est propre et définit un $r$-simplexe (possiblement dégénéré) de l'immeuble de Tits $\mathbf{T}$. On définit alors $\Delta (\mathbf{q})$ comme étant l'application simpliciale $\Delta_{k}' \to \mathbf{T}$ qui envoie chaque $r$-simplexe $\Delta_{v_0,\ldots,v_r}$ sur le $r$-simplexe associé à \eqref{E:flagq} dans $\mathbf{T}$ (on laisse au lecteur le soin de vérifier que cette application est bien simpliciale, autrement dit qu'elle est compatible aux applications de faces et de dégénérescence).

Supposons maintenant que $k=n-1$ et que les vecteurs $q_0 , \ldots , q_{n-1}$ soient linéairement indépendants. Soit 
\begin{equation} \label{E:g}
g = (q_0 |\cdots|q_{n-1}) \in \mathrm{GL}_n ( \C)
\end{equation}
la matrice dont les vecteurs colonnes sont précisément les vecteurs $q_0 , \ldots , q_{n-1}$. On définit alors une sous-variété $\Delta^\circ (\mathbf{q})$ dans $X$ de la manière suivante. Soit $B$ le sous-groupe parabolique minimal de $\SL_n (\C)$ associé au drapeau maximal 
$$0 \subsetneq \langle e_1 \rangle \subsetneq  \langle e_1 , e_2 \rangle  \subsetneq \cdots \subsetneq \langle e_1 , \ldots , e_{n-1} \rangle \subsetneq \C^n.$$
En notant simplement $A$ le groupe  
$$A_B=\{ \mathrm{diag}(t_1,\ldots,t_n) \in \SL_n (\C) \; : \; t_j \in \R_{>0},  \ t_1 \cdots t_n = 1 \} \cong \R_{>0}^{n-1},$$
on pose
\begin{equation}
\Delta^\circ (\mathbf{q}) := g A K \R_{>0}  \subset G/K \R_{>0}=X,
\end{equation}
muni de l'orientation induite par les coordonnées $\mathrm{diag}(t_1,\ldots,t_n) \mapsto t_i$ identifiant $\Delta^\circ(\mathbf{q})$ à $\R_{> 0}^{n-1}$ (ce dernier étant muni de l'orientation standard). Son adhérence dans $\overline{X}^T$ est naturellement identifiée à la première subdivision barycentrique d'un $(n-1)$-simplexe dont le bord 
est la réunion dans $\mathbf{T}$ des translatés par $g$ de tous les $\overline{\mathfrak{a}_Q^+} (\infty )$ où $Q$ est un sous-groupe parabolique propre de $\SL_n (\C)$ contenant $A$. On construit ainsi une  application \eqref{E:appDelta} pour $k=n-1$ dont la restriction au bord  
$\partial \Delta_{n-1} '$ coïncide avec les applications construites précédemment.  

\medskip

Concluons ce paragraphe en expliquant comment recouvrir l'image de $\Delta (\mathbf{q})$ par des ensembles de Siegel généralisés~: à chaque sommet $v$ de $\Delta_{n-1} '$, autrement dit un sous-ensemble propre non vide de $\{0,\ldots , n-1\}$, il correspond un sous-espace 
\begin{equation}
W(\mathbf{q})_v = \langle q_k \; | \; k \in v  \rangle
\end{equation}
dans $\C^n$. Il découle de la proposition \ref{P:reduction} que l'on peut recouvrir l'image de $\Delta (\mathbf{q})$ par des ensembles de Siegel généralisés associés aux cusps $W_\bullet$ formés de sous-espaces $W(\mathbf{q})_v$~: 

\begin{proposition} \label{P33}
Soit $\mathbf{q} = (q_0 , \ldots , q_{n-1})$ un $n$-uplet de vecteurs non nuls dans $V$ et soit $t$ un réel strictement positif. Il existe alors un sous-ensemble relativement compact $\Omega \subset X$ et un nombre fini d'ensembles de Siegel généralisés 
$\mathfrak{S}_{W_\bullet} (g , t , \omega)$, où chaque drapeau $W_\bullet$ est formé de sous-espaces $W(\mathbf{q})_v$ et chaque $g$ est une matrice dont les vecteurs colonnes sont des $q_j$, tels que l'image de l'application $\Delta (\mathbf{q})$ dans $\overline{X}^T$ soit contenue dans la réunion finie
$$\Omega \cup \bigcup_{W_\bullet, g , \omega} \overline{\mathfrak{S}_{W_\bullet} (g , t , \omega)}.$$
\end{proposition}

\medskip
\noindent
{\it Remarque.} L'adhérence de $\Delta^\circ (\mathbf{q})$ dans $\overline{X}^S$ est égale à l'enveloppe convexe conique 
$$\mathrm{P} \left\{ \sum_{j=0}^{n-1} t_j m_j  \in \mathcal{H} \; : \;  \forall j \in \{0, \ldots , n-1 \}, \  t_j >0  \right\}  \subset \overline{X}^S$$
des formes hermitiennes semi-définies positives $m_j  =  q_j^* (\overline{q_j^*} )^\top$ où\footnote{On prendra garde au fait que Ash et Rudolph \cite[p. 5]{AshRudolph} commettent une légère erreur en identifiant $\Delta^\circ (\mathbf{q})$ avec l'enveloppe convexe conique des formes 
hermitiennes semi-définies positives $m_j  =  q_j \overline{q_j}^\top$.}
$$g^{-*} = (q_0^* | \cdots | q_{N-1}^* ), \quad 
\mbox{de sorte que } (\overline{q_i^*} )^\top q_j  = \delta_{ij} \quad \mbox{et} \quad q_j^* = g^{-*} e_{j+1} .$$

\medskip
\begin{proof} Pour tout $j \in \{0, \ldots , n-1 \}$, on a $ge_{j+1} = q_{j}$. On peut donc se ramener au cas où $\mathbf{q} = (e_1 , \ldots , e_n)$ et   $$\Delta (\mathbf{q})^\circ  = A K \R_{>0}  \subset G/K \R_{>0}=X$$
ce qui prouve immédiatement la remarque. Le cas général s'en déduit en translatant par $g$. Noter qu'alors la forme $m_j$ est égale à $q_j^* (q_j^* )^\top$ qui a bien pour noyau 
$$W(\mathbf{q})^{(j)} = W(\mathbf{q})_{\{0 , \ldots , \widehat{j} , \ldots , n-1 \}}. $$
\end{proof}

\section{\'Evaluation de $\eta$ sur les symboles modulaires} 

Soit $\mathbf{q} = (q_0 , \ldots , q_{k})$ un $(k+1)$-uplet de vecteurs de non nuls dans $V$ avec $k \leq n-1$. Les propositions \ref{P32} et \ref{P33} impliquent que la forme différentielle fermée
\begin{equation}
\eta (\mathbf{q} ) = (\Delta (\mathbf{q}) \times \mathrm{id} )^* \eta \in A^{2n-1} \left( \Delta_{k}' \times \left( V - \bigcup_{|v| <n} W(\mathbf{q})_v \right) \right),
\end{equation}
où $\Delta_{k}'$ est identifié à 
$$\R^{k} \cong \{ (t_0 , \ldots , t_{k} )  \in \R_+^{k+1} \; : \; t_0 + \ldots + t_{k} =1 \}$$
{\it via} les coordonnées barycentriques, est bien définie. 

Il découle en outre des propositions \ref{P32} et \ref{P33} que l'intégrale partielle 
$$\int_{\Delta_{k}'} \eta (\mathbf{q})$$
converge et définit une forme de degré $2n-k-1$ sur $V - \bigcup_{|v|<n} W(\mathbf{q})_v$. 

\begin{proposition} \label{P34}
1. Si $\langle q_0 , \ldots , q_{k} \rangle$ est un sous-espace propre de $V$, alors la forme $\int_{\Delta_{k}'} \eta (\mathbf{q})$ est identiquement nulle.

2. Supposons $k=n-1$ et que les vecteurs $q_0 , \ldots , q_{n-1}$ soient linéairement indépendants. Alors la $n$-forme $\int_{\Delta_{n-1}'} \eta (\mathbf{q})$ est égale à 
$$\frac{1}{(4i\pi)^n} \left( \frac{d \ell_0}{\ell_0} - \overline{\frac{d\ell_0}{\ell_0}} \right) \wedge \ldots \wedge \left( \frac{d \ell_{n-1}}{\ell_{n-1}} - \overline{\frac{d\ell_{n-1}}{\ell_{n-1}}} \right)  \in A^n \left( V - \bigcup_{j} W(\mathbf{q})^{(j)} \right),$$
où $\ell_j$ est la forme linéaire sur $\C^n$, de noyau $W(\mathbf{q})^{(j)}$, qui à $z$ associe $z^\top q_j^*$.
\end{proposition}
\begin{proof} 1.  Dans ce cas l'image de  $\Delta (\mathbf{q})$ est contenue dans le bord de $\overline{X}^T$ et il résulte de la proposition \ref{P32} que $\int_{\Delta_{k}'} \eta (\mathbf{q})$ est nulle sur tout ouvert relativement compact de $V - \bigcup_{|v| <n} W(\mathbf{q})_v$. 

2. Supposons donc $k=n-1$ et que les vecteurs $q_0 , \ldots , q_{n-1}$ soient linéairement indépendants.

En notant toujours $g$ l'élément \eqref{E:g}, la $G$-invariance de $\eta$ implique que 
$$g^* \left( \int_{\Delta_{n-1}'} \eta (\mathbf{q}) \right) = \int_{\Delta_{n-1}'} \eta (e_1 , \ldots , e_n ).$$
Comme par ailleurs $g^*\ell_j$ est la forme linéaire $e_{j+1}^*$ de noyau 
$$g^{-1} W(\mathbf{q})^{(j)} = \langle e_1 , \ldots , \widehat{e_{j+1}} , \ldots , e_n  \rangle,$$ 
on est réduit à vérifier la proposition dans le cas où $\mathbf{q} = (e_1 , \ldots , e_n )$.

Il nous reste donc à calculer l'intégrale 
$$\int_{AK\R_{>0}} \eta, \quad \mbox{où } A=\{ \mathrm{diag}(t_1,\ldots,t_n) \in \SL_n (\C) \; : \; t_j \in \R_{>0},  \ t_1 \cdots t_n = 1 \}.$$ D'après la remarque à la fin du paragraphe \ref{S:42} on a 
\begin{equation} \label{E:intsymbMod}
\int_{AK\R_{>0}} \eta = \int_{\{ \mathrm{diag}(t_1,\ldots,t_n ) \; : \; t_j \in \R_{>0} \} K} \varphi.
\end{equation}
Or, en restriction à l'ensemble des matrices symétriques diagonales réelles, le fibré en $\C^n$ se scinde {\it métriquement} en une somme directe de $n$ fibrés en droites, correspondant aux coordonnées $(z_j )_{j=1 , \ldots , n}$ de $z$ et la forme $\varphi$ se décompose en le produit de $2$-formes associées à ces fibrés en droites~: 
$$
\varphi^{(j)}  = \frac{i}{2\pi} e^{- t_j^2 |z_j |^2 }  \left( t_j^2 dz_j \wedge d \overline{z}_j 
-  t_j^2 ( z_j d\overline{z}_j - \overline{z}_j dz_j )  \wedge \frac{dt_j}{t_j}  \right),
$$
d'après (\ref{E:phiN1}).

Finalement, on obtient que l'intégrale \eqref{E:intsymbMod} est égale à 
\begin{multline*}
\frac{(-i)^n}{(2\pi)^n} \left( \prod_{j=1}^n \int_{\R_{>0}} t_j^2 |z_j|^2 e^{-t_j ^2 |z_j |^2}  \frac{dt_j }{t_j} \right) \wedge_{j=1}^n  \left( \frac{dz_j}{z_j} - \frac{d\overline{z}_j}{\overline{z}_j} \right)    \\ 
= \frac{1}{(4i\pi)^n} \wedge_{j=1}^n  \left( \frac{dz_j}{z_j} - \frac{d\overline{z}_j}{\overline{z}_j} \right) ,
\end{multline*}
comme attendu. 
\end{proof}

\chapter{Cocycles de $\GL_n (\C)$ explicites}  \label{S:6}

\resettheoremcounters

Dans ce chapitre on note à nouveau $G=\GL_n (\C)^{\delta}$. Dans un premier temps on explique comment associer à la forme de Mathai--Quillen $\eta \in A^{2n-1} (X \times (\C^n - \{ 0 \} ))^G$ un représentant explicite du relevé canonique $\Phi \in H_G^{2n-1} (\C^n -\{ 0 \})$ fourni par la proposition \ref{P4}. On utilise ensuite ce représentant explicite pour démontrer le théorème \ref{T:Sa}.

\section{Forme simpliciale associée à $\eta$} \label{S:61}

Soit $x_0 \in X$ le point base associé à la classe de l'identité dans $\SL_n (\C)$. L'application $\exp : T_{x_0} X \to X$ étant un difféomorphisme, il existe une rétraction 
\begin{equation} \label{E:R}
R : [0, 1] \times X  \to X 
\end{equation}
de $X$ sur $\{x_0 \} $. Elle est donnée par la formule 
$$R_s (x ) = \exp (s \exp^{-1} (x) ) \quad (x \in X,   \ s \in [0,1]).$$

Suivant \cite{Dupont} on déduit de $R$ une suite d'applications
\begin{equation} \label{E:mapr}
\rho_k : \Delta_k \times E_kG  \times \C^n \longrightarrow X \times \C^n
\end{equation}
définies de la manière suivante~: pour $t = (t_0 , \ldots , t_k ) \in \Delta_k$ on pose $s_j = t_j + t_{j+1} + \ldots + t_k$ ($j=1, \ldots , k$). \'Etant donné un $(k+1)$-uplet 
$$\mathbf{g} = (g_0 , \ldots , g_k ) \in E_kG$$ 
et un vecteur $z \in \C^n$, on a alors
\begin{multline} \label{E:mapr2}
\rho_k ( t , \mathbf{g} , z ) = (g_0^{-1} \cdot R_{s_1} ( g_0 g_1^{-1} \cdot R_{s_2 / s_1} ( g_1 g_2^{-1} \cdot \\ \ldots g_{j-1} g_j^{-1} \cdot R_{s_{j+1} / s_j} ( g_j g_{j+1}^{-1} \cdot \cdots R_{s_k / s_{k-1} } ( g_{k-1} g_k^{-1} \cdot x_0) \ldots )))  , z).
\end{multline}
La suite $(\rho_k )$ est constituée d'applications $G$-équivariantes qui font commuter le diagramme\footnote{Ici $\epsilon^k : \Delta_{k-1} \to \Delta_k$ désigne l'application d'inclusion de la $k$-ième face.} 
\begin{equation} \label{diag:simpl}
\xymatrix{
\Delta_{k-1} \times E_kG  \times \C^n   \ar[d]^{\  \mathrm{id} \times \partial_k \times \mathrm{id}} \ar[r]^{\epsilon^k \times \mathrm{id}} & \Delta_k \times E_kG   \times \C^n \ar[d]^{\rho_{k}} \\
\Delta_{k-1} \times E_{k-1} G  \times \C^n \ar[r]^{\quad \quad \rho_{k-1}} & X\times \C^n,
}
\end{equation}
de sorte que $\rho$ induit une application $G$-équivariante
$$\rho^* : A^\bullet (X \times \C^n ) \to \mathrm{A}^\bullet (EG \times \C^n),$$
où l'espace à droite est celui des formes différentielles simpliciales sur la variété simpliciale $EG \times \C^n$, cf. annexe \ref{A:A}. 

La proposition suivante précise la proposition \ref{P4}. 

\begin{proposition}
La forme simpliciale
$$\rho^* \eta \in \mathrm{A}^{2n-1} (EG \times (\C^n - \{ 0 \}))^G$$
est fermée et représente une classe dans $H_G^{2n-1} (\C^n - \{ 0 \} )$ qui relève la classe fondamentale dans $H^{2n-1} (\C^n - \{ 0 \} )$. 
\end{proposition}
\begin{proof} La forme $\rho^* \eta$ est fermée et $G$-invariante puisque $\eta$ l'est (proposition \ref{P:eta}) et, comme $\eta$ représente la classe fondamentale de $\C^n - \{ 0 \}$, la classe de cohomologie équivariante $[\rho^* \eta ]$ s'envoie sur la classe fondamentale dans $H^{2n-1} (\C^n - \{ 0 \} )$.   
\end{proof}

\section{Cocycle associé} \label{S:62}

Considérons maintenant les ouverts 
\begin{equation} \label{hypCU'}
\begin{split}
U(g_0 , \ldots , g_k )  & = \left\{  z  \in  \C^n \; : \; \forall j \in \{ 0 , \ldots , k \}, \ e_1^* (g_j z ) \neq 0 \right\} \\
& = \C^n - \cup_j H_j,
\end{split}
\end{equation}
où $H_j$ est le translaté par $g_j^{-1}$ de l'hyperplan $z_1=0$ dans $\C^n$. Ces sous-variétés sont toutes affines de dimension $n$ et n'ont donc pas de cohomologie en degré $>n$. 

La forme simpliciale $\rho^* \eta$ représente une classe de cohomologie équivariante dans $H^{2n-1}_G (\C^n - \{0 \})$. On peut lui appliquer la méthode décrite dans l'annexe \ref{A:A} et lui associer un $(n-1)$-cocycle de $G$ à valeurs dans 
$$\lim_{\substack{\rightarrow \\ H_j}} H^n (\C^n - \cup_j H_j ).$$
On calculera explicitement ce cocycle au paragraphe \S \ref{S:demTSa}. 

Avant cela remarquons qu'on aurait pu également considérer les ouverts 
\begin{equation} \label{hypC}
\begin{split}
U^* (g_0 , \ldots , g_k) & = \left\{  z \in \C^n \; \left| \;   \begin{array}{l} \forall J \subset \{ 0 , \ldots , k \}, \\ |J| <\min (k,n)  \Rightarrow z \notin \langle q_j  \; : \; j \in J \rangle \end{array} \right. \right\} \\
& = \C^n - \bigcup_j W(\mathbf{q})^{(j)},
\end{split}
\end{equation}
où $\mathbf{q} = (q_0 , \ldots , q_k)=( g_0^{-1} e_1 , \ldots  , g_{k}^{-1} e_1 )$ et $W(\mathbf{q})^{(j)} = \langle q_i \; : \; i \neq j \rangle$. 

En général la sous-variété \eqref{hypC} n'est pas affine. C'est toutefois le cas lorsque les vecteurs $g_0^{-1} e_1, \ldots , g_{k}^{-1} e_1$ engendrent $\C^n$.  L'argument de l'annexe \ref{A:A} implique donc encore que la forme simpliciale $\rho^* \eta$ détermine un $(n-1)$-cocycle de $G$ à valeurs dans 
$$\lim_{\substack{\rightarrow \\ H_j}} H^n (\C^n - \cup_j H_j ).$$
On détaille particulièrement le calcul explicite de ce cocycle dans les paragraphes suivants; ce sont en effet ces calculs que nous généraliserons dans les cas multiplicatifs et elliptiques.

\section{Section simpliciale et homotopie} \label{S:63}

Pour tout entier $k \in [0 , n-1]$, on définit, par récurrence, une subdivision simpliciale $\left[ \Delta_k \times [0,1] \right]'$ de $\Delta_k \times [0,1]$ en joignant tous les simplexes dans $\Delta_k \times \{0 \} \cup \partial \Delta_k \times [0,1]$ au barycentre de $\Delta_k \times \{1\}$, comme sur la figure ci-dessous.
\begin{center}
\includegraphics[width=0.5\textwidth]{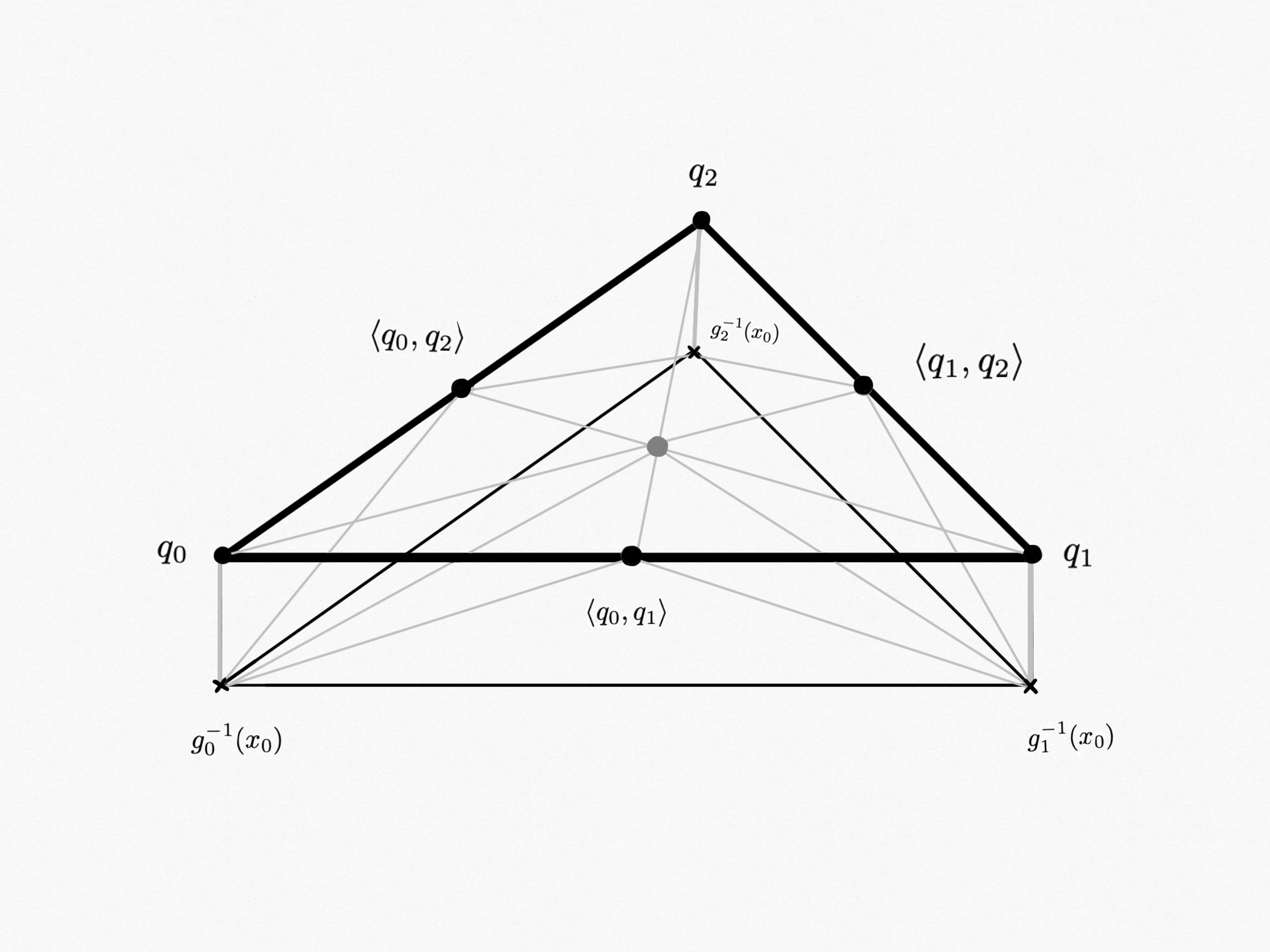}
\end{center}
L'ensemble des $(k+1)$-simplexes de $\left[ \Delta_k \times [0,1] \right]'$ est constitué des joints 
\begin{equation} \label{E:cone1}
\Delta_w \star \Delta_{v_0, \ldots , v_{k - |w|}}  \quad (w \subset v_0)
\end{equation}
où $\Delta_w \subset \Delta_k = \Delta_k \times \{ 0 \}$ est le simplexe correspondant à un sous-ensemble non-vide $w \subset \{0, \ldots , k \}$, de cardinal $|w|$, et $\Delta_{v_0, \ldots , v_{k-|w|}}$ est le $(k-|w|)$-simplexe de la subdivision barycentrique $\Delta_k' = \Delta_k ' \times \{ 1 \}$ associé à une suite croissante $v_0 \subset \cdots \subset v_{k-|w|}$ de sous-ensembles de $\{0 , \ldots , k \}$ contenant tous $w$. 

On définit maintenant une suite d'applications 
\begin{equation}
\varrho_k : \left[ \Delta _k \times [0,1] \right]' \times E_k G \times \C^n \to \overline{X}^T \times \C^n \quad (k \in \{ 0 , \ldots , n-1 \} )
\end{equation}
de manière à ce que la restriction de $\varrho_k$ à 
$$\Delta_k  \times E_k G \times \C^n = (\Delta _k \times \{0 \} ) \times E_k G \times \C^n$$ 
coïncide avec $\rho_k$ et qu'en restriction à 
$$\Delta_k '  \times E_k G \times \C^n \subset (\Delta _k ' \times \{1 \} ) \times E_k G \times \C^n$$
on ait  
$$\varrho_k (  - , \mathbf{g} , -) = \Delta (\mathbf{q} ) \times \mathrm{Id}_{\C^n },$$
où si $\mathbf{g}=(g_0 , \ldots , g_{k} ) \in E_{k} G$ on note toujours $\mathbf{q} = ( g_0^{-1} e_1 , \ldots , g_k^{-1} e_1 ).$ 

En procédant par récurrence sur $k$ on est ramené à définir $\varrho_k$ sur chaque simplexe \eqref{E:cone1}. Considérons donc un sous-ensemble non-vide $w \subset \{0, \ldots , k \}$ et une suite croissante $v_0 \subset \cdots \subset v_{k-|w|}$ de sous-ensembles de $\{0 , \ldots , k \}$ contenant tous $w$. L'application $\varrho_k$ étant définie sur $\Delta_w \subset \Delta _k \times \{0 \}$ et sur $\Delta_{v_0, \ldots , v_{k-|w|}} \subset \Delta _k ' \times \{1 \} $, l'expression
$$\varrho_k ( t ,  \mathbf{g} , z ) = s \left( t ; \varrho_k ( 0 ,  \mathbf{g} , z ) , \varrho_k ( 1 ,  \mathbf{g} , z ) \right)$$
définit une application 
$$|\Delta_w|  \times |\Delta_{v_0, \ldots , v_{k-|w|}}| \times [0,1] \to \overline{X}^T \times \C^n$$
qui se factorise en une application
$$| \Delta_w \star \Delta_{v_0, \ldots , v_{k - |w|}}  | \to \overline{X}^T \times \C^n.$$
On définit ainsi $\varrho_k$ sur les simplexes \eqref{E:cone1}; on laisse au lecteur le soin de vérifier les relations de compatibilité.

La proposition \ref{P33} implique que, pour tout $\mathbf{g} \in E_k G$, on peut recouvrir l'image
$$\varrho_k \left( \left[ \Delta _k \times [0,1] \right]'  \times \{ \mathbf{g} \} \times \C^n \right) \subset \overline{X}^T \times \C^n$$
par un nombre fini de produits $\mathfrak{S} \times \C^n$ d'ensembles de Siegel par $\C^n$. Il découle donc de la proposition \ref{P32} que la forme différentielle fermée 
$$\varrho_k^* \eta (\mathbf{g}) \in \mathrm{A}^{2n-1} \left( \left[ \Delta _k \times [0,1] \right]' \times U^* (g_0, \ldots , g_k) \right)$$
est bien définie. 

\begin{definition}
Pour tout entier $k \in [0, n-1]$ et pour tout $(k+1)$-uplet $(g_0 , \ldots , g_{k}) \in E_k G$, on pose 
$$H_k (g_0 , \ldots , g_k ) = \int_{\left[ \Delta _k \times [0,1] \right]'} \varrho_k^* \eta (g_0 , \ldots , g_k ) \in A^{2n-2-k}  \left(\C^n - \bigcup_{j=0}^k W(\mathbf{q})^{(j)} \right),$$
où $\mathbf{q} = (q_0 , \ldots , q_k)$ avec $q_j = g_j^{-1} e_1$. 
\end{definition}

\section{Calcul du cocycle}

Il résulte de la définition de l'application bord $\delta$ donnée dans l'annexe \ref{A:A} que, pour tout entier $k \in [0, n-1]$, on a
$$\delta H_{k-1} (g_0 , \ldots , g_{k}) = \sum_{j=0}^k H_{k-1} (g_0 , \ldots , \widehat{g}_j , \ldots , g_k ),$$
vue comme forme différentielle sur $\C^n - \bigcup_{j=0}^k W(\mathbf{q})^{(j)}$. 

\begin{theorem} \label{T37}
Pour tout entier $k \in [0 , n-1]$ et pour tout $\mathbf{g} = (g_0 , \ldots , g_k)$ dans $E_k G$, l'intégrale $\int_{\Delta_k} \rho^* \eta (\mathbf{g})$ est égale à 
$$\delta H_{k-1} (g_0 , \ldots , g_{k}) \pm d H_k (g_0 , \ldots , g_k),  \quad \mbox{si} \quad  k < n-1,$$
et
$$\int_{\Delta_{n-1}'} \eta (\mathbf{q}) + \delta H_{n-2} (g_0 , \ldots , g_{n-1}) \pm dH_{n-1}  (g_0 , \ldots , g_{n-1}), \quad \mbox{si} \quad k=n-1,$$
dans $A^{2n-2-k} \left(\C^n - \bigcup_{j=0}^k W(\mathbf{q})^{(j)} \right),$ 
où $\mathbf{q} = (q_0 , \ldots , q_k)$ avec $q_j = g_j^{-1} e_1$. 
\end{theorem}
\begin{proof} Puisque $\eta$ est fermée on a~:
$$(d_{\left[ \Delta _k \times [0,1] \right]'} \pm d ) \varrho_k^* \eta = 0$$
et donc 
$$\int_{\left[ \Delta _k \times [0,1] \right]'} d_{\left[ \Delta _k \times [0,1] \right]'} \varrho_k^* \eta (g_0 , \ldots , g_k ) \pm d H_k (g_0 , \ldots , g_k ) =0.$$
Maintenant, d'après le théorème de Stokes on a
\begin{multline*}
\int_{\left[ \Delta _k \times [0,1] \right]'} d_{\left[ \Delta _k \times [0,1] \right]'} \varrho_k^* \eta (g_0 , \ldots , g_k ) = \int_{\Delta_k \times \{ 0 \}} \varrho_k^* \eta (g_0 , \ldots , g_k ) \\ +  \int_{\left[(\partial \Delta_k) \times [0,1]\right]'} \varrho_k^* \eta (g_0 , \ldots , g_k )  -  \int_{\Delta_k' \times \{ 1 \}} \varrho_k^* \eta (g_0 , \ldots , g_k ).
\end{multline*}
La dernière intégrale est égale à $\int_{\Delta_k '} \eta (\mathbf{q})$ et est donc nulle si $k < n-1$ d'après la proposition \ref{P34}.

Finalement, par définition de $\varrho_k$ on a 
$$\int_{\Delta_k \times \{ 0 \}} \varrho_k^* \eta (g_0 , \ldots , g_k ) = \int_{\Delta_k} \rho_k^* \eta (g_0 , \ldots , g_k) $$
et
$$\int_{\left[(\partial \Delta_k) \times [0,1]\right]'} \varrho_k^* \eta (g_0 , \ldots , g_k ) = - \delta H_{k-1} (g_0 , \ldots , g_k ).$$
\end{proof}

Le corollaire suivant découle du théorème \ref{T37} et de la proposition \ref{P34}. 

\begin{cor} \label{C38}
La forme différentielle simpliciale fermée $\rho^* \eta $ définit un $(n-1)$-cocycle de $G$ à valeurs dans 
$$\lim_{\substack{\rightarrow \\ H_j}} H^n (\C^n - \cup_j H_j )$$
qui est cohomologue au cocycle 
$$(g_0 , \ldots , g_{n-1} ) \mapsto \left[ \frac{1}{(4i\pi)^n} \left( \frac{d \ell_0}{\ell_0} - \overline{\frac{d\ell_0}{\ell_0}} \right) \wedge \ldots \wedge \left( \frac{d \ell_{n-1}}{\ell_{n-1}} - \overline{\frac{d\ell_{n-1}}{\ell_{n-1}}} \right) \right],$$
où $\ell_j$ est une forme linéaire sur $\C^n$, de noyau $W(\mathbf{q})^{(j)}$, qui à $z$ associe $z^\top q_j^*$.
\end{cor}

\section[Démonstration du théorème 2.2]{Démonstration du théorème \ref{T:Sa}} \label{S:demTSa}

Les identités immédiates suivantes entre formes différentielles sur $\C^*$
$$\frac{1}{2i\pi} \frac{dz}{z} = \frac{d\theta}{2\pi} + d \left( \frac{1}{2i\pi} \log r \right) \quad \mbox{et} \quad 
\frac{1}{4i\pi} \left( \frac{dz}{z} - \overline{\frac{dz}{z}} \right) = \frac{d\theta}{2\pi},$$
avec $z = re^{i\theta}$, impliquent que si $\ell$ est une forme linéaire sur $\C^n$ les formes différentielles 
$$\frac{1}{2i\pi} \frac{d\ell}{\ell} \quad \mbox{et} \quad \frac{1}{4i\pi} \left( \frac{d \ell}{\ell} - \overline{\frac{d\ell}{\ell}} \right)$$
sur $\C^n- \mathrm{ker} (\ell )$ sont cohomologues. Le théorème de Brieskorn \cite{Brieskorn} évoqué au \S \ref{S21} implique donc que l'application 
$$\Omega_{\rm aff} \to \lim_{\substack{\rightarrow \\ H_j}} H^\bullet (\C^n - \cup_j H_j ); \quad \frac{1}{2i\pi} \frac{d\ell}{\ell} \mapsto \left[ \frac{1}{4i\pi} \left( \frac{d \ell}{\ell} - \overline{\frac{d\ell}{\ell}} \right) \right]$$
est un isomorphisme d'algèbre. Comme par ailleurs cet isomorphisme est $G$-équivariant, il résulte du corollaire \ref{C38}
que la classe de $[\rho^* \eta]$ donne lieu à un $(n-1)$-cocycle de $G$ dans $\Omega^n_{\rm aff}$ cohomologue à 
$$\mathbf{S}_{\rm aff}^* : G^n \to \Omega^n_{\rm aff} ; \quad (g_0 , \ldots , g_{n-1}) \mapsto \frac{1}{(2i\pi )^n} \frac{d\ell_0}{\ell_0} \wedge \ldots \wedge \frac{d\ell_{n-1}}{\ell_{n-1}},$$
où $\ell_j$ est une forme linéaire de noyau $\langle g_0^{-1} e_1 , \ldots , \widehat{g_j^{-1} e_1} , \ldots , g_{n-1}^{-1} e_1 \rangle$.

Il nous reste à montrer que ce cocycle est cohomologue à $\mathbf{S}_{\rm aff}$ et que sa classe de cohomologie est non nulle. 

Pour ce faire on applique à nouveau l'argument de l'annexe \ref{A:A} à la forme simpliciale fermée $\rho^*\eta$ mais en utilisant cette fois les ouverts \eqref{hypCU'}. Rappelons que  dans ce cas
\begin{equation*}
U (g_0 , \ldots , g_k)   = \C^n - \cup_j H_j 
\end{equation*}
où $H_j$ est le translaté par $g_j^{-1}$ de l'hyperplan $z_1=0$ dans $\C^n$. 

On procède alors de la même manière que pour obtenir $\mathbf{S}_{\rm aff}^*$ mais en remplaçant la compactification de Tits par celle de Satake $\overline{X}^S$. La convexité de l'ouvert des formes hermitiennes non-nulles semi-définies positives dans $\mathcal{H}$ permet de rétracter $\overline{X}^S$ sur le point (à l'infini) associé à la forme hermitienne $|e_1^* (\cdot )|^2$ de noyau $\langle e_2 , \ldots , e_n \rangle$; notons
$$\overline{R} : [0,1] \times \overline{X}^S \times \C^n \to \overline{X}^S \times \C^n$$
l'application correspondante. 

Comme pour $R$, il correspond à $\overline{R}$ une suite d'applications $G$-équivariantes 
$$\overline{\rho}_k : \Delta_k \times E_k G \times \C^n \to \overline{X}^S \times \C^n$$ 
qui font commuter le diagramme \eqref{diag:simpl} de sorte que $\overline{\rho} = (\overline{\rho}_k)$ induit une application $G$-équivariante 
$$(\overline{\rho})^* : A^\bullet (\overline{X}^S \times \C^n ) \to \mathrm{A}^\bullet (EG \times \C^n ).$$
Mieux, il découle encore de la proposition \ref{P32} que les formes différentielles fermées 
$$(\overline{\rho}_{n-1})^* \eta (\mathbf{g}) \in A^{2n-1} (\Delta_{n-1}  \times U (g_0 , \ldots , g_{n-1})  )$$
sont bien définies, se recollent en une forme simpliciale fermée dans  
$$\varinjlim A^{n}(  \C^n - \cup_j H_j )$$
et définissent un $(n-1)$-cocycle
$$(g_0 , \ldots , g_{n-1} ) \mapsto \int_{\Delta_{n-1}} (\overline{\rho}_{n-1})^* \eta (g_0 , \ldots , g_{n-1} ).$$
Ce dernier étant obtenu en appliquant l'argument de l'annexe \ref{A:A}, il est cohomologue au cocycle du corollaire \ref{C38}. 
Or, d'après la remarque suivant la proposition \ref{P33}, on a 
$$\overline{\rho}_k (\Delta_k \times \{ \mathbf{g} \} \times \{ z \} ) = \overline{\Delta^\circ (\mathbf{\sigma}^*)} \times \{ z \}$$
où cette fois $\mathbf{\sigma} = (\sigma_0 , \ldots , \sigma_{n-1})$, avec $\sigma_j = g_j^\perp e_1 $, de sorte que
$$\mathrm{ker} \ e_1^* (g_j \cdot) = \langle \sigma_0^* , \ldots , \widehat{\sigma_j^*} , \ldots , \sigma_{n-1}^* \rangle.$$
On a donc
$$\int_{\Delta_{n-1}} (\overline{\rho})^* \eta (g_0 , \ldots , g_{n-1} ) = \int_{\Delta_{n-1} '} \eta (\mathbf{\sigma}^*)$$
et la proposition \ref{P34} implique finalement que l'intégrale $\int_{\Delta_{n-1}} (\overline{\rho})^* \eta (\mathbf{g})$ est nulle si les $n$ formes linéaires $\ell_j = e_1^* (g_j \cdot)$ ($j \in \{0, \ldots , n-1\}$) sont linéairement dépendantes et qu'elle est égale à 
$$\frac{1}{(4i\pi)^n} \left( \frac{d \ell_0}{\ell_0} - \overline{\frac{d\ell_0}{\ell_0}} \right) \wedge \ldots \wedge \left( \frac{d \ell_{n-1}}{\ell_{n-1}} - \overline{\frac{d\ell_{n-1}}{\ell_{n-1}}} \right)$$
sinon. 

En faisant à nouveau appel au théorème de Brieskorn on retrouve que $\mathbf{S}_{\rm aff}$ définit un cocycle mais surtout que celui-ci représente la même classe que $\mathbf{S}_{\rm aff}^*$. Le fait que la classe de cohomologie correspondante $S_{\rm aff} \in H^{n-1} (G , \Omega_{\rm aff}^n)$ soit non nulle résulte finalement de \cite[Theorem 3]{Sczech93} où $\mathbf{S}_{\rm aff}$ est évalué sur un $(n-1)$-cycle d'éléments unipotents.  \qed
 
\medskip
\noindent
{\it Remarque.} Soit $W$ un sous-espace propre et non nul de $\C^n$, autrement dit un sommet de l'immeuble de Tits $\mathbf{T}$. Soit $X(W)$ le sous-espace de $\overline{X}^S$ constitué des matrices hermitiennes semi-définies positives dont le noyau contient $W$; il est homéomorphe à $\overline{X}_{n-\dim W}^S$ et donc contractile. Les sous-ensembles $X(\ell)$, avec $\ell \subset \C^n$ droites, forment un recouvrement acyclique du bord $\partial \overline{X}^S$ de la compactification de Satake. De plus, une intersection $X(\ell_1) \cap \ldots \cap X(\ell_k)$ est non-vide si et seulement si les droites $\ell_1, \ldots , \ell_k$ engendrent un sous-espace propre $W$ de $\C^n$, auquel cas 
$$X(\ell_1) \cap \ldots \cap X(\ell_k) = X(W).$$ 
Comme ensemble simplicial, la première subdivision barycentrique du nerf du recouvrement acyclique de $\partial \overline{X}^S$ par les $X(\ell )$ est donc égal à $\mathbf{T}$. C'est la source de la dualité qui relie les deux cocycles $\mathbf{S}_{\rm aff}$ et $\mathbf{S}_{\rm aff}^*$.

On a en effet des isomorphismes 
$$H_{n-1} ( \overline{X}^S , \partial \overline{X}^S ) \cong  \widetilde{H}_{n-2} (\partial \overline{X}^S) \cong \underbrace{\widetilde{H}_{n-2} (\mathbf{T})}_{= \mathrm{St} (\C^n )} \cong H_{n-1} ( \overline{X}^T , \partial \overline{X}^T ).$$
Explicitement, l'isomorphisme 
$$\mathrm{St} (\C^n ) \stackrel{\sim}{\longrightarrow} H_{n-1} ( \overline{X}^T , \partial \overline{X}^T )$$
associe à l'élément $[q_0 , \ldots , q_{n-1}] \in \mathrm{St} (\C^n )$ la classe de $\Delta (\mathbf{q})$ dans 
$$H_{n-1} ( \overline{X}^T , \partial \overline{X}^T );$$ 
il est $G$-équivariant ce qui se traduit par le fait que $\mathbf{S}_{\rm aff}^*$ s'étende en un isomorphisme $G$-équivariant de $\mathrm{St} (\C^n )$ vers $\Omega^n_{\rm aff}$. L'isomorphisme 
$$\mathrm{St} (\C^n ) \stackrel{\sim}{\longrightarrow} H_{n-1} ( \overline{X}^S , \partial \overline{X}^S )$$
associe quant à lui à un élément $[q_0 , \ldots , q_{n-1}] \in \mathrm{St} (\C^n )$ la classe de $[q_0^* , \ldots , q_{n-1}^*]$ dans 
$H_{n-1} ( \overline{X}^S , \partial \overline{X}^S )$; la représentation de $G$ dans $H_{n-1} ( \overline{X}^S , \partial \overline{X}^S )$ est donc naturellement identifiée à la représentation $\mathrm{St} ((\C^n )^\vee)$. 

\medskip

Dans la suite, on globalise la construction ci-dessus en remplaçant la forme $\eta$ par la valeur en $0$ d'une série d'Eisenstein construite à partir de la fonction test $\psi$ à l'infini. On ne considère que la compactification de Tits, plus naturelle pour notre propos comme le montre la remarque précédente.

\chapter{Séries d'Eisenstein associées à $\psi$} \label{C:7}

\resettheoremcounters

Dans ce chapitre on restreint les formes $\varphi$ et $\psi$ à l'espace symétrique associé à $\SL_n (\R)$ que nous noterons $X$ en espérant ne pas créer de confusion. On note donc dorénavant
$$S =  \GL_n (\R) / \SO_n, \ S^+ =  \GL_n (\R)^+ / \SO_n \ \mbox{et} \ X = S^+/ \R_{>0} = \SL_n (\R) /\SO_n.$$
On globalise les constructions précédentes en formant une série d'Eisenstein à partir de la fonction $\psi$. La préimage de cette série d'Eisenstein par une section de torsion est étudiée dans un registre plus général par Bismut et Cheeger \cite{BismutCheeger}. Nous considérons ici le fibré en tores; la valeur en $s=0$ de la série d'Eisenstein est une manière de régulariser la moyenne de $\eta$ relativement à un réseau de $\R^n$. Nous travaillons adéliquement. 

\section{Quotients adéliques}

Soit $\A$ l'anneau des adèles de $\Q$ et soit
$$[\GL_n] = \GL_n (\Q) \backslash \GL_n (\A ) / \SO_n Z(\R)^+.$$
Le théorème d'approximation forte pour $\GL_n$ implique que, pour tout sous-groupe compact ouvert $K \subset \GL_n (\A_f )$, le quotient 
$$[\GL_n] / K = \GL_n (\Q)  \backslash \left( (\GL_n (\R ) / \SO_n Z(\R)^+ ) \times \GL_n (\A_f ) \right) / K$$
est une union finie de quotients de $X$ de volumes finis que l'on peut décrire de la manière suivante. \'Ecrivons
\begin{equation} \label{E:approxforte}
\GL_n (\A_f ) =  \bigsqcup_{j} \GL_n (\Q)^+ g_j K
\end{equation}
avec $\GL_n (\Q )^+ = \GL_n (\Q ) \cap \GL_n (\R)^+$. Alors
\begin{equation} \label{E:quot0}
[\GL_n] / K = \bigsqcup_{j} \Gamma_j \backslash X,
\end{equation}
où $\Gamma_j$ est l'image de $\GL_n (\Q )^+ \cap g_j K g_j^{-1}$ dans $\GL_n (\R)^+ / Z(\R)^+$. 
La composante connexe de la classe de l'identité dans $[\GL_n] / K$ est le quotient 
\begin{equation} \label{E:quot3}
\Gamma \backslash X 
\end{equation}
où $\Gamma = K \cap \GL_n (\Q)^+$.

Soit $V = \mathbf{G}_a^n$ vu comme groupe algébrique sur $\Q$; on a en particulier
$$V (\Q ) = \Q^n \quad \mbox{et} \quad V (\R ) = \R^n.$$
Le groupe $\GL_n$, algébrique sur $\Q$, opère naturellement (par multiplication matricielle à gauche) sur $V(\C) = \C^n$. On note 
$$\mathcal{G} = \GL_n \ltimes V$$
le groupe affine correspondant; on le voit comme groupe algébrique sur $\Q$. Soit 
$$[\mathcal{G} ] = \mathcal{G} (\Q ) \backslash  \left[ (\GL_n (\R ) \ltimes \C^n ) \cdot \mathcal{G} (\A_f ) \right] / \SO_n Z (\R)^+.$$

On explique maintenant comment associer à ces données une famille de groupes abéliens isomorphes à $(\C^* )^n$. 
Soit 
$$L_f \subset V(\A_f ) = \{ I_n \} \ltimes V(\A_f ) \subset \mathcal{G} (\A_f)$$ 
un sous-groupe compact ouvert; l'intersection $L= L_f \cap V (\Q)$ est un réseau dans $V$. 
On suppose dorénavant que $K \subset \GL_n (\A_f)$ préserve $L_f$. Le sous-groupe 
$$\mathcal{K} = \mathcal{K}_{L_f} = K \ltimes L_f \subset \mathcal{G} (\A_f)$$ 
préserve $L_f$; c'est un sous-groupe compact ouvert. 
Les quotients
\begin{equation} \label{E:quot1}
[\mathcal{G} ] / L_f = \mathcal{G} (\Q ) \backslash  \left[ (\GL_n (\R ) \ltimes \C^n ) \cdot \mathcal{G} (\A_f ) \right] / \SO_n Z (\R)^+ L_f \quad \mbox{et} \quad [\mathcal{G} ] / \mathcal{K} 
\end{equation}
sont des fibrés en quotients de $\C^n$ au-dessus de respectivement $[\GL_n]$ et $[\GL_n] / K$. 

De \eqref{E:approxforte} on déduit que 
\begin{equation} \label{E:TK}
[\mathcal{G} ] / \mathcal{K} \simeq \bigsqcup_j \Gamma_j \backslash (X \times \C^n/L_j ),
\end{equation}
où $L_j = g_j (L_f) \cap V(\Q)$. 
L'isomorphisme s'obtient de la manière suivante~: étant donné une double classe 
$$\mathcal{G} (\Q) [( x , z) , ( g_f , v_f ) ] \mathcal{K},$$ 
avec $x \in \GL_n (\R )/\SO_n Z (\R)^+$, $z \in \C^n$ et $(g_f , v_f) \in \mathcal{G} (\A_f)$, 
on peut d'abord supposer que $x$ appartient à $X$ en multipliant à gauche par un élément de $\mathcal{G} (\Q)$ si nécessaire. On écrit alors 
$$(g_f , v_f) = (h , w)^{-1} (g_j , 0) k, \quad \mbox{avec } (h,w) \in  \mathcal{G} (\Q )^+ \mbox{ et } k \in \mathcal{K}.$$
Alors
\begin{equation} \label{E:TK2}
\begin{split}
\mathcal{G} (\Q) [( x , z) , ( g_f , v_f ) ] \mathcal{K} & = \mathcal{G} (\Q) [( x , z) , (h , w)^{-1} (g_j , 0)  ] \mathcal{K} \\
& = \mathcal{G} (\Q)(h , w)^{-1} [( hx , hz+w) , (g_j,0)  ] \mathcal{K} 
\end{split}
\end{equation}
d'image $[hx, hz+w]$ dans $\Gamma_j  \backslash (X \times \C^n/L_j )$.
Dans la suite, on note 
$$\mathcal{T}_\mathcal{K} = \Gamma \backslash (X \times \C^n/L )$$ 
le fibré au-dessus de \eqref{E:quot3}.

\medskip
\noindent
{\it Exemple.} Soit $N$ un entier strictement supérieur à $1$. On note  
$$K_{0} (N) \subset \GL_n (\widehat{\Z}) = \prod_p \GL_n (\Z_p)$$
le sous-groupe défini par les relations de congruences suivantes aux nombres premiers $p$ divisant $N$~: si $N = \prod p^{v_p (N)}$ alors la $p$-composante de $K_{0} (N)$ est constituée des matrices de $\GL_n (\Z_p )$ de la forme 
$$\left( \begin{array}{cc} \Z_p^\times & * \\  0_{1,n-1} & \GL_{n-1} (\Z_p )   \end{array} \right)$$
modulo $p^{v_p (N)}$. Le groupe $K_0 (N)$ est compact ouvert et 
$$[\GL_n] / K_0 (N) = \Gamma_0 (N) \backslash X$$
avec 
$$\Gamma_0 (N) = \left\{ A \in \SL_n (\Z ) \; : \; A \equiv \left( \begin{array}{ccc} * & * & * \\ 0 & * & * \\ \vdots & \vdots & \vdots \\ 0 & * & * \end{array} \right) \ (\mathrm{mod} \ N ) \right\}.$$
Le groupe $\mathcal{K}_0 (N) = K_0 (N) \ltimes V (\widehat{\Z} )$ est compact et ouvert dans $\mathcal{G} (\A_f )$ et 
$$\mathcal{T}_{\mathcal{K}_0 (N)} = \Gamma_0 (N) \backslash \left[ X \times (\C^n / \Z^n) \right].$$

\medskip

\section{Fonctions de Schwartz et cycles associés} \label{S:15}

Soit $\mathcal{S} (V (\A_f ))$ l'espace de Schwartz de $V(\A_f)$ des fonctions $\varphi_f : V (\A_f) \to \C$ localement constantes et à support compact. Le groupe $\mathcal{G} (\A_f)$ opère sur $\mathcal{S} (V (\A_f ))$ par la ``représentation de Weil''
$$\omega (g , v ) : \mathcal{S} (V (\A_f )) \to \mathcal{S} (V (\A_f )); \quad \phi \mapsto \left( w \mapsto \phi (g^{-1} (w-v) \right).$$ 

Considérons maintenant l'espace $C^\infty \left( \mathcal{G} (\A_f ) \right)$ des fonctions lisses; on fait opérer le groupe $\mathcal{G} (\A_f)$ sur $C^\infty \left( \mathcal{G} (\A_f ) \right)$ par la représentation régulière droite~:
$$( (h ,w) \cdot f) (g ,v ) = f( g h , g w + v ).$$
Les fonctions lisses sont précisément celles qui sont invariantes sous l'action d'un sous-groupe ouvert de $\mathcal{G} (\A_f )$. 

L'application 
\begin{equation}
\mathcal{S} (V (\A_f )) \to C^\infty \left( \mathcal{G} (\A_f ) \right); \quad \phi \mapsto f_\phi : ( (g,v) \mapsto \phi (-g^{-1} v ))
\end{equation}
est $\mathcal{G} (\A_f)$-équivariante relativement aux deux actions définies ci-dessus. 

\'Etant donné une fonction $\varphi_f \in \mathcal{S} (V (\A_f ))$ on note $L_{\varphi_f}$ le réseau des périodes de $\varphi_f$. Si $K \subset \GL_n (\A_f )$ est un sous-groupe compact ouvert qui laisse $\varphi_f$ invariante (et préserve donc $L_{\varphi_f}$), alors la fonction $f_{\varphi_f}$ est invariante (à droite) sous l'action de $\mathcal{K} = K \ltimes L_{\varphi_f} \subset \mathcal{G} (\A_f)$.

\begin{definition}
Soit $\varphi_f \in \mathcal{S} (V (\A_f ))$ une fonction de Schwartz invariante sous l'action d'un sous-groupe compact ouvert $K\subset \GL_n (\A_f )$. 
\begin{itemize}
\item Soit $D_{\varphi_f}$, resp. $D_{\varphi_f , K}$, l'image de l'application 
$$\mathcal{G} (\Q ) \left[ \left( \GL_n (\R) \times \{ 0 \} \right) \cdot \mathrm{supp} (f_{\varphi_f} ) \right]  \to  [\mathcal{G}] /L_{\varphi_f}  ,$$
resp.
$$\mathcal{G} (\Q ) \left[ \left( \GL_n (\R) \times \{ 0 \} \right) \cdot   \mathrm{supp} (f_{\varphi_f} ) \right]  \to  [\mathcal{G}]/ \mathcal{K},$$
induite par l'inclusion du support de $f_{\varphi_f}$ dans $\mathcal{G} (\A_f)$. 
\item Soit 
$$U_{\varphi_f} \subset  [\mathcal{G}] /L_{\varphi_f} , \quad \mbox{resp.} \quad U_{\varphi_f, K} \subset  [\mathcal{G}]/ \mathcal{K},$$ 
le complémentaire de $D_{\varphi_f}$, resp. $D_{\varphi_f , K}$.
\end{itemize}
\end{definition}

\medskip
\noindent
{\it Remarque.} {\it Via} l'isomorphisme \eqref{E:TK} la projection de $D_{\varphi_f , K} \subset [\mathcal{G}]/ \mathcal{K}$ est égale à la réunion finie
\begin{equation} 
\bigsqcup_j \bigcup_\xi \Gamma_j \backslash (X \times (L_j + \xi )/L_j ),
\end{equation}
où $\xi$ parcourt les éléments de $V(\Q) / L_j$ tels que $\varphi_f (g_j^{-1} \xi )$ soit non nul. 

\medskip

En effet, en suivant  \eqref{E:TK2} on constate que 
$$(g_f , v_f ) = (h^{-1} g_j , - h^{-1} w) k$$
et, la fonction $f_{\varphi_f}$ étant $\mathcal{K}$-invariante à droite, on a 
$$f_{\varphi_f} (g_f , v_f ) = f_{\varphi_f} (h^{-1} g_j , -h^{-1} w ) = \varphi_f ( g_j^{-1} w ).$$

\medskip

L'espace $D_{\varphi_f}$ est donc un revêtement fini de $[\GL_n]$. La fonction $\varphi_f$ induit en outre une fonction localement constante sur $D_{\varphi_f}$, c'est-à-dire un élément de $H^0 (D_{\varphi_f })$.
Maintenant, l'isomorphisme de Thom implique que l'on a~:
\begin{equation} \label{E:thom}
H^0 (D_{\varphi_f }) \stackrel{\sim}{\longrightarrow} H^{2n} \left( [\mathcal{G}]/ L_{\varphi_f} , U_{\varphi_f } \right);
\end{equation}
on note 
$$[\varphi_f] \in H^{2n}   \left( [\mathcal{G}]/ L_{\varphi_f} , U_{\varphi_f } \right)$$ 
l'image de $\varphi_f$; cette classe est $K$-invariante, on désigne par $[\varphi_f]_K$ son image dans 
$$H^{2n} \left( [\mathcal{G}]/ \mathcal{K}, U_{\varphi_f , K } \right).$$ 

\begin{lemma} \label{L:40}
Supposons $\widehat{\varphi}_f (0) =0$, autrement dit que $\int_{V(\A_f )} \varphi_f (v) dv=0$. Alors, l'image de $[\varphi_f]$ par l'application degré 
$$H^0 (D_{\varphi_f }) \to \Z^{\pi_0 (D_{\varphi_f })}$$
est égale à $0$. 
\end{lemma}
\begin{proof} Montrons en effet que pour tout $j$ on a 
$$\sum_{\xi \in V(\Q) / L_j} \varphi_f ( g_j^{-1} \xi) = 0.$$
Quitte à remplacer $\varphi_{f}$ par $\omega (g_j ) \varphi_f$ on peut supposer que $g_j$ est l'identité. Mais 
\begin{equation*}
\begin{split}
\sum_{\xi \in V(\Q) / L} \varphi_f (\xi ) & = \sum_{ v \in V (\A_f ) / L_{\varphi_f} } \varphi_f (v) \\
& = \frac{1}{\vol (L_{\varphi_f})} \sum_{v \in V (\A_f ) / L_{\varphi_f} } \int_{L_{\varphi_f} } \varphi_f (v+u) du \\
& = \frac{1}{\vol (L_{\varphi_f} )} \int_{V(\A_f )} \varphi_f (v) dv \\
& = \frac{1}{\vol (L_{\varphi_f} )} \widehat{\varphi}_f (0) =0.
\end{split}
\end{equation*}
\end{proof}

Dans la suite on désigne par $D_{\varphi_f }^0$ l'intersection de $D_{\varphi_f }$ avec la composante connexe $\mathcal{T}_{\mathcal{K}}$.

\medskip
\noindent
{\it Exemple.} Soit $N$ un entier strictement supérieur à $1$. Alors, le réseau 
$$L_{\mathcal{K}_0 (N)} = (N^{-1} \widehat{\Z} ) \times \widehat{\Z} \times \ldots \times \widehat{\Z}  \subset V (\A_f)$$ 
est $\mathcal{K}_0 (N)$-invariant. La fonction 
\begin{equation} \label{E:exvarphif}
\sum_{j=0}^{N-1} \delta_{(\frac{j}{N} , 0 , \ldots , 0) + \widehat{\Z}^n} - N \delta_{\widehat{\Z}^n} \in \mathcal{S} (V (\A_f))
\end{equation}
est $\mathcal{K}_0 (N)$-invariante, à support dans $L_{\mathcal{K}_0 (N)}$ et de degré $0$. 

En conservant les notations de l'exemple précédent, on désigne par $D_0 (N)$ le sous-ensemble de 
$$\mathcal{T}_{\mathcal{K}_0 (N)} = \Gamma_0 (N) \backslash \left[ X \times (\C^n / \Z^n) \right]$$
associé à la fonction \eqref{E:exvarphif}. Il est constitué de tous les points dont la première coordonnée dans la fibre au-dessus de $\Gamma_0 (N) \backslash X$ est de $N$-torsion et dont toutes les autres coordonnées sont nulles.

\medskip

\section{Série theta adélique} \`A toute fonction $\varphi_f \in \mathcal{S} (V (\A_f))$ il correspond les formes différentielles
$$\widetilde{\varphi} \otimes \varphi_f \quad \mbox{et} \quad \widetilde{\psi}  \otimes \varphi_f \in A^{\bullet} \left( S \times \C^n , \mathcal{S} (\C^n ) \right) \otimes \mathcal{S} (V (\A_f ))$$
de degrés respectifs $2n$ et $2n-1$.

En appliquant la distribution theta dans les fibres, on obtient alors des applications 
\begin{equation} \label{appl-theta}
\theta_\varphi \quad \mbox{et} \quad \theta_\psi  : \mathcal{S} (V (\A_f)) \longrightarrow \left[ A^{\bullet} (S \times \C^n)  \otimes C^\infty \left( \mathcal{G} (\A_f ) \right) \right]^{\mathcal{G} (\Q )}.
\end{equation}

\medskip
\noindent
{\it Remarque.} Rappelons que 
$$A^{\bullet} (S \times \C^n) \cong \left[ \wedge^\bullet (\mathfrak{p} \oplus \C^n)^* \otimes C^{\infty} (\GL_n (\R) \ltimes \C^n ) \right]^{\SO_n}.$$
Le produit tensoriel dans \eqref{appl-theta} est donc plus rigoureusement égal à 
$$\mathrm{Hom}_{\SO_n} \left( \wedge^\bullet (\mathfrak{p} \oplus \C^n) , C^{\infty} \left( (\GL_n (\R) \ltimes \C^n) \times \mathcal{G} (\A_f ) \right) \right)^{\mathcal{G} (\Q )}$$
où $C^{\infty} \left( (\GL_n (\R) \ltimes \C^n) \times \mathcal{G} (\A_f ) \right) $ est l'espace des fonctions lisses sur un espace adélique. 

\medskip

L'application $\theta_\varphi$ est définie par 
\begin{equation} \label{appl-theta2}
\begin{split}
\theta_\varphi (g_f , v_f ; \varphi_f ) & = \sum_{\xi \in V (\Q ) }  \widetilde{\varphi} (\xi ) (\omega (g_f ,  v_f ) \varphi_f ) (\xi ) \\
& = \sum_{\xi \in V(\Q ) } \varphi_f \left( g_f^{-1} (\xi -v_f ) \right) \widetilde{\varphi} (\xi )
\end{split}
\end{equation}
et de même pour $\theta_\psi$. 

Rappelons que le groupe $\mathcal{G} (\R)$ opère naturellement sur $S \times \C^n$; étant donné un élément $(g,v) \in \mathcal{G} (\R)$ et une forme $\alpha \in \mathcal{A}^{\bullet} (S \times \C^n)$ on note $(g,v)^* \alpha$ le tiré en arrière de $\alpha$ par l'application
$$(g,v) :  S  \times \C^n \to S \times  \C^n.$$
L'invariance sous le groupe $\mathcal{G}(\Q)$ dans \eqref{appl-theta} signifie donc que pour tout $(g,v)$ dans $\mathcal{G} (\Q)$ on a~:
\begin{equation} \label{E:invtheta}
(g,v)^*\theta_\varphi (g g_f ,  gv_f  + v ;\varphi_f)=\theta_\varphi (g_f, v_f ; \varphi_f);
\end{equation}
ce qui découle de la $\mathcal{G} (\R)$-invariance de $\widetilde{\varphi}$, voir \S \ref{S:44}. 

Les applications $\theta_\varphi$ et $\theta_\psi$ entrelacent par ailleurs les actions naturelles de $\mathcal{G} (\A_f)$ des deux côtés~: pour tout $(h_f , w_f ) \in \mathcal{G} (\A_f )$ et pour $\theta = \theta_\varphi$ ou $\theta_\psi$, on a 
\begin{equation} \label{E:entrelacetheta}
\theta ( \omega (h_f , w_f ) \varphi_f ) = (h_f , w_f ) \cdot \theta (\varphi_f ).
\end{equation}
En particulier, on a 
$$\theta_\varphi (\varphi_f )  \quad \mbox{et} \quad \theta_\psi (\varphi_f ) \in \left[ A^{\bullet} (S \times \C^n)  \otimes C^\infty \left( \mathcal{G} (\A_f ) \right) \right]^{\mathcal{G} (\Q ) \times L_{\varphi_f}} = A^\bullet \left( \widehat{[\mathcal{G}]} / L_{\varphi_f} \right),$$
où 
$$\widehat{[\mathcal{G}]} = \mathcal{G} (\Q) \backslash \left( (\GL_n (\R) \ltimes \C^n) \times  \mathcal{G} (\A_f ) \right) / \SO_n.$$
Il découle en outre de \eqref{E:entrelacetheta} que si $\varphi_f$ est $K$-invariante alors les formes $\theta_\varphi (\varphi_f )$ et 
$\theta_\psi (\varphi_f )$ sont $K$-invariantes à droite. 

\medskip

\medskip

\subsection{Action de l'algèbre de Hecke} \label{algHecke}
Soit $p$ un nombre premier. On désigne par $\mathcal{H}_p$ l'algèbre de Hecke locale de $\mathcal{G} (\Q_p )$, c'est-à-dire les fonctions lisses et à support compact dans $\mathcal{G} (\Q_p )$ et le produit de convolution. La fonction caractéristique de $\mathcal{G} (\Z_p)$ appartient à $\mathcal{H}_p$. 

L'\emph{algèbre de Hecke globale} $\mathcal{H} (\mathcal{G} (\A_f) )$ est le produit restreint des algèbres de Hecke locales $\mathcal{H}_p$ relativement aux fonctions caractéristiques de $\mathcal{G} (\Z_p)$, cf. \cite{Flath}. 

Un élément de $\mathcal{H} (\mathcal{G} (\A_f) )$ est donc un produit tensoriel 
$\phi = \otimes \phi_p$, où pour presque tout $p$ la fonction $\phi_p$ est égale à la fonction caractéristique de $\mathcal{G} (\Z_p)$. On désigne par $\mathbf{T}_\phi$ l'{\it opérateur de Hecke} sur $C^\infty (\mathcal{G} (\A_f ))$ qui lui correspond. Il associe à une fonction $f \in  C^\infty (\mathcal{G} (\A_f ))$ la fonction 
$$\mathbf{T}_\phi (f) : (g,v)  \mapsto \int_{\mathcal{G} (\A_f)} f(gh , gw+v) \phi (h,w) d(h,w),$$
où la mesure de Haar est normalisée de sorte que pour tout $p$, le volume de $\mathcal{G} (\Z_p )$ soit égal à un.  
On note encore 
$$\mathbf{T}_\phi : A^\bullet (S \times \C^n) \otimes C^\infty (\mathcal{G} (\A_f )) \to A^\bullet (S \times \C^n) \otimes C^\infty (\mathcal{G} (\A_f ))$$
l'opérateur de Hecke induit; il commute à l'action (à gauche) de $\mathcal{G} (\Q)$.

L'algèbre de Hecke $\mathcal{H} (\mathcal{G} (\A_f) )$ opère également sur $ \mathcal{S} (V (\A_f))$ {\it via} les opérateurs 
$T_\phi : \mathcal{S} (V (\A_f)) \to  \mathcal{S} (V (\A_f))$ qui à une fonction de Schwartz $\varphi_f$ associe la fonction 
$$T_\phi (\varphi_f ) : v \mapsto \int_{\mathcal{G} (\A_f )} \phi (h,w) \varphi_f (h^{-1} (v-w)) d (h,w) .$$

La proposition suivante résulte alors des définitions.

\begin{proposition} \label{P:hecke1}
Soit $\phi \in \mathcal{H} (\mathcal{G} (\A_f ))$. Alors $\mathbf{T}_\phi$ préserve 
$$\left[ A^\bullet (S \times \C^n) \otimes C^\infty (\mathcal{G} (\A_f )) \right]^{\mathcal{G} (\Q)} $$
et 
$$\mathbf{T}_\phi  ( \theta_{\varphi} (\varphi_f) ) = \theta_{\varphi} ( T_\phi (\varphi_f )) \quad \mbox{et} \quad \mathbf{T}_\phi  ( \theta_{\psi} (\varphi_f) ) = \theta_{\psi} ( T_\phi (\varphi_f )).$$
\end{proposition}

\subsection{Classes de cohomologie associées} \label{cohomClass}

Notons $\widehat{U_{\varphi_f}}$ la préimage de $U_{\varphi_f}$ dans  $\widehat{[\mathcal{G}]} / L_{\varphi_f}$ et $\widehat{U_{\varphi_f , K}} $ la projection de $\widehat{U_{\varphi_f}}$ dans $ \widehat{[\mathcal{G}]} / \mathcal{K}$ de sorte que
$$\widehat{U_{\varphi_f , K}} = \widehat{[\mathcal{G}]} / \mathcal{K} - \mathcal{G} (\Q) \left( [S^+ \times \{ 0 \}] \cdot  \mathrm{supp}(f_{\varphi_f} ) \right) \mathcal{K}.$$ 
Notons enfin $\widehat{[\varphi_f]}$ l'image de $[\varphi_f]$ dans 
$$H^{2n} \left( \widehat{[\mathcal{G}]} / L_{\varphi_f} , \widehat{U_{\varphi_f}} \right).$$

\begin{proposition} \label{P:thetacohom}
La forme différentielle $\theta_\varphi (\varphi_f )$ est fermée et représente $\widehat{[\varphi_f]}$ dans $H^{2n} \left( \widehat{[\mathcal{G}]} / L_{\varphi_f} , \widehat{U_{\varphi_f}} \right)$.
\end{proposition}
\begin{proof} La forme $\theta_\varphi (\varphi_f )$ est fermée comme combinaison linéaire de formes fermées. Fixons un sous-groupe compact ouvert $K \subset \GL_n (\A_f)$ tel que $\varphi_f$ soit $K$-invariante. Alors $\theta_\varphi (\varphi_f )$ appartient à 
$$\left[ A^{2n} (S \times \C^n)  \otimes C^\infty \left( \mathcal{G} (\A_f ) \right) \right]^{\mathcal{G} (\Q ) \times \mathcal{K}} \cong \bigoplus_j A^{2n} (S^+ \times \C^n )^{\Gamma_j \ltimes L_j},$$
où l'isomorphisme ci-dessus est obtenu en évaluant en $(g_j , 0)$. D'un autre côté, on a montré que l'image de $[\varphi_f ]_{K}$ dans 
$$H^{2n} \left( \widehat{[\mathcal{G}]} / \mathcal{K} \right)  \cong \bigoplus_j H^{2n} (\Gamma_j \backslash (S^+ \times \C^n/L_j ))$$
est égale à 
$$\bigoplus_j \sum_{\xi \in V(\Q)/ L_j} \varphi_f (g_j^{-1} \xi ) \left[ \Gamma_j \backslash (S^+ \times (L_j +\xi )/L_j ) \right].$$
Quitte à remplacer  $\varphi_f$ par $\omega (g_j  ) \varphi_f$ on peut donc se restreindre à la composante connexe associée à l'identité. On a alors simplement $\Gamma_j = \Gamma$ et $L_j = L$.

Maintenant, pour tout $\xi \in V (\Q )$ la forme différentielle
$$\varphi_f (\xi ) \widetilde{\varphi} (\xi ) = \varphi_f (\xi) (1,-\xi)^* \varphi \in A^{2n} ( S^+ \times \C^n )$$
est une forme de Thom et représente  
$$\varphi_f (\xi) [S^+ \times \{  \xi \} ] \in H^{2n} (S^+ \times \C^n , S^+ \times (\C^n - \{ \xi \} )).$$
La moyenne 
$$\theta_\varphi (1,0;\varphi_f ) = \sum_{\xi \in V (\Q)} \varphi_f (\xi ) \widetilde{\varphi} (\xi) \in A^{2n} (S^+ \times \C^n )$$
représente donc l'image de  
$$\sum_{\xi \in V(\Q) / L} \varphi_f (\xi) \left[ \Gamma \backslash (S^+ \times (L +\xi )/L ) \right]$$
par l'isomorphisme de Thom. 
\end{proof}

La proposition \ref{P:eta} suggère de considérer les formes différentielles 
$$\theta_{[r]^*\varphi}  (\varphi_f ) \quad \mbox{et} \quad \theta_{[r]^*\psi}  (\varphi_f ) \in \left[ A^{\bullet} (S \times \C^n)  \otimes C^\infty \left( \mathcal{G} (\A_f ) \right) \right]^{\mathcal{G} (\Q )} \quad (r >0).$$
Le lemme suivant est une version globale du lemme \ref{L:convcourant}. On le déduit de la formule sommatoire de Poisson. 

\begin{lemma} \label{L:theta-asympt}
1. Lorsque $r$ tend vers $+\infty$, les formes $\theta_{[r]^* \varphi} (\varphi_f )$ convergent uniformément sur tout compact de 
$\widehat{U_{\varphi_f , K}}$ 
exponentiellement vite vers la forme nulle.

2. Vues comme courants dans 
$$ \left[\mathcal{D}^{2n} (S \times \C^n )  \otimes C^\infty \left( \mathcal{G} (\A_f ) \right) \right]^{\mathcal{G} (\Q ) \times \mathcal{K}}$$
les formes $\theta_{[r]^* \varphi} (\varphi_f )$ convergent exponentiellement vite, lorsque $r$ tend vers $+\infty$, vers le courant $[\widehat{D_{\varphi_f , K}}]$ associé à la fonction localement constante $\varphi_f$ et de support
$$\mathcal{G} (\Q) \left( [S^+ \times \{ 0 \}] \times \mathrm{supp}(f_{\varphi_f}) \right) \mathcal{K}.$$

3. L'application qui à $r \in \R_+^*$ associe la forme différentielle $\theta_{[r]^* \varphi} ( \varphi_f )$ sur $S \times \C^n $ se prolonge en une fonction lisse sur $[0, +\infty)$ qui s'annule en $0$.
\end{lemma}
\begin{proof} Les deux premiers points découlent des points (1) et (2) du lemme \ref{L:convcourant}. On peut en outre remarquer que, lorsque $r$ varie, les formes $\theta_{[r] ^* \varphi} (\varphi_f )$ sont toutes cohomologues d'après la démonstration de la proposition \ref{P:thetacohom}. 

Le troisième point découle lui du point (3) du lemme \ref{L:convcourant}. 
\end{proof}

\section{Séries d'Eisenstein adéliques} \label{SEA7}

Soit $\varphi_f \in \mathcal{S} (V(\A_f))$. Les trois propositions suivantes sont des incarnations adéliques du procédé classique de régularisation de Hecke, voir par exemple \cite[Chapitre IV]{Wielonsky}. Elles se démontrent de la même manière; on ne détaille que la démonstration de la première qui est légèrement plus subtile. 

\begin{proposition} \label{P:Eis1}
L'intégrale 
\begin{equation} \label{Eis1}
E_\varphi ( \varphi_f , s) = \int_0^{\infty} r^{s} \theta_{[r]^*\varphi}  (\varphi_f ) \frac{dr}{r},
\end{equation}
qui converge absolument, uniformément sur tout compact de $\widehat{U_{\varphi_f , K}}$ si $\mathrm{Re} (s) >0$, possède un prolongement méromorphe à $\C$ tout entier, à valeurs dans l'espace $A^{2n} (\widehat{U_{\varphi_f , K}})$, holomorphe en dehors de pôles au plus simples aux entiers strictement négatifs. 
\end{proposition}
\begin{proof} D'après le lemme \ref{L:theta-asympt} (1), l'intégrale \eqref{Eis1} 
est absolument convergente sur tout compact de $\widehat{U_{\varphi_f , K}}$ si $\mathrm{Re} (s) >0$. Elle définit donc une forme différentielle dans $A^{2n} (\widehat{U_{\varphi_f , K}})$. Pour prolonger cette fonction de $s$, il suffit de couper l'intégrale en deux morceaux, l'un allant de $0$ à $1$ où l'on utilise un développement limité de la fonction $r \mapsto \theta_{[r]^*\varphi}  (\varphi_f )$ au voisinage de $0$, qui est bien défini d'après le lemme \ref{L:theta-asympt} (3), et l'autre de $1$ à $+\infty$ qui ne pose pas de problème d'après le lemme \ref{L:theta-asympt} (1). Le fait que $0$ ne soit pas un pôle découle du fait que $r \mapsto \theta_{[r]^*\varphi}  (\varphi_f )$ s'annule en $0$.
\end{proof}

\begin{proposition} \label{P:Eis1bis}
L'intégrale 
\begin{equation} \label{Eis1bis}
E_\varphi ( \varphi_f , s) = \int_0^{\infty} r^{s} \left( \theta_{[r]^*\varphi}  (\varphi_f ) - [\widehat{D_{\varphi_f , K}}] \right) \frac{dr}{r},
\end{equation}
définit une application méromorphe en $s$ à valeurs dans l'espace des courants 
$$\left[\mathcal{D}_{\bullet} (S \times \C^n )  \otimes C^\infty \left( \mathcal{G} (\A_f ) \right) \right]^{\mathcal{G} (\Q ) \times \mathcal{K}}$$
avec un pôle simple en $s=0$ de résidu $- [\widehat{D_{\varphi_f , K}}]$. 
\end{proposition} 
\begin{proof} La démonstration est identique à celle de la proposition \ref{P:Eis1} à ceci près que cette fois on applique le lemme \ref{L:theta-asympt} (2) et que l'application qui à $r$ associe le courant $ \theta_{[r]^*\varphi}  (\varphi_f ) - [\widehat{D_{\varphi_f , K}}]$ ne s'annule plus en $0$ mais est égale à $- [\widehat{D_{\varphi_f , K}}]$. La fonction qui à $s$ associe le courant $E_\varphi ( \varphi_f , s)$ a donc cette fois un pôle (simple) en $0$ de résidu $- [\widehat{D_{\varphi_f , K}}]$.
\end{proof}

\begin{proposition} \label{P:Eis2}
L'intégrale 
\begin{equation} \label{Eis2}
E_\psi ( \varphi_f , s) = \int_0^{\infty} r^{s} \theta_{[r]^*\psi}  (\varphi_f ) \frac{dr}{r}
\end{equation}
qui converge absolument, uniformément sur tout compact de $\widehat{U_{\varphi_f , K}}$ si $\mathrm{Re} (s) >0$, possède un prolongement méromorphe à $\C$ tout entier, à valeurs dans l'espace $A^{2n-1} (\widehat{U_{\varphi_f , K}})$, holomorphe en dehors de pôles au plus simples aux entiers strictement négatifs. 
\end{proposition} 
\begin{proof} La démonstration est identique à celle de la proposition \ref{P:Eis1}. Le fait que la fonction $r \mapsto  \theta_{[r]^*\psi}  (\varphi_f )$ soit nulle en $0$ découle cette fois du fait que $\psi$ s'annule en $0$. 
\end{proof}

\medskip

Rappelons maintenant que 
$$[\mathcal{G}] = \widehat{[\mathcal{G} ]} / Z (\R)^+$$
où $Z (\R)^+$ opère naturellement sur $S^+$ et trivialement sur la fibre $\C^n$. L'action d'un élément $\lambda \in Z(\R)^+$ sur 
$S^+ \times \C^n$ s'obtient donc en composant l'action de $(\lambda , 0)$ dans $\mathcal{G} (\R)$ par la multiplication par $\lambda^{-1}$ dans la fibre $\C^n$. La $\mathcal{G} (\R)$-équivariance de $\widetilde{\varphi}$ implique alors que pour tout $w \in \C^n$ on a
$$\lambda^* \widetilde{\varphi} (w) = [\lambda^{-1}]^* ((\lambda , 0 )^* \widetilde{\varphi} (w)) = [\lambda^{-1}]^* \widetilde{\varphi} (\lambda^{-1} w ) = \widetilde{[\lambda^{-1}]^* \varphi} (w)$$
et de même pour $\psi$. Il s'en suit que pour tout $\lambda \in Z(\R)^+$ on a 
\begin{equation}
\lambda^* \widetilde{[r]^*\varphi} = \widetilde{[\lambda^{-1} r]^*\varphi} \quad \mbox{et} \quad \lambda^* \widetilde{[r]^*\psi} = \widetilde{[\lambda^{-1} r]^*\psi}
\end{equation}
et donc
\begin{equation} \label{E:actZ}
\lambda^* E_\varphi (\varphi_f , s) = \lambda^{s} E_{\varphi} (\varphi_f , s) \quad \mbox{et} \quad \lambda^*E_\psi (\varphi_f , s) =  \lambda^{s} E_\psi (\varphi_f , s).
\end{equation}
On pose  
\begin{equation} 
E_\psi (\varphi_f) = E_\psi (\varphi_f , 0) \in A^{2n-1} \left( [\mathcal{G}] / L_{\varphi_f}  - D_{\varphi_f } \right) .
\end{equation}

\begin{theorem}
La forme différentielle
$$E_\psi (\varphi_f) \in A^{2n-1} \left( [\mathcal{G}] / L_{\varphi_f}  - D_{\varphi_f } \right)$$ 
est \emph{fermée} et représente une classe de cohomologie qui relève la classe\footnote{Noter que dans le cas (multiplicatif) de ce paragraphe, l'image de $[\varphi_f]$ dans $H^{2n} \left(  [\mathcal{G}] / L_{\varphi_f} \right)$ est nulle puisque les fibres sont de dimension cohomologique $n$.} 
$$[\varphi_f] \in H^{2n} \left(  [\mathcal{G}] / L_{\varphi_f} , [\mathcal{G}] / L_{\varphi_f}  - D_{\varphi_f }\right)$$ 
dans la suite exacte longue
\begin{multline*}
\ldots \to H^{2n-1} \left( [\mathcal{G}] / L_{\varphi_f}  - D_{\varphi_f } \right) \\ \to H^{2n} \left(  [\mathcal{G}] / L_{\varphi_f} , [\mathcal{G}] / L_{\varphi_f}  - D_{\varphi_f } \right) \to H^{2n} \left(  [\mathcal{G}] / L_{\varphi_f} \right) \to \ldots
\end{multline*}
\end{theorem}
\begin{proof} C'est une version adélique de \cite[Theorem 19]{Takagi}. Fixons un sous-groupe compact ouvert $K \subset \GL_n (\A_f )$ tel que $\varphi_f$ soit $K$-invariante. Les intégrales \eqref{Eis1} et \eqref{Eis2} étant absolument convergentes sur tout compact de $\widehat{U_{\varphi_f , K}}$, il découle de (\ref{E:tddt1}) que 
\begin{equation*}
\begin{split}
d E_\psi ( \varphi_f , s) & = \int_0^\infty r^{s} d (\theta_{[r]^*\psi}^*  (\varphi_f )) \frac{dr}{r} \\ 
& = \int_0^\infty r^{s} \frac{d}{dr} \theta_{[r]^*\varphi}  (\varphi_f )  dr.
\end{split}
\end{equation*}
Une intégration  par parties donne donc
\begin{equation} \label{dEis}
d E_\psi ( \varphi_f , s) = - s  \int_0^\infty r^{s}  \theta_{[r]^*\varphi}  (\varphi_f )  \frac{dr}{r} = -s E_{\varphi} (\varphi_f , s )
\end{equation}
sur $\widehat{U_{\varphi_f , K}}$. 

En particulier la forme $E_\psi ( \varphi_f ) = E_\psi ( \varphi_f , 0)$ est fermée sur $\widehat{U_{\varphi_f , K}}$ et la première partie du théorème découle finalement du fait que, d'après \eqref{E:actZ}, la forme $E_\psi ( \varphi_f )$ est invariante sous l'action du centre $Z (\R)^+$. 

Enfin, la proposition \ref{P:Eis1bis} implique que l'identité \eqref{dEis} s'étend sur $\widehat{[\mathcal{G}]} / \mathcal{K}$ en une identité entre courants et qu'en ce sens
$$d E_\psi (\varphi_f) =[\widehat{D_{\varphi_f , K}}]$$
et que l'image de $[E_\psi (\varphi_f)] \in H^{2n-1} (\widehat{U_{\varphi_f , K}})$ dans $H^{2n} (\widehat{[\mathcal{G}] }/ \mathcal{K} , \widehat{U_{\varphi_f , K}})$ est égale à $[\varphi_f]_K$. Le théorème s'en déduit en remarquant encore que ces classes sont toutes invariantes sous l'action de $Z (\R )^+$. 
\end{proof}

\medskip

Pour conclure ce paragraphe notons que pour tout $g \in \GL_n (\Q) $, il découle de \eqref{E:invtheta} que l'on a
\begin{equation} \label{E:glnqinv}
g^* (E_\psi ( \varphi_f ) (gg_f , gv_f )) = E_\psi (\varphi_f) (g_f , v_f).
\end{equation} 
En prenant $g$ scalaire et $(g_f, v_f) = (g_j , 0)$ on obtient en particulier que pour $\alpha$ dans $\Q^*$, positif si $n$ est impair, on a 
$$E_\psi (\varphi_f (\alpha \cdot )) (g_j , 0) = E_{[\alpha^{-1}]^*\psi} (\varphi_f ) (g_j , 0) = E_{\psi} (\varphi_f ) (g_j , 0) .$$
Ce qui implique que 
$$[E_{\psi} (\varphi_f ) (g_j , 0) ] \in H_{\Gamma_j}^{2n-1} \left( ( \C^n - \cup_\xi (L_j +\xi) ) / L_j \right)^{(1)},$$
au sens de la définition \ref{Def1.7}. 

La restriction de $E_\psi (\varphi_f)$ à la composante connexe $\mathcal{T}_{\mathcal{K}}$ définit une forme fermée $E_\psi (\varphi_f)^0$ sur
$$\Gamma \backslash (X \times \C^n )/ L - D_{\varphi_f }^0$$
et donc une classe de cohomologie équivariante 
\begin{equation} \label{E:721}
[E_\psi (\varphi_f)^0] \in H_\Gamma^{2n-1} \left( ( \C^n - \cup_\xi (L+ \xi) ) / L \right)^{(1)} ,
\end{equation}
où $\xi$ parcourt des éléments de $V (\Q) \cap \mathrm{supp} (\varphi_f )$.

\medskip
\noindent
{\it Exemple.} Soit $N$ un entier strictement supérieur à $1$. En prenant pour $\varphi_f$ la fonction \eqref{E:exvarphif}, on obtient une classe de cohomologie équivariante  
$$E_{D_0 (N)} \in H_{\Gamma_0 (N)}^{2n-1} (T - T [N])^{(1)}$$
avec $T = \C^n / \Z^n$. Cette classe est associée au cycle invariant de $N$-torsion et de degré $0$  
$$D_0 (N) \in H_{\Gamma_0 (N)}^0 (T[N])$$
comme expliqué au paragraphe \ref{S:15}. 

\medskip

\section{Comportement à l'infini de $E_\psi (\varphi_f)$} \label{S192}

\begin{definition}
\'Etant donné un sous-espace rationnel $W \subset V$, on note 
$$\int_W : \mathcal{S} (V (\A_f)) \to \mathcal{S} (V (\A_f) / W (\A_f))$$
l'application naturelle d'intégration le long des fibres de la projection $V \to V/ W$. 
\end{definition}

Dans ce paragraphe on fixe une fonction $\varphi_f \in \mathcal{S}  (V(\A_f ))$ et un sous-groupe parabolique $Q = Q (W_\bullet )$ dans $\SL_n (\Q)$ associé à un drapeau de sous-espaces rationnels de $\Q^n$; voir (\ref{E:flag}) dont on reprend les notations. On se propose d'étudier le comportement de la forme différentielle $E_\psi (\varphi_f)$ en restriction aux ensembles de Siegel associés à $Q$. On fixe également 
\begin{itemize}
\item un réel strictement positif $t_0$,
\item un élément $g \in \SL_n (\Q)$ tel que $g^{-1} W_\bullet$ soit un drapeau standard $W_J$,
\item un sous-ensemble relativement compact $\omega\subset N_J M_J$, et 
\item un sous-ensemble compact $\kappa$ dans
\begin{equation*} 
\left\{ (z , (g_f , v_f ) ) \in \C^n \times \mathcal{G} (\A_f )  \; \left| \; \begin{array}{l} \forall \xi \in V (\Q), \\ \varphi_f (g_f^{-1} (\xi - v_f)) \neq 0 \Rightarrow z \notin  W_k (\C) + \xi  \end{array} \right. \right\}.
\end{equation*}
\end{itemize}

\begin{lemma} \label{L:thetaSiegel}
Il existe des constantes strictement positives $C$, $\alpha$ et $\beta$ telles que pour tout $t \geq t_0$ les deux propriétés suivantes sont vérifiées.
\begin{enumerate}
\item Les formes $\theta_{[r]^* \psi} (\varphi_f )$ ($r\geq 1$) sont de norme
$$||\theta_{[r]^* \psi} (\varphi_f ) ||_\infty \leq C e^{-r^2 \alpha t^\beta}$$
en restriction à $\mathfrak{S}_{W_\bullet} (g ,t, \omega) \times \kappa$.  
\item Si l'on suppose de plus que $\int_{W_1} \varphi_f$ est constante égale à $0$, alors les formes $\theta_{[r]^* \psi} (\varphi_f )$ ($r\leq 1$) sont de norme
$$||\theta_{[r]^* \psi} (\varphi_f ) ||_\infty \leq C e^{-r^{-2} \alpha t^\beta}$$
en restriction à $\mathfrak{S}_{W_\bullet} (g ,t, \omega) \times \C^n \times \mathcal{G} (\A_f )$.
\end{enumerate}
\end{lemma}
\begin{proof} La premier point découle du lemme \ref{L8} (la convergence uniforme de la somme sur $\xi$ résulte du fait qu'on somme des fonctions gaussiennes). Pour démontrer le deuxième point il suffit de majorer la norme ponctuelle de 
$$(h,v)^*\theta_{[r]^* \psi} (\varphi_f ) = [r]^* \left(\sum_\xi \varphi_f (\xi)  (\omega (r^{-1} h , v) \widetilde{\psi}) (\xi ) \right),$$
en le point base $([e] , 0)$ de $X \times \C^n$, pour $(h,v) \in \mathfrak{S}_{W_\bullet} (g ,t, \omega) \times \C^n$. Pour cela, on utilise à nouveau la formule de Poisson~:
\begin{equation*}
\sum_\xi \varphi_f (\xi)  (\omega (r^{-1} h , v) \widetilde{\psi}) (\xi )  = \sum_\xi \widehat{\varphi}_f (\xi)  \widehat{(\omega (r^{-1} h , v) \widetilde{\psi})} (\xi ) .
\end{equation*}
Maintenant, pour $(h , v) \in \GL_n (\C) \ltimes \C^n$ on a~:
$$| \widehat{\omega(h , v) \widetilde{\psi } } (\xi)|  = | \det (h) \widehat{\widetilde{\psi}} (h^{\top} \xi )|.$$
On est donc ramené à étudier la croissance de $\widehat{\widetilde{\psi}}$ sur $\mathfrak{S}_{W_\bullet} (g ,t, \omega)^{-\top} \times \{ 0 \}$ où l'image par $h \mapsto h^{-\top}$ de $\mathfrak{S}_{W_\bullet} (g ,t, \omega)$ est un 
ensemble de Siegel de la forme $\mathfrak{S}_{W_\bullet '} (g' ,t, \omega')$ associé au drapeau  
$$(0) \varsubsetneq W_k '  \varsubsetneq W_{k-1} '   \varsubsetneq \cdots   \varsubsetneq W_1 '  \varsubsetneq \Q^n,$$
avec $g' = w_J g^{-\top} w_J^{-1}$, $\omega' = w_J \omega^{-\top} w_J^{-1}$ et 
$$w_J =  \left(\begin{array}{ccc}
 & & 1_{n-j_k} \\ 
 & \reflectbox{$\ddots$} & \\
1_{n-j_1} & & \end{array} \right).$$ 
Or 
$$\widehat{\widetilde{\psi}} (h^{\top} \xi ) = (\omega (h^{-\top} , 0 ) \widehat{\widetilde{\psi}} ) (\xi)= ((h^{-\top})^* \widehat{\widetilde{\psi}} ) (\xi)$$ 
et la démonstration du lemme \ref{L8} --- ou le fait que $\widehat{\widetilde{\psi}}$ soit une forme différentielle dans $A^{2n-1} (E , \mathcal{S}(V))$ --- implique qu'en restriction à $\mathfrak{S}_{W_\bullet '} (g' ,t, \omega')$ la forme 
$$\widehat{\widetilde{\psi}} (h^\top \xi ) \quad \mbox{a pour norme} \quad O(e^{-\alpha' t^{\beta '} |p_{W_1} (\xi )|^2})$$ 
pour certaines constantes strictement positives $\alpha '$ et $\beta '$.

Finalement, par hypothèse $\int_{W_1} \varphi_f$ est identiquement nulle, de sorte que pour tout $\xi \in W_1 '$, on a\footnote{Ici $\chi$ est un caractère additif fixé.}
\begin{equation*}
\begin{split}
\widehat{\varphi}_f (\xi ) & = \int_{V (\A_f )} \varphi_f (x) \chi (\langle \xi , x \rangle) dx \\
& = \int_{W_1 '} \left( \int_{W_1} \varphi_f (w+w') dw' \right) \chi (\langle \xi , w \rangle) dw \\
& = 0.
\end{split}
\end{equation*}
La somme 
$$\sum_{\xi \in V (\Q ) } \widehat{\varphi}_f (\xi ) \widehat{\psi} (\xi )$$
ne porte donc que sur les $\xi \notin W_1'$ c'est-à-dire ceux tels que $p_{W_1} (\xi) \neq 0$ et le théorème s'en déduit.
\end{proof}

\begin{proposition} \label{P12}
Supposons $\int_{W_1} \varphi_f$ constante égale à $0$. Il existe alors des constantes strictement positives $C$, $\alpha$ et $\beta$ telles que pour tout $t \geq t_0$, la forme différentielle $E_\psi (\varphi_f )$ soit de norme $\leq C e^{- \alpha t^\beta}$ en restriction à
$$\mathcal{G} (\Q ) \left[ \mathfrak{S}_{W_\bullet} (g ,t, \omega) \times \kappa \right] / \SO_n .$$  
\end{proposition}
\begin{proof}
La forme différentielle $E_\psi (\varphi_f )$ est définie par l'intégrale absolument convergente 
$$E_\psi ( \varphi_f ) = \int_0^{\infty} \theta_{[r]^*\psi}  (\varphi_f ) \frac{dr}{r}.$$
La proposition se démontre en décomposant l'intégrale en une somme $\int_0^1 + \int_1^\infty$ et en appliquant à chacune de ces intégrales le lemme \ref{L:thetaSiegel}. 
\end{proof}

\section{\'Evaluation de $E_\psi (\varphi_f )$ sur les symboles modulaires} \label{S:7.eval}

Soit $\mathbf{q} = (q_0 , \ldots , q_{k})$ un $(k+1)$-uplet de vecteurs non nuls dans $V(\Q)$ avec $k \leq n-1$. 

Rappelons que l'on a associé à $\mathbf{q}$ une application continue \eqref{E:appDelta} de la première subdivision barycentrique $\Delta_k ' $ du $k$-simplexe standard vers la compactification de Tits~:
$$\Delta (\mathbf{q}) : \Delta_k ' \to \overline{X}^T.$$
On considère ici l'application
$$\Delta (\mathbf{q}) \times \mathrm{id}_{\C^n \times \mathcal{G} (\A_f)} : \Delta_k ' \times \C^n \times \mathcal{G} (\A_f) \to \overline{X}^T \times \C^n \times \mathcal{G} (\A_f).$$

Pour tout entier $j \in [0 , k]$ on désigne par $W(\mathbf{q})^{(j)}$ le sous-espace 
$$W(\mathbf{q})^{(j)} = \langle q_0 , \ldots , \widehat{q}_j , \ldots , q_{k} \rangle$$
dans $V$. 

Soit $\varphi_f \in \mathcal{S} (V (\A_f))$ une fonction de Schwartz telle que pour tout entier $j$ dans $[0 , k]$ on ait 
\begin{equation} \label{E:condphi}
\int_{W(\mathbf{q})^{(j)}} \varphi_f =0 \quad \mbox{dans} \quad \mathcal{S} (V (\A_f) / W(\mathbf{q})^{(j)} (\A_f) ).
\end{equation}

La proposition \ref{P33} et la proposition \ref{P12} impliquent que, pour tout sous-ensemble compact $\kappa$ dans  
$$\left\{ (z , (g_f , v_f ) ) \in \C^n \times \mathcal{G} (\A_f )  \; \left| \; \begin{array}{l}\forall \xi \in V (\Q) \mbox{ tel que } \varphi_f (g_f^{-1} (\xi - v_f)) \neq 0, \\ 
z \notin  \bigcup_{j=0}^k (W(\mathbf{q})^{(j)}_\C + \xi)  \end{array} \right. \right\},$$
la forme différentielle fermée $E_\psi (\varphi_f )$ est intégrable sur 
$$\Delta (\mathbf{q}) \times \mathrm{id}_{\C^n \times \mathcal{G} (\A_f)} ( \Delta_k ' \times \kappa  ).$$
On note
\begin{equation}
E_\psi (\varphi_f , \mathbf{q} ) = (\Delta (\mathbf{q}) \times \mathrm{id} )^* E_\psi (\varphi_f);
\end{equation}
son évaluation en $(g_f , v_f) \in \mathcal{G} (\A_f )$ donne une forme fermée dans 
$$A^{2n-1} \left( \Delta_{k}' \times \left( \C^n - \bigcup_{j=0}^k \bigcup_{\xi} (W(\mathbf{q})^{(j)}_\C + \xi) \right) \right),$$
où $\xi$ parcourt l'ensemble des vecteurs de $V(\Q)$ tels que $\varphi_f (g_f^{-1} (\xi - v_f )) \neq 0$. Dans la proposition suivante, on calcule son intégrale partielle sur $\Delta_k '$. 

Supposons maintenant que $k=n$ et que $\mathbf{q}$ soit constitué de vecteurs linéairement indépendants. Soit $g \in \GL_n (\Q)$ tel que $g \cdot \mathbf{e} = \mathbf{q}$. La considération du diagramme commutatif 
$$\xymatrixcolsep{5pc}\xymatrix{
\Delta_n ' \times \C^n \times \mathcal{G} (\A_f) \ar[d]^{\mathrm{id} \times g \times \mathrm{id}} \ar[r]^{\Delta (\mathbf{e}) \times \mathrm{id}} &  \overline{X}^T \times \C^n \times \mathcal{G} (\A_f) \ar[d]^{g} \\
\Delta_n ' \times \C^n \times \mathcal{G} (\A_f) \ar[r]^{\Delta (\mathbf{q}) \times \mathrm{id}} & \overline{X}^T \times \C^n \times \mathcal{G} (\A_f) }$$
et la formule d'invariance \eqref{E:glnqinv} impliquent que  
\begin{equation} \label{E:invformsimpl0}
g^* (E_{\psi} (\varphi_f , \mathbf{q}) (gg_f , gv_f )) = E_\psi ( \varphi_f , \mathbf{e} ) (g_f , v_f ).
\end{equation}
Quitte à remplacer $\varphi_f$ par $\omega (g_f , v_f ) \varphi_f$, les calculs explicites se ramènent au cas où $(g_f , v_f) = (1,0)$. On ne considère dans la suite de ce paragraphe que les formes 
$$E_\psi (\varphi_f , \mathbf{q} )^0 = E_\psi (\varphi_f , \mathbf{q} ) (1,0)$$ 
qui satisfont
\begin{equation} \label{E:invformsimpl}
g^* E_{\psi} ( \omega (g, 0) \varphi_f , \mathbf{q})^0 = E_\psi ( \varphi_f , \mathbf{e} )^0.
\end{equation}

\begin{proposition} \label{P34bis}
Soit $\varphi_f \in \mathcal{S} (V (\A_f))$ une fonction vérifiant \eqref{E:condphi}.

{\rm 1.} Si $\langle q_0 , \ldots , q_{k} \rangle$ est un sous-espace propre de $V$, alors la forme $\int_{\Delta_{k}'} E_\psi (\varphi_f , \mathbf{q} )^0$ est identiquement nulle.

{\rm 2.} Supposons $k=n-1$ et que les vecteurs $q_0 , \ldots , q_{n-1}$ soient linéairement indépendants. On pose $g = (q_0 | \cdots | q_{n-1} ) \in \GL_n (\Q)$ et on fixe $\lambda \in \Q^\times$ tel que la matrice $h=\lambda g$ envoie $V(\Z) = \Z^n$ dans $L_{\varphi_f}$. Alors la $n$-forme $\int_{\Delta_{n-1}'} E_\psi (\varphi_f , \mathbf{q} )^0$ est égale à 
\begin{multline*}
\sum_{v \in \Q^n / L_{\varphi_f} }  \varphi_f (v )  \sum_{\substack{\xi \in \Q^n/\Z^n \\ h \xi  = v \ (\mathrm{mod} \ L_{\varphi_f})}} \mathrm{Re} \left( \varepsilon (\ell_1 - \xi_1 ) d\ell_1 \right) \wedge \cdots \wedge \mathrm{Re} \left( \varepsilon (\ell_{n} - \xi_n)  d \ell_{n} \right),
\end{multline*}
où $\ell_j$ est la forme linéaire sur $\C^n$, de noyau $W(\mathbf{q})^{(j-1)}$, telle que $h^* \ell_j = e_j^*$.  
\end{proposition}

\medskip
\noindent
{\it Remarque.} Le fait que l'expression soit en fait indépendante du choix de $\lambda$ découle des relations de distribution 
$$\sum_{j=0}^{m -1} \cot (\pi (z+ j /m )) = m \cot (\pi m z ).$$

\medskip

\begin{proof} 1.  Dans ce cas l'image de  $\Delta (\mathbf{q})$ est contenue dans le bord de $\overline{X}^T$ et il résulte de l'hypothèse faite sur $\varphi_f$, de la proposition  \ref{P33} et de la proposition \ref{P12} que l'intégrale $\int_{\Delta_{k}'} E_\psi (\varphi_f , \mathbf{q} )^0$ est nulle sur tout ouvert relativement compact de 
$$\C^n - \bigcup_{j=0}^k \bigcup_{\xi} (W(\mathbf{q})^{(j)}_\C + \xi).$$ 

2. Supposons donc $k=n-1$ et que les vecteurs $q_0 , \ldots , q_{n-1}$ soient linéairement indépendants.

En notant toujours $g = (q_0 | \cdots | q_{n-1} ) \in \GL_n (\Q)$ l'élément (\ref{E:g}), il découle de \eqref{E:invformsimpl} que 
\begin{equation*}
\int_{\Delta_{n-1}'} E_\psi (\varphi_f , \mathbf{q} )^0  = (h^{-1})^* \left(  \int_{\Delta_{n-1}'} E_\psi ( \varphi_f (h \cdot) , \mathbf{e})^0 \right),
\end{equation*}
où $\mathbf{e} = (e_1 , \ldots , e_n )$ et $h=\lambda g$.

On est donc ramené à 
calculer l'intégrale 
$$\int_{A\SO_n \R_{>0}} E_\psi (\varphi_f (h \cdot )) (1,0) ,$$ 
où 
$$A=\{ \mathrm{diag}(t_1,\ldots,t_n ) \in \SL_n (\R) \; : \; t_j \in \R_{>0},  \ t_1 \cdots t_n = 1 \}.$$
D'après la remarque à la fin du paragraphe \ref{S:42}, on a 
\begin{multline} \label{E:intti}
\int_{A\SO_n \R_{>0}} E_\psi (\varphi_f (h \cdot )) (1,0) \\ = \int_{\{ \mathrm{diag}(t_1,\ldots,t_n) \; : \; t_j \in \R_{>0} \} \SO_n} (t_1 \cdots t_n )^{s} \theta_\varphi (1 , 0 ;\varphi_f (h \cdot))  \ \Big|_{s=0}.
\end{multline}
Or, en restriction à l'ensemble des matrices symétriques diagonales réelles, le fibré en $\C^n$ se scinde {\it métriquement} en une somme directe de $n$ fibrés en droites, correspondant aux coordonnées $(z_j )_{j=1 , \ldots , n}$ de $z$ et la forme $\varphi$ se décompose en le produit de $n$ formes associées à ces fibrés en droites et égales, d'après (\ref{E:phiN1}), à
$$
\varphi^{(j)}  = \frac{i}{2\pi} e^{- t_j^2 |z_j |^2 }  \left( t_j^2 dz_j \wedge d \overline{z}_j 
- t_j^2 (z_j d\overline{z}_j - \overline{z}_j dz_j )  \wedge \frac{dt_j}{t_j}  \right), \quad j \in \{ 1 , \ldots , n \}.
$$
Comme $(1 , \xi )^* \widetilde{\varphi} (\xi ) = \varphi$ on obtient que 
\begin{multline*}
\int_{\{ \mathrm{diag}(t_1,\ldots,t_n) \; : \; t_j \in \R_{>0} \} \SO_n} (t_1 \cdots t_n )^s \widetilde{\varphi} (\xi) \\ = \frac{(-i)^n}{(4\pi)^n} \Gamma (1+ \frac{s}{2})^n \wedge_{j=1}^n  \left(  \frac{dz_j}{(z_j - \xi_j) | z_j - \xi_j |^{s} } -\frac{d\overline{z}_j}{\overline{z_j - \xi_j} | z_j - \xi_j |^{s}} \right). 
\end{multline*}
L'intégrale \eqref{E:intti} est donc égale à la valeur en $s=0$ de 
\begin{equation*}
\Gamma (1+ \frac{s}{2})^n \sum_{\xi \in V(\Q)} \varphi_f (h \xi )  \wedge_{j=1}^n  \mathrm{Re} \left( \frac{1}{2i\pi} \frac{dz_j}{(z_j - \xi_j) | z_j - \xi_j |^{s} } \right)   
\end{equation*}
c'est-à-dire,
\begin{equation*}
\Gamma (1+ \frac{s}{2})^n \sum_{\xi \in V(\Q) / V (\Z)} \varphi_f (h \xi)  \wedge_{j=1}^n  \mathrm{Re} \left( \frac{1}{2i\pi}  \sum_{m \in  \Z}  \frac{dz_j}{(z_j - \xi_j +m) | z_j - \xi_j +m |^{s} }  \right) ,
\end{equation*}
où l'on a utilisé que la fonction $\varphi_f ( h \cdot)$ est $V(\Z )$-invariante. Rappelons maintenant que pour tout $z \in \C$, la fonction 
$$s \mapsto \frac{1}{2i\pi} \sideset{}{'} \sum_{m \in \Z} \frac{1}{(z+m) |z +m|^{s}}$$ 
admet un prolongement méromorphe au plan des $s \in \C$, qui est égal à $\varepsilon (z) = \frac{1}{2i}\cot ( \pi z)$ en $s=0$; cf. \cite[VII, \S 8, p.56]{Weil}. On en déduit que l'intégrale \eqref{E:intti} 
n'est autre que 
$$\sum_{\xi \in V(\Q) / V (\Z)} \varphi_f (h \xi)  \wedge_{j=1}^n  \mathrm{Re} \left(\varepsilon (z_j - \xi_j ) dz_j \right).$$
Finalement, comme $(h^{-1})^*  e_{j}^*$ est égale à la forme linéaire $\ell_j$, on  conclut que 
\begin{multline*}
\begin{split}
\int_{\Delta_{n-1}'} E_\psi (\varphi_f , \mathbf{q} )^0  
& = \sum_{\xi \in V(\Q) / V (\Z)} \varphi_f (h \xi) \cdot  (h^{-1})^* \left(  \wedge_{j=1}^n  \mathrm{Re} (\varepsilon (z_j - \xi_j ) dz_j )  \right)  \\
& =  \sum_{\xi \in V(\Q) / V (\Z)} \varphi_f (h \xi)  \wedge_{j=1}^{n}  \mathrm{Re} (\varepsilon ( \ell_j - \xi_{j} ) d\ell_j ) .
\end{split}
\end{multline*}
En rassemblant les vecteurs $\xi$ envoyés par $h$ sur un même vecteur modulo $L_{\varphi_f}$ on obtient la formule annoncée. 
\end{proof}

\chapter{Cocycle multiplicatif du groupe rationnel $\GL_n (\Q)^+$} \label{S:chap9}

\resettheoremcounters

Au chapitre précédent on a défini une forme $E_\psi (\varphi_f)^0$ représentant une classe de cohomologie équivariante \eqref{E:721} qui, d'après le théorème \ref{T:cocycleM}, induit une classe 
\begin{equation*}
S_{\rm mult} [D_{\varphi_f}^0] \in H^{n-1} (\Gamma, \Omega^n_{\rm mer} ( \C^n ) ) .
\end{equation*}
Dans ce chapitre, on détermine explicitement des cocycles qui représentent ces classes de cohomologie. On suit la même démarche qu'au chapitre  \ref{S:6} mais en considérant les séries d'Eisenstein $E_\psi (\varphi_f)$ plutôt que la forme $\eta$. Autrement dit plutôt que la distribution ``évaluation en zéro'' on considère cette fois la distribution theta. On travaille uniquement avec la bordification de Tits car il faut ici prendre garde au fait que la forme $E(\varphi_f)$ ne s'étend pas à tout le bord. 

\section{Forme simpliciale associée à $E_\psi$} \label{S8.1}

L'application (\ref{E:R}) induit une rétraction 
$$[0,1] \times X \times \C^n  \to X \times \C^n ,$$
encore notée $R$, de $X \times \C^n$ sur $\{ x_0 \} \times \C^n$. 

Soit $\Gamma$ un sous-groupe de congruence dans $\SL_n (\Z)$. Il lui correspond un sous-groupe compact ouvert $K$ dans $\GL_n (\mathbf{A}_f )$ tel que
$$\Gamma = K \cap \GL_n (\Q)^+ .$$
Dans ce chapitre on prend $L = \Z^n$; c'est un réseau $\Gamma$-invariant dans $V(\Q)$.

De la même manière qu'au paragraphe \ref{S:61}, la rétraction $R$ induit une suite d'applications
\begin{equation} 
\rho_k : \Delta_k \times E_k\Gamma  \times \C^n/ \Z^n  \longrightarrow X \times \C^n/ \Z^n .
\end{equation}
\'Etant donné un $(k+1)$-uplet 
$$\mathbf{g} = (\gamma_0 , \ldots , \gamma_k ) \in E_k\Gamma$$ 
et un élément $z \in \C^n/ \Z^n$, l'application $\rho_k ( \cdot , \mathbf{g} , z)$ envoie le simplexe  $\Delta_k$ sur le simplexe géodésique dans $X$ de sommets $\gamma_0^{-1} x_0$, $\ldots$, $\gamma_k^{-1} x_0$ défini, par récurrence, en prenant le cône géodésique sur le $(k-1)$-simplexe géodésique de sommets $\gamma_1^{-1} x_0 , \ldots , \gamma_k^{-1} x_0$ depuis $\gamma_0^{-1} x_0$.  

La suite $\rho=(\rho_k )$ est constituée d'applications $\Gamma$-équivariantes et induit donc une application 
$\Gamma$-équivariante
$$\rho^* : A^\bullet (X \times \C^n / \Z^n ) \to \mathrm{A}^\bullet (E\Gamma \times  \C^n /\Z^n ),$$
où l'espace de droite est celui des formes différentielles simpliciales sur la variété simpliciale 
\begin{equation} \label{VS1}
E\Gamma \times  \C^n / \Z^n .
\end{equation}
Le groupe $\Gamma$ opère (diagonalement) sur \eqref{VS1} par 
$$\left(  \mathbf{g} , z  \right) \stackrel{(h,w)}{\longmapsto} \left(  \mathbf{g} h^{-1}, h z  \right).$$

Un élément $D \in \mathrm{Div}_\Gamma$ peut-être vu comme une fonction $\Gamma$-invariante sur $\C^n / \Z^n$ à support dans les points de torsion; il lui correspond donc une fonction $\mathcal{K}$-invariante $\varphi_f \in \mathcal{S} (V (\A_f ))$. On a 
\begin{equation} \label{E:fibresj}
\C^n / \Z^n - \mathrm{supp} \ D  = \C^n/\Z^n - \bigcup_\xi (\xi + \Z^n )/ \Z^n , 
\end{equation}
où $\xi$ parcourt l'ensemble des éléments de $\Q^n / \Z^n$ tels que $\varphi_f (\xi)$ soit non nul.

La proposition suivante découle des définitions.

\begin{proposition} \label{P:50}
La forme simpliciale
$$\mathcal{E}_\psi (\varphi_f ) := \rho^* E_\psi (\varphi_f )^0 \in \mathrm{A}^{2n-1} \left( E \Gamma \times ( \C^n / \Z^n - \mathrm{supp} \ D ) \right)^\Gamma$$
est fermée et représente la classe de cohomologie équivariante $[E_{\psi} (\varphi_f )^{(0)} ]$. 
\end{proposition}

L'ouvert \eqref{E:fibresj} étant de dimension cohomologique $n$, il correspond à la classe de cohomologie équivariante $[E_{\psi} (\varphi_f )^{(0)} ]$ une classe de cohomologie dans 
$$H^{n-1} ( \Gamma , H^n (( \C^n / \Z^n - \mathrm{supp} \ D )^{(1)}).$$
Il découle même de \cite[Theorem 2.3]{Dupont} que l'intégration sur les $(n-1)$-simplexes associe à la forme simpliciale $\mathcal{E}_\psi (\varphi_f )$ un $(n-1)$-cocycle qui à un $n$-uplet d'éléments de $\Gamma$ 
associe une $n$-forme fermée sur \eqref{E:fibresj}. Pour obtenir des cocycles à valeurs dans les formes méromorphes on procède comme dans la démonstration du théorème \ref{T:cocycleM}. Il faut pour cela effacer quelques hyperplans de manière à pouvoir invoquer le théorème \ref{P:Brieskorn}. C'est ce que l'on détaille dans le paragraphe qui vient. 

\section[Les cocycles $\mathbf{S}_{{\rm mult}, \chi_0}$]{Les cocycles $\mathbf{S}_{{\rm mult}, \chi_0}$, démonstration du théorème \ref{T:mult}} \label{S:3.2.2}

Fixons un morphisme primitif $\chi_0 : \C^n /\Z^n \to \C / \Z$. 

Pour tout $(k+1)$-uplet $\mathbf{g} = (\gamma_0 , \ldots , \gamma_k )$ d'éléments de $\Gamma$, on note 
\begin{equation} \label{Vchi0S2}
U(\mathbf{g}) = \C^n / \Z^n -  \bigcup_{\xi } \bigcup_{i=0}^k  (\xi + \mathrm{ker} (\chi_0 \circ \gamma_i) ),
\end{equation}
où $\xi$ décrit les éléments du support de $D$. 
Les ouverts $U(\mathbf{g})$ sont des variétés affines. On peut donc leur appliquer le théorème \ref{P:Brieskorn}. 

On obtient un cocycle
\begin{equation*}
\mathbf{S}_{{\rm mult}, \chi_0} [D] :   \Gamma^n  \longrightarrow   \Omega_{\rm mer}^n (\C^n / \Z^n )
\end{equation*}
qui représente la classe $S_{\rm mult} [D]$ et est à valeurs dans les formes méromorphes qui sont régulières en dehors des hyperplans  affines $\xi + \mathrm{ker} (\chi_0  \circ \gamma )$, avec $\xi \in \mathrm{supp} \ D$ et $\gamma \in \Gamma$. Ce sont les cocycles annoncés dans le théorème \ref{T:mult}; il nous faut encore vérifier les propriétés attendues sous l'action des opérateurs de Hecke.

Considérons donc un sous-monoïde $S$ de $M_n (\Z)^\circ$ contenant $\Gamma$. À toute double classe $\Gamma a \Gamma$, avec $a \in S$, il correspond la fonction caractéristique de $KaK$, elle appartient à $\mathcal{H} (\GL_n (\A_f) , K)$. 

L'application $\GL_n (\A_f) \to \mathcal{G} (\A_f )$ qui à un élément $g_f$ associe $(g_f , 0)$ induit un plongement 
$$K \backslash \GL_n (\A_f)  / K \hookrightarrow \mathcal{K} \backslash \mathcal{G} (\A_f ) / \mathcal{K}$$
et donc une inclusion 
$$\mathcal{H} (\GL_n (\A_f) , K) \hookrightarrow \mathcal{H} (\mathcal{G} (\A_f ) , \mathcal{K})$$
``extension par $0$''. Dans la suite on identifie une fonction $\phi$ dans $\mathcal{H} (\GL_n (\A_f) , K)$ à son image dans $\mathcal{H} (\mathcal{G} (\A_f ) , \mathcal{K})$. 

Il découle alors du \S \ref{algHecke} que $\phi$ induit un opérateur de Hecke $\mathbf{T}_\phi$ sur 
$$H^{n-1} (\Gamma , \Omega_{\rm mer}^n (\C^n / \Z^n ));$$ 
lorsque $\phi$ est la fonction caractéristique de $KaK$, l'opérateur $\mathbf{T}_\phi$ coïncide avec $\mathbf{T}(a)$ du \S~\ref{S:2-2}. 

Une fonction $\phi  \in \mathcal{H} (\GL_n (\A_f) , K)$ induit aussi un opérateur 
\begin{equation} \label{E:applphi}
T_\phi  : \mathcal{S} (V(\mathbf{A}_f ))^K \to \mathcal{S} (V(\mathbf{A}_f ))^K
\end{equation}
sur l'espace $\mathcal{S} (V(\mathbf{A}_f ))^K$ des fonctions de Schwartz $K$-invariantes. On a 
\begin{equation*}
T_\phi ( \varphi_f) =  \sum_{g \in \Gamma \backslash \GL_n (\Q) / \Gamma} \phi (g) \sum_{h \in \Gamma g \Gamma / \Gamma} \varphi_f (h^{-1} \cdot ).
\end{equation*}

La fonction $\phi$ induit finalement une application
\begin{equation} \label{E:applphi}
[\phi ]  : \mathrm{Div}_\Gamma \to \mathrm{Div}_\Gamma  \quad \mbox{avec} \quad [\phi] =   \sum_{g \in \Gamma \backslash \GL_n (\Q) / \Gamma} \phi (g) \sum_{h \in \Gamma g \Gamma / \Gamma}  h .
\end{equation}
Lorsque $\phi$ est la fonction caractéristique de $KaK$, on a $[\phi]=[\Gamma a \Gamma]$ (cf. \S~\ref{S:2-2}). 

Il résulte des définitions que  
\begin{equation} \label{E:DetT}
D_{T_{\phi} \varphi_f} = [\phi]^* D
\end{equation}
et la proposition \ref{P:hecke1} implique~:

\begin{proposition} \label{P:hecke2}
Soit 
$\phi \in \mathcal{H} (\GL_n (\A_f) , K)$. On a 
\begin{equation}
\mathbf{T}_\phi \left[ \mathbf{S}_{{\rm mult}, \chi_0}[D] \right]= \left[ \mathbf{S}_{{\rm mult}, \chi_0 }[[\phi]^* D] \right] .
\end{equation}
\end{proposition}

Ceci conclut la démonstration du théorème \ref{T:mult}.

\section{Le cocycle $\mathbf{S}^*_{{\rm mult}}$}

Comme dans le cas additif, on peut obtenir un cocycle explicite en appliquant l'argument de l'annexe \ref{A:A} à la forme simpliciale $\mathcal{E}_\psi (\varphi_f )$ et aux ouverts, indexés par les éléments $\mathbf{g} = (\gamma_0 , \ldots , \gamma_k) $ de $E_k \Gamma$ (avec $k \in \mathbf{N}$),
\begin{equation} \label{E:86} 
U^* (\mathbf{g}) =  \C^n /\Z^n -  \bigcup_\xi \bigcup_{i=0}^k   (W(\mathbf{q})_{\C}^{(i)} + \xi +\Z^n ) / \Z^n ,
\end{equation}
où $\mathbf{q} = ( \gamma_0^{-1} e_1 , \ldots  , \gamma_{k}^{-1} e_1 )$ et $\xi$ parcourt l'ensemble des éléments de $\Q^n / \Z^n$ tels que $\varphi_f (\xi)$ soit non nul.

\subsection{Section simpliciale et homotopie}

Au paragraphe \ref{S:63}, on a défini des applications $\varrho_k$ que l'on restreint maintenant à $E\Gamma$. Pour tout entier $k \in [0 , n-1]$ on note encore 
\begin{equation}
\varrho_k : \Delta _k \times [0,1] \times E_k \Gamma \times \C^n  \to \overline{X}^T \times \C^n / \Z^n
\end{equation}
les applications induites. Les images des projections dans $\overline{X}^T$ sont cette fois contenues dans la bordification rationnelle de Tits obtenue en n'ajoutant que les sous-groupes paraboliques $Q=Q (W_\bullet )$ associés à des drapeaux de sous-espaces rationnels de $V(\Q)= \Q^n$ engendrés par des vecteurs $q=\gamma^{-1} e_1$ avec $\gamma \in \Gamma$. 

La proposition \ref{P33} implique que pour tout entier $k \in [0, n-1]$, pour tout $\mathbf{g} \in E_k \Gamma$, pour tout $z \in U^* (\mathbf{g})$ et pour tout réel strictement positif $t$, l'image  
$$\varrho_k ( \Delta_k \times [0,1] \times \{ \mathbf{g} \} \times \{z \}  ) \subset \overline{X}^T \times \C^n / \Z^n $$
est contenue dans une réunion finie
$$\left( \Omega \cup_{W_\bullet , h , \omega} \overline{\mathfrak{S}_{W_\bullet } (h , t , \omega )} \right) \times \{z \} ,$$
où $\Omega \subset X$ est relativement compact et chaque drapeau $W_\bullet$ est formé de sous-espaces engendrés par certains des vecteurs $\gamma_i^{-1} e_1$ où $\mathbf{g} = (\gamma_0 , \ldots , \gamma_k)$.

La proposition \ref{P12} motive la définition suivante.

\begin{definition} \label{def67}
Soit $\mathcal{S} (V (\A_f))^{\circ}$ le sous-espace de $\mathcal{S} (V(\A_f))$ constitué des fonctions $\varphi_f$ telles que pour tout sous-espace $W$ contenant $e_1$, l'image $\int_W  \varphi_f $ de $\varphi_f$ dans $\mathcal{S} (V (\A_f) / W(\A_f))$ est constante égale à $0$.
\end{definition}

\medskip
\noindent
{\it Remarque.} On réserve la notation $\mathcal{S} (V (\A_f))^{0}$ pour l'espace des fonctions $\varphi_f$ dans $\mathcal{S} (V(\A_f))$ telles que $\widehat{\varphi}_f (0) =0$. Noter que puisque $V$ contient $e_1$ on a 
$$\mathcal{S} (V (\A_f))^{\circ} \subset \mathcal{S} (V (\A_f))^{0}.$$

\medskip

La forme $E_\psi (\varphi_f )^{(0)}$ est $(\Gamma \ltimes \Z^n)$-invariante et la proposition \ref{P12} implique que si $\varphi_f \in \mathcal{S} (V (\A_f ))^{\circ}$ alors pour tout $\mathbf{g} \in E_k \Gamma$ la restriction de la forme différentielle fermée 
$\varrho_k^* E_\psi (\varphi_f )^{(0)}$ à $\Delta_k \times [0,1] \times \mathbf{g} \times U^* ( \mathbf{g})$ est bien définie et fermée. 

\begin{definition}
Supposons $\varphi_f \in \mathcal{S} (V (\A_f ))^{\circ}$. Pour tout entier $k \in [0, n-1]$ et pour tout $(k+1)$-uplet $\mathbf{g} = (\gamma_0 , \ldots , \gamma_k) \in E_k \Gamma$, on pose 
$$\mathcal{H}_k[\varphi_f] (\gamma_0 , \ldots , \gamma_k ) = \int_{\Delta_k \times [0,1]} \varrho_k^*E_\psi (\varphi_f )^{(0)} (\gamma_0 , \ldots , \gamma_k ).$$
C'est une forme différentielle de degré $2n-2-k$ sur $U^* (\mathbf{g})$. 
\end{definition}

\subsection{Calcul du cocycle} 

En remplaçant l'appel à la proposition \ref{P32} par l'utilisation de la proposition  \ref{P12}, la démonstration du théorème \ref{T37} conduit au résultat suivant. On renvoie à l'annexe \ref{A:A} pour les définitions des opérateurs $\delta$ et $d$. 

\begin{theorem} \label{T8.8}
Supposons $\varphi_f \in \mathcal{S} (V(\A_f))^{\circ}$. Pour tout entier $k \in [0 , n-1]$ et pour tout $\mathbf{g} = (\gamma_0 , \ldots , \gamma_k) \in E_k \Gamma$, 
l'intégrale $\int_{\Delta_k } \mathcal{E}_{\psi} (\varphi_f ) (\mathbf{g})$ est égale à 
$$\delta \mathcal{H}_{k-1} [\varphi_f ]  (\mathbf{g}) \pm d \mathcal{H}_k [\varphi_f ]  (\mathbf{g}), \quad \mbox{si } k < n-1,$$
et 
$$\int_{\Delta_{n-1}'} E_\psi (\varphi_f , \mathbf{q} )^{(0)}
+ \delta  \mathcal{H}_{n-2} [\varphi_f ]  (\mathbf{g})  \pm d \mathcal{H}_{n-1} [\varphi_f ]  (\mathbf{g}), \quad \mbox{si } k=n-1,$$
dans $A^{2n-2-k} \left( U^* (\mathbf{g}) ) \right)$ avec toujours $\mathbf{q} = ( \gamma_0^{-1} e_1 , \ldots  , \gamma_{k}^{-1} e_1 )$.
\end{theorem}
\begin{proof} Puisque $E_\psi (\varphi_f)^{(0)}$ est fermée on a~:
$$(d_{\Delta_k \times [0,1]} \pm d ) \varrho_k^*E_\psi (\varphi_f )^{(0)} = 0$$
et donc 
$$\int_{\Delta_k \times [0,1]} d_{\Delta_k \times [0,1]} \varrho_k^* E_\psi (\varphi_f )^{(0)} (\gamma_0 , \ldots , \gamma_k ) \pm d \mathcal{H}_k (\gamma_0 , \ldots , \gamma_k ) =0.$$
Maintenant, d'après le théorème de Stokes on a
\begin{multline*}
\int_{\Delta_k \times [0,1]} d_{\Delta_k \times [0,1]} \varrho_k^* E_\psi (\varphi_f )^{(0)} (\gamma_0 , \ldots , \gamma_k ) = \int_{\Delta_k \times \{ 0 \}} \varrho_k^* E_\psi (\varphi_f )^{(0)} (\gamma_0 , \ldots , \gamma_k ) \\ +  \int_{(\partial \Delta_k) \times [0,1]} \varrho_k^* E_\psi (\varphi_f )^{(0)} (\gamma_0 , \ldots , \gamma_k )  -  \int_{\Delta_k \times \{ 1 \}} \varrho_k^* E_\psi (\varphi_f )^{(0)} (\gamma_0 , \ldots , \gamma_k ).
\end{multline*}
La dernière intégrale est égale à $\int_{\Delta_k '} E_\psi (\varphi_f , \mathbf{q} )^{(0)} $ et est donc nulle si $k < n-1$ d'après la proposition \ref{P34bis}.

Finalement, par définition on a 
$$\int_{\Delta_k \times \{ 0 \}} \varrho_k^* E_\psi (\varphi_f )^{(0)} (\gamma_0 , \ldots , \gamma_k ) = \int_{\Delta_k} \rho^* E_\psi (\varphi_f ) (\gamma_0 , \ldots , \gamma_k) $$
et
$$\int_{(\partial \Delta_k ) \times [0,1]} \varrho_k^* E_\psi (\varphi_f )^{(0)} (\gamma_0 , \ldots , \gamma_k ) = - \delta \mathcal{H}_{k-1} (\gamma_0 , \ldots , \gamma_k ).$$
\end{proof}

Le théorème précédent motive les définitions suivantes. 

\begin{definition}
\'Etant donné une fonction $K$-invariante $\varphi_f \in \mathcal{S} (V(\A_f))^{\circ}$ on désigne par
$$\mathbf{S}_{\rm mult}^* [\varphi_f] : \Gamma^{n} \longrightarrow \Omega^{n}_{\rm mer} (\C^n / \Z^n )$$
l'application qui à un $n$-uplet $(\gamma_0 ,  \ldots , \gamma_{n-1})$ associe $0$ si les vecteurs $\gamma_j^{-1} e_1$ sont liés et sinon la forme différentielle méromorphe dans $\Omega^n_{\rm mer} (\C^n / \Z^n )$ d'expression
\begin{equation} \label{E:87} 
\sum_{v \in \Q^n / \Z^n} \varphi_f (v) \\ \sum_{\substack{\xi \in \Q^n/\Z^n \\ h \xi  = v \ (\mathrm{mod} \ \Z^n )}} \varepsilon (\ell_1 - \xi_1 ) \cdots \varepsilon (\ell_{n} - \xi_n)  \cdot d\ell_1 \wedge \cdots \wedge d \ell_{n} ,
\end{equation}
où $h = ( \gamma_0^{-1}  e_1 | \cdots | \gamma_{n-1}^{-1} e_1)$ et $h^* \ell_j = e_{j}^*$.  
\end{definition}

On déduit du théorème \ref{T8.8} le théorème suivant.

\begin{theorem} \label{T8.8cor}
Supposons $\varphi_f \in \mathcal{S} (V(\A_f))^{\circ}$. L'application 
$$\mathbf{S}_{\rm mult}^*[\varphi_f] : \Gamma^{n} \longrightarrow  \Omega^{n}_{\rm mer} (\C^n / \Z^n )$$
définit un $(n-1)$-cocycle homogène non nul du groupe $\Gamma$. Il représente la même classe de cohomologie que $\mathbf{S}_{\rm mult , \chi_0} [\varphi_f]$.   
\end{theorem}
\begin{proof} Il découle du théorème \ref{T8.8} que l'application 
$$(\gamma_0 , \ldots , \gamma_{n-1}) \mapsto \int_{\Delta_{n-1}'} E_\psi (\varphi_f , \mathbf{q} )^0$$
définit un $(n-1)$-cocycle à valeurs dans $H^n (U^* (\mathbf{g}))$. 
Or, d'après la proposition \ref{P34bis} la $n$-forme $\int_{\Delta_{n-1}'} E_\psi (\varphi_f , \mathbf{q} )^0$ est nulle si les vecteurs $q_j=\gamma_j^{-1} e_1$ sont liés et elle est égale à 
\begin{multline} \label{E:8.fd} 
\sum_{v \in \Q^n / \Z^n} \varphi_f (v) \\ \sum_{\substack{\xi \in \Q^n/\Z^n \\ h \xi  = v \ (\mathrm{mod} \ \Z^n)}} \mathrm{Re} (\varepsilon (\ell_1 - \xi_1 ) d \ell_1 ) \wedge \cdots \wedge \mathrm{Re} (\varepsilon (\ell_{n} - \xi_n) d \ell_{n})  ,
\end{multline}
sinon.

Remarquons maintenant qu'en posant $q=e^{2i\pi z}$, on a
\begin{equation} \label{E:8.cot}
\varepsilon (z) dz = \frac{1}{2i} \cot (\pi z) dz = \frac{dz}{e^{2i\pi z} -1} +  \frac12 dz = \frac{1}{2i\pi} \left( \frac{dq}{q-1} - \frac{dq}{2q} \right).
\end{equation}
En particulier les $1$-formes différentielles 
$$\mathrm{Re} (\varepsilon (z) dz ) \quad \mbox{ et } \quad \varepsilon (z) dz$$
sur $\C / \Z$ sont cohomologues. On en déduit que la forme \eqref{E:8.fd} est cohomologue à  $\mathbf{S}_{\rm mult}^* [\varphi_f] (\gamma_0 , \ldots , \gamma_{n-1})$. On conclut alors la démonstration, en suivant celle du théorème \ref{T:Sa} au paragraphe \ref{S:demTSa}, mais en remplaçant le théorème de Brieskorn par sa version multiplicative, le théorème \ref{P:Brieskorn}. Celui-ci s'applique car les relations de distribution pour $\varepsilon$ impliquent que la forme \eqref{E:8.fd} représente bien une classe dans le sous-espace caractéristique associé à la valeur propre $1$ dans la cohomologie de la fibre.
\end{proof}

La proposition suivante donne une autre expression, parfois plus maniable, du cocycle $\mathbf{S}_{\rm mult}^*$.

\begin{proposition} 
Supposons toujours $\varphi_f \in \mathcal{S} (V (\A_f ))^{\circ}$. L'expression \eqref{E:87} est égale à 
$$\sum_{v \in \Q^n / \Z^n} \varphi_f (v) \sum_{\substack{\xi \in \Q^n/\Z^n \\ h \xi  = v \ (\mathrm{mod} \ \Z^n)}} \frac{d\ell_1 \wedge \cdots \wedge d \ell_{n}}{(e^{2i\pi (\ell_1 - \xi_1 )} -1) \ldots (e^{2i\pi (\ell_n - \xi_n )} -1)}.$$
\end{proposition}
\begin{proof} D'après \eqref{E:8.cot}, il suffit de montrer que pour tout sous-ensemble non vide et propre $J \subset \{ 1 , \ldots , n \}$ on a 
\begin{equation} \label{E:8.smoothing}
\sum_{v \in \Q^n / \Z^n} \varphi_f (v)  \sum_{\substack{\xi \in \Q^n/\Z^n \\ h \xi  = v \ (\mathrm{mod} \ \Z^n )}} \wedge_{j \notin J} \varepsilon (\ell_j - \xi_j )  =0.
\end{equation}
Posons $L = \Z^n$ et fixons un sous-ensemble $J \subset \{ 1 , \ldots , n \}$ non vide et propre de cardinal $k$. Soit $\overline{V}$ le quotient de $V$ par la droite engendrée par les vecteurs $\gamma_j^{-1} e_1$ avec $j \in J$. Désignons par $\overline{L}$ l'image de $L$ dans $\overline{V}$. 

En identifiant $\Q^{n-k}$ avec le quotient de $\Q^n$ par l'espace engendré par les vecteurs $e_j$ avec $j \in J$ et $\Z^{n-k}$ avec l'image de $\Z^n$ dans $\Q^{n-k}$, la matrice $h$ induit une application linéaire $\overline{h} : \Q^{n-k} \to \overline{V} $ telle que $\overline{h} (\Z^{n-k} )$ soit contenu dans $\overline{L}$. 

Pour tout vecteur $v \in V$ d'image $\overline{v}$ dans $\overline{V}$, la projection de $\Q^{n}$ sur $\Q^{n-k}$ induit alors une application surjective 
$$\{ \xi \in \Q^n / \Z^n \; : \; h \xi  = v \ (\mathrm{mod} \ L) \} \to \{ \overline{\xi} \in \Q^{n-k}/\Z^{n-k}  \; : \; \overline{h} \overline{\xi}  = \overline{v} \ (\mathrm{mod} \ \overline{L}) \}$$
dont les fibres ont toutes le même cardinal, égal à $\frac{[L : h (\Z^{n})]}{[\overline{L} : \overline{h} (\Z^{n-k})]}$. Le membre de gauche de \eqref{E:8.smoothing} est donc égal à 
\begin{equation*} 
\sum_{w\in \overline{V} / \overline{L}} \Big( \sum_{\substack{v \in V/L \\ \overline{v}=w}}  \varphi_f (v) \Big) 
\frac{[L : h (\Z^{n})]}{[\overline{L} : \overline{h} (\Z^{n-k})]} \sum_{\substack{\overline{\xi} \in \Q^{n-k}/\Z^{n-k} \\ \overline{h} \overline{\xi}  = w \ (\mathrm{mod} \ \overline{L})}} \wedge_{j \notin J} \varepsilon (\ell_j - \xi_j ),
\end{equation*}
où $\overline{\xi} = (\xi_j )_{j \notin J}$. 

Pour conclure, remarquons que le noyau de la projection $V \to \overline{V}$ contient un vecteur qui est un translaté de $e_1$ par un élément de $\Gamma$. L'image de ce vecteur dans $V(\A_f)$ est donc égale à $e_1$ modulo $K$. Comme la fonction $K$-invariante $\varphi_f$ appartient à $\mathcal{S} (V (\A_f ))^{\circ}$, on en déduit que chaque somme 
$$\sum_{\substack{v \in V/L \\ \overline{v}=w}}  \varphi_f (v) $$ 
est égale à zéro et l'on obtient l'identité annoncée \eqref{E:8.smoothing}.
\end{proof}

\subsection{Démonstration du théorème \ref{T:mult2}} \label{S:8.3.3}

Rappelons qu'un élément $D \in \mathrm{Div}_\Gamma$ peut-être vu comme une fonction $\Gamma$-invariante sur $\Q^n / \Z^n$ à support dans un réseau de $\Q^n$. Notons $\varphi_f$ la fonction $K$-invariante correspondante. Sous l'hypothèse que $D \in \mathrm{Div}_\Gamma^{\circ}$ la fonction $\varphi_f$ appartient à $ \mathcal{S} (V (\A_f ))^{\circ}$. 

On pose alors 
$$\mathbf{S}^*_{\rm mult} [D] = \mathbf{S}^*_{{\rm mult}}[\varphi_f].$$
C'est un $(n-1)$-cocycle de $\Gamma$ à valeurs dans $\Omega^n_{\rm mer} (\C^n / \Z^n)$. Les cocycles $\mathbf{S}^*_{\rm mult} [D]$ et $\mathbf{S}_{{\rm mult} , \chi_0 }[D]$ représentent la même classe de cohomologie puisqu'ils proviennent tous les deux de la restriction de $\mathcal{E}_\psi (\varphi_f)$. Cela démontre le premier point du théorème \ref{T:mult2}. 

Le deuxième point du théorème peut se vérifier par un calcul élémentaire; il résulte aussi de la proposition \ref{P:hecke1}. \qed

\medskip

On conclut ce chapitre en remarquant que la proposition \ref{P:DRgen} découle du théorème \ref{T:mult} sauf en ce qui concerne l'intégralité de la classe $d_n [\Phi_\delta ]$. Cette dernière propriété résulte de \cite{Takagi}. Le lecteur attentif aura pourtant noté que la série d'Eisenstein étudiée dans \cite{Takagi} est de degré total $n-1$ alors que la série d'Eisenstein $E_\psi (\varphi_f)$ étudiée ici est de degré $2n-1$. On passe de la deuxième à la première en remarquant que 
le fibré $\C^n$ se scinde métriquement en $\R^n \oplus (i \R^n)$ au-dessus de l'espace symétrique associé à $\SL_n (\R)$. La forme de Mathai--Quillen $\varphi$ associée au fibré $\C^n$ se décompose alors en le produit $\varphi_{\R^n} \wedge \varphi_{i\R^n}$ de deux formes égales à la forme de Thom (de degré $n$) de \cite[\S 5]{Takagi}. Si l'on applique à $\varphi_{\R^n}$ la distribution theta, qu'on évalue $\varphi_{i\R^n}$ en $0$ et que l'on contracte la forme obtenue à l'aide du multivecteur $\partial_{y_1} \wedge \ldots \wedge \partial_{y_n}$, on obtient la série d'Eisenstein étudiée dans \cite{Takagi}. La propriété d'intégralité annoncée résulte alors de \cite[\S 10 Remark p. 354]{Takagi}.

\medskip
\noindent
{\it Remarque.} L'énoncé de \cite[\S 10 Remark p. 354]{Takagi} comporte malheureusement une erreur. En adoptant les notations de \cite{Takagi}, nous affirmions que pour tout entier strictement positif $m$, la classe $(m^n -1) z(\mathbf{v})$ est $\Z_\ell$-entière pour tout $\ell$ premier à $m$. Toutefois la démonstration requiert que la multiplication par $m$ dans les fibres fixe la section de torsion $\mathbf{v}$, autrement que si $\mathbf{v}$ est associée à un vecteur $v$ dans $\Q^n$ alors $mv=v$ dans $\Q^n / \Z^n$. On ne peut donc pas montrer que $d_n z(\mathbf{v})$ est une classe entière comme annoncé dans la remarque. L'énoncé est d'ailleurs faux même lorsque $n=2$, comme le montre par exemple \cite[\'Equation (11.3)]{Takagi}. 

D'un autre côté, la classe $[\Phi_\delta ]$, considérée ici, est obtenue en évaluant le cocycle $\mathbf{S}_{{\rm mult}, e_1^*} [D_\delta ]$ en $0$. La section nulle étant  invariante par multiplication par \emph{tous} les entiers strictement positifs $m$, la démonstration de \cite[\S 10 Remark p. 354]{Takagi} implique bien que  
la classe $d_n [\Phi_\delta ]$ est entière.

\medskip

\chapter{Cocycle elliptique du groupe rationnel $\GL_n (\Q)^+$} \label{S:chap10}

\resettheoremcounters

Dans ce chapitre on explique comment modifier les constructions des deux chapitres précédents pour construire les cocycles  elliptiques des théorèmes \ref{T:ell} et~\ref{T:ellbis}. On commence par détailler le quotient adélique avec lequel il nous faut cette fois travailler. On explique ensuite quelles sont les principales différences avec le cas multiplicatif. Elles sont au nombre de trois~: 
\begin{enumerate}
\item De la même manière que le cycle $D$ doit être supposé de degré $0$ dans le théorème \ref{T:cocycleE}, on doit supposer $\widehat{\varphi}_f (0)=0$ pour construire le cocycle $\mathbf{S}_{{\rm ell}, \chi}$; voir Théorème \ref{T:9.4}.
\item Pour pouvoir appliquer le théorème \ref{P:Brieskorn} ``à la Orlik--Solomon'', on a besoin d'effacer suffisamment d'hyperplans pour que les fibres soient affines. Cela nous contraint à considérer un $n$-uplet $\chi = (\chi_1 , \ldots , \chi_n )$ de morphismes primitifs;  voir \S \ref{S:9.3.3}.
\item Le fait que la série d'Eisenstein $E_1$ ne soit pas périodique nous contraint à considérer la série d'Eisenstein non holomorphe $E_1^*$ lors de la construction de $\mathbf{S}_{{\rm ell}}^*$; voir \S \ref{S:9.4.1}. 
\end{enumerate}

\section{Quotients adéliques}

Les quotients adéliques que l'on considère dans ce chapitre sont associés au groupe (algébrique sur $\Q$)
$$\mathcal{G}  = (\GL_n   \times \SL_2 ) \ltimes M_{n, 2} .$$
\`A l'infini l'espace est 
$$S \times \mathcal{H} \times \C^n \cong [(\GL_n (\R) \times \SL_2 (\R)) \ltimes M_{n, 2} (\R)] / \SO_n \times \SO_2,$$
où $\mathcal{H} = \SL_2 (\R) / \SO_2$
est le demi-plan de Poincaré. L'action de $\mathcal{G} (\R)$ sur $S \times \mathcal{H} \times \C^n$ se déduit des actions suivantes~: 
\begin{itemize}
\item le groupe $\SL_2 (\R)$ opère par
$$(gK, \tau ,  z) \stackrel{B}{\longmapsto} \left( gK , \frac{a\tau+b}{c\tau+d} , \frac{z}{c \tau +d}  \right), \quad \mbox{où } B=\left( \begin{smallmatrix} a & b \\ c & d \end{smallmatrix} \right) \in \SL_2 (\R ),$$
\item le groupe $\GL_n (\R)$ opère par\footnote{Ici le $z =(z_1 , \ldots , z_n )$ est vu comme un vecteur colonne, l'action $z \mapsto hz$ correspond donc à l'action usuelle sur les vecteur colonnes. On pourrait aussi bien considérer l'action $z \mapsto h^{-\top}z$ mais celle-ci parait plus artificielle, une alternative similaire se présente dans l'étude de la représentation de Weil.}
$$(gK , \tau , z) \stackrel{h}{\longmapsto} (hgK, \tau , hz) \quad \left(h \in \GL_n (\R) \right),$$
et 
\item une matrice
$$\left( \begin{array}{cc}
u_1 & v_1 \\ 
\vdots & \vdots \\
u_n & v_n 
\end{array} \right) \in M_{n, 2} (\R)$$ 
opère par translation 
$$(gK , \tau ,  (z_1 , \ldots , z_n)) \mapsto (gK , \tau ,  (z_1 + u_1 \tau + v_1 , \ldots , z_n +u_n \tau + v_n)).$$
\end{itemize}
Noter que la loi de groupe dans le produit semi-direct $$(\GL_n (\R) \times \SL_2 (\R)) \ltimes M_{n, 2} (\R)$$ est donnée par 
$$(g,B , x ) \cdot (g' , B' , x') = (gg' , BB' ,gx' B^{-1} + x ).$$

L'action décrite ci-dessus munit le fibré
\begin{equation} \label{E:9.fibre2}
\xymatrix{
S \times \mathcal{H} \times \C^n \ar[d]   \\
S \times  \mathcal{H} 
}
\end{equation}
d'une structure de fibré $\mathcal{G} (\R)$-équivariant. 

Les structures rationnelles usuelles sur $\GL_n$, $\SL_2$ et $M_{n,2}$ munissent $\mathcal{G}$
d'une structure de groupe algébrique sur $\Q$. Dans la suite de ce chapitre on identifie $M_{n,2}$ au produit $V^2$ avec $V (\Q) = \Q^n$ et on fixe~: 
\begin{itemize}
\item un sous-groupe compact ouvert $K \subset \GL_n (\widehat{\Z})$, et
\item un sous-groupe compact ouvert $L \subset \SL_2 (\widehat{\Z})$.
\end{itemize}
Il correspond à ces données le sous-groupe compact ouvert
$$\mathcal{K} = (K \times L) \ltimes V(\widehat{\Z})^2 \subset \mathcal{G} (\A_f)$$ 
qui préserve le réseau $V(\widehat{\Z})^2$ dans $V (\A_f )^2$. 
Les quotients
$$[\mathcal{G} ] / (L \ltimes V(\widehat{\Z})^2) = \mathcal{G} (\Q ) \backslash  \left[ \mathcal{G} (\R) \cdot \mathcal{G} (\A_f ) \right] / (\SO_n \times \SO_2) Z (\R)^+ (L \ltimes V(\widehat{\Z})^2)$$
et
$$[\mathcal{G} ] / \mathcal{K}$$
sont des fibrés en groupes au-dessus de respectivement $[\GL_n]$ et $[\GL_n] / K$. Les fibres sont des produits (fibrés) $E^n$ de (familles de) courbes elliptiques 
$$E (= E_{L}) = \Lambda \backslash \left[ (\mathcal{H} \times \C) / \Z^2 \right] \quad \mbox{où} \quad \Lambda = L \cap \SL_2 (\Z).$$ 
On notera simplement $Y$ la base de cette famille de courbes elliptiques; c'est une courbe modulaire $Y = \Lambda \backslash \mathcal{H}$. 

Comme dans le cas multiplicatif, on déduit du théorème d'approximation forte que 
\begin{equation} \label{E:9.TK}
Z( \A_f ) \backslash [\mathcal{G} ] / \mathcal{K} \simeq \Gamma \backslash (X \times E^n ),
\end{equation}
où $\Gamma = K \cap \GL_n (\Q)$; c'est un fibré au-dessus de $\Gamma \backslash X$ que, comme dans le cas multiplicatif, nous noterons $\mathcal{T}_\mathcal{K}$. 

\medskip
\noindent
{\it Exemple.} Soit $N$ un entier strictement supérieur à $1$. En pratique on considérera surtout le cas où    
$K=K_{0} (N)$ et $L=L_0 (N)$ de sorte que $Y$ soit égale à la courbe modulaire $Y_0 (N)$ et 
$$\mathcal{T}_{\mathcal{K}} = \Gamma_0 (N) \backslash \left[ X \times E^n \right].$$

\medskip

\section{Fonctions de Schwartz et cycles}

\subsection{Formes de Mathai--Quillen}

Le fibré \eqref{E:9.fibre2} est naturellement muni d'une orientation et d'une métrique hermitienne $\GL_n (\R ) \times \SL_2 (\R)$-invariantes. Le formalisme de Mathai--Quillen s'applique donc encore pour fournir des fonctions de Schwartz $\varphi$ et $\psi$ à l'infini. Les formules sont identiques à ceci près qu'elles dépendent maintenant du paramètre $\tau$. 

Au-dessus d'un point  
$$(gK, \tau ) \in S \times \mathcal{H}$$
la métrique sur la fibre $\C^n$ de \eqref{E:9.fibre2} est en effet égale à   
$$z \mapsto  \frac{1}{\mathrm{Im} \ \tau } | g^{-1} z |^2.$$

La construction de Mathai--Quillen expliquée au chapitre \ref{C:4} appliquée au fibré hermitien équivariant \eqref{E:9.fibre2} conduit cette fois à une forme de Thom qui est $(\GL_n (\R) \times \SL_2 (\R))$-invariante
\begin{equation} \label{9.varphi}
\varphi \in \mathcal{A}^{2n} (S \times \mathcal{H} \times \C^n)^{\GL_n (\R) \times \SL_2 (\R)}
\end{equation}
à décroissance rapide dans les fibres. On définit encore 
\begin{equation} \label{9.psi}
\psi = \iota_X \varphi \in \mathcal{A}^{2n-1} (S \times \mathcal{H} \times \C^n)^{\GL_n (\R) \times \SL_2 (\R)}
\end{equation}
mais il n'est plus vrai dans ce contexte que la forme 
$$\eta = \int_0^{+\infty} [s]^* \psi \frac{ds}{s} \in A^{2n-1} (X \times \mathcal{H} \times \C^n)^{\GL_n (\R) \times \SL_2 (\R)}$$
soit fermée; il découle par exemple de \cite[\S 7.4]{Takagi} que déjà dans le cas $n=1$ on a
\begin{equation} \label{9.eta1}
\eta = \frac{1}{8\pi} \left( \frac{d\tau}{y} + \frac{d \overline{\tau}}{y} \right) - \frac{i}{4\pi} \left( \frac{dz}{z} - \frac{d\overline{z}}{\overline{z}} \right)
\end{equation}
dont la différentielle est égale à la forme d'aire $\frac{dx \wedge dy}{4\pi y^2}$ sur $\mathcal{H}$. 

On appelle maintenant {\it représentation de Weil} la représentation $\omega$ du groupe $\mathcal{G} (\R)$ dans l'espace de Schwartz $\mathcal{S} (M_{n , 2} (\R))$ donnée par 
$$\omega (g , B , x) : \mathcal{S} (M_{n , 2} (\R))  \longrightarrow  \mathcal{S} (M_{n , 2} (\R)); \quad \phi \mapsto \left(y \mapsto \phi \left( g^{-1} (y -x ) B \right) \right) ,$$
avec $(g,B,x) \in (\GL_n (\R) \times \SL_2 (\R)) \ltimes M_{n, 2} (\R)$. 

Il correspond encore à $\varphi$ et $\psi$ des formes
\begin{equation} \label{9.varphi2}
\widetilde{\varphi} \in \mathcal{A}^{2n} \left( S \times \mathcal{H} \times \C^n ; \mathcal{S} (M_{n , 2} (\R)) \right)^{\mathcal{G}(\R)} 
\end{equation}
et
\begin{equation} \label{9.psi2}
\widetilde{\psi} \in \mathcal{A}^{2n-1} \left( S \times \mathcal{H} \times \C^n ; \mathcal{S} (M_{n , 2} (\R)) \right)^{\mathcal{G}(\R)} 
\end{equation}
qui, après évaluation en la matrice nulle, sont respectivement égales à $\varphi$ et $\psi$. Le lemme \ref{L:convcourant} reste valable sauf le dernier point; les formes 
$$\widetilde{[s]^*\varphi} (0) = [s]^* \varphi$$ 
convergent cette fois uniformément sur tout compact vers une forme non-nulle mais qui reste invariante sous l'action de $\mathcal{G} (\R)$.

\subsection{Fonctions de Schwartz aux places finies et cycles invariants}

Soit 
$$\mathcal{S} (M_{n,2} (\A_f)) = \mathcal{S} ( V (\A_f )^2 )$$ 
l'espace de Schwartz des fonctions $\varphi_f : M_{n,2} (\A_f) \to \C$ localement constantes et à support compact. Le groupe $\mathcal{G} (\A_f)$ opère sur $\mathcal{S} (M_{n,2} (\A_f ))$ par la représentation de Weil
$$\omega (g , B ,x) : \mathcal{S} (M_{n,2} (\A_f )) \to \mathcal{S} (M_{n,2} (\A_f )); \quad \phi \mapsto \left( w \mapsto \phi (g^{-1} (w-x) B \right).$$ 

Considérons maintenant l'espace $C^\infty \left( \mathcal{G} (\A_f ) \right)$ des fonctions continues localement constantes; on fait opérer le groupe $\mathcal{G} (\A_f)$ sur $C^\infty \left( \mathcal{G} (\A_f ) \right)$ par la représentation régulière à droite~:
$$( (h ,y , C) \cdot f) (g ,x , B ) = f( g h , g yB^{-1} + x , BC ).$$

L'application 
\begin{equation}
\mathcal{S} (M_{n,2} (\A_f )) \to \mathcal{G} (\A_f); \quad \phi \mapsto f_\phi : ( (g, x ,B ) \mapsto \phi (-g^{-1} x B))
\end{equation}
est $\mathcal{G} (\A_f)$-équivariante relativement aux deux actions définies ci-dessus.

\begin{definition}
Soit $\varphi_f \in \mathcal{S} (V (\A_f )^2)$ une fonction de Schwartz $\mathcal{K}$-invariante. 
\begin{itemize}
\item Soit $D_{\varphi_f}$, resp. $D_{\varphi_f , K}$, l'image du cycle
$$\mathcal{G} (\Q ) \left[ (\GL_n (\R) \times \SL_2 (\R )) \cdot \mathrm{supp} (f_{\varphi_f} ) \right]  \to  [\mathcal{G}] /(L \ltimes V(\Z )^2)  ,$$
resp.
$$\mathcal{G} (\Q ) \left[ (\GL_n (\R) \times \SL_2 (\R ))  \cdot   \mathrm{supp} (f_{\varphi_f} ) \right]  \to  [\mathcal{G}]/ \mathcal{K},$$
induite par l'inclusion du support de $f_{\varphi_f}$ dans $\mathcal{G} (\A_f)$. 
\item Soit 
$$U_{\varphi_f} \subset   [\mathcal{G}] /(L \ltimes V(\Z )^2), \quad \mbox{resp.} \quad U_{\varphi_f, K} \subset  [\mathcal{G}]/ \mathcal{K},$$ 
le complémentaire de $D_{\varphi_f}$, resp. $D_{\varphi_f , K}$.
\end{itemize}
\end{definition}

Comme dans le cas multiplicatif, {\it via} l'isomorphisme \eqref{E:9.TK} la projection de $D_{\varphi_f , K} \subset [\mathcal{G}]/ \mathcal{K}$ est égale à la réunion finie
\begin{equation}  \label{E:9.9}
\bigcup_\xi \Gamma \backslash (X \times \{  [\tau , u \tau + v] \in E^n \; : \; \tau \in \mathcal{H} \} ),
\end{equation}
où $\xi = (u , v)$ parcourt les éléments de $V(\Q)^2 / V (\Z)^2$ tels que $\varphi_f ( \xi )$ soit non nul.

\medskip

L'espace $D_{\varphi_f}$ est donc un revêtement fini de $[\GL_n] \times Y$ et $\varphi_f$ induit une fonction localement constante sur $D_{\varphi_f}$ c'est-à-dire un élément de $H^0 (D_{\varphi_f })$.
Maintenant, l'isomorphisme de Thom implique que l'on a~:
\begin{equation} \label{E:thom}
H^0 (D_{\varphi_f }) \stackrel{\sim}{\longrightarrow} H^{2n} \left( [\mathcal{G}]/ (L \ltimes V(\Z )^2) , U_{\varphi_f } \right);
\end{equation}
on note 
$$[\varphi_f] \in H^{2n}   \left( [\mathcal{G}]/ (L \ltimes V(\Z )^2) , U_{\varphi_f } \right)$$ 
l'image de $\varphi_f$; cette classe est $K$-invariante et l'on désigne par $[\varphi_f]_K$ son image dans 
$$H^{2n}   \left( [\mathcal{G}]/ \mathcal{K}, U_{\varphi_f , K } \right).$$ 

En reprenant la démonstration du lemme \ref{L:40} on obtient le lemme analogue suivant.  

\begin{lemma} \label{L:9.40}
Supposons $\widehat{\varphi}_f (0) =0$. Alors, l'application degré 
$$H^0 (D_{\varphi_f }) \to \Z$$
est égale à $0$. 
\end{lemma}

Suivant la définition \ref{def67} on isole finalement pour la suite un espace privilégié de fonctions de Schwartz sur $M_{n , 2} (\A_f)$. 

\begin{definition} 
Soit $\mathcal{S} (M_{n , 2} (\A_f ))^{\circ}$ le sous-espace de $\mathcal{S} (M_{n , 2} (\A_f))$ constitué des fonctions de la forme 
$$\mathbf{1}_{V (\widehat{\Z})} \oplus \overline{\varphi}_f \in \mathcal{S} (V (\A_f) \oplus V (\A_f)) = \mathcal{S} (M_{n,2} (\A_f )),$$
où $\overline{\varphi}_f \in \mathcal{S} (V (\A_f))^\circ$ est $(K \ltimes V (\widehat{\Z}))$-invariante. 
\end{definition}

\section[Séries theta et séries d'Eisenstein adéliques]{Séries theta et séries d'Eisenstein adéliques; \\ démonstration du théorème \ref{T:ell}}

\subsection{Séries theta adéliques}

Comme dans le cas multiplicatif, on associe à toute fonction $\varphi_f \in \mathcal{S} (M_{n , 2} (\A_f))$ des formes différentielles
$$\widetilde{\varphi} \otimes \varphi_f  \quad \mbox{et} \quad \widetilde{\psi} \otimes \varphi_f \in A^{\bullet} \left( S \times \mathcal{H} \times \C^n , \mathcal{S} (M_{n,2} (\A )) \right)^{(\GL_n (\R)  \times \SL_2 (\R)) \ltimes \C^n}.$$

En appliquant la distribution theta dans les fibres, on obtient alors une application
\begin{equation} \label{9.appl-theta}
\theta_\varphi \quad \mbox{et} \quad \theta_\psi : \mathcal{S} (M_{n , 2} (\widehat{\Z})) \longrightarrow \left[ A^{\bullet} (S \times \mathcal{H} \times \C^n)  \otimes C^\infty \left( \mathcal{G} (\A_f ) \right) \right]^{\mathcal{G} (\Q )}
\end{equation}
définie cette fois par 
\begin{equation} \label{9.appl-theta2}
\begin{split}
\theta_\varphi^* (g_f , B_f , x_f ; \varphi_f ) & = \sum_{\xi \in M_{n,2} (\Q )}\widetilde{\varphi} (\xi ) (\omega (g_f , B_f , x_f ) \varphi_f ) (\xi ) \\ 
& = \sum_{\xi \in M_{n,2} (\Q )} \varphi_f \left( g_f^{-1} (\xi -x_f ) B_f \right) \widetilde{\varphi} (\xi )  
\end{split}
\end{equation}
et de même pour $\theta_\psi$.
Rappelons que le groupe $\mathcal{G} (\R)$ opère naturellement sur $S \times \mathcal{H} \times \C^n$; étant donné un élément $(g,B,x) \in \mathcal{G} (\R)$ et une forme $\alpha \in A^{\bullet} (S \times \mathcal{H} \times \C^n)$ on note $(g,B,x)^* \alpha$ le tiré en arrière de $\alpha$ par l'application
$$(g,B,x) :  S \times \mathcal{H}  \times \C^n \to S \times \mathcal{H} \times \C^n.$$
L'invariance sous le groupe $\mathcal{G}(\Q)$ dans \eqref{9.appl-theta} signifie alors que pour tout élément $(g,B,x) \in \mathcal{G} (\Q)$ on a~:
$$(g,B,x)^*\theta_\varphi^*(g g_f , BB_f , gx_f B^{-1} + x ;\varphi_f)=\theta_\varphi^*(g_f, B_f , x_f ; \varphi_f);$$ 
ce qui découle de la $\mathcal{G} (\R)$-invariance of $\widetilde{\varphi}$, cf. \eqref{9.varphi2}. 

L'application $\theta_\varphi$ entrelace les actions naturelles de $\mathcal{G} (\A_f)$ des deux côtés. En particulier, en supposant $\varphi_f$ invariante sous le compact ouvert $\mathcal{K}$ on obtient des formes différentielles
\begin{multline*}
\theta_\varphi (\varphi_f ) \quad \mbox{et} \quad \theta_\psi (\varphi_f ) \in \left[ A^{\bullet} (S \times \mathcal{H} \times \C^n)  \otimes C^\infty \left( \mathcal{G} (\A_f ) \right) \right]^{\mathcal{G} (\Q ) \times \mathcal{K}} \\ = A^\bullet (S \times \mathcal{H}  \times \C^n)^{(\Gamma \times \Lambda) \ltimes M_{n,2} (\Z )}
\end{multline*}
autrement dit des formes différentielles sur $\Gamma \backslash (S \times E^n)$.

Comme au \S \ref{algHecke}, les séries theta $\theta_\varphi$ et $\theta_\psi$ sont équivariantes relativement aux actions naturelles des opérateurs de Hecke et, comme au \S \ref{cohomClass}, la forme différentielle $\theta_\varphi (\varphi_f)$ est fermée et représente $[\varphi_f]_K$.

Finalement les propriétés de décroissance des formes $\theta_{[r]^* \varphi} (\varphi_f )$ sont analogues à celles décrites dans le lemme \ref{L:theta-asympt} à ceci près que cette fois $\theta_{[r]^* \varphi} (\varphi_f )$ ne tend pas nécessairement vers $0$ avec $r$. C'est pour garantir cela que nous supposerons dorénavant que $\widehat{\varphi}_f (0) = 0$; en procédant comme dans la démonstration du lemme \ref{L:thetaSiegel}, il découle en effet alors de la formule de Poisson que sous cette hypothèse $\theta_{[r]^* \varphi} (\varphi_f )$ tend vers $0$ avec $r$. 

\subsection{Séries d'Eisenstein adéliques} 

Soit toujours $\varphi_f \in \mathcal{S} (M_{n , 2} (\A_f ))$ une fonction $\mathcal{K}$-invariante dont on supposera de plus qu'elle vérifie $\widehat{\varphi}_f (0) = 0$. Comme au \S \ref{SEA7} on peut alors associer à $\varphi_f$ les séries d'Eisenstein
$$E_\varphi (\varphi_f , s) = \int_0^\infty r^s \theta_{[r]^* \varphi} (\varphi_f ) \frac{dr}{r} \quad \mbox{et} \quad E_\psi (\varphi_f , s) = \int_0^\infty r^s \theta_{[r]^* \psi} (\varphi_f ) \frac{dr}{r};$$
ce sont des formes différentielles sur la préimage de $U_{\varphi_f}$ dans $S \times \mathcal{H} \times \C^n$ qui sont bien définies, et invariantes par l'action du centre $Z(\R)^+$, en $s=0$. On pose alors 
\begin{equation}
E_\psi (\varphi_f ) = E_\psi (\varphi_f , 0) \in A^{2n-1} \left( [\mathcal{G}]/ \mathcal{K} - D_{\varphi_f , K} \right).
\end{equation}

Comme dans le cas multiplicatif, on obtient le théorème suivant.
\begin{theorem} \label{T:9.4}
Supposons $\widehat{\varphi}_f (0) =0$. Alors la forme différentielle
$$E_\psi (\varphi_f ) \in A^{2n-1} \left( [\mathcal{G}]/ \mathcal{K} - D_{\varphi_f , K} \right)$$ 
est \emph{fermée} et représente une classe de cohomologie qui relève la classe $[\varphi_f]_K$ dans la suite exacte longue
\begin{multline*}
\ldots \to H^{2n-1} \left( [\mathcal{G}]/ \mathcal{K} - D_{\varphi_f , K} \right) \\ \to H^{2n} \left( [\mathcal{G}]/ \mathcal{K} ,  [\mathcal{G}]/ \mathcal{K} - D_{\varphi_f , K} \right) \to H^{2n} \left(  [\mathcal{G}]/ \mathcal{K} \right) \to \ldots 
\end{multline*}
\end{theorem}
Noter que, sous l'hypothèse $\widehat{\varphi}_f (0) =0$, le degré de $D_{\varphi_f , K}$ est nul et l'image de $[\varphi_f]_K$ dans $H^{2n} \left(  [\mathcal{G}]/ \mathcal{K} \right)$ est égale à $0$.

\medskip

\subsection{Démonstration du théorème \ref{T:ell}} \label{S:9.3.3}

Il correspond à un sous-groupe de congruence $\Gamma$ dans $\SL_n (\Z)$ un sous-groupe compact ouvert $K$ dans $\GL_n (\widehat{\Z})$ tel que $K \cap \SL_n (\Z) = \Gamma$.

Un élément $D \in \mathrm{Div}_\Gamma$ peut-être vu comme une fonction $\Gamma$-invariante sur $E^n$ à support dans les points de torsion de $E^n$; il lui correspond donc une fonction $\mathcal{K}$-invariante $\varphi_f \in \mathcal{S} (M_{n , 2} (\A_f ))$. Si de plus $D$ est de degré $0$ alors $\widehat{\varphi}_f (0) = 0$. 

En procédant comme au \S \ref{S8.1} on associe à $E_\psi (\varphi_f )$ une forme simpliciale 
$$\mathcal{E}_\psi (\varphi_f )  \in \mathrm{A}^{2n-1} \left( E\Gamma \times (E^n - \mathrm{supp} \ D) \right)^{\Gamma}$$
qui est fermée et représente la même classe de cohomologie équivariante que $E_{\psi} (\varphi_f )$. 

Fixons $n$ morphismes primitifs linéairement indépendants
$$\chi_1 , \ldots , \chi_n : \Z^n \to \Z.$$ 
On note encore $\chi_j : \Q^n \to \Q$ les formes linéaires correspondantes et $\chi_j : E^n \to E$ les morphismes primitifs qu'ils induisent. On pose 
$$\chi = (\chi_1 , \ldots , \chi_n).$$

Suivant la démonstration du théorème \ref{T:cocycleE} on restreint maintenant la forme simpliciale $\mathcal{E}_\psi (\varphi_f )$ aux ouverts de $E^n$, indexés par les éléments $\mathbf{g} = (\gamma_0 , \ldots , \gamma_k)$ de $E_k \Gamma$, 
\begin{equation} \label{E:9.860}
U (\mathbf{g}) = E^n - \bigcup_\xi \bigcup_{i=1}^n \bigcup_{j=0}^k  [\mathrm{ker} (\chi_i \circ \gamma_j) +\xi ],
\end{equation}
où $\xi$ parcourt l'ensemble des éléments du support de $D$.

Il découle du lemme \ref{affineb} que les variétés \eqref{E:9.860} sont affines. On peut donc appliquer le théorème \ref{P:Brieskorn}.

Comme dans la démonstration du théorème \ref{T:cocycleE}, l'argument expliqué dans l'annexe \ref{A:A} permet alors d'associer à la forme simpliciale $\mathcal{E}_\psi (\varphi_f )$ un $(n-1)$-cocycle homogène
\begin{equation} 
\mathbf{S}_{{\rm ell}, \chi} [D]  \in C^{n-1} \left( \Gamma  , \Omega_{\rm mer} (E^n )  \right)^{\Gamma}
\end{equation}
qui représente la classe $S_{\rm mult} [D]$ et qui est à valeurs dans les formes méromorphes régulières en dehors des hyperplans  affines $\mathrm{ker} (\chi_i \circ g) +\xi$, avec $\xi$ dans le support de $D$ et $g$ dans $\Gamma$. Cela démontre les points (1), (2) et (3) du théorème \ref{T:ell}.

Comme dans le cas multiplicatif, l'application $\GL_n (\A_f) \to \mathcal{G} (\A_f )$ induit une inclusion entre algèbres de Hecke 
$$\mathcal{H} (\GL_n (\A_f) , K) \hookrightarrow \mathcal{H} (\mathcal{G} (\A_f ) , \mathcal{K})$$
et à une fonction $\phi \in \mathcal{H} (\GL_n (\A_f) , K)$ on associe un opérateur de Hecke $\mathbf{T}_\phi$. L'analogue de la proposition \ref{P:hecke2} se démontre de la même manière, de sorte que 
\begin{equation}
\mathbf{T}_\phi \left[ \mathbf{S}_{{\rm ell}, \chi}[\varphi_f] \right]= \left[ \mathbf{S}_{{\rm ell}, \chi }[ T_{\phi} \varphi_f] \right] .
\end{equation}
En prenant pour $\phi$ la fonction caractéristique d'une double classe $KaK$ avec $a$ dans $M_n (\Z) \cap \GL_n (\Q)$, l'opérateur $\mathbf{T}_\phi$ est égal à l'opérateur $\mathbf{T} (a)$ du chapitre \ref{C:2}. 

L'unicité dans le théorème \ref{T:cocycleE} ne nécessite plus de passer au quotient par les formes invariantes, l'analogue de la proposition \ref{P:hecke2} devient alors 
\begin{equation*}
\mathbf{T} (a) \left[ \mathbf{S}_{{\rm ell}, \chi}[D] \right] = [ \mathbf{S}_{{\rm ell}, \chi}[[\Gamma a \Gamma ]^* D] ] \quad \mbox{dans} \quad  H^{n-1} (\Gamma , \Omega^n_{\rm mer} (E^n)),
\end{equation*}
ce qui démontre le point (4) du théorème \ref{T:ell}.

Finalement, l'application $\SL_2 (\A_f) \to \mathcal{G} (\A_f )$ induit une inclusion entre algèbres de Hecke 
$$\mathcal{H} (\SL_2 (\A_f) , L) \hookrightarrow \mathcal{H} (\mathcal{G} (\A_f ) , \mathcal{K}).$$
Il lui correspond une deuxième famille d'opérateurs de Hecke qui contiennent en particulier ceux notés $T(b)$ au chapitre \ref{C:2}. Le point (5) du théorème \ref{T:ell} se démontre alors de la même manière que le point (4). \qed

\section[\'Evaluation sur les symboles modulaires]{\'Evaluation sur les symboles modulaires et \\ démonstration du théorème \ref{T:ellbis}}

L'étude du comportement à l'infini des séries d'Eisenstein $E_\psi (\varphi_f)$ est similaire à ce que l'on a fait dans le cas multiplicatif au \S \ref{S:7.eval}. On ne détaille pas plus, on calcule par contre l'intégrale de $E_\psi (\varphi_f)$ sur les symboles modulaires.

\subsection{\'Evaluation de $E_\psi (\varphi_f)$ sur les symboles modulaires} \label{S:9.4.1}

Soit 
$$\mathbf{q} = (q_0 , \ldots , q_k)$$ 
un $(k+1)$-uplet de vecteurs non nuls dans $V(\Q)$ avec $k \leq n-1$. 

Soit $\overline{\varphi}_f \in \mathcal{S} (V(\A_f ))$ une fonction de Schwartz $(K \ltimes V( \widehat{\Z} ))$-invariante telle que pour tout entier $j$ dans $[0 , k]$ on ait 
\begin{equation} \label{E:condphi2}
\int_{W(\mathbf{q})^{(j)}} \overline{\varphi}_f =0 \quad \mbox{dans} \quad \mathcal{S} (V (\A_f) / W(\mathbf{q})^{(j)} (\A_f) ).
\end{equation}
Soit 
$$\varphi_f = \mathbf{1}_{V (\widehat{\Z})} \oplus \overline{\varphi}_f \in \mathcal{S} (V (\A_f) \oplus V (\A_f)) = \mathcal{S} (M_{n,2} (\A_f )).$$

On considère cette fois l'application
$$\Delta (\mathbf{q}) \times \mathrm{id}_{\mathcal{H} \times \C^n}  : \Delta_k ' \times \mathcal{H} \times \C^n  \to \overline{X}^T \times \mathcal{H} \times \C^n$$
et l'on note
\begin{equation}
E_\psi (\varphi_f , \mathbf{q} ) = (\Delta (\mathbf{q}) \times \mathrm{id} )^* E_\psi (\varphi_f).
\end{equation}
On obtient ainsi une forme fermée dans 
$$A^{2n-1} \Big( \Delta_{k}' \times \big( E^n - \bigcup_{j=0}^k \bigcup_{\xi} (H_\mathbf{q}^{(j)} + \xi) \big) \Big),$$
où $H_\mathbf{q}^{(j)}$ désigne l'image de 
$$\mathcal{H} \times \langle q_0 , \ldots , \widehat{q_j} , \ldots , q_k \rangle_\C \subset \mathcal{H} \times \C^n$$ 
dans $E^n$, et $\xi$ est comme dans \eqref{E:9.9}

Il s'agit maintenant de calculer les intégrales $\int_{\Delta_{k}'} E_\psi (\varphi_f , \mathbf{q} )$; ce sont des formes différentielles sur des ouverts de $E^n$. Pour simplifier on se contente de calculer, dans la proposition suivante, la restriction de ces formes aux fibres de $E^n \to Y$.

\begin{proposition} \label{9.P34bis}
1. Si $\langle q_0 , \ldots , q_{k} \rangle$ est un sous-espace propre de $V$, alors la forme $\int_{\Delta_{k}'} E_\psi (\varphi_f , \mathbf{q} )$ est identiquement nulle.

2. Supposons $k=n-1$ et que les vecteurs $q_0 , \ldots , q_{n-1}$ soient linéairement indépendants. On pose $g = (q_0 | \cdots | q_{n-1} ) \in \GL_n (\Q)$ et l'on fixe $\lambda \in \Q^\times$ tel que la matrice $h=\lambda g$ soit entière. Alors, en restriction à la fibre $E_\tau^n$ au-dessus d'un point $[\tau] \in Y$, la $n$-forme $\int_{\Delta_{n-1}'} E_\psi (\varphi_f , \mathbf{q} )$ est égale à 
\begin{multline*}
\frac{1}{\det h} \sum_{v \in V(\Q) / V(\Z) }  \overline{\varphi}_f (v ) \cdot \\ \sum_{\substack{\xi \in E^n \\ h \xi  = v }} \mathrm{Re} \left( E_1 ( \tau , \ell_1 - \xi_{1} ) dz_1 \right) \wedge \cdots \wedge \mathrm{Re} \left( E_1 (\tau , \ell_{n} - \xi_{n} ) dz_n \right)  ,
\end{multline*}
où un vecteur $v \in  V(\Q) / V(\Z) = \Q^n / \Z^n$ est identifié à un élément de la courbe elliptique $E_\tau^n = \C^n / (\tau \Z^n + \Z^n )$ et $\ell_j$ est la forme linéaire sur $\C^n$ caractérisée par $h^* \ell_j = e_j^*$.  
\end{proposition}
\begin{proof} On explique comment modifier la démonstration de la proposition \ref{P34bis}. La première partie est identique. Supposons donc $k=n-1$ et que les vecteurs $q_0 , \ldots , q_{n-1}$ sont linéairement indépendants. La propriété d'invariance (\ref{E:invformsimpl}) s'étend naturellement et implique cette fois encore que 
\begin{equation} \label{E:9.20-}
\int_{\Delta_{n-1}'} E_\psi (\varphi_f , \mathbf{q} )  = (h^{-1})^* \left(  \int_{\Delta_{n-1}'} E_\psi ( \varphi_f (h \cdot) , \mathbf{e}) \right).
\end{equation}
On en est donc réduit à calculer l'intégrale 
\begin{multline} \label{9.E:intti}
\int_{A\SO_n \R_{>0}} E_\psi (\varphi_f (h \cdot ))  \\ = \int_{\{ \mathrm{diag}(t_1,\ldots,t_n) \; : \; t_j \in \R_{>0} \} \SO_n} (t_1 \cdots t_n )^{s} \theta_\varphi (\varphi_f (h \cdot) (1 , 0) \ \Big|_{s=0}.
\end{multline}
En restriction à l'ensemble des matrices symétriques diagonales réelles, le fibré en $\C^n$ se scinde encore {\it métriquement} en une somme directe de $n$ fibrés en droites, correspondant aux coordonnées $(z_j )_{j=1 , \ldots , n}$ de $z$ et la forme $\varphi$ se décompose en le produit de $n$ formes associées à ces fibrés en droites. Ces formes font cette fois intervenir la variable $\tau$ (cf. \cite[\S 6.1]{Takagi}) mais d'après \eqref{9.eta1}, en restriction à une fibre $E_\tau^n$, on a
\begin{multline*}
\int_{\{ \mathrm{diag}(t_1,\ldots,t_n) \; : \; t_j \in \R_{>0} \} \SO_n} (t_1 \cdots t_n )^s \widetilde{\varphi} (\xi) \\ = \frac{(-i)^n}{(4\pi)^n} \Gamma (1+ \frac{s}{2})^n \wedge_{j=1}^n  \left(  \frac{dz_j}{(z_j - \xi_j) | z_j - \xi_j |^{s} } -\frac{d\overline{z}_j}{\overline{z_j - \xi_j} | z_j - \xi_j |^{s}} \right). 
\end{multline*}
où $\xi \in M_{n,2} (\Q) = V (\Q)^2$ est identifié à un élément de $E_\tau^n$ {\it via} l'application 
$$V(\Q)^2 \to E_\tau^n; \quad (u,v) \mapsto \tau u + v.$$
L'intégrale \eqref{9.E:intti} est donc égale à 
\begin{multline} \label{9.int}
\Gamma (1+ \frac{s}{2})^n \sum_{\xi \in M_{n,2} (\Q)} \varphi_f (h \xi ) \wedge_{j=1}^n \mathrm{Re} \left( \frac{1}{2i\pi}  \frac{dz_j}{(z_j - \xi_j) | z_j - \xi_j |^{s} }  \right)   \\ 
 = \Gamma (1+ \frac{s}{2})^n \sum_{\xi \in M_{n,2} (\Q) / M_{n,2} (\Z)} \varphi_f (h \xi)  \wedge_{j=1}^n  \mathrm{Re} \left( K_1 (z_j , 0 , 1+s/2) dz_j \right),
\end{multline}
où $K_1$ est la série de Eisenstein--Kronecker \cite[Chap. VIII, \ (27)]{Weil} définie, pour $\mathrm{Re} (s) > 3/2$, par\footnote{On prendra garde au fait qu'on a ajouté un facteur $1/2i\pi$.}
$$K_1 (z,0, s) = \frac{1}{2i \pi} \sideset{}{'} \sum_{w \in \Z \tau + \Z} (\overline{z} + \overline{w}) | z + w|^{-2s}.$$
Il découle de \cite[p. 81]{Weil} que pour tout $z \in \C$, cette fonction 
admet un prolongement méromorphe au plan des $s \in \C$ qui, en $s=1$, est égal à la série d'Eisenstein {\it non-holomorphe}
\begin{equation} \label{E*E}
E_1^* (\tau , z) = E_1 (\tau , z) + \frac{1}{y} \mathrm{Im} (z),
\end{equation}
où $y= \mathrm{Im} (\tau)$. Noter que la série $E_1^* (\tau , z)$ est $(\Z \tau + \Z)$-périodique en $z$ et modulaire de poids $1$.
 
En prenant $s=0$ dans l'expression \eqref{9.int} on obtient que 
$$\int_{\Delta_{n-1}'} E_\psi ( \varphi_f (h \cdot) , \mathbf{e}) = \sum_{\xi \in M_{n,2} (\Q) / M_{n,2} (\Z)} \varphi_f (h \xi)  \wedge_{j=1}^n   \mathrm{Re} \left( E_1^* (\tau , z_j - \xi_j ) dz_j \right) $$
et donc que l'intégrale \eqref{E:9.20-} est égale à
\begin{multline*}
\sum_{\xi \in M_{n,2} (\Q) / M_{n,2} (\Z)} \varphi_f (h \xi) \cdot  (h^{-1})^* \left(  \wedge_{j=1}^n  \mathrm{Re} \left( E_1^* (\tau , z_j - \xi_j ) dz_j \right) \right) \\
\begin{split}
& = \frac{1}{\det h} \sum_{\xi \in M_{n,2} (\Q) / M_{n,2} (\Z)} \varphi_f (h \xi)  \wedge_{j=1}^{n}  \mathrm{Re} \left( E_1^* (\tau , \ell_j - \xi_j ) dz_j \right) \\
& = \frac{1}{\det h} \sum_{v \in V(\Q) / V(\Z) }  \overline{\varphi}_f (v ) \cdot \\
& \quad \quad \sum_{\substack{\xi \in E^n \\ h \xi  = v }} \mathrm{Re} \left( E_1^* ( \tau , \ell_1 - \xi_{1} ) dz_1 \right) \wedge \cdots \wedge \mathrm{Re} \left( E_1^* (\tau , \ell_{n} - \xi_{n} ) dz_n \right)  ,
\end{split}
\end{multline*}
où la dernière expression, qui découle de la définition de $\varphi_f$, est obtenue en groupant les $\xi$ envoyés par $h$ sur un même élément dans $E^n$. Rappelons que l'on identifie une classe $\xi $ dans le quotient $M_{n,2} (\Q) / M_{n,2} (\Z)$ à l'élément $\tau \alpha + \beta$ dans $E^n$, où $\alpha$ et $\beta$ sont les vecteurs colonnes de $\xi$. La dernière somme porte donc sur les classes $\alpha$ et $\beta$ dans $V(\Q) / V(\Z)$ telles que $h \alpha$ soit nul modulo $V(\Z)$ et $h \beta$ soit égal à $v$ modulo $V(\Z)$. 

\medskip
\noindent
{\it Remarque.} Le fait que l'expression soit indépendante des choix de $\lambda$ et $h$ découle des relations de distributions pour $E_1^*$
$$\sum_{\xi \in E_\tau [m]} E_1^* (\tau , z- \xi ) = m E_1^* (\tau , m z ).$$

\medskip

Il nous reste à expliquer que dans l'expression ci-dessus on peut partout remplacer $E_1^*$ par $E_1$. Pour cela on commence par remarquer que l'expression
$$\sum_{v \in V(\Q) / V(\Z) }  \overline{\varphi}_f (v ) \sum_{\substack{\xi \in E^n \\ h \xi  = v }} \mathrm{Re} \left( E_1^* ( \tau , \ell_1 - \xi_{1} ) dz_1 \right) \wedge \cdots \wedge \mathrm{Re} \left( E_1^* (\tau , \ell_{n} - \xi_{n} ) dz_n \right)$$
peut se réécrire
\begin{multline} \label{E:sans*}
\sum_{\substack{\alpha \in \Q^n / \Z^n \\ h \alpha = 0 \ (\mathrm{mod} \ \Z^n) }} \sum_{v \in V(\Q) / V(\Z) } \overline{\varphi}_f (v ) \\ \sum_{\substack{\beta \in \Q^n / \Z^n \\ h \beta  = v \ (\mathrm{mod} \ \Z^n)}} \mathrm{Re} \left( E_1^* ( \tau , \ell_1 + \alpha_1 \tau + \beta_1 ) dz_1 \right) \wedge \cdots \wedge \mathrm{Re} \left( E_1^* (\tau , \ell_{n} + \alpha_n \tau + \beta_n ) dz_n \right).
\end{multline}
Puisque d'après \eqref{E*E} les différences
$$E_1^* ( \tau , \ell_j + \alpha_j \tau + \beta_j ) - E_1 (\tau , \ell_j + \alpha_j \tau + \beta_j) = \frac{1}{y} \mathrm{Im} (\ell_j ) + \alpha_j$$
sont indépendantes de $v$ et $\beta$, il résulte bien de la démonstration de (\ref{E:8.smoothing}) que l'on peut partout remplacer $E_1^*$ par $E_1$ dans \eqref{E:sans*}. Cela conclut la démonstration de la proposition \ref{9.P34bis}.
\end{proof}

\subsection{Démonstration du théorème \ref{T:ellbis}} \label{S:9.4.2} 

On procède comme dans la démonstration du théorème \ref{T:mult2}. On utilise cette fois le résultat suivant. 

\begin{lemma}
Les  $1$-formes différentielles  
$$\mathrm{Re}(E_1^*(\tau,z) dz) \quad \textrm{and} \quad E_1^*( \tau,z) dz$$ 
sont cohomologues sur $E_\tau - \{ 0 \}$, où $E_\tau$ est la courbe elliptique $\C/(\tau \Z+\Z)$. Autrement dit, la forme
$\mathrm{Im} (E_1^*( \tau,z) dz)$ est exacte sur $E_\tau - \{ 0 \}$.
\end{lemma}
\begin{proof}
D'après \cite[Eq. (2) p. 57]{deShalit}, on a
\begin{equation}\label{eqE1*tolog}
2\pi i E_1^* (\tau , z) =\frac 1 {12} \partial_z \log \theta (\tau , z ) - \pi \frac{\overline{z}}{2y},
\end{equation}
où
$$\theta (\tau , z ) =e^{6\pi\frac{z(z-\overline{z})}{y}}q_\tau(q_z^{\frac 12}-q_z^{-\frac 12})^{12}\prod_{n\geq 1}(1-q_\tau^nq_z)^{12}(1-q_\tau^nq_z^{-1})^{12}.$$

La fonction réelle $z \mapsto | \theta (\tau , z)| $ est régulière sur $\C - (\tau \Z + \Z)$ et $(\tau \Z + \Z)$-périodique (cf. \cite[p. 49]{deShalit}). Calculons maintenant sa différentielle~:
\begin{equation*}
\begin{split}
d \log | \theta (\tau , z) | & = \mathrm{Re} (d \log \theta (\tau , z) ) \\
& = \mathrm{Re} \left(\partial_z  \log  \theta (\tau , z) dz + \partial_{\overline{z}}  \log \theta (\tau , z) d\overline{z}\right)\\ 
& = \mathrm{Re} \left(\partial_z   \log \theta (\tau , z) dz -6 \pi  \frac{z}{y} d\overline{z} \right)\\
& = \mathrm{Re} \left( 12 (2\pi i) E_1^*(\tau,z) dz + 6\pi \frac{\overline{z}}{y} dz - 6\pi  \frac{z}{y} d\overline{z} \right),
\end{split}
\end{equation*}
où la dernière ligne se déduit de (\ref{eqE1*tolog}). Comme 
$$\mathrm{Re} \left(\overline{z}dz-  zd\overline{z}\right)=0,$$
on obtient que 
$$d \log | \theta (\tau , z) | = -24 \mathrm{Im} (E_1^*( \tau,z) dz).$$
\end{proof}

En particulier en restriction à $E_\tau$ les formes 
\begin{multline*}
\sum_{v \in V(\Q) / V(\Z) }  \overline{\varphi}_f (v ) \cdot \\ \sum_{\substack{\xi \in E^n \\ h \xi  = v }} \mathrm{Re} \left( E_1^*  ( \tau , \ell_1 - \xi_{1} ) dz_1 \right) \wedge \cdots \wedge \mathrm{Re} \left( E_1^* (\tau , \ell_{n} - \xi_{n} ) dz_n \right)  ,
\end{multline*}
et 
$$
\sum_{v \in V(\Q) / V(\Z) }  \overline{\varphi}_f (v ) \sum_{\substack{\xi \in E^n \\ h \xi  = v }} E_1^* ( \tau , \ell_1 - \xi_{1} )  \cdots E_1^* (\tau , \ell_{n} - \xi_{n} ) dz_1 \wedge \cdots \wedge dz_n  ,
$$
sont cohomologues. La démonstration du théorème \ref{T8.8cor} implique alors que pour toute fonction $\mathcal{K}$-invariante $\varphi_f \in \mathcal{S} (V (\A_f ))^\circ$, l'application 
$$\mathbf{S}_{\rm ell}^* [\varphi_f ] : \Gamma^{n} \longrightarrow \Omega^{n}_{\rm mer} (E^n )$$
qui à un $n$-uplet $(\gamma_0 ,  \ldots , \gamma_{n-1})$ associe $0$ si les vecteurs $\gamma_j^{-1} e_1$ sont liés et sinon la forme différentielle méromorphe 
\begin{equation*}
\frac{1}{\det h} \sum_{v \in V(\Q) / V(\Z) }  \overline{\varphi}_f (v ) \sum_{\substack{\xi \in E^n \\ h \xi  = v }} E_1 ( \tau , \ell_1 - \xi_{1} )  \cdots  E_1 (\tau , \ell_{n} - \xi_{n} ) dz_1 \wedge \cdots \wedge dz_n,
\end{equation*}
où $h = ( \gamma_0^{-1}  e_1 | \cdots | \gamma_{n-1}^{-1} e_1) \in \GL_n (\Q) \cap M_n (\Z )$ et $h^* \ell_j = e_{j}^*$, définit un $(n-1)$-cocycle homogène non nul du groupe $\Gamma$ qui représente la même classe de cohomologie que $\mathbf{S}_{\rm ell , \chi}$.   

Rappelons que, pour chaque élément $D \in \mathrm{Div}_{\Gamma,K}$ il correspond une fonction $\mathcal{K}$-invariante $\varphi_f$ dans $\mathcal{S} (V (\A_f ))$ qui, sous l'hypothèse que $D \in \mathrm{Div}_{\Gamma,K}^{\circ}$, appartient à $ \mathcal{S} (V (\A_f ))^{\circ}$. Il s'en suit finalement que le cocycle 
$$\mathbf{S}^*_{\rm ell} [D] = \mathbf{S}^*_{{\rm ell}}[\varphi_f]$$
vérifie le théorème \ref{T:ellbis}. \qed

\medskip

\newpage

\setcounter{chapter}{1}

\setcounter{equation}{0}

\numberwithin{equation}{chapter}

\begin{appendix}

\chapter{Cohomologie équivariante et complexe de de Rham simplicial} \label{A:A}

\resettheoremcounters

Les résultats de ce volume sont formulés dans le langage de la cohomologie équivariante. Il s'agit d'une théorie cohomologique pour les espaces topologiques munis d'une action de groupe. Comme la cohomologie habituelle, elle peut être calculée à l'aide de chaînes simpliciales ou, alternativement, à l'aide d'une version du complexe de de Rham. Dans cette annexe, nous rassemblons les principales définitions et faits concernant la cohomologie équivariante que nous utilisons. Pour plus de détails, le lecteur peut consulter \cite{DupontBook,Dupont}.

\section{Définition de la cohomologie équivariante} 
Soit $G$ un groupe discret et $X$ un espace topologique raisonnable muni d'une action de $G$. Dans cette situation, on peut définir les groupes de cohomologie équivariante
$$
H^*_G(X).
$$
Ces groupes contiennent des informations à la fois sur la cohomologie usuelle de $X$ et sur l'action de $G$. L'idée derrière la construction est la suivante : si $G$ opère librement sur $X$ alors on peut définir
$$
H^*_G(X) = H^*(X/G),
$$
c'est-à-dire la cohomologie usuelle du quotient $X/G$. Pour des actions plus générales, cependant, beaucoup d'informations sont perdues en passant à l'espace quotient $X/G$. L'idée de la cohomologie équivariante est de remplacer le quotient $X/G$ par un {\it quotient d'homotopie}
$$
EG \times_G X := |EG \times X|/G.
$$
On définit $|EG \times X|$ ci-dessous, retenons pour l'instant que c'est un espace topologique obtenu comme le produit de $X$ avec un espace contractile et sur lequel $G$ opère librement. Ainsi, le remplacement de $X$ par $|EG \times X|$ ne change pas le type d'homotopie sous-jacent, mais maintenant $G$ opère librement sur $|EG \times X|$. Il est donc raisonnable de définir
$$
H^*_G(X) = H^*(|EG \times X|/G).
$$
Par exemple, quand $X$ est réduit à un point (ou contractile), on a
$$
H^*_G(X) = H^*(BG),
$$
où $BG:=|EG|/G$ est l'espace classifiant de $G$ et $H^*(BG)$ est la cohomologie du groupe $G$. D'un autre côté, quand $G$ est trivial la cohomologie équivariante se réduit à la cohomologie habituelle. 

Ci-dessous, nous définissons $EG$ en utilisant des ensembles simpliciaux et le foncteur de réalisation simpliciale. Puis nous décrivons les résultats de Dupont \cite{Dupont} qui permettent de calculer $H^*_G(X)$ en utilisant des formes différentielles lorsque $X$ est une variété.

\section{La construction de Borel} Soit $G$ un groupe discret. Un espace classifiant pour $G$ est un espace topologique dont le groupe fondamental est isomorphe à $G$ et dont les groupes d'homotopie supérieurs sont triviaux. Un tel espace existe toujours. On en donne une construction concrète qui fait naturellement le lien avec la cohomologie des groupes.

Soit $EG_\bullet$ l'ensemble simplicial dont les $m$-simplexes sont les $(m+1)$-uplets ordonnés $(g_0 , \ldots , g_m )$ d'éléments de $g$. On note $EG_m = G^{m+1}$ l'ensemble de ces $m$-simplexes. Un élément de $(g_0 , \ldots , g_m ) \in EG_m$ est recollé aux $(m-1)$-simplexes $(g_0 , \ldots , g_{i-1},g_{i+1}, \ldots , g_m)$ de la même manière qu'un simplexe standard est recollé à ses faces. Les applications de faces et de dégénérescence sont donc
\begin{align}
\partial_i(g_0,\ldots,g_m) &= (g_0,\ldots,g_{i-1},g_{i+1},\ldots,g_m),  \quad \mbox{et} \\
\sigma_i(g_0,\ldots,g_m) &= (g_0,\ldots,g_i,g_i,\ldots,g_m), \quad i=0,\ldots, m.
\end{align}
Le complexe $EG_\bullet$ est contractile. Le groupe $G$ opère librement sur $EG_\bullet$ par 
\begin{equation}
g \cdot (g_0,\ldots,g_m) = (g_0g^{-1},\ldots,g_m g^{-1})
\end{equation}
de sorte que l'application quotient $EG_\bullet \to EG_\bullet / G$ est un revêtement universel. La base est donc un espace classifiant pour $G$.  

L'application
\begin{equation}
(g_0,\ldots,g_m) \mapsto (g_1g_0^{-1},\ldots,g_m g_{m-1}^{-1})
\end{equation}
identifie $EG_\bullet / G$ avec l'ensemble simplicial $BG_\bullet$ défini par $BG_m = G^m$ et  
\begin{align}
\partial_i(g_1,\ldots,g_m) &= \begin{cases} (g_2,\ldots,g_m) & i=0 \\ (g_1,\ldots,g_ig_{i+1},\ldots, g_m) & 0 < i < m  \\  (g_1,\ldots,g_{m-1}) & i=m, \end{cases} \\
\sigma_i(g_1,\ldots,g_m) &= (g_1,\ldots,g_i,1,g_{i+1},\ldots,g_m), \quad i=0,\ldots, m.
\end{align}
Notons $|X_\bullet|$ la réalisation géométrique d'un ensemble simplicial $X_\bullet$ : c'est l'espace topologique défini comme
$$
|X_\bullet | = \bigsqcup_{k \geq 0}(\Delta_k \times X_k)/\sim,
$$
où $\Delta_k$ désigne le $k$-simplexe standard et $\sim$ est la relation d'équivalence donnée par 
$$(t,f^*x) \sim (f_* t, x)$$
pour tous $t \in \Delta_k$, $x \in X_l$ et toute application croissante $f:[k] \to [l]$. Pour les principales propriétés du foncteur de réalisation simpliciale, nous renvoyons le lecteur à \cite{Segal}. L'espace $|BG_\bullet|$ est donc un espace classifiant pour $G$ et $|EG_\bullet| \to |BG_\bullet|$ est le $G$-fibré universel associé (cf. \cite[Prop. 2.7]{Burgos}).

\medskip

Considérons maintenant un espace topologique $M$ sur lequel le groupe $G$ opère par homéomorphismes, autrement dit un $G$-espace. Le groupe $G$ opère alors sur 
$$EG_\bullet  \times M$$
par 
\begin{equation}
g \cdot (g_0,\ldots,g_m, x) = (g_0g^{-1},\ldots,g_mg^{-1}, gx  ).
\end{equation}
La \emph{cohomologie équivariante} de $M$ est définie comme étant la cohomologie ordinaire du quotient 
$$X=|EG_{\bullet} \times_G M|;$$
autrement dit
$$H_G^\bullet (M) = H^\bullet (X) = H^\bullet (|EG_{\bullet} \times_G M|).$$
Si $G$ est le groupe trivial, c'est juste la cohomologie de $M$. Si $M$ est contractile, l'espace $X$ est homotopiquement équivalent à l'espace classifiant $BG_\bullet$ et $H_G^\bullet (M) = H^\bullet (BG_\bullet)$ est la cohomologie du groupe $G$. 

Noter que si $G$ opère librement sur $M$, la projection $EG_{\bullet} \times_G M \to M/G$ est une équivalence d'homotopie et on a $H_G^\bullet (M) = H^\bullet (M/G)$.

Donnons quelques exemples qui apparaissent dans le texte principal.

\begin{example}
Supposons que $G$ soit égal à $\Z^n$ pour un certain entier naturel $n$. Puisque $\Z^n$ opère librement sur $\R^n$, on a
$$
BG = |EG/G| \simeq |(\R/\Z)^n|.
$$
\end{example}
\begin{example}
Soit $\Gamma$ un sous-groupe discret de $\SL_n (\R )$ et supposons que l'image de $\Gamma$ dans $\mathrm{PSL}_n (\R ) $ est sans torsion. De tels groupes $\Gamma$ abondent car on peut montrer que tout groupe arithmétique de $\SL_n (\R )$ contient un tel sous-groupe d'indice fini. Considérons maintenant l'espace symétrique
$$
X=\SL_n (\R )/\SO_n .
$$
Cet espace est contractile et porte une action de $\Gamma$ induite par la multiplication à gauche dans $\mathrm{SL}_n(\R)$. Notre hypothèse sur $\Gamma$ garantit que cette action est libre. Il s'ensuit que
$$
B\Gamma = |E\Gamma /\Gamma| \simeq \Gamma \backslash \mathrm{SL}_n(\R)/\mathrm{SO}_n ,
$$
c'est-à-dire que le quotient de Borel est un espace localement symétrique.
\end{example}

\begin{example}
Les deux exemples ci-dessus peuvent être combinés : considérons le produit semi-direct $\mathrm{SL}_n(\Z) \ltimes \Z^n$ et soit
$$
X = (\mathrm{SL}_n(\mathbf{R})/\mathrm{SO}_n ) \times \mathbf{R}^n.
$$
Supposons que $\Gamma \subset \mathrm{SL}_n(\mathbf{Z})$ opère librement sur $\mathrm{SL}_n(\mathbf{R})/\mathrm{SO}_n$ et soit $L \subset \mathbf{Z}^n$ un sous-réseau de rang $n$ préservé par $\Gamma$. Alors $\Gamma \ltimes L$ opère librement sur l'espace contractile $X$ et donc
$$
B(\Gamma \rtimes L ) \simeq \Gamma \backslash (\mathrm{SL}_n(\mathbf{R})/\mathrm{SO}_n \times \mathbf{R}^n/L)
$$
est un fibré sur l'espace localement symétrique $\Gamma \backslash\mathrm{SL}_n(\mathbf{R})/\mathrm{SO}_n$ dont les fibres sont des tores compacts de dimension $n$.
\end{example}

\section{Formes différentielles simpliciales} 
Supposons maintenant que $M$ est une $G$-variété. Le quotient $X$ est alors une \emph{variété simpliciale}, c'est-à-dire un ensemble semi-simplicial dont les $m$-simplexes 
\begin{equation}
X_m = \left( EG_{\bullet} \times_G  M \right)_m = \left( EG_m \times M \right)/G
\end{equation}
sont des variétés (lisses) et dont les applications de faces et de dégénérescence sont lisses, voir par exemple \cite{Dupont}. 
La projection canonique 
\begin{equation}
EG_{\bullet} \times_G M \to BG_\bullet
\end{equation}
réalise $EG_{\bullet} \times_G M$ comme un fibré simplicial au-dessus de $BG_\bullet$ de fibre $M$.

Dans ce contexte, la cohomologie équivariante $H_G^\bullet(M)$ peut également être calculée à l'aide de formes différentielles, comme l'a montré Dupont \cite{Dupont}. Plus précisément, désignons par $\mathrm{A}^\bullet (X)$ le complexe de de Rham simplicial où, par définition, une $p$-forme simpliciale $\alpha$ sur $X$ est une collection d'applications 
$$
\alpha^{(m)} : G^{m+1} \to A^p (\Delta_m \times M)
$$
définies pour $m \geq 0$ et vérifiant les relations de compatibilité
\begin{equation}\label{eq:app_simp_1}
(\partial_i \times \mathrm{id})^* \alpha^{(m)}(g) = \alpha^{(m-1)}(\partial_i g), \qquad g \in G^{m +1},
\end{equation}
dans $\Delta_{m-1} \times M$ pour tout $i \in \{ 0 , \ldots , m\}$ et pour tout $m \geq 1$, ainsi que la relation de $G$-invariance 
\begin{equation} \label{eq:app_simp_2}
g^*\alpha^{(m)}(g \cdot g') = \alpha^{(m)}(g'), \quad g \in G, \quad g' \in G^{m+1}.
\end{equation}
En d'autres termes, une $p$-forme simpliciale $\alpha$ sur $X$ est une $p$-forme sur la \emph{réalisation grossière}  
$\| X \|$ de $X$, c'est-à-dire le quotient 
$$
\|X\| = \bigsqcup_{k \geq 0} (\Delta_k \times X_k)/\sim,
$$
où la relation d'équivalence est donnée par $(t,f^*x) \sim (f_*t, x)$ pour toutes les applications croissantes \emph{injectives}  $f:[k] \to [l]$ (une telle $f$ est généralement appelée application ``de face''). Sous des hypothèses légères sur $X$ l'application canonique $\|X\| \to |X|$ est une équivalence d'homotopie ; ces hypothèses s'appliquent pour $X=EG \times_G M$ comme le montre Segal \cite[A.1]{Segal}.

Comme dans le cas des variétés usuelles, la différentielle extérieure et le cup-produit usuel font de $\mathrm{A}^\bullet (X)$ une algèbre différentielle graduée et Dupont démontre que la cohomologie de ce complexe est isomorphe à la cohomologie de $\|X \|$. 

Enfin, chaque $\mathrm{A}^p (X)$ se décompose en la somme directe 
$$\mathrm{A}^p (X) = \bigoplus_{k+l = p} \mathrm{A}^{k,l} (X)$$
où $\mathrm{A}^{k,l} (X)$ est constitué des formes dont la restriction à $\Delta_m \times X_m$ est localement somme de formes 
$$a dt_{i_1} \wedge \ldots \wedge dt_{i_k} \wedge dx_{j_1} \wedge \ldots \wedge dx_{j_l}$$ 
où les $x_j$ sont des coordonnées locales de $X_m$ et $(t_0 , \ldots , t_m)$ sont les coordonnées barycentriques de $\Delta_m$. Les différentielles extérieures $d_\Delta$ et $d_X$ relativement aux variables respectives $t$ et $x$ décomposent $(\mathrm{A}^\bullet (X) , d)$ en un double complexe $(\mathrm{A}^{\bullet , \bullet} , d_\Delta , d_X )$.

On peut introduire un autre double complexe de formes différentielles 
$$(\mathcal{A}^{\bullet,\bullet}(X_\bullet),\delta,d_X).$$
Ici $\mathcal{A}^{k, l}(X_\bullet)=A^l(X_k)$ est l'espace les formes différentielles de degré $l$ sur $X_k$. On pose $\delta=\sum_{i}(-1)^i \partial_i^*$. Dupont montre que pour tout $l$, l'application d'intégration sur les simplexes
$$\mathcal{I} : (\mathrm{A}^{\bullet ,l} (X) , d_\Delta) \to (A^l (X_\bullet ) , \delta)$$
définit une équivalence d'homotopie entre complexes de chaines. Ces équivalences induisent des isomorphismes entre les suites spectrales calculant la cohomologie des complexes totaux \cite[Cor. 2.8]{Dupont}.

Dans le texte principal, on utilise ces équivalences pour obtenir des cocycles du groupe $G$ à partir de certaines classes de cohomologie équivariante. Dans la suite de cette annexe nous expliquons cette construction. L'argument consiste à considérer les applications au bord d'une suite spectrale, mais notre objectif est de décrire ces applications explicitement. L'argument dans les cas affine et multiplicatif étant analogue, nous nous concentrerons sur le cas elliptique.

Soit $E$ une courbe elliptique et considérons $E^n$ pour un entier $n \geq 2$. Fixons un sous-groupe d'indice fini $G$ de $\mathrm{SL}_n(\Z)$ et un $0$-cycle $G$-invariant $D$ dans $E^n$ de degré zéro et constitué de points de torsion.  Soit  $M=E^n-|D|$. Au chapitre \ref{C:1} nous expliquons comment $D$ donne naissance à une classe 
$$
E[D]  \in H^{2n-1}_G (M) = H^{2n-1}(\|(EG \times M)/G\|).
$$
Nous expliquons maintenant comment associer à $E[D]$ une classe de cohomologie dans 
$$
H^{n-1}(G,\varinjlim_U H^n(U))
$$
où $U$ décrit l'ensemble des ouverts affines dans $E^n$ obtenus en supprimant un nombre fini d'hyperplans. 

La construction est la suivante. Choisissons une $(2n-1)$-forme simpliciale fermée $\alpha \in A^\bullet(EG \times_G M)$ représentant $E[D]$. Autrement dit, $\alpha$ est une collection d'applications
$$
\alpha^{(m)}: G^{m+1} \to A^{2n-1}(\Delta_m \times M), \qquad m \geq 0,
$$
vérifiant \eqref{eq:app_simp_1}, \eqref{eq:app_simp_2} et
$$
(d_\Delta+d_M)\alpha^{(m)}=0.
$$
Considérons les formes
$$
\mathcal{I}\alpha^{(m)} : G^{m+1} \to A^{2n-1-m}(M), \qquad m \geq 0,
$$
definies par
$$
\mathcal{I}\alpha^{(m)}(g_0,\ldots,g_m) = \int_{\Delta_m} \alpha^{(m)}(g_0,\ldots,g_m).
$$
Il découle du théorème de Stokes que 
\begin{equation}
d_M \mathcal{I}\alpha^{(0)}=0
\end{equation}
et
\begin{equation}\label{eq:app_simp_4}
\delta \mathcal{I}\alpha^{(m)} + d_M \mathcal{I}\alpha^{(m+1)}= 0
\end{equation}
pour tout $m \geq 0$. Considérons maintenant les formes 
$$
\widetilde{\alpha}^{(m)} \in \widetilde{A}^{2n-1-m}:=\varinjlim A^{2n-1-m}(U)
$$
où $\widetilde{\alpha}^{(m)}$ désigne l'image de $\mathcal{I}\alpha^{(m)}$ par l'application naturelle 
$$A^{2n-1-m}(M) \to \varinjlim A^{2n-1-m}(U).$$ 
On construit maintenant une suite d'applications $G$-équivariantes 
\begin{equation} 
\beta^{(m)}:G^{m+1} \to \widetilde{A}^{2n-2-m}, \qquad 0\leq m < n-1,
\end{equation}
vérifiant $d_M \beta^{(0)}=\widetilde{\alpha}^{(0)}$ et 
\begin{equation} \label{eq:app_simp_5}
d_M \beta^{(m)} \pm \delta \beta^{(m-1)} = \widetilde{\alpha}^{(m)}
\end{equation}
pour tout $m \in \{1,\ldots,n-2\}$. Pour ce faire nous procédons comme suit :
notez que tout ouvert affine $U \subset E^n$ satisfait $H^k(U)=0$ pour tout $k>n$. Puisque $d_M \mathcal{I}\alpha^{(0)}(e)=0$, il existe $\beta^{(0)}(e) \in \widetilde{A}^{2n-2} $ tel que
$$
d_M \beta^{(0)}(e) = \widetilde{\alpha}^{(0)}(e).
$$
On pose $\beta^{(0)}(g)=g^*(\beta^{(0)}(e))$. La forme $\widetilde{\alpha}^{(0)}$ étant $G$-invariante, pour tout $g \in G$ on a $d_M \beta^{(0)}(g) = \widetilde{\alpha}^{(0)}(g)$. Il découle alors de \eqref{eq:app_simp_4} que
$$
d_M(\widetilde{\alpha}^{(1)}+\delta\beta^{(0)}) = 0,
$$
et en procédant comme ci-dessus on peut trouver une application $G$-invariante
$$
\beta^{(1)}: G^2 \to \widetilde{A}^{2n-3}
$$
telle que 
$$
d_M \beta^{(1)} = \widetilde{\alpha}^{(1)}+\delta \beta^{(0)}.
$$
En itérant cet argument, on obtient des applications $G$-équivariantes $\beta^{(0)},\ldots, \beta^{(n-2)}$ vérifiant  \eqref{eq:app_simp_5}. Maintenant, il découle de \eqref{eq:app_simp_4} que
$$
d_M(\widetilde{\alpha}^{(n-1)}\pm \delta\beta^{(n-2)})=0
$$
et, en passant aux classes de cohomologie, on obtient une application
$$
c: G^{n} \to \varinjlim H^n(U), \qquad c(g):=[\widetilde{\alpha}^{(n-1)}(g)\pm \delta \beta^{(n-2)}(g)].
$$
Comme $\widetilde{\alpha}$ et $\beta$
sont $G$-équivariantes, l'application $c$ est aussi $G$-équivariante. Noter aussi que $c$ est un $(n-1)$-cocycle :
\begin{equation}
\begin{split}
\delta c(g) &= [\delta(\widetilde{\alpha}^{(n-1)}(g)\pm \delta \beta^{(n-2)}(g))] \\
&= [\delta \widetilde{\alpha}^{(n-1)}(g)] \\
&= [-d_M \widetilde{\alpha}^{(n)}(g)] \\
&= 0.
\end{split}
\end{equation}
Le cocycle $c$ dépend des choix faits dans la construction des applications $\beta^{(k)}$, mais un argument standard montre que sa classe de cohomologie
$$
[c] \in H^{n-1}(G,\varinjlim H^n(U))
$$
est indépendante de ces choix. Elle est aussi indépendante de la forme $\alpha$ sélectionnée pour représenter la classe de cohomologie originale dans $H^{2n-1}_G(M)$. 

\begin{remark}
En pratique on trouve souvent des $\widetilde{\alpha}^{(m)}$ telles que pour tout $k$ dans $\{0,\ldots ,n-2\}$ la forme $\widetilde{\alpha}^{(k)}$ soit identiquement nulle. Dans ce cas on peut prendre $\beta^{(0)},\ldots,\beta^{(n-2)}$ identiquement nulles et l'application 
$$
g \mapsto [\widetilde{\alpha}^{(n-1)}(g)]: G^n \to \varinjlim H^n(U).
$$
est un cocycle qui représente la classe de $[c] $.
\end{remark}

\setcounter{equation}{0}

\chapter{Classe d'Eisenstein affine et théorie de l'obstruction} \label{A:B}

\resettheoremcounters

Soit $F$ un corps, par exemple $\C$, et soit  $\mathbf{T}_n = \mathbf{T}_n (F)$ l'immeuble de Tits de $V=F^n$ avec $n \geq 2$. Le groupe $G = \mathrm{GL}_n(F)$ opère naturellement sur $\mathbf{T}_n$ et on considère le fibré 
\begin{equation}
\begin{tikzcd}
EG \times_G \mathbf{T}_n \arrow["\pi" ',d] \\ BG
\end{tikzcd}
\end{equation}
de fibre $\mathbf{T}_n$. L'immeuble de Tits $\mathbf{T}_n$ étant $(n-3)$-connexe, la première obstruction à l'existence d'une section de $\pi$ est une classe
\begin{equation} \label{E:A1}
\omega_{n-2} \in H^{n-1}(BG,\pi_{n-2}(\mathbf{T}_n )),
\end{equation}
voir par exemple \cite[Obstruction Theory, p. 415]{Hatcher}. Dans \eqref{E:A1} on voit le groupe d'homotopie $\pi_{n-2}(\mathbf{T}_n)$ de la fibre de $\pi$ comme un système local. L'immeuble de Tits $\mathbf{T}_n$ étant $(n-3)$-connexe, on a 
$$\pi_{n-2}(\mathbf{T}_n) = \widetilde{H}_{n-2} (\mathbf{T}_n ) = \mathrm{St} (F^n).$$
On peut donc voir $\omega_{n-2}$ comme un élément de $H^{n-1}(G,\mathrm{St} (F^n))$. 
\begin{proposition}
La classe d'obstruction $\omega_{n-2}$ coincide avec la classe de cohomologie associée au symbole universel de Ash--Rudolph représenté par le cocycle
$$G^n \to \mathrm{St} (F^n ); \quad (g_0 , \ldots , g_{n-1} ) \mapsto [g_0^{-1} e_1 , \ldots , g_{n-1}^{-1} e_1 ] .$$
\end{proposition}
\begin{proof} On commence par construire une section explicite $s$ de $\pi$ sur le squelette de dimension $n-2$. On étudie ensuite l'obstruction à étendre cette section au $(n-1)$-squelette. 

On définit la section $s$ en commençant par le $0$-squelette $BG_0$ puis on l'étend par récurrence de $BG_{\leq k}$ à $BG_{\leq k+1}$~:
\begin{equation}
\cdots
\begin{tikzcd}
(EG \times_G \mathbf{T}_n )_1 \arrow[d,"\pi_1"' ] \arrow[r, shift left] \arrow[r, shift right] & (EG \times_G \mathbf{T}_n )_0 \arrow[d,"\pi_0"' ] \\
BG_1 \arrow[u, dashed, bend right, "s_1"'] \arrow[r, shift left] \arrow[r, shift right] & BG_0 \arrow[u, dashed, bend right, "s_0"'].
\end{tikzcd}
\end{equation} 
Les $k$-simplexes de $EG \times_G \mathbf{T}_n$ sont les couples $(\mathbf{g} , D_\bullet)$, où $\mathbf{g}$ est un $(k+1)$-uplet $(g_0 , \ldots , g_k) \in G^{k+1}$ et $D_\bullet \in (\mathbf{T}_n )_k$ est un $(k+1)$-drapeau de sous-espaces propres non nuls de $V$, modulo équivalence 
\begin{equation}
((g_0,\ldots,g_k),D_\bullet) \sim ((g_0g^{-1},\ldots,g_kg^{-1}) , gD_\bullet), \quad g \in G.
\end{equation}

Tentons maintenant de définir une section $s$ de $\pi$ en commençant par définir $s$ sur le $0$-squelette $BG_0$ puis en l'étendant par récurrence de $BG_{\leq k}$ à $BG_{\leq k+1}$:
\begin{equation}
\cdots
\begin{tikzcd}
(EG \times_G \mathbf{T}_n )_1 \arrow[d,"\pi_1"' ] \arrow[r, shift left] \arrow[r, shift right] & (EG \times_G \mathbf{T}_n )_0 \arrow[d,"\pi_0"' ] \\
BG_1 \arrow[u, dashed, bend right, "s_1"'] \arrow[r, shift left] \arrow[r, shift right] & BG_0 \arrow[u, dashed, bend right, "s_0"'].
\end{tikzcd}
\end{equation} 
Les $k$-simplexes dans $EG \times_G \mathbf{T}_n $ sont les couples $(\mathbf{g},F_\bullet)$ (où $\mathbf{g} \in G^{k+1}$ et $F_\bullet \in (\mathbf{T}_n )_k$) modulo l'équivalence
\begin{equation}
((g_0,\ldots,g_k),F_\bullet) \simeq ((g_0g^{-1},\ldots,g_kg^{-1}),gF_\bullet), \quad g \in G.
\end{equation}
On note dans la suite $[\mathbf{g},F_\bullet]$ la classe d'équivalence de $(\mathbf{g},F_\bullet)$. Pour définir la section $s_0$ de $\pi_0$ (l'application induite par $\pi$ sur le $0$-squelette) on choisit une droite $L=\langle v \rangle \subset F^n$ et on pose
\begin{equation}
s_0 ( g_0 ) = [g_0 , g_0^{-1} L].
\end{equation}
On pose ensuite
\begin{equation}
s_1(g_0 , g_1) = [(g_0 ,g_1), \Delta(g_0^{-1} L , g_1^{-1} L)]
\end{equation}
où 
\begin{equation}
\Delta(g_0^{-1} L , g_1^{-1}L) = \begin{tikzcd} \stackrel{\langle g_0^{-1} v \rangle}{\bullet} \arrow[r, dash] & \stackrel{\langle g_0^{-1} v , g_1^{-1}v \rangle}{\bullet} & \arrow[l, dash] \stackrel{\langle g_1^{-1} v \rangle}{\bullet}, \end{tikzcd}
\end{equation}
et les deux segments correspondent aux drapeaux\footnote{Ce chemin allant de $g_0^{-1} L$ à $g_1^{-1}L$ dans $\mathbf{T}_n$ n'est pas unique mais tout autre chemin est plus long ou passe par un $0$-simplex correspondant  à un sous-espace de $F^n$ de dimension strictement supérieure à $2$.}  
$$\langle g_0^{-1} v \rangle \subseteq \langle g_0^{-1} v , g_1^{-1}v \rangle \quad \mbox{et} \quad \langle g_1^{-1}v \rangle \subseteq \langle g_0^{-1} v , g_1^{-1}v \rangle.$$
En général, on pose 
\begin{equation}
s_k(g_0 ,\ldots , g_k) = [(g_0 , \ldots , g_k ), \Delta( g_0^{-1} L ,  \ldots, g_k^{-1} L)],
\end{equation}
où $\Delta( g_0^{-1} L ,  \ldots, g_k^{-1} L)$ est le sous-complexe de $\mathbf{T}_n$ isomorphe à la première subdivision barycentrique d'un $k$-simplexe de $( \mathbf{T}_n )_k$ de sommets $g_0^{-1} L, \ldots ,g_k^{-1}L$ et de barycentre le sous-espace $\langle g_0^{-1} v , \ldots , g_k^{-1} v \rangle$ de dimension $k+1$ dans $V$; de sorte que les $k$-simplexes de ce complexe sont associés aux drapeaux  
$$\langle g_{\sigma (0)}^{-1} v \rangle \subset \langle g_{\sigma (0)}^{-1} v, g_{\sigma (1)}^{-1} v \rangle \subset \cdots \subset \langle g_{\sigma (0)}^{-1} v, g_{\sigma (1)}^{-1} v ,\ldots, g_{\sigma (k)}^{-1}v \rangle \quad (\sigma \in \mathfrak{S}_{k+1} ) .$$
Cette construction inductive est possible tant que $k \leq n-2$, c'est-à-dire tant que $k+1$ est strictement inférieur à la dimension de $V$. On définit ainsi section 
\begin{equation}
s_{\leq n-2}:BG_{\leq n-2} \to (EG \times_G \mathbf{T}_n )_{\leq n-2}.
\end{equation}
L'obstruction à étendre $s$ au $(n-1)$-squelette $BG_{\leq n-1}$ est alors représentée par l'application $BG_{n-1} \to \pi_{n-2}(\mathbf{T}_n)$ définie par 
\begin{equation}
(g_0 , \ldots , g_{n-1)} \mapsto \sum_{j=0}^{n-1} (-1)^j [s_{n-2}(g_0 , \ldots , \widehat{g_j} , \ldots , g_{n-1})].
\end{equation}
La théorie de l'obstruction implique que cette application définit un cocycle 
$$c_{n-2} \in C^{n-1}(BG,\pi_{n-2}(\mathbf{T}_n ))$$ 
qui représente $\omega_{n-2}$; de sorte que $[c_{n-2}]$ est indépendant du choix de $L$. Vue comme classe dans $H^{n-1}(G,\mathrm{St}_n(F))$ la classe d'obstruction $\omega_{n-2}$ est finalement donnée par
\begin{equation}
\omega_1(g_0,\ldots,g_{n-1}) = [g_0^{-1} v ,g_1^{-1}v,\ldots ,g_{n-1}^{-1}v],
\end{equation}
où $[\cdot]$ désigne le symbole modulaire universel de Ash--Rudolph, cf. \S \ref{S:ARuniv}.
\end{proof}

\end{appendix}

\bibliographystyle {plain}
\bibliography{Eisenstein}

\end{otherlanguage}

\end{document}